\documentclass[a4paper,english,twoside,11pt]{book}

\usepackage[english]{babel}		 	
\usepackage[utf8]{inputenc}	 		
\usepackage[T1]{fontenc}			
\usepackage{lmodern}  				
\usepackage{microtype}
\usepackage{color}

\usepackage{mathrsfs} 				
\usepackage{amsmath}
\usepackage{amsthm}
\usepackage{amsfonts}
\usepackage{amssymb}
\usepackage{latexsym}

\usepackage{indentfirst}
\usepackage{enumerate}
\usepackage{enumitem}
\usepackage{fancyhdr} 				

\renewcommand{\sectionmark}[1]{\markboth{\textsf{#1}}{}}

\usepackage{lipsum}
\usepackage{graphicx}					
\usepackage[all]{xy}					
\input xypic
\usepackage{subfig} 					
\usepackage{tikz}						
\usetikzlibrary{matrix,arrows,decorations.pathmorphing,backgrounds,positioning,fit,petri}
\usepackage{tikz-cd}

\DeclareMathOperator{\im}{im}			
\DeclareMathOperator{\coker}{coker}		
\DeclareMathOperator{\rank}{rank}		
\DeclareMathOperator{\ab}{ab}			
\DeclareMathOperator{\I}{I}				
\DeclareMathOperator{\Log}{Log}			
\DeclareMathOperator{\Rad}{Rad}			
\DeclareMathOperator{\R}{R}				
\DeclareMathOperator{\gr}{gr}			
\DeclareMathOperator{\Z}{Z}				
\DeclareMathOperator{\GL}{GL}			
\DeclareMathOperator{\SL}{SL}			
\DeclareMathOperator{\Sp}{Sp}			
\DeclareMathOperator{\UU}{U}				
\DeclareMathOperator{\incl}{incl}		
\DeclareMathOperator{\id}{id}			
\DeclareMathOperator{\Id}{Id}			
\DeclareMathOperator{\Gal}{Gal}			
\DeclareMathOperator{\Aut}{Aut}			
\DeclareMathOperator{\Irr}{Irr}			
\DeclareMathOperator{\Rep}{Rep}			
\DeclareMathOperator{\Ind}{Ind}			
\DeclareMathOperator{\Ad}{Ad}			
\DeclareMathOperator{\ch}{char}			
\DeclareMathOperator{\End}{End}			
\DeclareMathOperator{\Hom}{Hom}			
\DeclareMathOperator{\Mor}{Mor}
\DeclareMathOperator{\A}{A}				
\DeclareMathOperator{\K}{K}				
\DeclareMathOperator{\M}{M}				
\DeclareMathOperator{\rk}{rk}			
\DeclareMathOperator{\Spec}{Spec}		
\DeclareMathOperator{\Fun}{Fun}			
\DeclareMathOperator{\rad}{Rad}			
\DeclareMathOperator{\kk}{k}				
\DeclareMathOperator{\Wh}{Wh}
\DeclareMathOperator{\HH}{H}
\DeclareMathOperator{\SK}{SK}
\DeclareMathOperator{\B}{B}			
\DeclareMathOperator{\St}{St}
\DeclareMathOperator{\trf}{trf}		
\DeclareMathOperator{\ord}{ord}		
\DeclareMathOperator{\Alt}{Alt}			
\DeclareMathOperator{\Frob}{Frob}		
\DeclareMathOperator{\Der}{Der}			
\DeclareMathOperator{\Lie}{Lie}			
\DeclareMathOperator{\Lin}{Lin}			
			
\DeclareMathOperator{\f}{F}			
\DeclareMathOperator{\tr}{tr}		
\DeclareMathOperator{\Real}{Re}		
\DeclareMathOperator{\PGL}{PGL}		

\newcommand{\NN}{\mathbb{N}}			
\newcommand{\QQ}{\mathbb{Q}}			
\newcommand{\ZZ}{\mathbb{Z}}			
\newcommand{\CC}{\mathbb{C}}			
\newcommand{\Fp}{\mathbb{F}_{p}}    	
\newcommand{\Fq}{\mathbb{F}_{q}}		
\newcommand{\F}{\mathbb{F}}				
\newcommand{\FF}{\mathbb{F}}			
\newcommand{\OO}{\mathcal{O}}			
\newcommand{\mm}{\mathfrak{m}}			
\newcommand{\pp}{\mathfrak{p}}			
\newcommand{\LL}{\mathcal{L}}			
\newcommand{\RR}{\mathbb{R}}			
\newcommand{\Fl}{\mathbb{F}_{l}}	

\newcommand{\note}[1]{}							
\newcommand{\marginnote}[1]{}					
\newcommand{\qnote}[1]{}						
\newcommand{\qmarginnote}[1]{}					
\newcommand{\enote}[1]{}						

\theoremstyle{plain}
	\newtheorem{theorem}{Theorem}[chapter]
	\newtheorem*{theorem*}{Theorem}
 	\newtheorem{proposition}{Proposition}[section]
	\newtheorem{lemma}{Lemma}[section]
	\newtheorem*{lemma*}{Lemma}
	\newtheorem{corollary}{Corollary}[section]
	\newtheorem{question}{{\rm Question}}
	\newtheorem{teorema}{Teorema}[chapter]
	\newtheorem*{teorema*}{Teorema}

\theoremstyle{definition}
	\newtheorem*{definition*}{Definition}

\theoremstyle{remark}
	\newtheorem*{remark}{Remark}
	\newtheorem*{example}{Example}
	\newtheorem{claim}{Claim}
	\newtheorem{conjecture}{Conjecture}
	\newtheorem*{conjecture*}{Conjecture}

\begin{document}

\pagestyle{empty} 

\pagestyle{empty}
\begin{titlepage}

\newcommand{\HRule}{\rule{\linewidth}{0.5mm}} 

\center 
 

\textsc{\LARGE Universidad Autónoma de Madrid}\\[1.5cm] 
\textsc{\Large Tesis Doctoral}\\[0.5cm] 
\textsc{\large }\\[0.5cm] 


\HRule \\[0.4cm]
{ \huge \bfseries Representation Growth}\\[0.4cm] 
\HRule \\[1.5cm]
 

\begin{minipage}{0.4\textwidth}
\begin{flushleft} \large
\emph{Autor:}\\
Javier García Rodríguez 
\end{flushleft}
\end{minipage}
~
\begin{minipage}{0.4\textwidth}
\begin{flushright} \large
\emph{Director:} \\
Andrei Jaikin Zapirain 
\end{flushright}
\end{minipage}\\[2cm]



{\large Octubre 2016}\\[2cm] 


 

\vfill 

\end{titlepage}
\clearpage{\pagestyle{empty}\cleardoublepage}


\newpage
\pagestyle{empty}
\vspace*{7cm}
\hfill\begin{minipage}[t]{0.5\textwidth}
\raggedleft
{\Large{\it a Laura.
}}
\end{minipage}

\frontmatter
\pagestyle{plain}

\chapter*{Abstract}
\addcontentsline{toc}{chapter}{Abstract}
The main results in this thesis deal with the representation theory of certain classes of groups.
More precisely, if $r_n(\Gamma)$ denotes the number of non-isomorphic $n$-dimensional complex representations of a group $\Gamma$, we study the numbers $r_n(\Gamma)$ 
and the relation  of this arithmetic information with structural properties of $\Gamma$.

In chapter $1$ we present the required preliminary theory.
In chapter $2$ we introduce the Congruence Subgroup Problem for an algebraic group $G$ defined over a global field $k$.

In chapter $3$ we consider $\Gamma=G(\OO_S)$ an arithmetic subgroup of a semisimple algebraic $k$-group for some global field $k$ with ring of $S$-integers $\OO_S$.
If the Lie algebra of $G$ is perfect, Lubotzky and Martin showed in \cite{LuMa} that if $\Gamma$ has the weak Congruence Subgroup Property then $\Gamma$ has Polynomial Representation Growth, that is, $r_n(\Gamma)\leq p(n)$ for some polynomial $p$.
By using a different approach, we show that the same holds for any semisimple algebraic group $G$ including those with a non-perfect Lie algebra.

In chapter $4$ we apply our results on representation growth of groups of the form $\Gamma=G(\OO_S)$ to show that if $\Gamma$ has the weak Congruence Subgroup Property then $s_n(\Gamma)\leq n^{D\log n}$ for some constant $D$,
where $s_n(\Gamma)$ denotes the number of subgroups of $\Gamma$ of index at most $n$.
As before, this extends similar results of Lubotzky \cite{Lu}, Nikolov, Abert, Szegedy \cite{NiAlSze} and Golsefidy \cite{Gol} for almost simple simply connected groups with perfect Lie algebra to any almost simple simply connected algebraic $k$-group $G$.

In chapter $5$ we consider $\Gamma=1+J$, where $J$ is a finite nilpotent associative algebra, this is called an algebra group.
The Fake Degree Conjecture says that for algebra groups the numbers $r_n(\Gamma)$ may be obtained by 
looking at the square root of the sizes of the orbits arising from Kirillov's Orbit Method.
In \cite{Ja1} Jaikin gave a $2$-group that served as a counterexample to this conjecture.
We provide counterexamples for any prime $p$ by looking at groups of the form $\Gamma=1+\I_{\F_q}$, where $\I_{\F_q}$
is the augmentation ideal of the group algebra $\F_q[\pi]$ for some $p$-group $\pi$.
Moreover, we show that for such groups $r_1(\Gamma)=q^{\kk(\pi)-1}|\B_0(\pi)|$, where $\B_0(\pi)$ is the Bogomolov multiplier of $\pi$.

Finally in chapter $6$, we consider $\Gamma=\prod_{i\in I} S_i$, where the $S_i$ are nonabelian finite simple group.
We give a characterization of the groups of this form having Polynomial Representation Growth.
We also show that within this class one can obtain any 
rate of representation growth, i.e., for any $\alpha>0$ there exists $\Gamma=\prod_{i\in I}S_i$ where the $S_i$ are finite simple groups of Lie type such that $r_n(\Gamma)\sim n^\alpha$.
Moreover, we may take all $S_i$ with a fixed Lie type. This complements results of Kassabov and Nikolov in \cite{KaNi}.

\clearpage{\pagestyle{empty}\cleardoublepage}

\chapter*{Agradecimientos}
\addcontentsline{toc}{chapter}{Agradecimientos}
Escribir 100 páginas en 4 años puede no parecer gran cosa. Confieso además que no lo habría podido hacer sin la ayuda de los que me han acompañado estos últimos años. Me gustaría darles las gracias por todas las líneas con las que, de una manera u otra, han contribuido.

Gracias en primer lugar a Andrei por estar siempre dispuesto y de buen humor para trabajar, por la barbaridad de cosas que he aprendido y por el placer y la motivación que da trabajar con alguien con tanto talento.
Gracias también a Joan por ejercer de hermano mayor. Por los ratos hablando de y pensando en grupos y sus cosas y por todos los buenos ratos que hemos pasado fuera de los grupos y sus cosas.
Gracias a Urban porque los meses de trabajo y estudio juntos han sido los más satisfactorios. 
Gracias a Benjamin por acogerme en Düsseldorf y a toda la gente de allí que me hizo sentir como uno más durante mi estancia.

I would like to thank A. S. Golsefidy and B. Martin whose work as reviewers significantly improved this thesis. I would also like to thank Y. Antolin, B. Klopsch and B. Martin for being part of the Ph. D. Thesis committee and their helpful comments.

La mejor parte de hacer el doctorado en la UAM ha sido sin duda el grupo de gente con el que he compartido estos años y que han conseguido que madrugara con ganas de irme a la otra punta de Madrid.

Gracias a toda mi familia por apoyarme en todas las etapas hasta llegar hasta aquí.

Y gracias a ti Laura. Por acompañarme y ayudarme en todo desde siempre.
\clearpage{\pagestyle{empty}\cleardoublepage}


\tableofcontents 

\mainmatter


\pagestyle{fancy} 

\setcounter{page}{1}
\pagenumbering{arabic}


\chapter*{Introducción}
\addcontentsline{toc}{chapter}{Introducción}

Sea $\Gamma$ un grupo.
Una representación de $\Gamma$ es un homomorfismo de grupos $
\rho:\Gamma \to \GL(V)$
para algún espacio vectorial complejo finito dimensional $V$.
La Teoría de Representaciones estudia la relación entre propiedades estructurales de $\Gamma$ y sus representaciones.
El Teorema de Burnside fue la primera gran contribución de la Teoría de Representaciones y desde entonces se ha convertido en un área de investigación en sí misma. Este trabajo se centra en la teoría de representaciones de ciertas clases de grupos.
El bloque principal está dedicado al estudio de la teoría de representaciones (más concretamente, del crecimiento de representaciones) de grupos aritméticos, aunque también abordamos cuestiones relacionadas con la teoría de representaciones de grupos asociados a álgebras y el crecimiento de representaciones de (productos directos de) grupos finitos simples.

Sea $\Gamma$ un grupo y denotemos por $r_n(\Gamma)$ el número de clases de isomorfía de representaciones irreducibles $n$-dimensionales de $\Gamma$.
Si $\Gamma$ es finito sabemos que $r_n(\Gamma)=0$ para $n\geq |\Gamma|^{1/2}$.
El estudio de grupos finitos a través de las dimensiones de sus representaciones irreducibles y clases de conjugación es una gran área de investigación, e.g., ver \cite{Isaacs} o \cite{Hu}  y referencias allí indicadas.
Por otro lado para $\Gamma$ infinito $r_n(\Gamma)$ puede ser infinito.
El comportamiento asintótico de la sucesión $r_n(\Gamma)$ es el crecimiento de representaciones del grupo $\Gamma$. 
Supongamos que $r_n(\Gamma)<\infty$ para todo $n$.
Entonces podemos asociar a la sucesión $\{r_n(\Gamma)\}$ una zeta función de Dirichlet
\[
\zeta_\Gamma(s)=\sum_{n=1}^\infty r_n(\Gamma)n^{-s} \quad (s\in\mathbb{C}).
\]
La abscisa de convergencia $\alpha(\Gamma)$ de $\zeta_\Gamma(s)$ es el ínfimo de todos los $\alpha\in\mathbb{R}_{\geq 0}$ tales que la serie $\zeta_\Gamma(s)$ converge en el semiplano derecho $\{s\in\mathbb{C}\; |\; \Real(s)>\alpha\}$.
La abscisa $\alpha(\Gamma)$ es finita si y solo si $\Gamma$ tiene Crecimiento Polinomial de Representaciones (CPR), es decir, si $r_n(\Gamma)\leq p(n)$ para algún polinomio $p$.

\subsubsection*{Crecimiento de Representaciones de Grupos Aritméticos}
Sea $k$ un cuerpo global, $S$ un conjunto finito de valuaciones de $k$ que contiene a todas las arquimedianas y $\OO_S$ el anillo de  $S$-enteros correspondiente.
Sea $G$ un $k$-grupo conexo simplemente conexo y semisimple y supongamos $G\subseteq\GL_N$.
Un grupo $\Gamma$ es un grupo aritmético si es conmensurable con 
\[
G(\OO_S):=G(k)\cap\GL_N(\OO_S).
\]
Nótese que en términos del crecimiento de representaciones de un grupo aritmético $\Gamma$ podemos asumir $\Gamma=G(\OO_S)$.
En \cite{LuMa}, Lubotzky y Martin iniciaron el estudio del crecimiento de representaciones de grupos aritméticos.
Establecieron una conexión entre tener CPR y el famoso Problema de los Subgrupos de Congruencia.
Decimos que $G$ tiene la Propiedad de los Subgrupos de Congruencia con respecto a $S$ si todo subgrupo de índice finito $H\leq G(\OO_S)$ contiene algún subgrupo de congruencia principal $G(\mathfrak{p})$ para algún ideal $\mathfrak{p}\subset\OO_S$.
En el capítulo \ref{sec:CSP} presentamos el Problema de los Subgrupos de Congruencia y explicamos la  condición (ligeramente más débil) de tener la Propiedad débil de los Subgrupos de Congruencia (PdSC).

\begin{teorema*}[Lubotzky y Martin]
Supongamos que $G$ tiene la Propiedad débil de los Subgrupos de Congruencia con respecto a $S$.
Si $\ch k=2$, asumimos también que $G$ no contiene factores de tipo $A_1$ o $C_m$ para cualquier $m$. 
Entonces $\Gamma$ tiene Crecimiento Polinomial de Representaciones.
Además, si $\ch k=0$ el recíproco también es cierto.
\end{teorema*}

Si $G$ tiene la Propiedad débil de los Subgrupos de Congruencia con respecto a $S$, la compleción profinita $\widehat{G(\OO_S)}$ puede estudiarse a través de la compleción de congruencia $\overline{G(\OO_S)}$, la cuál,  por el Teorema Fuerte de Aproximación, es isomorfa a $G(\widehat{\OO}_S)$.
Es decir, el crecimiento de representaciones de $\Gamma$ viene dado por el de 
\[
G(\widehat{\OO}_S)=\prod_{v\notin S}G(\OO_v).
\]

Cada factor local $G(\OO_v)$ es un grupo compacto $k_v$-analítico. 
Para obtener un crecimiento polinomial de representaciones la clave es encontrar una cota uniforme de tal manera que $r_n(G(\OO_v))\leq cn^c$ para todo $v\notin S$.
Lubotzky y Martin obtuvieron esta cota uniforme cuando $L_G$, el álgebra de Lie del grupo algebraico $G$, es perfecta, equivalentemente, si $\ch k=2$ el grupo $G$ no contiene factores de tipo $A_1$ o $C_m$.
Cuando el álgebra de Lie no es perfecta Jaikin obtuvo en \cite{JaiPre} crecimiento polinomial de representaciones para cada factor local, es decir, para cada $v$ existe una constante $c_v$ tal que $r_n(G(\OO_v))\leq c_vn^{c_v}$.
Siguiendo las ideas de Jaikin para el caso local abordamos el caso global al obtener una cota polinomial uniforme para todo $v$, lo que finalmente lleva al siguiente teorema (cf. Theorem \ref{th:CSP_implies_PRG}).

\begin{teorema}\label{main}
Sean $k,\,S,\,\Gamma$ y $G$ como anteriormente.
Si $G$ tiene la Propiedad débil de los Subgrupos de Congruencia con respecto a $S$, entonces $\Gamma$ tiene Crecimiento Polinomial de Representaciones.
\end{teorema}

El trabajo de Lubotzky y Martin en \cite{LuMa} inició el estudio del crecimiento de representaciones de grupos aritméticos (con la Propiedad débil de Subgrupos de Congruencia).
En \cite{LarLu} Larsen y Lubotzky estudiaron la abscisa de convergencia $\alpha(\Gamma)$ y Avni demostró en \cite{Av} que ésta es un número racional para $\ch k=0$.
Una gran parte del trabajo reciente en el área ha sido realizado por Avni, Klopsch, Onn y Voll en una seria de artículos (véase \cite{AKOV1}, \cite{AKOV2}, \cite{AKOV3} y \cite{AKOV_base_change}). Recomendamos \cite{Klopsch} para una visión general sobre el área.

\subsubsection{Crecimiento de Subgrupos de Grupos Aritméticos}

Sea $\Gamma$ un grupo finitamente generado.
Escribiremos $a_n(\Gamma)$ para denotar el número de subgrupos de $\Gamma$ de índice $n$ y definimos $s_n(\Gamma):=\sum_{i=1}^n a_n(\Gamma)$.
En analogía con el Crecimiento de Representaciones\footnote{El estudio de propiedades estructurales de un grupo $\Gamma$ a través de su crecimiento de subgrupos en realidad precede al enfoque desde el punto de vista del crecimiento de representaciones, de hecho, este último está motivado de alguna manera por la rica teoría de Crecimiento de Subgrupos.} el Crecimiento de Subgrupos de $\Gamma$ viene dado por el comportamiento asintótico de $a_n(\Gamma)$ (equivalentemente $s_n(\Gamma)$).
El estudio de propiedades estructurales de $\Gamma$ a través de su crecimiento de subgrupos ha dado pie a una teoría profunda con muchos resultados y aplicaciones, véase \cite{Subgroup_Growth} y referencias allí.

Al igual que antes, nos centraremos en el crecimiento de subgrupos de grupos aritméticos, más concretamente, nos fijaremos en el crecimiento de subgrupos de un grupo $\Gamma=G(\OO_S)$, donde $G$ es un grupo algebraico sobre $k$ que es conexo, simplemente conexo y casi simple y $\OO_S$ es el anillo de $S$-enteros del cuerpo global $k$.

Si el grupo $G$ tiene la Propiedad de los Subgrupos de Congruencia con respecto a $S$, entonces todo subgrupo de índice finito es un subgrupo de congruencia y por tanto $s_n(\Gamma)=\gamma_n(\Gamma)$, donde $\gamma_n(\Gamma)$ denota el número de subgrupos de congruencia de índice como mucho $n$.
El estudio del crecimiento de subgrupos de congruencia de grupos aritméticos fue iniciado por Lubotzky en \cite{Lu}.
Para $\ch k=0$ Lubotzky encuentra cotas inferiores y superiores para $\gamma_n(\Gamma)$.
\begin{teorema*}[{\cite[Teorema A y F]{Lu}}]
Sea $\Gamma=G(\OO_S)$ como anteriormente.
Si $\ch k=0$, entonces existen constantes $c_1,c_2$ tales que
\[
n^{\frac{c_1\log^2 n}{\log\log n}}\leq\gamma_n(\Gamma)\leq n^{\frac{c_2\log^2 n}{\log\log n}}.
\]
Si $\ch k>0$ y si $\ch k=2$ asumimos que $G$ no es de tipo $A_1$ o $C_m$, entonces existen constantes $c_3,c_4$ tales que
\[
n^{c_3 \log n}\leq \gamma_n(\Gamma)\leq n^{c_4\log^2n}.
\]
\end{teorema*}
Abért, Nikolov y Szegedy mejoraron la cota superior en \cite{NiAlSze}.
\begin{teorema*}[{\cite[Teorema 1]{NiAlSze}}]
Sea $\Gamma=G(\OO_S)$ como anteriormente y supongamos que $\ch k>0$.
Supongamos que $G$ es split sobre $k$ y que si $\ch k=2$ entonces  $G$ no es de tipo $A_1$ o $C_m$.
Entonces existe una constante $D$ que solo depende del rango de Lie de $G$ tal que
\[
\gamma_n(\Gamma)\leq n^{D\log n}.
\] 
\end{teorema*} 
En \cite{Gol}, entre otros resultados, Golsefidy obtiene una cota similar sin la necesidad de que $G$ sea split sobre $k$.
\begin{teorema*}
Sea $\Gamma=G(\OO_S)$ como anteriormente y supongamos que $\ch k\geq 5$.
Entonces existen constantes $C$ y $D$ tales que
\[
n^{C\log n}\leq \gamma_n(\Gamma)\leq n^{D\log n}.
\] 
\end{teorema*}

En los resultados previos es necesario que el álgebra de Lie (graduada) asociada a cada factor local $G(\OO_v)$ sea perfecta.
De ahí la necesidad de excluir los casos en los que $G$ es de tipo $A_1$ o $C_m$.
Abordamos esta dificultad aplicando el Teorema \ref{main} y cierta relación entre el crecimiento de subgrupos y el crecimiento de representaciones.
\begin{teorema}[cf. Theorem \ref{th:log_subgroup_growth} y \ref{th:PSG}]
Sea $\Gamma=G(\OO_S)$ como anteriormente.
Entonces existe una constante $D>0$ tal que
\[
\gamma_n(\Gamma)\leq n^{D\log n}.
\]
Supongamos que $G$ tiene la Propiedad débil de los Subgrupos de Congruencia con respecto a $S$.
Entonces existe $D'>0$ tal que
\[
s_n(\Gamma)\leq n^{D'\log n}.
\]
\end{teorema}

\subsubsection{La Conjetura de los Pseudogrados}

Sea $J$ un álgebra associativa finito dimensional sobre un cuerpo finito $\F_q$.
Asumamos además que $J$ es nilpotente, es decir, $J^n=0$ para algún $n$.
Podemos entonces dar al conjunto formal $1+J=\{1+j\;:\; j\in J\}$ una estructura de grupo definiendo $(1+j)(1+h):=1+j+h+jh$ y $(1+j)^{-1}:=\sum_{k=0}^{n-1}(-1)^kj^k$.
Llamamos a los grupos de la forma $1+J$ grupos asociados a álgebras; éstos fueron introducidos por Isaacs en \cite{Isaacs} como una generalización de los grupos unitriangulares $U_n(\Fq)=1+\mathfrak{u}_n(\F_q)$, donde $\mathfrak{u}_n(\F_q)$ es el álgebra de matrices cuadradas de tamañano $n$ estrictamente triangulares superiores.

El grupo $1+J$ actúa de manera natural en el álgebra $J$ y esto induce una acción de $1+J$ en $\Irr(J)$, los caracteres irreducibles del grupo aditivo de $J$.
En este contexto es posible adaptar el método de las órbitas de Kirillov y tratar de construir una biyección
\[
\begin{array}{ccc}
(1+J)\textrm{-órbitas en }\Irr(J)&\to&\Irr(1+J).
\end{array}
\]
Existen serias dificultades para llevar a cabo esta construcción y una tal biyección solo se obtiene cuando $J^p=0$, donde $p=\ch \F_q$.
Sin embargo la adaptación del método de las órbitas de Kirillov sugiere una función
\[
\begin{array}{ccc}
(1+J)\textrm{-órbitas en }\Irr(J)&\to&\{\chi(1)\;:\;\chi\in\Irr(1+J)\}\\
\Omega&\mapsto&|\Omega|^{1/2}.\\
\end{array}
\]
Isaacs conjeturó que esta función está bien definida y que los $|\Omega|^{1/2}$ (los pseudogrados), contando con multiplicidad, son los grados de los caracteres irreducibles de $\Irr(1+J)$ para cualquier grupo asociado a un álgebra $1+J$.
Esta conjetura es conocida como la Conjetura de los Pseudogrados.
En \cite{Ja1} Jaikin refutó la Conjetura de los Pseudogrados construyendo un grupo asociado a una $\F_2$-álgebra que servía de contraejemplo.
No obstante, el comportamiento especial del primo $p=2$ en correspondencias de caracteres y varios cálculos experimentales sugerían que la conjetura podría ser cierta para grupos asociados a $\F_q$-álgebras donde $\ch \F_q>2$.
Consideremos la siguiente familia de grupos asociados a $\F_q$-álgebras.
Si $\ch \F_q=p$, consideremos un $p$-grupo $\pi$ y sea $\I_{\F_q}$ el ideal de augmentación del álgebra de grupo $\F_q[\pi]$.
Entonces $1+\I_{\F_q}$ es un grupo asociado a una $\F_q$-algebra.
Si la Conjetura de los Pseudogrados fuera cierta es fácil ver (Lemma \ref{lm:FD_ab_lie_conjugacy_classes}) que se debe tener
\[
|(1+\I_{\F_q})_{\ab}|=q^{\kk(\pi)-1},
\]
donde $(1+\I_{\F_q})_{\ab}$ es la abelianización del grupo $1+\I_{\F_q}$ y $\kk(\pi)$ denota el número de clases de conjugación del grupo $\pi$.
El siguiente teorema da una descripción del tamaño de la abelianización de los grupos de la forma $1+\I_{\F_q}$. 
\begin{teorema}[cf. Theorem \ref{th:sizeequality}]
Sea $\pi$ un $p$-grupo y $\I_{\F_q}$ el ideal de augmentación de $\F_q[\pi]$. Entonces 
\[
|(1+\I_{\F_q})_{\ab}|=q^{\kk(\pi)-1}|B_0(\pi)|.
\]
\end{teorema}
El grupo $B_0(\pi)$ que aparece en el teorema es el multiplicador de Bogomolov del grupo $\pi$.
Como para cada primo $p$ existen $p$-grupos con multiplicador de Bogomolov no trivial el teorema anterior ofrece contrajemplos a la Conjetura de los Pseudogrados para todo primo $p$.

La inesperada aparición del multiplicador de Bogomolov en el teorema anterior nos llevó a seguir investigando los grupos $(1+\I_{\F_q})_{\ab}$.
Nótese que para toda $\F_q$-álgebra $R$, $1+\I_{\F_q}\otimes R$ da un grupo asociado a una $\F_q$-álgebra.
Esto define un $\F_q$-grupo algebraico afín $G$ tal que $1+\I_{\F_q}=G(\F_q)$.
\begin{teorema}[cf. Theorem \ref{t:algebraicgroup}]
Sea $\pi$ un $p$-grupo y $1+\I_{\F_q}$ el grupo asociado a la $\F_q$-álgebra correspondiente.
Sea $G$ el $\F_q$-grupo algebraico asociado.
Entonces el grupo derivado $[G,G]$ es un grupo algebraico afín definido sobre $\F_q$ y se tiene
\[
[G,G](\F_q)/[G(\F_q),G(\F_q)]\cong B_0(\pi).
\]
\end{teorema}

Los resultados que aparecen en este capítulo fueron obtenidos en colaboración con Andrei Jaikin y Urban Jezernik y pueden encontrarse en \cite{GRJaiJe}.
\subsubsection{Productos Directos de Grupos Finitos Simples} 

Consideremos la siguiente clase de grupos 
\[
\mathcal{C}:=\{\mathbf{H}=\prod_{i\in I}S_i\ :\ S_i \ \textrm{es un grupo finito simple no abeliano}\}.
\]
Dago un grupo $\Gamma$ definimos $R_n(\Gamma):=\sum_{i=1}^n r_i(\Gamma)$. Hemos obtenido una caracterización de los grupos en $\mathcal{C}$ que tienen Crecimiento Polinomial de Representaciones.
\begin{teorema}[c.f Theorem \ref{th:CharPRG}]
Sea $\mathbf{H}=\prod_{i\in I} S_i$ un producto cartesiano de grupos finitos simples no abelianos y sea $l_{\mathbf{H}}(n):=|\{i\in I : R_n(S_i)>1\}|$.
Entonces $\mathbf{H}$ tiene Crecimiento Polinomial de Representaciones si y solo si $l_{\mathbf{H}}(n)$ esta acotado polinomialmente.
\end{teorema}
Es natural preguntarse qué tipos de crecimiento de representaciones pueden ocurrir dentro de los grupos de $\mathcal{C}$ que tienen CPR. 
En \cite{KaNi}, Kassabov y Nikolov estudiaron cuestiones relacionadas con la clase de grupos $\mathcal{C}$, y más concretamente con la subclase $\mathcal{A}\subset\mathcal{C}$ donde
\[
\mathcal{A}:=\{\mathbf{H}=\prod_{i\in I}S_i\ :\ S_i=\Alt(m) \ \mbox{para algún} \  m\geq 5\}
\]
y $\Alt(m)$ denota el grupo alternado de $m$ letras.
En particular mostraron que cualquier tipo de crecimiento de representaciones es posible.

\begin{teorema*}[{\cite[Theorem 1.8]{KaNi}}]
Para cualquier $b>0$, existe un grupo $G$ tal que $\alpha(G)=b$.
\end{teorema*}
Los grupos que aparecen en la prueba del teorema anterior son frames de grupos de la forma
\[
\mathbf{H}=\prod_{i\geq 5}\Alt(i)^{f(i)},
\]
para alguna $f:\NN\to \NN$.
En particular los factores simples de $\mathbf{H}$ tienen rango no acotado \footnote{Dado $S$ un grupo finito simple no abeliano definimos $\rk S=m$ si $S=\Alt(m)$ y $\rk S=\rk L$ si $S$ es un grupo simple de tipo Lie.} como grupos alternados.
Queremos obtener un resultado similar para la subclase $\mathcal{L}\subset\mathcal{C}$, donde
\[
\mathcal{L}:=\{\mathbf{H}=\prod_{i\in I}S_i\ :\ S_i \ \mbox{es un grupo finito simple de tipo Lie}\}
.\]
Obtenemos este resultado en el siguiente teorema, donde probamos que un resultado análogo es cierto para esta subclase.
En nuestro caso, todos los factores simples pueden tener el mismo rango de Lie.
\begin{teorema}[cf. Theorem \ref{th:arb_abscissa_quasisimple}]
Para cualquier $c>0$, existe un grupo $\mathbf{H}=\prod_{i\in\NN}S_i\in\mathcal{L}$ tal que $\alpha(\mathbf{H})=c$.
Además $\mathbf{H}$ puede ser elegido de tal manera que solo ocurra un tipo de Lie entre sus factores.
\end{teorema}

Los resultados presentados en este capítulo fueron obtenidos en colaboración con Benjamin Klopsch durante mi estancia en la Universidad de Düsseldorf.
\clearpage{\pagestyle{empty}\cleardoublepage}

\chapter*{Introduction}
\addcontentsline{toc}{chapter}{Introduction}

Let $\Gamma$ be a group.
A representation of $\Gamma$ is a homomorphism of groups $
\rho:\Gamma \to \GL(V)$
for some finite dimensional complex vector space $V$.
Representation Theory studies the interplay between the structural properties of $\Gamma$ and its representations.
Burnside's Theorem was the first important application of Representation Theory and since then it has become a research area on its own. This work focuses on the representation theory of certain classes of groups.
The main body is dedicated to the study of the representation theory (more precisely, the representation growth) of arithmetic groups although we also address some questions related to the representation theory of algebra groups and  the representation growth of (direct products of) finite simple groups.

Let $\Gamma$ be a group and write $r_n(\Gamma)$ for the number of isomorphism classes of $n$-dimensional irreducible complex representations of $\Gamma$.
If $\Gamma$ is finite then $r_n(\Gamma)=0$ for $n\geq |\Gamma|^{1/2}$.
The study of finite groups by means of their
irreducible character degrees and conjugacy classes is a well established research
area, e.g., see \cite{Isaacs} or \cite{Hu}  and references therein. On the other hand for $\Gamma$ infinite $r_n(\Gamma)$ might well be infinite.
The asymptotic behaviour of the sequence $r_n(\Gamma)$ gives the Representation Growth of the group $\Gamma$. 
Let us assume that $r_n(\Gamma)<\infty$ for every $n$.
Then we can associate to the sequence $\{r_n(\Gamma)\}$ a Dirichlet zeta function
\[
\zeta_\Gamma(s)=\sum_{n=1}^\infty r_n(\Gamma)n^{-s} \quad (s\in\mathbb{C}).
\]
The abscissa of convergence $\alpha(\Gamma)$ of $\zeta_\Gamma(s)$ is the infimum of all $\alpha\in\mathbb{R}_{\geq 0}$ such that the series $\zeta_\Gamma(s)$ converges on the right half plane $\{s\in\mathbb{C}\; |\; \Real(s)>\alpha\}$.
The abscissa $\alpha(\Gamma)$ is finite if and only if $\Gamma$ has Polynomial
Representation Growth (PRG), that is, if $r_n(\Gamma)\leq p(n)$ for some polynomial $p$.

\subsubsection*{Representation Growth of Arithmetic Groups}
Let $k$ be a global field, $S$ a finite set of valuations of $k$ containing all archimedean ones and $\OO_S$ the corresponding ring of $S$-integers.
Consider $G$ a connected simply connected semisimple algebraic $k$-group and suppose $G\subseteq\GL_N$.
A group $\Gamma$ is an arithmetic group if it is commensurable with
\[
G(\OO_S):=G(k)\cap\GL_N(\OO_S).
\]
Note that in terms of the representation growth of an arithmetic group $\Gamma$ we may assume that $\Gamma=G(\OO_S)$.
In \cite{LuMa}, Lubotzky and Martin initiated the study of the representation growth of arithmetic groups.
They established a strong connection between having PRG and the famous Congruence Subgroup Problem.
Recall that $G$ is said to have the Congruence Subgroup Property with respect to $S$ if every finite index subgroup $H\leq G(\OO_S)$ contains some principal congruence subgroup $G(\mathfrak{p})$ for some ideal $\mathfrak{p}\subset\OO_S$.
In Chapter \ref{sec:CSP} we present the corresponding Congruence Subgroup Problem and explain the (slightly weaker) condition of having the weak Congruence Subgroup Property (wCSP).

\begin{theorem*}[Lubotzky and Martin]
Suppose $G$ has the weak Congruence Subgroup Property with respect to $S$. If $\ch k=2$, assume that $G$ does not contain factors of type $A_1$ or $C_m$ for any $m$. 
Then $\Gamma$ has Polynomial Representation Growth.
Moreover, if $\ch k=0$ the converse holds.
\end{theorem*}

If $G$ has the weak Congruence Subgroup Property with respect to $S$, the profinite completion $\widehat{G(\OO_S)}$ can be studied through the congruence completion $\overline{G(\OO_S)}$, which,  by the Strong Approximation Theorem, is isomorphic to $G(\widehat{\OO}_S)$. that is, the representation growth of $\Gamma$ is given by that of
\[
G(\widehat{\OO}_S)=\prod_{v\notin S}G(\OO_v).
\]

Each local factor $G(\OO_v)$ is a compact $k_v$-analytic group. 
To obtain polynomial representation growth the key is to find a uniform bound such that $r_n(G(\OO_v))\leq cn^c$ for every $v\notin S$.
Lubotzky and Martin obtained this uniform bound when $L_G$, the Lie algebra of the algebraic group $G$, is perfect, i.e., if $\ch k=2$ the group $G$ does not contain factors of type $A_1$ or $C_m$.
When the Lie algebra is not perfect Jaikin obtained in \cite{JaiPre} polynomial representation growth for each local factor, that is, for every $v$ there exists a constant $c_v$ such that $r_n(G(\OO_v))\leq c_vn^{c_v}$.
Following Jaikin's ideas for the local case we handle the global case by obtaining a uniform polynomial bound for every $v$, which ultimately gives the following theorem (cf. Theorem \ref{th:CSP_implies_PRG}).

\begin{theorem}\label{main}
Let $k,\,S,\,\Gamma$ and $G$ be as above.
If $G$ has the weak Congruence Subgroup Property with respect to $S$, then $\Gamma$ has Polynomial Representation Growth.
\end{theorem}

The work of Lubotzky and Martin in \cite{LuMa} initiated the study of representation growth of arithmetic groups (with weak CSP).
In \cite{LarLu} Larsen and Lubotzky studied the abscisa of convergence $\alpha(\Gamma)$ and Avni showed in \cite{Av} that this is a rational number for $\ch k=0$. Most of the recent work  has been done by Avni, Klopsch, Onn and Voll in a series of papers  (see \cite{AKOV1}, \cite{AKOV2}, \cite{AKOV3} and \cite{AKOV_base_change}). We recommend \cite{Klopsch} for a survey on representation growth.
\subsubsection{Subgroup Growth of Arithmetic Groups}

Let $\Gamma$ be a finitely generated group.
Let us write $a_n(\Gamma)$ for the number of subgroups of $\Gamma$ of index $n$ and put $s_n(\Gamma):=\sum_{i=1}^n a_n(\Gamma)$.
Analogously as for the Representation Growth\footnote{The study of structural properties of a group $\Gamma$ my means of its subgroup growth actually precedes the representation growth approach and the latter is somehow motivated by the fruitful theory of Subgroup Growth.} the Subgroup Growth of $\Gamma$ is given by the asymptotic behaviour of $a_n(\Gamma)$ (equivalently $s_n(\Gamma)$).
The study of structural properties of $\Gamma$ via its subgroup growth has evolved into a deep theory with many important results and applications, see \cite{Subgroup_Growth} and references therein.

As above, we focus on the subgroup growth of arithmetic groups, more precisely, we look at the subgroup growth of a group $\Gamma=G(\OO_S)$, where $G$ is a connected simply connected simple $k$-group and $\OO_S$ is the ring of $S$-integers of the global field $k$.

If the group $G$ has the Congruence Subgroup Property with respect to $S$, then every finite index subgroup is a congruence subgroup and so $s_n(\Gamma)=\gamma_n(\Gamma)$, where $\gamma_n(\Gamma)$ stands for the number of congruence subgroups of index at most $n$.
The study of the subgroup growth of congruence subgroups of arithmetic groups was initiated by Lubotzky in \cite{Lu}.
For $\ch k=0$ Lubotzky finds lower and upper bounds for $\gamma_n(\Gamma)$.
\begin{theorem*}[{\cite[Theorem A and F]{Lu}}]
Let $\Gamma=G(\OO_S)$ be as above.
If $\ch k=0$, then there exist constants $c_1,c_2$ such that
\[
n^{\frac{c_1\log^2 n}{\log\log n}}\leq\gamma_n(\Gamma)\leq n^{\frac{c_2\log^2 n}{\log\log n}}.
\]
If $\ch k>0$ and in case $\ch k=2$ $G$ is not of type $A_1$ or $C_m$, then there exists constants $c_3,c_4$ such that
\[
n^{c_3 \log n}\leq \gamma_n(\Gamma)\leq n^{c_4\log^2n}.
\]
\end{theorem*}
Abért, Nikolov and Szegedy improved the upper bound in \cite{NiAlSze}.
\begin{theorem*}[{\cite[Theorem 1]{NiAlSze}}]
Let $\Gamma=G(\OO_S)$ be  as above and suppose $\ch k>0$.
Suppose that $G$ splits over $k$ and if $\ch k=2$ suppose $G$ is not of type $A_1$ or $C_m$.
Then there exists a constant $D$ which depends only on the Lie rank of $G$ such that
\[
\gamma_n(\Gamma)\leq n^{D\log n}.
\] 
\end{theorem*} 
In \cite{Gol}, among other things, Golsefidy obtained a similar bound without the split assumption.
\begin{theorem*}
Let $\Gamma=G(\OO_S)$ be as above and suppose $\ch k\geq 5$.
Then there exist constants $C$ and $D$ such that
\[
n^{C\log n}\leq \gamma_n(\Gamma)\leq n^{D\log n}.
\] 
\end{theorem*}

In the results above it is necessary that the (graded) Lie algebra associated to each local factor $G(\OO_v)$ is perfect.
This excludes the cases where $G$ is of type $A_1$ or $C_m$.
We overcome this difficulty by applying Theorem \ref{main} and a certain relation between subgroup and representation growth.
\begin{theorem}[cf. Theorem \ref{th:log_subgroup_growth} and \ref{th:PSG}]
Let $\Gamma=G(\OO_S)$ be as above.
Then there exists a constant $D>0$ such that
\[
\gamma_n(\Gamma)\leq n^{D\log n}.
\]
Suppose that $G$ has the weak Congruence Subgroup Property with respect to $S$. Then there exists $D'>0$ such that
\[
s_n(\Gamma)\leq n^{D\log n}.
\]
\end{theorem}

\subsubsection{The Fake Degree Conjecture}

Let $J$ be a finite dimensional associative algebra over a finite field $\F_q$.
Assume further that $J$ is nilpotent, that is, $J^n=0$ for some $n$.
Then one can give to the formal set $1+J=\{1+j\;:\; j\in J\}$ a group structure by defining $(1+j)(1+h):=1+j+h+jh$ and $(1+j)^{-1}:=\sum_{k=0}^{n-1}(-1)^kj^k$.
Groups of the form $1+J$ are called Algebra Groups and were introduced by Isaacs in \cite{Isaacs} as a generalization of the unitriangular groups $U_n(\Fq)=1+\mathfrak{u}(\F_q)$, where $\mathfrak{u}_n(\F_q)$ is the algebra of strictly upper triangular square matrices of size $n$.

The group $1+J$ acts naturally on the algebra $J$ and this induces an action of $1+J$ on $\Irr(J)$, the irreducible characters of the additive group of $J$.
In this setting, it is possible to adapt Kirillov's Orbit Method and try to build a bijection
\[
\begin{array}{ccc}
(1+J)\textrm{-orbits in }\Irr(J)&\to&\Irr(1+J).
\end{array}
\]
This approach has serious difficulties and such a bijection is only obtained when $J^p=0$, where $p=\ch \F_q$.
Nevertheless, the adaptation of Kirillov's Orbit Method suggests a map
\[
\begin{array}{ccc}
(1+J)\textrm{-orbits in }\Irr(J)&\to&\{\chi(1)\;:\;\chi\in\Irr(1+J)\}\\
\Omega&\mapsto&|\Omega|^{1/2}.\\
\end{array}
\]
Isaacs conjectured that this map is well defined and that the $|\Omega|^{1/2}$'s (the fake degrees), counting multiplicities, give the irreducible character degrees of $\Irr(1+J)$ for every algebra group $1+J$.
This is known as the Fake Degree Conjecture.
In \cite{Ja1} Jaikin disproved the Fake Degree Conjecture by providing an $\F_2$-algebra group which served as a counterexample.
Nevertheless, the particular behaviour of the prime $p=2$ in character correspondences and experimental computations suggested that the conjecture might hold for $\F_q$-algebra groups with $\ch \F_q>2$.
Let us consider the following family of $\F_q$-algebra groups. If $\ch \F_q=p$, take a $p$-group $\pi$ and let $\I_{\F_q}$ the augmentation ideal of the group algebra $\F_q[\pi]$.
Then $1+\I_{\F_q}$ is an $\F_q$-algebra group.
If the Fake Degree Conjecture holds it is easy to see (Lemma \ref{lm:FD_ab_lie_conjugacy_classes}) that we must have
\[
|(1+\I_{\F_q})_{\ab}|=q^{\kk(\pi)-1},
\]
where $(1+\I_{\F_q})_{\ab}$ is the abelianization of the group $1+\I_{\F_q}$ and $\kk(\pi)$ stands for the number of conjugacy classes of the group $\pi$.
The following theorem gives a description of the size of the abelianization of groups of the form $1+\I_{\F_q}$. 
\begin{theorem}[cf. Theorem \ref{th:sizeequality}]
Let $\pi$ be a $p$-group and $\I_{\F_q}$ the augmentation ideal of $\F_q[\pi]$. Then 
\[
|(1+\I_{\F_q})_{\ab}|=q^{\kk(\pi)-1}|B_0(\pi)|.
\]
\end{theorem}
The group $B_0(\pi)$ appearing in the theorem is the Bogomolov multiplier of the group $\pi$.
Since for every prime $p$ there exist $p$-groups with nontrivial Bogomolov multiplier, the previous theorem provides counterexamples for the Fake Degree Conjecture for every prime $p$.

The somehow unexpected appearance of the Bogomolov multiplier in the previous theorem led us to further investigation of the groups $(1+\I_{\F_q})_{\ab}$.
Note that for every $\F_q$-algebra $R$, $1+\I_{\F_q}\otimes R$ gives an $\F_q$-algebra group. This defines an affine algebraic $\F_q$-group $G$ such that $1+\I_{\F_q}=G(\F_q)$.
\begin{theorem}[cf. Theorem \ref{t:algebraicgroup}]
Let $\pi$ be a $p$-group and $1+\I_{\F_q}$ the associated $\F_q$-algebra group. Let us consider the associated algebraic $\F_q$-group $G$ as above.
Then the derived group $[G,G]$ is an affine algebraic group defined over $\F_q$ and we have
\[
[G,G](\F_q)/[G(\F_q),G(\F_q)]\cong B_0(\pi).
\]
\end{theorem}

The results appearing in this chapter were obtained in joint work with Andrei Jaikin and Urban Jezernik and can be found in \cite{GRJaiJe}.

\subsubsection{Direct Products of Finite Simple Groups} 

Let us consider the class of groups 
\[
\mathcal{C}:=\{\mathbf{H}=\prod_{i\in I}S_i\ :\ S_i \ \textrm{is a nonabelian finite simple group}\}.
\]
Given a group $\Gamma$ let us put $R_n(\Gamma):=\sum_{i=1}^n r_i(\Gamma)$. We obtain a characterization of groups in $\mathcal{C}$ having Polynomial Representation Growth.
\begin{theorem}[c.f Theorem \ref{th:CharPRG}]
Let $\mathbf{H}=\prod_{i\in I} S_i$ be a cartesian product of nonabelian finite simple groups and let $l_{\mathbf{H}}(n):=|\{i\in I : R_n(S_i)>1\}|$.
Then $\mathbf{H}$ has Polynomial Representation Growth if and only if $l_{\mathbf{H}}(n)$ is polynomially bounded.
\end{theorem}
It is natural to ask what types of representation growth can occur among groups in $\mathcal{C}$ having PRG. 
In \cite{KaNi}, Kassabov and Nikolov studied questions related to the class $\mathcal{C}$, and more in particular about the subclass $\mathcal{A}\subset\mathcal{C}$ where
\[
\mathcal{A}:=\{\mathbf{H}=\prod_{i\in I}S_i\ :\ S_i=\Alt(m) \ \mbox{for some} \  m\geq 5\}
\]
and $\Alt(m)$ denotes the alternating group on $m$ letters.
In particular they showed that any type of representation growth is possible.

\begin{theorem*}[{\cite[Theorem 1.8]{KaNi}}]
For any $b>0$, there exists a group $G$ such that $\alpha(G)=b$.
\end{theorem*}
The groups appearing in the proof of the previous theorem are frames of groups of the form
\[
\mathbf{H}=\prod_{i\geq 5}\Alt(i+1)^{f(i)},
\]
for some $f:\NN\to \NN$.
In particular the simple factors of $\mathbf{H}$ have unbounded rank\footnote{For $S$ a nonabelian finite simple group we define $\rk S=m$ if $S=\Alt(m)$ and $\rk S=\rk L$ if $S$ is a simple group of Lie type $L$.} as alternating groups.
We want to obtain a similar result for the subclass $\mathcal{L}\subset\mathcal{C}$, where
\[
\mathcal{L}:=\{\mathbf{H}=\prod_{i\in I}S_i\ :\ S_i \ \mbox{is a finite simple group of Lie type}\}
.\]
We do this in the following theorem, where it is shown that an analogous result holds for this class.
In our case, all of the simple factors may have the same Lie rank.
\begin{theorem}[cf. Theorem \ref{th:arb_abscissa_quasisimple}]
For any $c>0$, there exists a group $\mathbf{H}=\prod_{i\in\NN}S_i\in\mathcal{L}$ such that $\alpha(\mathbf{H})=c$. Moreover $\mathbf{H}$ can be chosen so that only one Lie type occurs among its factors.
\end{theorem}

The results presented in this chapter where obtained in collaboration with B. Klopsch during my research stay at Heinrich Heine Universtät Düsseldorf.

\clearpage{\pagestyle{empty}\cleardoublepage}
\pagestyle{fancy}

\chapter{Preliminaries}
\section{Global and local fields}\label{sec:global_and_local}

By a \textbf{global field} we mean a finite field extension  of $\mathbb{Q}$ or of $\F_q(t)$ for some prime power $q$.
Global fields have interesting arithmetic properties, resembling those of $\mathbb{Q}$.

Let $k$ be a field.
A multiplicative \textbf{valuation} or an absolute value on $k$ is a function $v: k \to \mathbb{R}_{\geq 0}$ which satisfies:
\begin{enumerate}[label=\roman*)]
\item $v(x)=0$ if and only if $x=0$.
\item $v(xy)=v(x)v(y)$ for every $x,y\in k$.
\item $v(x+y)\leq v(x)+v(y)$ for every $x,y\in k$.
\end{enumerate}
A valuation $v$ is said to be non-archimedean if it satisfies the ultrametric inequality
\[
v(x+y)\leq\max\{v(x),v(y)\}\, \mbox{for every}\,x,y\in k,
\]
and is said to be archimedean otherwise.

We will always assume that a valuation is nontrivial, i.e., $v(x)\neq 1$ for some $x\in k\setminus\{0\}$.
Any valuation $v$ defines a metric $d_v$ on $k$, given by $d_v(x,y)=v(x-y)$ for $x,y\in k$. Two valuations are said to be equivalent if the metrics they induce define the same topology on $k$.
This gives an equivalence relation among valuations of a field $k$.
From now on a valuation will be identified with its equivalence class.

The set of (equivalence classes of) valuations of $k$ will be denoted by $V_k$.
Let us write $V_\infty$ for the set of archimedean valuations and $V_f$ for the non-archimedean ones.
The set of archimedean valuations  $V_\infty$ is non-empty if and only if $\ch k=0$.
Note that given a valuation $v$, addition in $k$ is continuous with respect to the topology induced by $v$.
We can consider its completion $k_v$ which is again a valued field where multiplication, addition and the valuation naturally extend those of $k$.

Given $S\subseteq V_k$ we define the ring of $S$-integers to be
\[
\OO_S:=\{x\in k\,:\, v(x)\leq 1\;\textrm{for every}\; v\in V_f\setminus{S}\}.
\]
The ultrametric inequality for non-archimedean valuations guarantees that $\OO_S$ is a subring of the field $k$.

A \textbf{local field}  is a field isomorphic either to $\mathbb{C}$ or $\mathbb{R}$ or a finite field extension of the $p$-adic numbers $\mathbb{Q}_p$ or of the field of formal Laurent series over a finite field $\F_q((t))$.

Suppose now that $k$ is a global field.
For every $v\in V_k$ the completion $k_v$ is a local field and $v$ extends to a valuation on $k_v$ which we will continue to denote by $v$. The following cases may occur:

If $v\in V_\infty$ then $k_v$ is isomorphic to $\mathbb{C}$ or $\mathbb{R}$.
If $v\in V_f$ and $\ch k=0$ then $k_v$ is isomorphic to a finite extension of $\mathbb{Q}_p$ for some prime $p$.
If $\ch k=p>0$ then $k_v$ is isomorphic to a finite extension of $\F_q((t))$ for some $p$-power $q$.

Suppose now $v$ is non-archimedean. Then 
\[
\OO_v:=\{x\in k_v\,:\, v(x)\leq 1\}
\]
is a ring. It is the unique maximal compact subring of $k_v$ and is called the \textbf{ring of integers} of $k_v$.
The ring $\OO_v$ is a discrete valuation ring with maximal ideal
\[
\mm_v:=\{x\in \OO_v\; :\; v(x)<1\}.
\]
It follows that $\F_v:=\OO_v/\mm_v$ is a field, called the residue field of $v$ (or of $k_v$).
For $x\in\mm_v$ we will write
\[
\ord_v(x):=\max\{k\;:\; x\in\mm_v^k\}.
\]

We define the \textbf{adèle ring or ring of adèles} of the global field $k$ as 
\[
\mathbb{A}_k:=\{(x_v)_{v\in V_k}\in\prod_{v\in V_k}k_v\; :\; x_v\in\OO_v\;\textrm{for almost every}\; v\in V_f\}.
\]
Note that $\mathbb{A}_k$ is a subring of the ring $\prod_{v\in V_k}k_v$, where ring operations are made component-wise.
The ring $\mathbb{A}_k$ is endowed with the restricted direct product topology with respect to the $\OO_v$ ($v\in V_f$), that is, the open sets are of the form $\prod_{v\in V_k} X_v$, where $X_v$ is open in $k_v$ for every $v\in V_k$ and $X_v=\OO_v$ for almost every $v\in V_f$.
With this topology $\mathbb{A}_k$ is a locally compact topological group.
Note that if $x\in k$, then $x\in\OO_v$ for almost every $v\in V_f$.
It follows that we can embed $k$ diagonally into $\mathbb{A}_k$ as the set of \textbf{principal adèles}, $x\mapsto (x)_{v\in V_k}$.
Moreover $k$ is discrete in $\mathbb{A}_k$.
Given a finite subset $S\subset V_f$ we will write $\mathbb{A}_S$ for the image of $\mathbb{A}_k$ under the natural projection onto the direct product $\prod_{v\notin S}k_v$.

The following lemma gives a bound for the number of distinct valuations in terms of the size of their residue field, see \cite[Lemma 4.7.(a)]{LuMa} for a proof.
\begin{lemma} \label{pre:lm:numberofval}
Let $k$ be a global field. There exists $a\in\NN$, depending only on the field $k$, such that the number of $v\in V_f$ such that $|\F_v|\leq n$ is bounded by $an$, and hence also by $n^b$ for some $b\in\NN$.
\end{lemma}


We shall also need an application of Chebotarev Density Theorem, see \cite{Mil} for more details.

\begin{theorem}[Chebotarev Density Theorem]\label{pre:th:Chebotarev}
Let $k$ be a global field and $K$ a finite Galois extension of $k$.
Let $C$ be a conjugacy class in $\Gal(K/k)$.
Let $\Sigma_C$ be the set of non-archimedean places $v\in V_f$ that are unramified in $k$ and have Frobenius conjugacy class $C$.
The set $\Sigma_C$ has a positive Dirichlet density and it is equal to $|C|/|\Gal(K/k)|.$
In particular, $\Sigma_C$ is infinite.
\end{theorem}

\begin{lemma}\label{pre:lm:totally_split_v}
Let $k$ be a global field and $K$ a finite Galois field extension of $k$.
Then for infinitely many $v\in V_f$ we have $K\subset k_v$.
\end{lemma}
\begin{proof}
By Chebotarev Density Theorem applied to the trivial conjugacy class $C=\{1\}$ we obtain that there are infinitely many $v\in V_f$ such that $v$ completely splits in $K$, i.e., $K_w=k_v$, where $w$ is the extension of $v$ to $K$.
But this in particular implies that $K\subseteq k_v$. 
\end{proof}

\section{Topological groups}

We begin by recalling some definitions and concepts generally related to topological spaces.
Let us recall that a \textbf{filter} $\mathfrak{F}$ in a set $X$ is a collection of non-empty subsets of $X$ such that:
\begin{enumerate}[label=\roman*)]
\item If $F\in\mathfrak{F}$ and $F\subseteq E$, then $E\in\mathfrak{F}$.
\item If $E,F\in\mathfrak{F}$ then $E\cap F\in\mathfrak{F}$.
\end{enumerate}
The main example for us is the set of neighbourhoods $\mathfrak{N}(x)$ of an element $x$ in a topological space $X$, which we called the \textbf{neighbourhood filter} of $x$.
In this setting we say that a filter $\mathfrak{F}$ on $X$ \textbf{converges} to $x\in X$ if $\mathfrak{N}(x)\subseteq\mathfrak{F}$.
Recall that a \textbf{filter basis} $\mathfrak{B}$ on a set $X$ is a collection of subsets of $X$ such that:
\begin{enumerate}[label=\roman*)]
\item If $F_1,\, F_2\in\mathfrak{B}$, then there exists $F\in\mathfrak{B}$ such that $F\subseteq F_1\cap F_2$.
\item $\emptyset\not\in\mathfrak{B}$.
\end{enumerate}
The collection of supersets of a filter basis on $X$ gives a filter on $X$.

To define the completion of a topological group we first need to establish a notion of nearness of a pair of points in a topological space $X$. This is done via uniform structures.
A uniform structure on a set $X$ consists of a family $E$ of subsets of $X\times X$ called entourages satisfying:
\begin{enumerate}[label=E\arabic*)]
\item If $V\in E$ and $V\subseteq U$, then $U\in E$.
\item If $U,V\in E$, then $U\cap V\in E$.
\item If $V\in E$, then $\Delta=\{(x,x)\, :\, x\in X\}\subseteq V$.
\item If $V\in E$, then $V^{-1}:=\{(x,y)\, :\, (y,x)\in V\}\in E$.
\item If $V\in E$, then there exists $W\in E$ such that $W\circ W\subset V$, where $W\circ W:=\{(x,z)\, :\,(x,y),(y,z)\in W\,\mbox{for some}\, y\in W\}$.
\end{enumerate}

If  $(x,y)\in V$ for some $V\in E$  we say that $x$ is {\bf $\mathbf{V}$-close} to $y$.
We say that $A\subseteq X$ is $\mathbf{V}${\bf-small} if $A\times A\subseteq V$, equivalently, if every pair of points in $A$ are $V$-close.
Hence, entourages provide a way to measure the nearness of a pair of points in $X$.
Given a uniform structure on $X$ one can endow $X$ with a topology as follows.
For every $V\in E$ and $x\in X$ put $V(x):=\{y\in X\,:\, (x,y)\in V\}$.
Then, for every $x\in X$ we define a system of neighbourhoods of $x$ to be
\[
\mathcal{N}(x):=\{V(x)\,:\, V\in E\}.
\]
It follows that $X$ is Hausdorff if and only if $\bigcap_{V\in E}V=\Delta$.
The uniform structure allows us to define a notion of uniform continuity.
Namely, a function $f:X\to Y$ between uniform spaces $(X,E)$ and $(Y,E')$ is said to be \textbf{uniformly continuous} if for every $W\in E'$ there exists $V\in E$ such that $(f(x),f(x'))\in W$ for every $(x,x')\in V$.

In a similar fashion as a fundamental system of neighbourhoods of a given topology we can define the notion of a fundamental system of entourages of a uniform structure.
A family $\mathcal{F}\subseteq E$ is called a \textbf{fundamental system of entourages} for $E$ if for every $V\in E$ there exists some $W\in\mathcal{F}$ such that $W\subseteq V$.
Equivalently given a fundamental system of entourages $\mathcal{F}$ of a uniform structure $E$, we can recover $E$ as
\[
E=\{V\subseteq X\times X\, :\, W\subseteq V\,\textrm{for some}\, W\in\mathcal{F}\}.
\]
A family $\mathcal{F}$ of subsets of $X\times X$ serves as a fundamental system of entourages for a uniform structure on $X$ if and only if $\mathcal{F}$ satisfies the following properties:
\begin{enumerate}[label=F\arabic*)]
\item For every $U,V\in\mathcal{F}$ there exists $W\in\mathcal{F}$ such that $W\subseteq U\cap V$.
\item If $V\in \mathcal{F}$, then $\Delta\subseteq V$.
\item If $V\in \mathcal{F}$, then there exists $W\in \mathcal{F}$ such that $W\subseteq V^{-1}$.
\item If $V\in\mathcal{F}$, then there exists $W\in\mathcal{F}$ such that $W\circ W\subseteq V$.
\end{enumerate}

An entourage $V$ such that $V=V^{-1}$ is called \textbf{symmetric}.
If $V$ is symmetric, then $V\cap V^{-1}$ and $V\cup V^{-1}$ are entourages and from this one can check that the set of symmetric entourages forms a fundamental system of entourages.

Given a uniform space $X$ with uniform structure $E$ we want to construct the completion $\widehat{X}$ of $X$.
To this end we need to introduce the notion of a \textbf{Cauchy filter}.
A filter $\mathfrak{F}$ of $X$ is Cauchy if for every entourage $V\in E$ there exists $F\in\mathfrak{F}$ such that $F$ is $V$-small, i.e., $F\times F\subseteq V$.
\begin{proposition}[{\cite[II.3.2.5]{Bour}}] \label{pre:prop:minimal_Cauchy_filters}
Let $X$ be a uniform space.
For each Cauchy filter $\mathfrak{F}$ on $X$ there is a unique minimal Cauchy filter $\mathfrak{F}_0$ coarser than $\mathfrak{F}$.
If $\mathfrak{B}$ is a base of $\mathfrak{F}$ and $\mathfrak{S}$ is a fundamental system of symmetric entourages of $X$, then the sets $V(M):=\{\bigcup_{x\in M} V(x)\}$ for $M\in\mathfrak{B}$ and  $V\in\mathfrak{S}$ form a base of $\mathfrak{F}_0$.
\end{proposition}

A uniform space $X$ is called \textbf{complete} if every Cauchy filter on $X$ converges to some point.
\begin{theorem}[{\cite[Theorem II.3.7.3]{Bour}}]\label{th:completion}
Let $X$ be a uniform space.
Then there exists a complete Hausdorff
uniform space $\widehat{X}$ and a uniformly continuous mapping $i:X\to\widehat{X}$ having the following property:
Given any uniformly continuous mapping $f:X\to Y$ for some complete Hausdorff uniform space $Y$, there is a unique uniformly continuous mapping $\hat{f}:\widehat{X}\to Y$ such that $f=\hat{f}\circ i$.

Moreover if $(i',X')$ is another pair consisting of a complete Hausdorff uniform space $X'$ and a uniformly continuous mapping $i':X\to X'$ having the previous property, then there is a unique isomorphism $\phi:\widehat{X}\to X'$ such that $i'=\phi\circ i$.
\end{theorem}
\begin{proof}[Sketch of proof]
The space $\widehat{X}$ is defined as the set of minimal Cauchy filters on $X$ (these are the minimal elements with respect to inclusion of the set of Cauchy filters on $X$).
We shall define a uniform structure on $\widehat{X}$.
For this purpose, if $V$ is any symmetric entourage of $X$, let $\widehat{V}$ denote the set of all pairs $(\mathfrak{F},\mathfrak{G})
$ of minimal Cauchy filters which have in common a $V$-small set.
One can check that the sets $\widehat{V}$ form a fundamental system of entourages of a uniform structure on $\widehat{X}$ and that $\widehat{X}$ with this uniform structure is a complete space. 
The map $i:X\to\widehat{X}$ sends every element $x\in X$ to $N(x)$, which is a minimal Cauchy filter. One can check that the pair $(i,\widehat{X})$ satifies all the required properties.
\end{proof}
Before moving to topological group completions we give the following theorem about extensions of uniformly continuous functions.

\begin{theorem}[{\cite[Theorem II.3.6.2]{Bour}}]
Let $f:A\to X'$ be a function from a dense subset $A$ of a uniform
space $X$ to a complete Hausdorff uniform space $X'$ and suppose
that $f$ is uniformly continuous on $A$.
Then f can be extended to the whole of $X$ by continuity, and the extended function $\widehat{f}$ is uniformly continuous.
\end{theorem}
This has as an immediate consequence:
\begin{corollary}\label{pre:cor:extension_to_completion}
Let $X$ and $X'$ be two complete Hausdorff uniform spaces, and
let $A$, $A'$ be dense subsets of $X$ and $X'$ respectively. Then every isomorphism $f:A\to A'$ extends to an isomorphism $\widehat{f}:X\to X'$.
\end{corollary}

\subsection{Topological groups and group completions}\label{sec:completions_and_group_completions}

A \textbf{topological group} is a topological space $G$ together with a group structure such that multiplication $m:G\times G\to G$ and inversion $i:G\to G$ are continuous.
We will write $H\leq_o G$, respectively $H\leq_c G$, whenever $H$ is an open subgroup, respectively closed subgroup, of $G$.
We similarly write $H\trianglelefteq_o G$ in case $H$ is a normal subgroup.
Let us denote by $\mathfrak{N}(e)=\mathfrak{N}$ the neighbourhood filter of the identity element $e$.
Since multiplication by an element $g\in G$ is a homeomorphism, $g\cdot\mathfrak{N}$ or $\mathfrak{N}\cdot g$ are the neighbourhood filter of the element $g$.
Hence the topology on a topological group $G$ is determined by $\mathfrak{N}$.
Since in a topological group multiplication and inversion (and hence also conjugation) are continuous it follows that $\mathfrak{N}$ satisfies:
\begin{enumerate}[label=(GB\arabic*)]
\item For every $U\in\mathfrak{N}$ there exists $V\in\mathfrak{N}$ such that $V\cdot V\subseteq U$.\label{GB1}
\item For every $U\in\mathfrak{N}$ we have $U^{-1}\in\mathfrak{N}$.\label{GB2}
\item For every $g\in G$ and every $U\in\mathfrak{N}$ we have $g U g^{-1}\in\mathfrak{N}$.\label{GB3}
\end{enumerate}
Conversely, we have:
\begin{proposition}[{\cite[III.1.2.1]{Bour}}]\label{pre:base_group_topology_char}
Let $\mathfrak{N}$ be a filter on a group $G$ such that $\mathfrak{N}$ satisfies \ref{GB1}, \ref{GB2} and \ref{GB3}.
Then there is a unique topology on $G$, compatible with the group structure of $G$, for which $\mathfrak{N}$ is the neighbourhood filter of the identity element.
\end{proposition}

Let $G$ be an abstract group and $\mathfrak{F}$ a family of subgroups of $G$. Consider $\mathfrak{B}$ the set of finite intersections of groups of the form $g^{-1}Fg$ for $g\in G$, $F\in\mathfrak{F}$. Then $\mathfrak{B}$ is a filter base and by construction its associated filter satisfies \ref{GB1}, \ref{GB2} and \ref{GB3}. From any family of subgroups $\mathfrak{F}$ of a given group $G$, one can construct a topology on $G$ such that $F$ is open for every $F\in\mathfrak{F}$ and with this topology $G$ becomes a topological group.

Let us now focus on Hausdorff topological groups. Recall that we denote by $\mathfrak{N}(e)$ the neighbourhood filter of the identity element in $G$.
It is easy to check that the families
\begin{align*}
& \mathcal{F}_r:=\{V_r=\{(x,y)\,:\, yx^{-1}\in V \}\quad :\, V\in \mathfrak{N}(e)\}\\
& \mathcal{F}_l:=\{V_l=\{(x,y)\,:\, x^{-1}y\in V\}\quad :\, V\in\mathfrak{N}(e)\}
\end{align*}
satisfy the required properties to be a fundamental system of entourages, and hence define a uniform structure $E_r$ (repectively $E_l$) on $G$ called the right (respectively left) uniform structure of $G$ which is compatible with the original topology of $G$, i.e., this uniformity induces the original topology on $G$.
If we now consider $G_l=(G,E_l)$, we want to give to $\widehat{G}_l$ (the completion of $G_l$ as a uniform space, see Theorem \ref{th:completion}) a topological group structure such that $G$ is a dense subgroup of $\widehat{G}_l$.
For this we need to define a continuous multiplication and inversion on $\widehat{G}_l$ extending those of $G$.

\begin{proposition}[\cite{Bour}, Proposition III.3.4.6]
Let $\mathfrak{F}$ and $\mathfrak{G}$ be two Cauchy filters on $G
_l$.
Then the image of the filter $\mathfrak{F}\times\mathfrak{G}$ under the map $(x,y)\mapsto (xy)$ is a Cauchy filter base on $G_l$.

\end{proposition}

It follows that we can define a multiplication on $\widehat{G}_l$ that extends the original one in $G$.
To show that we can also extend inversion to $\widehat{G}_l$ we must assume that the function $x\mapsto x^{-1}$ maps Cauchy filters to Cauchy filters\footnote{There are examples of groups where this assumption is not satisfied, see \cite[Exercise 16 X.3]{Bour}.}.
However let us note that if $G$ has a neighbourhood base of the identity consisting of subgroups then this condition is automatically satisfied. Moreover one can show that the group $\widehat{G}_l$ obtained in this way is unique, in particular, $\widehat{G}_l\cong\widehat{G}_r$.
In the end we obtain the following theorem.

\begin{theorem}[\cite{Bour}, Theorem III.3.4.1]\label{th:group_completion}
A Hausdorff topological group $G$ is isomorphic to a dense subgroup
of a complete group $\widehat{G}$ if and only if the image under the symmetry $x\mapsto x^{-l}$ of a Cauchy filter with respect to the left uniformity of $G$ is a Cauchy filter with
respect to this uniformity.
The complete group $\widehat{G}$, which is called the completion
of G, is then unique up to isomorphism.
\end{theorem}

\subsection{Profinite groups}\label{pre:profinite_groups}

A \textbf{directed system} $(\Lambda,\leq)$ is a partially ordered set such that for every pair of elements $\lambda_1,\lambda_2\in\Lambda$ there exists an element $\lambda_0$ such that $\lambda_0\geq\lambda_1,\lambda_2$.
The collection of (normal) subgroups of a group is a directed system ordered by reverse inclusion.

A \textbf{projective system} is a collection of sets $(X_\lambda)_{\lambda\in\Lambda}$ for some directed system $\Lambda$ together with maps $f_{\lambda\mu}:X_\lambda\to X_\mu$ for every $\lambda\geq\mu$. Given such a directed system we define its \textbf{projective limit} as

\[
\varprojlim_{\lambda\in\Lambda}:=\{(x_\lambda)_{\lambda\in\Lambda}\in\prod_{\lambda\in\Lambda}X_{\lambda}\;:\; f_{\lambda\mu}(x_{\lambda})=x_\mu\;\textrm{for every}\; \lambda\geq\mu\}
\]
If we have a family of normal subgroups $\mathfrak{N}$ of a given group $G$ closed under intersection, then we have group homomorphisms $f_{NM}:G/N\to G/M$ whenever $N\leq M$.
This way $\varprojlim_{N\in\mathfrak{N}} G/N$ becomes a subgroup of the direct product $\prod_{N\in\mathfrak{N}}G/N$.
If all quotients are finite and we give them the discrete topology, then the inverse limit becomes a closed subgroup of the compact group $\prod_{N\in\mathfrak{N}}G/N$ called the pro-$\mathfrak{N}$ completion of $G$.
If we take $\mathfrak{N}$ to be the family of all finite index normal subgroups of $G$, then
\[
\widehat{G}:=\varprojlim_{N\in\mathfrak{N}} G/N
\]
is called the profinite completion of $G$\footnote{Proposition \ref{pre:pro:profinite_uniform_completion_coincide} justifies the overlaping of notation.}.

A \textbf{profinite group} $G$ is a compact Hausdorff topological group whose open subgroups form a base for the neighbourhoods of the identity.
The following theorem collects equivalent characterizations of profinite groups:

\begin{theorem}[\cite{AnalyticProP}, Chapter 1]\label{pre:th:profinite_char}
Let $G$ be a topological group. The following are equivalent:
\begin{enumerate}[label=\roman*)]
\item $G$ is profinite.
\item $G$ is totally disconnected.
\item $G\cong\varprojlim_{N\trianglelefteq_oG}G/N$
\item $G$ has a neighbourhood base of the identity consisting of finite index open subgroups and $G$ is complete.
\end{enumerate}
\end{theorem}

Consider now an abstract group $G$  and $\mathfrak{N}$ a family of finite index normal subgroups closed under intersection and such that $\bigcap_{N\in\mathfrak{N}}N=e$.
On the one hand $\mathfrak{N}$ satisfies \ref{GB1}, \ref{GB2} and \ref{GB3} so it induces a topology on $G$, which  makes $G$ into a Hausdorff topological group, and this gives $G$ a uniform structure as well.
On the other hand $\varprojlim_{N\in\mathfrak{N}}G/N$ is a profinite group.
\begin{proposition} \label{pre:pro:profinite_uniform_completion_coincide}
Let $G$ and $\mathfrak{N}$ be as above. Then $\widehat{G}_l\cong \varprojlim_{N\in\mathfrak{N}}G/N$.
\end{proposition}
\begin{proof}
Note that $G$ is dense both in $\varprojlim_{N\in\mathfrak{N}}G/N$ and $\widehat{G}_l$. Since the identity map $\id: G\to G$ is continuous with respect to these structures, Corollary \ref{pre:cor:extension_to_completion} gives the result.
\end{proof}

A profinite group is called \textbf{pro-nilpotent} if it can be realized as an inverse limit of nilpotent groups. Similarly, a profinite group is called a \textbf{pro-$\mathbf{p}$ group} if it can be realized as an inverse limit of $p$-groups.
\subsection{Analytic groups}

Throughout this section $k$ is a local field complete with respect to a non-archimedean valuation, $\OO$ denotes its valuation ring and $\mm$ its maximal ideal. We refer to Part II of \cite{Ser} for a more detailed exposition.

Let $G$ be a topological group and suppose further that $G$ has the structure of an analytic manifold over $k$.
We say that $G$ is an \textbf{analytic group} over $k$ or a $k$-analytic group if multiplication $m:G\times G\to G$ and inversion $i:G\to G$ are analytic maps.
If $G$ is a $\mathbb{Q}_p$-analytic group we say that $G$ is a $\mathbf{p}${\bf-adic analytic group}.
The following lemma shows that every $k_v$-analytic group can be seen either as a $p$-adic analytic group or an $F_p((t))$-analytic group, see \cite[Exercise 13.4]{AnalyticProP}.

\begin{lemma}\label{pre:lm:Zp_and_Fpt_analytic} Suppose $k'$ is a finite separable extension of $k$ with ring of integers $\OO'$. Let $G$ be a $k'$-analytic group and $U\leq_o G$ an $\OO'$-standard subgroup (see definition below).
Then $G$ has the structure of a $k$-analytic group as well and $U$ becomes an $\OO$-standard subgroup with respect to this structure.
\end{lemma}

\subsubsection{ Formal group laws}
Let us write ${\bf x}=(x_1,\ldots,x_n)$ to denote a tuple of $n$ variables.
For $\alpha=(\alpha_1,\ldots,\alpha_n)\in\mathbb{N}^n$ let us write ${\bf x}^\alpha=x_1^{\alpha_1}\ldots,x_n^{\alpha_n}$ and $|\alpha|=\sum_{i=1}^n\alpha_i$.
Given a commutative ring $R$ consider the formal power series ring $R[[x_1,\ldots,x_n]]=R[[{\bf x}]]$.
A \textbf{formal group law} over $R$ in $n$ variables is an $n$-tuple $ F=(F_1,\ldots,F_n)$ of formal power series in $2n$ variables, $F_i\in R[[{\bf x},{\bf y}]]$ such that:
\begin{enumerate}
\item $F({\bf x},0)={\bf x}$ and $F(0,{\bf y})={\bf y}$.
\item $F(F({\bf x},{\bf y}),{\bf z})=F({\bf x},F({\bf y},{\bf z}))$.
\end{enumerate}
If there is no danger of confusion we will write $\mathbf{x}\mathbf{y}=F(\mathbf{x},\mathbf{y})$.
\begin{proposition}\label{pre:formal_group_laws}
Let $F=(F_1,\ldots,F_n)$ be a formal group law over $R$ and let $O(3)$ denote some formal power series whose homogeneous components of degree strictly less than $3$ are zero. Then:
\begin{enumerate}[label=\roman*)]

\item For every $1\leq i\leq n$ we have
\[
F_i(\mathbf{x},\mathbf{y})=x_i+y_i+\sum_{|\alpha|,|\beta|\geq 1} c_{\alpha\beta}\mathbf{x}^\alpha\mathbf{y}^\beta.
\]
\item There exists a unique $\phi=(\phi_1,\ldots,\phi_n)$ where $\phi_i\in R[[\mathbf{x}]]$ called the formal inverse such that $\phi(0)=0$ and
\[
F(\mathbf{x},\phi(\mathbf{x}))=F(\phi(\mathbf{x}),\mathbf{x})=0.
\]
We will write $\mathbf{x}^{-1}:=\phi(\mathbf{x})$. 
\item $F(\mathbf{x},\mathbf{y})=\mathbf{x}+\mathbf{y}+ B(\mathbf{x},\mathbf{y})+O(3)$, where $B$ is an $R$-bilinear form. We define
\[
[\mathbf{x},\mathbf{y}]:=B(\mathbf{x},\mathbf{y})-B(\mathbf{y},\mathbf{x}).
\].\label{e-bracket-formal-law}
\item $\mathbf{x}^{-1}=-\mathbf{x}+ B(\mathbf{x},\mathbf{x})+O(3)$.
\item \label{e_formal_adjoint}$F(F(\mathbf{x},\mathbf{y}),\mathbf{x}^{-1})=\mathbf{y}+[\mathbf{x},\mathbf{y}]+O(3)$ and more precisely
\[
\mathbf{x}\mathbf{y}\mathbf{x}^{-1}=\mathbf{y}+[\mathbf{x},\mathbf{y}]+\sum_{|\alpha|,|\beta|\geq 1, |\alpha|+|\beta|\geq 3} d_{\alpha,\beta}\mathbf{x}^\alpha\mathbf{y}^\beta,
\]
where $d_{\alpha,\beta}\in R$.
\item $\mathbf{x}^{-1}\mathbf{y}^{-1}\mathbf{x}\mathbf{y}=[\mathbf{x},\mathbf{y}]+O(3)$. \label{e_formal_commutator}

\item $[\mathbf{x},[\mathbf{y},\mathbf{z}]+[\mathbf{y},[\mathbf{z},\mathbf{x}]]+[\mathbf{z},[\mathbf{x},\mathbf{y}]]=0$. \label{e_fomal_jacobi}
\item ${\bf x}^m=\sum_{i=1}^\infty {m\choose i}\psi_i({\bf x})$, where $\psi_i({\bf x})$ has order at least $i$.\label{e_power_map}
\end{enumerate}
\end{proposition}

\subsubsection{Standard groups}\label{sec:standard_groups}

Let $F=(F_1,\ldots,F_n)$ a group law over $\OO$ and put
\[
G=\{(g_1,\ldots,g_n)\;:\; x_i\in\mm\}=(\mm)^n.
\]
Given $g=(g_1,\ldots,g_n)$, $h=(h_1,\ldots,h_n)\in G$, define  multiplication on $G$ by $gh:=F(g,h)$ and inversion by $g^{-1}:=\phi(g)$ where $\phi$ is the inverse associated to $F$.
Then $G$ is an analytic $k$-group.
Any group isomorphic to one of the form of $G$ is called an $\mathbf{\OO}$\textbf{-standard group}.
To every $\OO$-standard group $G$ one can associate a corresponding $\OO$-Lie algebra, namely we take $\mathcal{L}_{G}:=(\OO^n,[\quad ])$ and define for ${\bf x},{\bf y}\in\OO^n$ 
\[
[{\bf x},{\bf y}]=B(x,y)-B(y,x),
\]
where $B$ is the $\OO$-bilinear form associated to the formal group law $F$ as in \ref{pre:formal_group_laws}.\ref{e-bracket-formal-law}, which by \ref{pre:formal_group_laws}.\ref{e_fomal_jacobi} defines a Lie bracket.
Note that by \ref{pre:formal_group_laws}.\ref{e_formal_adjoint} we have
\[
\mathbf{x}\mathbf{y}\mathbf{x}^{-1}=\mathbf{y}+[\mathbf{x},\mathbf{y}]+\sum_{|\alpha|,|\beta|\geq 1, |\alpha|+|\beta|\geq 3} d_{\alpha\beta}\mathbf{x}^\alpha\mathbf{y}^\beta={\bf y}\Ad_{\OO}(\bf{x})+O_y(2),
\]
where $O_y(2)$ denotes some power series whose terms have degree at least $2$ in $y_i$.
Hence $\Ad_{\OO}:G\to \Aut_{\OO}(\mathcal{L}_{G})$ gives an $\OO$-linear action of $G$ on $\mathcal{L}_{G}$ called the \textbf{adjoint action}. For $g\in G$, $\Ad_{\OO}(g)$ is given by
\begin{equation}\label{eq:adjoint_locally}
\mathbf{x}\mapsto\mathbf{x}+[\mathbf{g},\mathbf{x}]+\sum_{|\beta|=1} d_{\alpha,\beta}{\bf g}^\alpha{\bf x}^\beta.
\end{equation}

\begin{theorem}\label{pre:th:analytic_open_standard}

Every $k$-analytic group $G$ has an open $\OO$-standard subgroup.

\end{theorem}

\subsubsection{The Lie algebra of an analytic group and the adjoint action}\label{sec:Lie_algebra_analytic}

Given a $k$-analytic group $G$, let $U\leq_o G$ be an open standard subgroup with a homeomorphism: $f: U\to k^n$ such that $f(1)=0$.
Then the group law in $U$ gives a formal group law $F$ on $k^n$.
Now let us denote by $\mathcal{L}(G)$ the tangent space of $G$ at $1$.
Then the differential of $f$ at $1$ gives an isomorphism $df:\mathcal{L}(G)\to k^n$.
For every $x,y\in\mathcal{L}(G)$ we define
\[
df([x,y])=[df(x),df(y)]_{F},
\]
where the Lie bracket on the right is the one associated to the formal group law $F$. Note that since $U$ is standard, the formal group $F$ can actually be interpreted as a formal group law over $\OO$. The associated Lie bracket is thus $\OO$-linear and extends to a $k$-linear bracket. 

If $U'$ is another standard subgroup with chart $f':U'\to k^n$ and associated formal group law $F'$, then $f^{-1}\circ f'_{|f(U\cap U')}\to f'(U\cap U')$ gives a local analytic isomorphism at $0$ and induces a formal homomorphism from $F$ to $F'$. The following lemma shows that $\mathcal{L}(G)$ is well defined.

\begin{lemma}\label{pre:lm:lie_analytic_well_defined}
Let $F$ and $F'$ be two formal group laws and let $f$ be a formal homomorphism from $F$ to $F'$ (that is, $f(F(X,Y))=F'(f(X),f(Y))$). Let $f_1$ be the linear part of $f$. Then:
\[
[f_1(X),f_1(Y)]_F'=f_1([X,Y]_F).
\]
\end{lemma}

Equivalently, if we denote by $k[G]$ the $k$-algebra of analytic functions on $G$, there is a natural isomorphism from $\LL(G)$ to the space of left invariant derivations of $k[G]$, which is a Lie algebra, see \ref{sec:Lie_algebra_algebraic_group} below.
One can endow $\LL(G)$ with a Lie algebra structure via this isomorphism and this gives the same Lie algebra structure on $\LL(G)$  as before.

For $g\in G$, let $\varphi_g$ denote conjugation by $g$. Then $\varphi_g$ defines an automorphism of $G$ and by the previous lemma its differential $d\varphi_g$ gives an automorphism of $\mathcal{L}(G)$.
The adjoint action of $G$ is $\Ad_{k_v}: G\to\Aut_{k_v}(\mathcal{L}(G))$, given by $g\mapsto d\varphi_g$.
Moreover, $\Ad_{k_v}$ is an analytic map.
Indeed, since $\Ad_{k_v}$ is clearly a group homomorphism, it suffices to check that $\Ad_{k_v}$ is analytic at $e$.
Now in local coordinates, it follows from \eqref{eq:adjoint_locally} that $\Ad_{k_v}$ is analytic at $e$.

\section{Affine algebraic varieties}
We refer to \cite{Bor} for a more detailed presentation.
Let $K$ be an algebraically closed field. A $K$-space $(X,\mathcal{O}_X)$ is a topological space $X$ together with a sheaf $\mathcal{O}_X$ of $K$-algebras on $X$ whose stalks are local rings. We will often write $X$ for $(X,\mathcal{O}_X)$ and $K[X]=\OO_X(X)$.
For $x\in X$ we write $\mathcal{O}_{x}$ for the stalk over $x$. Its maximal ideal is denoted  $\mm_x$ and its residue class field by $K(x)$.
A morphism $(Y,\mathcal{O}_Y)\to (X,\mathcal{O}_X)$ is a continuous function $f:Y\to X$ together with $K$-algebra homomorphisms $f_U^V: \mathcal{O}_X(U)\to\mathcal{O}_Y(V)$ whenever $U\subset X$, $V\subset Y$ are open subsets such that $f(V)\subset U$.
These maps 
are required to be compatible with the respective restriction homomorphisms 
in $\mathcal{O}_X$ and $\mathcal{O}_Y$.
For $y\in Y$ we can pass to the limit over neighborhoods $V$ of $y$ and $U$ of $x=f(y)$ to obtain a homomorphism $f_y:\mathcal{O}_x\to\mathcal{O}_y.$
It is further 
required of a morphism that this always be a local homomorphism, that is, that $f_y(\mm_x)\subseteq\mm_y$. 

An \textbf{affine $\mathbf{K}$-algebra} $A$ is a finitely generated $K$-algebra.
We write
\[
\Spec_m(A):=\{\mm\subset A\,:\,\mm\,\mbox{is a maximal ideal}\}.
\] 
Recall that the Nullstellensatz gives a canonical bijection
\[
\begin{array}{rccc}
e:	&\Hom_K(A,K)&	\to 		&X:=\Spec(A)\\
	&g			&	\mapsto		&\ker g.
\end{array}
\]
Given $x\in X$ and $f\in A$ we will write $f(x)=e^{-1}(x)(f)$. 
The resulting function $f:X\to K$ determines $f$ modulo the nilradical of $A$.
Hence if $A$ is reduced (i.e. the nilradical of $A$ is trival), we can identify $A$ with a ring of 
$X$-valued functions on $X$.
We give $X$ the structure of a $K$-space as follows.
We endow $X$ with the Zariski topology, that is, the closed sets are of the form 
\[
V(J)=\{\mm\in X\,:\,J\subseteq\mm\}
\]
 for some $J\subseteq A$.
Moreover, to every open subset $U\subseteq X$ we associate the $K$-algebra $A_{S(U)}$, the localization of $A$ at $A\setminus{S(U)}$, where $S(U)$ is the set of $f\in A$ such that $f(x)\neq 0$ for every $x\in U$.
This gives a presheaf on the topological space $X$, which induces a sheaf $\tilde{A}$ on $X$.
It is easy to see that for every $x\in X$ the stalk at $x$ is isomorphic to the local ring $A_x$.

Any $K$-space isomorphic to one of the form $(\Spec_m(A),\tilde{A})$ will be called an \textbf{affine $\mathbf{K}$-space}.
A homomorphism $\alpha^*:A\to B$ of affine $K$-algebras induces a continuous 
function $\alpha:Y\to X$, where $Y=\Spec_m(B).$ If $U\subset X$ and $V\subset Y$ are open and 
$\alpha(V)\subseteq U$ then $\alpha(S(U))\subseteq S(V)$ so there is a natural homomorphism $A_{S(U)}\to B_{S(V)}$. These induce a morphism on the associated 
$K$-spaces $(Y,\tilde{B})\to(X,\tilde{A})$.
Hence the assignment $A\mapsto \Spec_m (A)$ gives a contravariant functor from the category of affine $K$-algebras to the category of affine $K$-spaces.
\begin{theorem}\label{th:equivalence_categories_K_schemes}
Let $X=\Spec_m(A)$, $Y=\Spec_m(B)$ be affine $K$-spaces.
The natural map $A\mapsto\tilde{A}(X)$ is an isomorphism, and the map
\[
\begin{array}{ccc}
\Mor_{K}(Y,X)	&\to		&\Hom_{K}(A,B)\\
		\alpha	&\mapsto 	&\alpha^*
\end{array}
\]
is bijective. Thus the assignment $A\to\Spec_m(A)$ is a contravariant equivalence from the 
category of affine $K$-algebras to the category of affine $K$-spaces.
\end{theorem}

Let $X$ and $Y$ be affine $K$-spaces.
The product $X\times Y$ is characterized by the property that morphisms from an affine $K$-space $Z$ 
to $X\times Y$ are pairs of morphisms to the two factors.
Applying this to $Z=\Spec_m(K$) we find that the underlying set of $X\times Y$ is the usual cartesian 
product.
From \ref{th:equivalence_categories_K_schemes} it follows immediately that the product of affine 
$K$-spaces $\Spec_m(A)$ and $\Spec_m(B)$ exists and equals 
$\Spec_m(A\otimes_K B)$.

An affine $K$-space $X=\Spec_m(A)$ is said to be an \textbf{affine $\mathbf{K}$-variety} if $A$ is reduced.
By gluing affine $K$-spaces one can construct the more general concepts of $K$-schemes and $K$-varieties. However, throughout this work we will only encounter affine $K$-spaces and affine varieties, so whenever we use the terms $K$-space and $K$-variety we mean affine $K$-space and affine $K$-variety. 

From now on let $X=\Spec_m(K[X])$ be a $K$-variety.
A subvariety $Y$ of $X$ is a closed subset $Y\subseteq X$ where $\mathcal{O}_Y=\mathcal{O}_{X|Y}$. Every such $Y$ is given by an ideal $I\subseteq K[X]$ and $Y=\Spec_m(K[X]/I)\hookrightarrow \Spec_m(K[X])$. 

We recall that a topological space $X$ is said to be irreducible if it is not empty and is not the union of two proper closed subsets.
An algebraic variety is said to be \textbf{irreducible}, if it is not empty and can not be represented as a union of two proper algebraic subvarieties.
Every algebraic variety $X$ is a union of finitely many maximal algebraic irreducible subvarieties which are called the  \textbf{irreducible components} of $X$.
We define the \textbf{dimension} of $X$, $\dim X$, as the supremum of the lengths, $n$, of chains $F_0\subset F_1\subset\ldots\subset F_n$ of distinct irreducible closed sets in $X$.

\subsubsection{$k$-structures, $k$-varieties and $k$-morphisms}\label{pre:k_structures}

Let $k$ be a subfield of $K$. Given a $K$-algebra $A$ a $\mathbf{k}$\textbf{-structure} on $A$ is a $k$-subalgebra $A_k$ such that we have a natural isomorphism $K\otimes_k A_k\cong A$.

Let $Y$ be an algebraic subvariety of the affine space $K^n=\Spec_m(K[x_1,\ldots,x_n])$.
The subvariety $Y$ is called $\mathbf{k}$\textbf{-closed} or closed over $k$ if $X$ is the set of common zeros of a finite system of polynomials with coefficients in $k$.
A subvariety $Y$ is $k$-closed if and only if it is invariant under the natural action on $K^n$ of the Galois group of the field $K$ over $k$.
Let us denote by $I(Y)\subseteq K[x_1,\ldots,x_n]$ the ideal of polynomials vanishing on $Y$, and put $I_k(Y)=I(Y)\cap k[x_1,\ldots,x_n]$.
A subvariety $Y$ is said to be $\mathbf{k}$\textbf{-defined}, or defined over $k$, if $I_k(Y)$ is a $k$-structure on $I(Y)$, i.e., if $I(Y)=K\otimes_k I_k(Y)$.
A subvariety $Y$ is closed over $k$ if
and only if it is defined over some purely inseparable field extension of $k$.
If $Y$ is defined over $k$, then $K[Y]=K\otimes_k k[Y],$ where $k[Y]=k[x_1,\ldots,x_n]/I_k(Y)$.
In other words, if $Y$ is defined over $k$, then the $k$-algebra $k[Y]$ is a $k$-structure on the $K$-algebra $K[Y]$.
A function $f\in K[Y]$ is said to be defined over $k$ if $f\in k[Y]$.

To define a $k$-structure on an affine algebraic variety $X$ is the same thing as to define an isomorphism $\alpha: X\twoheadrightarrow Y$ of the variety $X$ onto an algebraic subvariety $Y$ of an affine space which is defined over $k$.
An affine variety equipped with a $k$-structure is called an \textbf{affine $\mathbf{k}$-variety}.
The image of the $k$-algebra $k[Y]$ under
the comorphism $\alpha^*$ is called the algebra of $k$-regular functions on the $k$-variety $X$ and is denoted by $k[X]$.
Note that $k[X]$ is a $k$-structure of the $k$-algebra $K[X]$ of regular functions on $X$.
Two $k$-structures on an affine variety are called equivalent if the corresponding algebras of $k$-regular functions coincide.
We say that a morphism $\alpha: X\to X'$ between two affine $k$-varieties is defined over $k$ or that $\alpha$ is a $k$-morphism if the image of $k[X']$ under the comorphism $\alpha^*$ is contained in $k[X]$.
If $Y$ and $Y'$ are subvarieties of affine spaces defined over
$k$, then a morphism $\alpha: Y\twoheadrightarrow Y'$ is defined over $k$ if and only if $\alpha$ is defined by
polynomials with coefficients in $k$.

If $X$ and $Y$ are affine $k$-varieties then $X\times Y$ becomes a $k$-variety with $k$-structure given by $k[X]\otimes_k k[Y]$.

If $X$ is an affine $k$-variety , we denote by $X(k)=\Hom_k(k[X],k)$ the set of its $\mathbf{k}$\textbf{-rational points}.

We shall identify the $k$-variety $X$ with the set $X(K)$. In particular, a susbset $A\subseteq X$ is said to be \textbf{Zariski-dense} if it is dense in $X(K)$.

\begin{theorem}[{\cite[AG.14.4]{Bor}}]\label{lm:sub_G(k)_has_closure_defined_k}
Let $V$ be a $k$-variety and $X\subseteq V(k)$. Then the Zariski closure of $X$ is defined over $k$.
\end{theorem}

\subsubsection{Tangent spaces and differentials}
Let $X$ be an affine $K$-variety.
Recall that $\OO_x$ denotes the local ring at $x$ and $\mm_x$ is the unique maximal ideal in $\OO_x$.
Note that $K$ is naturally isomorphic to the quotient ring $K(x)=\OO_x/\mm_x$.
We denote the image of the element $f\in\OO_x$ under the projection $\OO_x\to\OO_x/\mm_x\cong K$ by $f(x)$.
A \textbf{derivation} at the point $x$ is a $K$-linear map $\delta:\OO_x\to K$ such that $\delta(fg)=f(x)\delta(g)+\delta(f)g(x)$ for every $f,g\in\OO_x$.
We define $\mathcal{T}_x(X)$ to be the set of derivations at $x$, which has a natural structure of $K$-vector space.
The space $\mathcal{T}_x(X)$ is called the \textbf{tangent space of $\mathbf{X}$ at} $\mathbf{x}$.
Note that there is a natural isomorphism between $\mathcal{T}_x(X)$ and $\Hom_K(\mm_x/\mm_x^2)$.
Equivalently, one can interpret $\mathcal{T}_x(X)$ via dual numbers as follows. Recall that given a ring $R$, the ring of dual numbers of $R$ is $R[\epsilon]\cong R[T|T^2=0]$ and there is a natural projection $R[\epsilon]\twoheadrightarrow R$. This induces a morphism $\Hom_K(K[X],K[\epsilon])\to\Hom_K(K[X],K)$ whose fibre above $x$ is $K$-isomorphic to $\mathcal{T}_x(X)$.

If $\alpha:X\to Y$ is a morphism of algebraic varieties, we can define a map $d\alpha_x: \mathcal{T}_x(X)\to\mathcal{T}_{\alpha(x)}(Y)$, called the \textbf{differential of $\mathbf{\alpha}$ at $\mathbf{x}$}, as follows: if $v\in\mathcal{T}_x(X),\, f\in\OO_{f(x)}$ and $\alpha^*$ is the comorphism of $\alpha$, we set $d\alpha_{x}(v)(f)=v(\alpha^*(f))$.
If $X$ is a $k$-variety for some subfield $k$ one can give a $k$-structure $\OO_{x,k}$ to the local ring $\OO_x$ and and this induces a $k$-structure $\mathcal{T}_{x,k}(X)$ on $\mathcal{T}_x(X)$ given by the $k$-linear derivations from $\OO_{x,k}$ to $k$ and we have $\mathcal{T}_{x,k}\cong\ker\left(\Hom_k(k[X],k[\epsilon])\to\Hom_k(k[X],k)\right)$.
Moreover, if $\alpha$ is a $k$-morphism then $d\alpha_x(\mathcal{T}_{x,k}(X))\subseteq\mathcal{T}_{f(x),k}(Y)$.
We call a point $x\in X$ a \textbf{simple point} if $\OO_x$ is a regular local ring.
If $X$ is an irreducible variety this means that $\dim X =\dim\mathcal{T}_x(X)$ (if $X$ is irreducible we always have $\dim \mathcal{T}_x(X)\geq\dim X$).
The set of simple points of $X$ is Zariski dense and open in $X$.
If all points are simple, then we say that $X$ is \textbf{smooth}.

\subsubsection{Local fields and the space of adèles}\label{sec:analytic_varieties_adeles_varieties}

Let $X$ be an affine $k$-variety and let $\iota: X\to\mathbb{A}^n$ be an embedding onto a closed subvariety of an affine space. 
We make an abuse of notation an identify $X=\iota(X)$.
For every $v\in V_k$ the subspace $X(k_v)=\Hom(k[X],k_v)$ can be given a topology induced by that of $k_v$. Indeed, since $X\subseteq\mathbb{A}^n$ we can see $X(k_v)\subseteq k_v^n$ and give $X(k_v)$ the topology as a subspace of the topological space $k_v^n$. Hence $X(k_v)$ is a closed subspace of a locally compact topological space, which is itself locally compact.
It is easily checked that this topology does not depend on the embedding $\iota$ of $X$ into $\mathbb{A}^n$.

For $v\in V_f$ let us write $X(\OO_v)$ for the points in $X$ all of whose coordinates lie in $\OO_v$, so we have $X(\OO_v)\subseteq X(k_v)$ is a compact subspace.
We define the \textbf{adèle space of} $\mathbf{X}$ as follows:
\[
X(\mathbb{A})=\{(x_v)_{v\in V_k}\in\prod_{v\in V_k}X(k_v)\;:\; x_v\in X(\OO_v)\;\textrm{for almost every}\;v\in V_f\}.
\]
We endow $X(\mathbb{A}_k)$ with a topology by considering it as a subspace of the restricted topological direct product of the $X(k_v)$ with respect to $X(\OO_v)$, that is, open sets are of the form $\prod_{v\in v_k}\OO_v$ where $\OO_v$ is open in $X(k_v)$ and $\OO_v=X(\OO_v)$ for almost every $v\in V_k$.
One can check that the space $X(\mathbb{A}_k)$ does not depend on the initial embedding $\iota$ of $X$ into an affine space.
Since the $X(\OO_v)$ are compact and $X(k_v)$ is locally compact, $X(\mathbb{A}_k)$ is a locally compact space.
In the same way that $k$ embeds diagonally  as a discrete subspace in $\mathbb{A}_k$, $X(k)$ embeds diagonally as a discrete subspace in $X(\mathbb{A}_k)$, that is, $X(k)$ is naturally identified with the set of principal adèles of $X(\mathbb{A}_k)$ via the diagonal embedding $x\mapsto (x)_{v\in V_f}$.

\section{Algebraic groups}
Let $G$ be an affine algebraic variety over an algebraically closed field $K$.
Suppose that we have morphisms of algebraic varieties $m:G\times G\to G$, $i:G\to G$ and $e: G\to K$ (i.e. $e\in G(K)$)  such that $G$ is a group with multiplication given by $m$, inversion given by $i$ and $e$ as an identity element.
Then we call $G$ an \textbf{algebraic group}.

If $G$ is a $k$-variety and the morphisms $m,i$ and $e$ are defined over $k$, then $G$ is said to be defined over $k$ or an \textbf{algebraic $\mathbf{k}$-group} or simply a $k$-group.

A subgroup $H$ of an algebraic group $G$ is called algebraic if $H$ is an algebraic subvariety of $G$.
Algebraic subgroups defined over $k$ (as algebraic subvarieties) are called $k$-subgroups.
An algebraic subgroup of an algebraic group is called
$k$-closed or closed over $k$ (respectively $k$-defined or defined over $k$) if it is $k$-closed (respectively $k$-defined) as an algebraic subvariety.

By the center of a $k$-group $G$ we mean the center of the group $G(K)$, which is an algebraic subgroup of $G$ that is $k$-closed and will be denoted $Z(G)$.

An algebraic group morphism $\alpha:G\to G'$ is defined as a group homomorphism which is
an algebraic variety morphism. 
The morphism $\alpha$ is said to be $k$-defined if it is $k$-defined as morphism of algebraic varieties.
If $G$ and $G'$ are $k$-groups we say that they are isomorphic over $k$ or $k$-isomorphic if there exists a $k$-isomorphism $\alpha: G\to G'$.

The following proposition shows that every affine algebraic $k$-group is linear.

\begin{proposition}[{\cite[Proposition I.1.10]{Bor}}]\label{pre:prop:G(k)_linear}
Let $G$ be an affine $k$-group, then $G$ is $k$-isomorphic to a $k$-subgroup of the group $\GL_n$ for some $n$. 
\end{proposition}
Algebraic subgroups of $\GL_n$ are called \textbf{linear algebraic groups}.
The algebra $k[\GL_n]$ has the form $k[GL_n]=
k[x_{11},x_{12},\ldots,x_{nn},\det^{-1}]$, where we have chosen coordinates such that for a matrix $g=(g_{ij})\in\GL_n$ we have $x_{ij}(g)=g_{ij}-\delta_{ij}$ and $\det(g)=\det((g_{ij}))$. 
Note that with respect to this coordinates, the identity matrix has zero coordinates except for $\det \Id=1$.

An algebraic group is called \textbf{connected} if it is connected in the Zariski topology.
We will denote by $G^0$ the connected component of the identity in the group $G$, that is the maximal connected algebraic subgroup of $G$.

\begin{proposition}[{\cite[I.1.2]{Bor}}]\label{pre:connected_components}
Let $G$ be an affine algebraic $k$-group.
The factor group $G/G^0$ is finite, $G^0$ is defined over $k$ and
every algebraic subgroup of finite index in $G$ contains $G^0$.
\end{proposition}

\subsubsection{Reductive, semisimple and (almost) simple groups}
Let $K$ be an algebraically closed field and $G$ a linear algebraic group. Fix an embedding $\iota: G\to\GL_n(K)$ for some $n$ and assume from now on that $G\subseteq\GL_n(K)$.
An element $g\in G$ is said to be \textbf{semisimple} if it is diagonalizable as an element of $\End_n(K)$.
An element $g\in G$ is said to be \textbf{unipotent} if $g-\Id$ is a nilpotent element in $\End_n(K)$.
These concepts are independent of the chosen embedding $\iota$.
We say that $G$ is unipotent if all of its elements are unipotent.

The \textbf{unipotent radical} $\R_u(G)$ of an algebraic group $G$ is the maximal connected algebraic unipotent normal subgroup in $G$.
If $\R_u(G)$ is trivial then $G$ is said to be a \textbf{reductive} group.
We present some results concerning reductive groups.

\begin{theorem}[{\cite[2.14 and 2.15]{BorTits1}}]\label{th:G(k)_reductive}
Let $G$ be a connected reductive $k$-group. Then
\begin{enumerate}[label=\roman*)]
\item If $k$ is infinite then $G(k)$ is Zariski dense in $G$. \label{e_G(k)_Zariski_dense}
\item $Z(G)$ is defined over $k$.\label{e_Z_k_defined}
\end{enumerate} 

\end{theorem}

The following theorem by Mostow (see \cite[7.1]{Mos} or \cite[Proposition 5.1]{BorSer}) shows that in charactersitic $0$ any $k$-group $G$ is a semidirect product of a reductive $k$-subgroup $H$ and its unipotent radical $\R_u(G)$ ($\R_u(G)$ is a $k$-defined subgroup of $G)$.
\begin{theorem}\label{th:mostow}
Suppose $\ch k=0$. Let $G$ be a $k$-group and $\R_u(G)$ its unipotent radical. Then there exists a reductive $k$-subgroup $H$ such that $G=H\ltimes R_u(G)$ (as $k$-groups).
\end{theorem}

\begin{lemma}\label{pre:lm:reductive_center}
Let $G$ a connected reductive $k$-group.
Then
$G=[G,G]\cdot Z(G)^0$.\label{e_derived_plus_Z(G)}
\end{lemma}

The \textbf{solvable radical} $\R_s(G)$ of an algebraic group $G$ is the maximal connected algebraic solvable normal subgroup in $G$.
If $\R_s(G)$ is trivial then $G$ is said to be \textbf{semisimple}.

An algebraic group $G$ is called (absolutely) \textbf{simple} (respectively (absolutely) \textbf{almost simple}) if the trivial subgroup is the only proper algebraic normal subgroup of $G$ (respectively all such subgroups are finite).
If $G$ is an algebraic $k$-group, then $G$ is called $k$-simple (respectively almost $k$-simple) if this condition holds for $k$-closed normal subgroups.

\subsubsection{Tori, rank and root systems}

A commutative algebraic group $T$ is said to be a \textbf{torus} if it is connected and $T$ is isomorphic to the product of $\dim T$ many copies of the group $\mathbb{G}_m$.
A torus $T$ is called \textbf{$\mathbf{k}$-split} or split over $k$ if it is defined over $k$ and is $k$-isomorphic to the direct product of $\dim T$ many copies of $\mathbb{G}_m$. Every torus defined over $k$ is split over some finite separable field extension of $k$.

Let $G$ be an algebraic $k$-group.
If $G$ is reductive,
then (see \cite[4.21 and 8.2]{BorTits1}) the maximal $k$-split tori of $G$ are conjugate by elements of $G(k)$ and, hence, all have the same dimension.
If $G$ is reductive we denote by $\rank_k G$ the $k$-rank of the group $G$, i.e., the common dimension of maximal $k$-split tori in $G$.
If $\rank_k G> 0$, then $G$ is said to be \textbf{$\mathbf{k}$-isotropic} or isotropic over $k$.
Otherwise $G$ is said to be \textbf{$\mathbf{k}$-anisotropic} or anisotropic over $k$.
An almost $k$-simple factor of $G$ which is $k$-isotropic (resp.
$k$-anisotropic) will be called a $k$-isotropic (resp. $k$-anisotropic) factor.
A reductive $k$-group $G$ is called $k$-split or split over $k$, if $G$ contains a $k$-split torus which is a maximal torus of the group $G$, in other words, if $\rank_k G=\rank_K G$ where $K$ is the algebraic closure of $k$.

Let $G$ be a connected semisimple algebraic $k$-group and let $T$ be a maximal $k$-split torus in $G$. Then one can associate to the pair $(T,G)$ a root system $\Phi(T,G)$. We refer to \cite[Chapters 4 and 5]{Bor} for a detailed exposition on this subject.

\subsubsection{Isogenies and simply connected groups}
A surjective group homomorphism with finite kernel is called an \textbf{isogeny}.

A morphism $f:G\to G'$ of algebraic groups is called \textbf{quasi-central} if $\ker f\subset Z(G)$. In other words, $f$ is quasi-central if there exists a mapping $\chi: f(G)\times f(G)\to G$ such that $\chi(f(x),f(y))=xyx^{-1}y^{-1}$ for all $x,y\in G$. We say that $f$ is \textbf{central} if it is quasi-central and the mapping $\chi$ is a morphism of algebraic varieties.
We say that a connected  semisimple group $G$ is \textbf{simply connected} if every central isogeny 
$f:G'\to G$ between connected algebraic groups is an algebraic group isomorphism.

\begin{theorem}\label{pre:th:simply_covering}
If $G$ is a semisimple $k$-group there exists a connected simply connected semisimple $k$-group $\widetilde{G}$ and a central $k$-isogeny $p:\widetilde{G}\to G$.
\end{theorem} 

\begin{theorem}[{\cite[2.15]{BorTits1}}]\label{pre:th:ss_decomposes_simple_factors}
Let $G$ be an algebraic $k$-group. Suppose that $G$ is connected and semisimple. Then:
\begin{enumerate}[label=\roman*)]
\item $G$ decomposes uniquely (up to permutation of the factors) into an almost direct product of connected non-commutative almost simple algebraic subgroups $G_1,\ldots,G_m$, i.e., there exists a central isogeny
\[
p:G_1\times\ldots\times G_m\to G.
\]
\label{e-semisimple-in_simples}
\item $G$ decomposes uniquely (up to permutation of the factors) into an almost direct product of connected non-commutative almost $k$-simple $k$-subgroups $G_1',\ldots,G_r'$, i.e., there exists a central $k$-isogeny
\[
p':G_1'\times\ldots\times G_r'\to G.
\]
\label{e-k-semisimple-in-k-simples}
\end{enumerate}
Moreover, the almost simple factors (almost $k$-simple factors) are invariant under isogenies ($k$-isogenies), that is, if $f:G\to H$ is an isogeny ($k$-isogeny), the image of an almost simple factor (almost $k$-simple factor) is an almost simple factor (almost $k$-simple factor).

\end{theorem}
\subsubsection{The Lie algebra of an algebraic group}\label{sec:Lie_algebra_algebraic_group}

Recall that given an $R$-algebra $A$ for some commutative ring $R$, an $R$-derivation of $A$ is an $R$-linear map $D:A\to A$ such that $D(ab)=D(a)b+aD(b)$.
We denote by $\Der_R(A)$ the set of $R$-derivations of $A$.
It is easy to check that for $D,E\in\Der_R(A)$ $[D,E]:=DE-ED$ belongs to $\Der_R(A)$ as well.
Hence $\Der_R(A)$ has a Lie algebra structure.

Let $G$ be an algebraic $k$-group.
We will write $L_G=\mathcal{T}_e(G)$ and give it the structure of a Lie algebra as follows.
For every $x\in G$, left multiplication by $g$ gives an automorphism $\lambda_g:K[G]\to K[G]$. 
We define
\[
\Lie(G):=\{D\in\Der_K(A)\;|\; \lambda_gD=D\lambda_g,\,\textrm{for every}\ g\in G\}.
\]
One checks that $\Lie(G)$ is a Lie subalgebra of $\Der_K(A)$.
Moreover, one can give $\Lie(G)$ a $k$-structure via
\[
\Lie(G)_k:=\{D\in\Lie(G)\;|\; D(k[G])\subseteq k[G]\}.
\]
Now note that evaluation at $e$ gives a map $D\mapsto D_e$ from $\Lie(G)\to\mathcal{T}_e(G)$ such that $D_e(f)=(D(f))(e)$.
Moreover, this is a linear isomorphism, so $L_G$ becomes a Lie algebra via this isomorphism.

\subsubsection{Restriction of scalars}\label{sec:restriction_scalars}

We refer to \cite[6.17-2.21]{BorTits1}, \cite[1.7]{Mar} and \cite[1.3]{Weil} for a detailed presentation. Let $l$ be a finite separable extension of $k$.
If $V$ is an $l$-variety we want to find a $k$-variety $W$ such that in some sense $W(k)$ corresponds to $V(l)$.
Suppose $V$ is an affine $l$-variety given by some ideal $I\leq l[x_1,\ldots,x_n]$, i.e., $l[V]=l[x_1,\ldots,x_n]/I$.
Let $\sigma\in\Gal(l/k)$ and define $V^{\sigma}$ to be the $l$-variety defined by $I^{\sigma}:=\sigma(I)$, i.e., $l[V^\sigma]=l[x_1,\ldots,x_n]/I^\sigma$.
Analogously if $p: V\to V'$ is an $l$-morphism of $l$-varieties we obtain from $\sigma$ a morphism $p^{\sigma}:V^\sigma\to V'^\sigma$.
Let $\{\sigma_1=\Id,\sigma_2,\ldots,\sigma_d\}$ be the set of all distinct $k$-embeddings of the field $l$ into the algebraic closure of $k$.
Then for each $l$-variety $V$ there exists a $k$-variety $W$ and an $l$-morphism $p:W\to V$ such that the map $(p^{\sigma_1},\ldots,p^{\sigma_d}):W\to V^{\sigma_1}\times\ldots V^{\sigma_d}$ is an isomorphism of algebraic varieties.
Moreover, the pair $(W,p)$ is defined uniquely up to $k$-isomorphism.

The variety $W$ is denoted by $R_{l/k}(V)$ and is called the \textbf{restriction of scalars} from $l$ to $k$ of $V$.
The restriction of scalars from $l$ to $k$, $R_{l/k}$, gives a functor from the category of $l$-varieties into the category of $k$-varieties.
The map $p:R_{l/k}(V)\to V$ has the following universal property: for a $k$-variety $X$ and an $l$-morphism $f:X\to V$ there exists a unique $k$-morphism $\phi:X\to R_{l/k}(V)$ such that $f=p\circ\phi$. Notice that if $f$ is an $l$-group morphism, then $\phi$ is a $k$-group morphism.
The restriction of the projection projection $p:W\to V$ to $W(k)$ induces a bijection between $W(k)$ and $V(l)$.
We denote by $R_{l/k}^0$ the inverse map to $p_{|V(k)}$.
If $V$ is an $l$-group, then $R_{l/k}(V)$ is a $k$-group, $p$ is an $l$-group morphism and $R_{l/k}^0:V(l)\to W(k)$ is a group isomorphism.

\begin{theorem}[{\cite[3.1.2]{Ti2}]}]\label{pre:th:k_simple_is_restriction_abs_simple}

Let $G$ be a simply connected almost $k$-simple $k$-group. Then there exist a finite separable extension $k'$ of $k$ and a connected simply connected absolutely almost simple $k'$-group $H$ such that $G=R_{k'/k}(H)$.
\end{theorem}

\section{Affine group schemes}\label{sec:affine_group_schemes}
We refer to \cite[Chapter 1 and 2]{Wat} for a more detailed presentation.
Throughout this section let $k$ be a commutative ring with $1$.
A functor $F: k$-algebras$\to \textrm{Sets}$ is said to be \textbf{representable} if there exists a $k$-algebra $A$ and a natural correspondence between $F(R)$ and $\Hom_k(A,R)$ for every $k$-algebra $R$, in this case we will say that $A$ represents $F$.
If $F$ is representable, $k[F]$ will denote a $k$-algebra which represents $F$.

An \textbf{affine group scheme over $\mathbf{k}$} or an affine group $k$-scheme is a representable functor from the category of $k$-algebras to groups.
Suppose $G$ is an affine group scheme over $k$ represented by a $k$-algebra $k[G]$, then multiplication and inversion induce $k$-algebra maps $\Delta:k[G]\to k[G]\otimes_{k}k[G]$ (comultiplication), $\epsilon:k[G]\to k$ (counit) and $S:k[G]\to k[G]$ (antipode) such that $(k[G],\Delta,\epsilon,S)$ becomes a Hopf Algebra.
Equivalently, every Hopf $k$-algebra $A$, gives an affine group $k$-scheme $G$ given by $G(R)=\Hom_k(A,R)$ for every $k$-algebra $R$, where multiplication, unit and inverse are induced from $\Delta$, $\epsilon$ and $S$.

\begin{theorem}
Affine group schemes over $k$ correspond to Hopf algebras over $k$.
\end{theorem}

\subsubsection{Dual numbers and the functor Lie}
Given a $k$-algebra $R$, the algebra of dual numbers of $R$ is the $k$-algebra $R[\epsilon]:=R[x]/\langle x^2\rangle$.
More precisely, $R[\epsilon]=R+R\epsilon$ is a free $R$-module of rank $2$ with multiplication given by $(a+b\epsilon)(c+d\epsilon)=ac+(ad+bc)\epsilon$ for $a,b,c,d\in R$.
Note that we have two natural $k$-algebra morphism $\iota_R:R\to R[\epsilon]$ and $p_R:R[\epsilon]\to R$ such that $p_r\circ\iota_R=\id_R$.

Consider now an affine group $k$-scheme $G$.
Then we have a map $G(p_R): G(R[\epsilon])\to G(R)$.
We define $\Lie(G)(R):=\ker G(p_R)$.
Let us write $I_G:=\ker \left(\epsilon: k[G]\to k\right)$.

\begin{theorem}
For every $k$-algebra $R$ there are isomorphisms of abelian groups between the following:
\begin{enumerate}[label=\roman*)]
\item $\Der_k(A, R)$.
\item $\Hom_k(I_G/I_G^2,R)$.
\item $\Lie(G)(R)$.
\end{enumerate}
These isomorphisms are natural in $R$ and compatible with the action of $R$.
\end{theorem}

\subsubsection{Affine algebraic groups and affine group schemes}
Let now $k$ be a field. Given an affine algebraic $k$-group $G$ the algebra of regular functions $k[G]$ defines a representable functor, which by abuse of notation we will continue to denote by $G$,
\[
\begin{array}{rccc}
G: &k\textrm{-algebras}&\to &\textrm{Sets}\\
&R&\mapsto &\Hom_k(k[G],R).
\end{array}
\]
Note that the multiplication, inversion and unit maps in $G$ gives comaps $m^*: k[G]\to k[G]\otimes_k k[G]$, $i^*:k[G]\to k[G]$ and $\epsilon^*:k\to k[G]$.
Since $G$ is an algebraic group these comaps make $k[G]$ a Hopf algebra over $k$.
It follows that we can see an affine algebraic $k$-group as an affine group $k$-scheme which is represented by $k[G]$.

\section{Representation Theory}\label{sec:representation_theory}
\begin{definition*}
Let $\Gamma$ be a group.
A \textbf{complex representation} of $\Gamma$ is a group homomorphism
\[
\rho: \Gamma\to \GL(V),
\]
where $V$ is a finite dimensional $\mathbb{C}$-vector space. Equivalently, a representation of $\Gamma$ is a linear action of $\Gamma$  on a finite dimensional complex vector space $V$.
We say write $\dim\rho=\dim V$ for the dimension of $\rho$.

\item If the group $\Gamma$ has a topology, we require $\rho$ to be \textbf{continuous}, where $\GL(V)$ is given the usual topology\footnote{Often in the literature $\GL(V)$ is given the discrete topology. This ensures that a continuous homomorphism from a profinite group $\Gamma$ to $\GL(V)$ has finite image. Nevertheless, one gets the same conclusion if $\GL(V)$ is given the usual topology, see Lemma \ref{lm:RG_profinite_finite_image}.} as a subset of the $\mathbb{C}$-vector space $\End_{\mathbb{C}}(V)$.

\item We say that $\rho$ is \textbf{finite} if it has finite image, otherwise we say that $\rho$ is infinite.

\item A \textbf{subrepresentation} of $\rho$ is given by a subspace $W\subseteq V$ such that $\rho(\Gamma)(W)\subseteq W$.
Note that if $W$ is such a subspace, $\rho$ gives a linear action of $\Gamma$ on $W$ and hence a representation $\rho_{|W}:\Gamma\to\GL(W)$.

\item A representation  $\rho$ is \textbf{irreducible} if it has no nontrivial subrepresentation $0\subset W\subset V$. 

\item Two representations $\rho:\Gamma\to\GL(V)$ and $\rho':\Gamma\to\GL(W)$ are \textbf{isomorphic}, $\rho\cong\rho'$, if there exists an isomorphism of vector spaces $f:V\to W$  such that $f(\rho(g)(v))=\rho'(g)(f(v))$ for every $v\in V$ and $g\in\Gamma$.
The relation $\cong$ gives an equivalence relation.
We write $\Rep_n(\Gamma)$ (respectively $\Irr_n(\Gamma)$) for the set of isomorphism classes of $n$-dimensional complex representations (respectively irreducible representations) and put $\Rep(\Gamma):=\bigcup_{n}\Rep_n(\Gamma)$ (repectively $\Irr(\Gamma)=\bigcup_n\Irr_n(\Gamma)$). We will also write $\Irr_{\leq n}(\Gamma)=\bigcup_{i\leq n}\Irr_i(\Gamma)$.
We will usually make an abuse of notation by identifying a class $\overline{\rho}\in\Irr(\Gamma)$ with one of its representatives $\rho$.

\item We say that $\rho\in\Irr(\Gamma)$ is an \textbf{irreducible constituent} of a representation $\theta:\Gamma\to\GL_n(V)$ if there exists a subrepresentation $W\subseteq V$ such that $\rho\cong\theta_{|W}$.

\end{definition*}

Given two representation $\rho_1:\Gamma\to\GL(V)$ and $\rho_2:\Gamma\to\GL(W)$ we can construct a new representation $\rho_1\oplus\rho_2:\Gamma\to\GL(V\oplus W)$, called the direct sum of $\rho_1$ and $\rho_2$, by diagonally embedding $\GL(V)$ and $\GL(W)$ into $\GL(V\oplus W)$.

A representation $\rho\in\Rep(\Gamma)$ is called \textbf{completely reducible} if it is isomorphic to a direct sum of irreducible subrepresentations.

\begin{theorem}[Maschke's Theorem]
If $\Gamma$ is a finite group, then every representation of $\Gamma$ is completely reducible.
\end{theorem}

Given $\rho\in\Rep(\Gamma)$ by Maschke's Theorem we have $\rho=m_1\lambda_1\oplus\ldots\oplus m_r\lambda_r$ for some $\lambda_i\in\Irr(\Gamma)$. The $\lambda_i$ are called the irreducible constituents of $\rho$ and $m_i\in\NN$ is the multiplicity of $\lambda_i$ in $\rho$, where $m \lambda:=\oplus_{i=1}^n\lambda$. If $\lambda\in\Irr(\Gamma)$, we write $\langle\rho,\lambda\rangle$ for the multiplicity of $\lambda$ in $\rho$ (which might be equal to zero). Both the irreducible constituents of a representation and their multiplicities are uniquely determined.

Suppose now $\Gamma'\leq\Gamma$ is a finite index subgroup such that $|\Gamma:\Gamma'|=m$.
If $\rho:\Gamma\to \GL(V)$ is a representation of $\Gamma$, we can consider the restriction of $\rho$ to $\Gamma'$, $\rho_{|\Gamma'}:\Gamma'\to\GL(V)$, which gives a representation of $\Gamma'$.  

Given $\rho:\Gamma\to\GL(V)$  we say that $\rho$ is \textbf{imprimitive} if there exist $W_1,\ldots,W_n$ subspaces of $V$ such that $V=\bigoplus W_i$ and the $W_i$ are transitively permuted by the action of $\Gamma$, otherwise we say it is \textbf{primitive}.
Given a representation $\rho\in \Rep_n(\Gamma')$ there exists a representation  $\Ind_{\Gamma'}^\Gamma(\rho):\Gamma\to \GL(V)$ having an imprimitive decomposition $V=\bigoplus_{i=1}^n W_i$, where $\Gamma'$ is the stabilizer of $W_1$ and $\rho\cong \left(\Ind_{\Gamma'}^\Gamma(\rho)_{|\Gamma'}\right)_{|W_1}$.

\begin{theorem}[Frobenius reciprocity]\label{pre:Frobenius}
Let $\Gamma$ be a group and $\Gamma'\leq\Gamma$ a finite index subgroup.
If $\rho\in\Irr(\Gamma)$ and $\rho'\in\Irr(\Gamma')$ then
\[
\langle\rho,\Ind_{\Gamma'}^\Gamma \rho'\rangle=\langle\rho_{|\Gamma'},\rho'\rangle.
\]
\end{theorem}

Consider now $N\trianglelefteq\Gamma$. If $\lambda\in\Rep(N)$, then for $g\in\Gamma$ we can define $\lambda^g$ given by $\lambda^g(n)=(gng^{-1})$.
If $\lambda$ is irreducible then so is $\lambda^g$.
This gives an action of $\Gamma$ on  $\Irr(N)$.
We have the following theorem.

\begin{theorem}[Clifford Theory]\label{pre:th:Clifford}
Let $N\trianglelefteq\Gamma$ and consider $\lambda\in\Irr(N)$.
Let $H=\St_\Gamma(\lambda)=\{g\in \Gamma\; :\; \lambda^g\cong\lambda\}$.
Then the map
\[
\begin{array}{rccc}
\Ind: &\{\psi\in\Irr(H)\; :\; <\psi_{|N},\lambda>\neq 0\} &\to &\{\rho\in\Irr(\Gamma)\; :\; <\rho_{|N},\lambda>\neq 0\}\\
&\psi &\mapsto &\Ind_H^\Gamma\psi\\
\end{array}
\]
is a bijection.
In other words, every irreducible representation $\rho\in\Irr(\Gamma)$ such that $<\rho_{N},\lambda>\neq 0$ is induced from an irreducible representation of the stabilizer of $\lambda$.
\end{theorem}

Given a representation $\rho:\Gamma\to\GL(V)$ we define its \textbf{character}
\[
\begin{array}{rccc}
\chi_\rho:&\Gamma &\to &\mathbb{C}\\
& g &\mapsto &\tr\rho(g),
\end{array}
\]
where $\tr\rho(g)$ stands for the trace of the matrix $\rho(g)$.
Note that $\dim\rho=\chi_{\rho}(1)$, so given a character $\chi$ its dimension or degree is given by $\chi(1)$.
A character of an irreducible representation is said to be an irreducible character.
It is easy to check that character are constant on conjugacy classes and that two isomorphic representations have the same character. Conversely we have the following theorem.
\begin{theorem}
Let $\Gamma$ be a finite group. Then two representations $\rho_1,\,\rho_2\in\Rep_n(\Gamma)$ are isomorphic if and only if $\chi_{\rho_{1}}=\chi_{\rho_2}$.
Moreover, $\rho\cong m_1\lambda_1\oplus\ldots m_k\lambda_k$ if and only if $\chi_{\rho}=m_1\chi_{\lambda_1}+\ldots +\chi_{\lambda_k}$.
\end{theorem}
\noindent It follows that a representation of a finite group is completely determined by its character.

A \textbf{projective representation} of $\Gamma$ is a group homomorphism
\[
\rho:\Gamma\to \PGL(V)
\]
where $V$ is a finite dimensional complex vector space and $\PGL(V)$ stands for the projective linear group of $V$. We write $\dim\rho=\dim V$.

\section{Finite groups of Lie type}\label{sec:FGLT}
We refer to \cite[Part III]{MaTes} and references therein for a more detailed presentation.
Let $p$ be a prime, let $q$ denote some $p$-power and let us write $\F$ for an algebraic closure of $\F_p$.
Suppose $G$ is a reductive algebraic group defined over $\F_q$.
The map $F^0:\Fq[G]\to\Fq[G]$ given by $x\mapsto x^q$ is an $\Fq$-endomorphism of the algebra $\Fq[G]$.
Hence, the map $\Frob_q:=F^0\otimes\Id$ gives a map $\F[G]\to\F[G]$ called the \textbf{Frobenius endomorphism} of $G$. 
A \textbf{Steinberg endomorphism} of a linear algebraic group $G$ is an endomorphism $F:G\to G$ such that $F^m=\Frob_q$ for some $\Fq$-structure of $G$ and some $m\geq 1$.
A \textbf{finite group of Lie type} is a group of the form $G^F$ where $G$ is a reductive algebraic group and $F$ is a Steinberg endomorphism of $G$.
Finite groups of Lie type give rise to an important part of the finite simple groups.  
\begin{theorem}[{Tits, see \cite[24.17]{MaTes}}]\label{pre:th:FGLT_simple}
Let $G$ be a simply connected absolutely simple linear algebraic $\F_q$-group with Steinberg endomorphism $F$.
Then, except for a finite number of exceptions (in particular whenever $q\neq 2,3$) $G^F$ is perfect and $G^F/Z(G^F)$ is a finite simple group.
\end{theorem}
\noindent A finite simple group of the form $G^F/Z(G^F)$ is called a \textbf{finite simple group of Lie type}.

Let $G$ be a connected reductive algebraic group with root system $\Phi$ and let $F$ be a Frobenius endomorphism of $G$ with respect to some $\Fq$-structure of $G$.
Then $F$ induces an automorphism $\tau$ of the Coxeter diagram associated to $\Phi$.
Moreover if $G$ is an absolutely almost simple simply connected algebraic group, $F$ is uniquely determined (up to inner automorphism) by the pair $(\tau,q)$, see \cite[Theorem 22.5]{MaTes}.
Given a reductive algebraic group $G$ and a Frobenius automorphism $F$, we will say that the finite group $G^F$ is of type $L=(\Phi,\tau)$, where $\Phi$ is the root system of $G$ and $\tau$ is the automorphism induced by $F$ on the corresponding Coxeter diagram.
By a {\bf Lie type} $L$ we mean a pair $(\Phi,\tau)$ obtained as the type of some finite group of Lie type $G^F$.
For a root system $\Phi$, we denote by $\rk\Phi$ its rank
and by $\Phi^{+}$ a choice of positive roots.
For a Lie type $L=(\Phi,\tau)$ we write $\rk L=\rk\Phi$.
Finally, given an irreducible root system $\Phi$ its Coxeter number $h$ is defined by $h=|\Phi|/\rk \Phi$.

The following is Lemma 4.11 in \cite{LuMa}.
\begin{lemma}\label{lm:projrep}
Let $\Phi$ be a perfect finite central extension of a finite simple group of Lie type $H(\Fq)/Z(H(\Fq))$. Then there exists $\delta>0$, depending neither on $q$ nor on the Lie type, such that every nontrivial complex projective representation of $\Phi$ has dimension  at least $q^{\delta}$. In particular, any proper subgroup $\Phi '\leq\Phi$ has index at least $q^{\delta}$.
\end{lemma}

\section{Commutative algebra}

Let $R$ be a ring and $M$ an $R$-module (we will consider all modules to be right modules).
If $R$ and $M$ carry topological structures, the action of $R$ on $M$ is assumed to be continous and all submodules are considered to be closed.
Given $r\in R$ and $m\in M$, we will write $[m,r]=mr-m$.
If $M'\subseteq M$ and $R'\subseteq R$, denote by $[M',R']$ the smallest $R$-module containing all elements of the form $[m,r]$, where $m\in M',\ r\in R'$.

If a group $H$ acts by automorphisms on an abelian group $M$, we call $M$ an $H$-module.
We adopt similar conventions and notation for $H$-modules as for $R$-modules. 

\begin{lemma}\label{lm:finite_generation_pro_p_modules}

Let $H$ be a pro-$p$ group, $M$ an abelian pro-$p$ group and
suppose that $M$ is an $H$-module.
If $M/[M,H]$ is finite, then $M$ is a finitely generated $H$-module.
Conversely, if $M$ has finite exponent and $M$ is a finitely generated $H$-module then $M/[M,H]$ is finite.
\end{lemma}
\begin{proof}
Suppose that $M/[M,H]$ is finite.
Let $m_1,\ldots,m_r\in M$ be representatives of $M/[M,H]$ and write $N$ for the $H$-submodule they generate.
Suppose $N\lneq M$, then there exists an $H$-module $L\leq_o M$ such that $N+L\lneq M$ (\cite[Lemma 5.3.3(c)]{RiZa}).
Pick $M'$ maximal $H$-module such that $N+L\leq M'$, then
$M/M'$ is a finite simple $H$-module (i.e., it has no nontrivial proper $H$-submodules).
But $H$ being a pro-$p$ group, this implies that $H$ acts trivally on $M/M'$ (otherwise the set of fixed points would be a nontrivial proper $H$-submodule as we have a $p$-group acting on a set of $p$-power order), that is, $[M,H]\subseteq M'$.
It follows that $N+[M,H]\leq M'$, which contradicts the construction of $N$.
Therefore, $N=M$ and $m_1,\ldots,m_r$ generate $M$.

Conversely, suppose that $M$ is finitely generated and $M$ has finite exponent. Then the action of $H$ on $M/[M,H]$ is trivial, hence it is just a finitely generated abelian group with finite exponent, hence finite. 

\end{proof}

\clearpage{\pagestyle{empty}\cleardoublepage}

\chapter{The Congruence Subgroup Problem}\label{sec:CSP}

\section{Introduction}
Let $k$ be a global field, $S$ a finite set of valuations  and $\OO_S$ the ring of $S$-integers.
Prototypical examples of global fields are $\mathbb{Q}$ and $\Fp(t)$ with rings of $S$-integers $\mathbb{Z}$ ($S=\emptyset$) and $\Fp[t]$ ($S=\{v_{\infty}\}$\footnote{$v_{\infty}$ stands for the degree valuation on $\F_p(t)$: $v_\infty(0)=0$ and $v_\infty(\frac{f}{g})=e^{-(\deg g - deg f)}$.}).
Let us consider an algebraic group $G$ defined over $k$.
The relation between the group $G(k)$ of $k$-rational points of $G$ and the arithmetic of $k$ represents a source of interesting questions.
The arithmetic of $k$ is best reflected on its rings of $S$-integers $\OO_S$.
Analogously, if we fix a $k$-embedding  $\iota: G\to \GL_n$, we can see $G(k)$ as a subgroup of $\GL_n(k)$ and set \[G(\OO_S):=G(k)\cap \GL_n(\OO_S).\]
Note that, in general, the group $G(\OO_S)$ depends on the choice of the $k$-embedding.
For instance, if we take $G=(k,+)$, the additive group of $k$, and two different embeddings $\iota_j:G\to \GL_2$, given by
\[
\iota_1:\ x\mapsto \begin{pmatrix}
						1&\frac{x}{2}\\
						0&1\\
					\end{pmatrix}
\quad \iota_2:\ x\mapsto \begin{pmatrix}
							1&\frac{x}{3}\\
							0&1\\
						\end{pmatrix},
\]
we obtain different groups $G(\OO_S)$ for each embedding. Nevertheless certain properties have an intrinsic nature that are independent of the chosen embedding. The Congruence Subgroup Property is one of these. Let us start by giving an example that lies in the origin of the subject.

Take $k=\mathbb{Q}$, $\OO_S=\mathbb{Z}$ and $G=\SL_n$ ($n\geq 2$).
For every integer  $m\in\mathbb{Z}$ we have a ring homomorphism $\mathbb{Z}\to\mathbb{Z}/m\mathbb{Z}$ and this gives a group homomorphism $\SL_n(\mathbb{Z})\to\SL_n(\mathbb{Z}/m\mathbb{Z})$.
The kernel of this map is called the principal congruence subgroup of level $m$.
Any subgroup of $\SL_n(\mathbb{Z})$ which contains a principal congruence subgroup of some level is called a congruence subgroup.
Note that $\SL_n(\mathbb{Z}/m\mathbb{Z})$ is a finite subgroup, hence every congruence subgroup has finite index in $\SL_n(\mathbb{Z})$.
The Congruence Subgroup Problem asks whether every finite index subgroup is a congruence subgroup, or more generally, how far this is from being true.

By the end of the 19 century Fricke and Klein showed that the answer is negative for $\SL_2(\mathbb{Q})$.
In the 1960s Bass, Lazard and Serre \cite{Bass_Laz_Ser} and independently Mennicke \cite{Me} solved the problem in the affirmative for $\SL_n$, $n\geq 3$.
The formulation of the problem for $G$ any linear algebraic group and $k$ a global field was introduced in \cite{Bass_Mil_Ser}.
We refer to the surveys \cite{Rag2}, \cite{Pra_Rap_developments} or the book \cite{Sury} for more references and results on the Congruence Subgroup Problem.

\section{Arithmetic groups}\label{sec:arithmetic_groups}
Throughout this section we follow \cite[1.3]{Mar}. Let $k$ be a global field. We refer to section \ref{sec:global_and_local} for notation and definitions concerning global and local fields. 
Fix a finite subset $S$ of $V_k$ that is supposed to be non-empty\footnote{Note that otherwise the ring of $S$-integers is empty.} in case $\ch k>0$.
Recall that 
\[\OO_S=\{x\in k \: |\: x\in\OO_v\:\: \forall \, v\in V_f\setminus{S}\}\] is the ring of $S$-integers of the field $k$.

Let $G$ be an algebraic $k$-group.
We fix a $k$-embedding\footnote{By Proposition \ref{pre:prop:G(k)_linear} there exists a $k$-isomorphism onto a $k$-subgroup of some $\GL_n$ but it is not unique.} $\iota: G\hookrightarrow\GL_N$ and consider from now on $G\subseteq \GL_N$.
We define the group of $S$-integral points of $G$ as
\[
G(\OO_S):=G(k)\cap\GL_N(\OO_S).
\]
Let us note that $G$ is an affine $k$-variety, that is, we can see $G$ as the vanishing set of a radical ideal $J\leq K[x_{11},\ldots,x_{NN}]$, where $K$ is an algebraic closure of $k$. Since $G$ is a $k$-group, we have $J=J_k\otimes K$ where $J_k:=J\cap k[x_{11},\ldots,x_{NN}]$.
Now since $\GL_N$ is an affine group $\mathbb{Z}$-scheme it can be seen in particular as an affine group $\OO_S$-scheme represented by $\OO_S[\GL_N]=\mathbb{Z}[\GL_N]\otimes\OO_S$, see section \ref{sec:affine_group_schemes}.
If we put $J_{\OO_S}:=J\cap\OO_S[x_{11},\ldots,x_{NN}]$, we can consider the affine group $\OO_S$-subscheme $G_{\OO_S}$ defined by the ideal $J_{\OO_S}$, i.e., represented by $\OO_S[G_{\OO_S}]:=\OO_S[x_{11},\ldots,x_{NN}]/J_{\OO_S}$.
Note that by construction we have $k[G]=\OO_S[G_{\OO_S}]\otimes_{\OO_S}k$.
We will make an abuse of notation and write $G$ for $G_{\OO_S}$ if there is no danger of confusion.
Hence, from now on we will assume that $G$ has the structure of an affine group $\OO_S$-scheme. In this sense we have
\[
G(\OO_S)=\Hom_{\OO_s}(\OO_S[G],\OO_S)
\]
and the notation is consistent with both descriptions of $G(\OO_S)$.

For every ideal $\mathfrak{q}$ of $\OO_S$ we define the \textbf{principal $\mathbf{S}$-congruence subgroup of level $\mathbf{\mathfrak{q}}$} to be\[
G(\mathfrak{q}):=G(\OO_S)\cap\GL_N(\mathfrak{q}),
\]
where $\GL_N(\mathfrak{q})$ is the subgroup of $\GL_N(\OO_S)$ consisting of matrices congruent to the identity matrix modulo $\mathfrak{q}$.
Equivalently, we have
\[
G(\mathfrak{q})=G(\OO_S)\cap\ker\left(\GL_N(\OO_S)\to\GL_N(\OO_S/\mathfrak{q})\right).
\]
If we regard $G$ as an affine group $\OO_S$-scheme then $\OO_S/\mathfrak{q}$ is an $\OO_S$-algebra and we have a map $G(\OO_S)\to G(\OO_S/\mathfrak{q})$, so equivalently we have
\[
G(\mathfrak{q})=\ker\left(G(\OO_S)\to G(\OO_S/\mathfrak{q})\right).
\]
Subgroups of $G(\OO_S)$ containing some $G(\mathfrak{q})$ ($\mathfrak{q}\neq \mathfrak{0})$ are called $\mathbf{S}$\textbf{-congruence subgroups}.
The following lemma collects some facts that will be used in the sequel.
\begin{lemma}[{\cite[1.3.1.1]{Mar}}]\label{lm:embeddings_and_arithmetic}
Let $f: G\to G'$ be a $k$-morphism (as $k$-varieties) between linear algebraic $k$-groups, $G\subseteq GL_n$ and $G'\subseteq\GL_{m}$ 
. Then
\begin{enumerate}[label=\roman*)]
\item for every non-zero ideal $\mathfrak{p}\subset\OO_S$, there exists a non-zero ideal
$\mathfrak{q}\subset\OO_S$ such that $f(G(\mathfrak{q}))\subseteq G'(\mathfrak{p})$.\label{e_congruence_well_defined}
\item Suppose that $f$ is a homomorphism of $k$-groups.
Then for every $S$-congruence subgroup $H\leq G'(\OO_S)$, $f^{-1}(H)$ contains an $S$-congruence subgroup of the group $G(\OO_S)$.
In particular, the notion of $S$-congruence subgroup is independent of the chosen embedding.\label{e_congruence_continuous}
\item Every $S$-congruence subgroup of the group $G(\OO_S)$ is of finite index in $G(\OO_S)$.\label{e_finite_index}
\item Suppose that $f$ is an isomorphism of $k$-groups.
Then the subgroups $f(G(\OO_S))$ and $G'(\OO_S)$
are commensurable. \label{e_arithmetic_well_defined}
\item For each $g\in G(k)$ the subgroups $gG(\OO_S)g^{-1}$ and $G(\OO_S)$ are commensurable.\label{e_G(k)_commensurates}
\end{enumerate}
\end{lemma}

A subgroup of $G$ is called an $\mathbf{S}$\textbf{-arithmetic subgroup} if it is commensurable with $G(\OO_S)$.
From Lemma \ref{lm:embeddings_and_arithmetic}.\ref{e_arithmetic_well_defined} it follows that the notion of $S$-arithmetic subgroup is intrinsic with respect to the $k$-structure on $G$, that is, if $f:G\to G'$ is a $k$-isomorphism of $k$-groups, then a subgroup $\Gamma$ of $G$ is $S$-arithmetic if and only if $f(\Gamma)$ is an $S$-arithmetic subgroup of $G'$.

\begin{lemma}[{\cite[1.3.1.3]{Mar}}]\label{lm:arithmetic_subgroups}
Let $f:G\to G'$ be a $k$-morphism of $k$-groups and let $\Gamma\subset G$, $\Gamma'\subset G'$ be $S$-arithmetic subgroups.
Then:
\begin{enumerate}[label=\roman*)]
\item The subgroup $\Gamma\cap f^{-1}(\Gamma')$ is of finite index in $\Gamma$, in particular, it is an $S$-arithmetic subgroup.\label{e_arithmetic_continuous}
\item For every $g\in G(k)$ the subgroups $g\Gamma g^{-1}$ and $\Gamma$ are commensurable.\label{e_G(k)_commensurates_arithmetic}
\end{enumerate}
\end{lemma}

\begin{theorem}[{\cite[Theorem 4.1]{PlaRa}}]\label{th:CSP:arithmetic_open_map}
Suppose $\ch k=0 $ and let $f:G\to H$ be a surjective morphism of $k$-groups. Then for every $S$-arithmetic subgroup $\Gamma$ of $G$, $f(\Gamma)$ is an $S$-arithmetic subgroup.
\end{theorem}

A discrete subgroup $\Gamma$ in a locally compact group $H$ is a \textbf{lattice} if $\mu_H(\Gamma\setminus{H})<\infty$, where $\mu_H$ stands for the Haar measure of $H$.
We present now some results about lattices in semisimple groups and see that $S$-arithmetic groups are a main source to find them.
Recall that $k$ is discrete in $\mathbb{A}_k$, the ring of adèles of $k$, and so is $G(k)$ in $G(\mathbb{A}_k)$, see sections \ref{sec:global_and_local} and \ref{sec:analytic_varieties_adeles_varieties}.
Moreover, if $G$ is connected and reductive we have a characterisation of the groups $G$ for which $G(k)$ is a lattice in $G(\mathbb{A}_k)$. Let $X_k(G)$ be the group of $k$-rational characters of $G$, i.e., the group of $k$-homomorphisms from $G$ to $\mathbb{G}_m$. We have the following theorem. 

\begin{theorem} [\cite{Bor3}  for $\ch k=0$ and \cite{Be2} and \cite{Har2} for $\ch k\neq 0$]\label{th:G(k)lattice_char}
Let $G$ be a connected reductive $k$-group.
Then $G(k)$ is a lattice in $G(\mathbb{A}_k)$ if and only if $X_k(G)=1$.
\end{theorem}

It is natural to ask whether $G(\OO_S)$ gives a lattice in an appropriate subgroup of $G(\mathbb{A}_k)$.
Before proceeding we need to understand when the local factor $G(k_v)$ gives a compact group.
Recall that for $v\in V_k$ we say that $G$ is $k_v$-isotropic if $G(k_v)$ contains a nontrivial $k_v$-split subtorus.
Otherwise we say that $G$ is $k_v$-anisotropic.
We write $\rank_{k_v}G$ for the dimension of a maximal $k_v$-split subtorus of $G$, hence $G$ is isotropic (anisotropic) if and only if $\rank_{k_v}G\geq 1\ (=0)$.
Put $\mathcal{A}=\mathcal{A}(G):=\{v\in V_k\ :\ G\ \mbox{is}\ k_v\mbox{-anisotropic}\}$.
The following theorem, originally showed by Bruhat and Tits, can be found in \cite{Pra_isotropic_char}.

\begin{theorem}\label{th:kv_anisotropic_compact}
Let $G$ be a reductive $k$-group.
Then $G(k_v)$ is compact if and only if $G$ is $k_v$-anisotropic.
\end{theorem}

For every $S\subseteq V_k$ let us denote by $G_S$ the subgroup in $G(\mathbb{A}_k)$ consisting of the adèles whose $v$-component is equal to the identity for every $v\in V_k\setminus{S}$.
Note that if $S$ is finite then $G_S=\prod_{v\in S} G(k_v)$.
Let us note that the $S$-integral points $G(\OO_S)$ can be written as

\begin{equation}\label{eq:OS_description}
G(\OO_S)=G(k)\cap (G_{S\cup V_\infty}\cdot\prod_{v\in V_f\setminus{S}}G(\OO_v)).
\end{equation}

Let $G$ be a reductive $k$-group.
Note that for $v\in\mathcal{A}\cap V_f$, $G(k_v)$ is a compact group and hence $G(\OO_v)$, being an open subgroup, is of finite index.
Since $\mathcal{A}$ is finite (\cite[Lemma 4.9]{Spr}), it follows that $G(\OO_S)$ is of finite index in $G(\OO_{S\cap\mathcal{A}})$.

On the other hand, since $G(k)$ is discrete in $G(\mathbb{A}_k)$ it follows from \eqref{eq:OS_description} that $G(\OO_S)$ is discrete in 
\[
\prod_{v\in S\cup V_\infty}G(k_v)=\prod_{v\in (S\cup V_\infty)\setminus{\mathcal{A}}}G(k_v)\times \prod_{v\in (S\cup V_\infty)\cap\mathcal{A}}G(k_v).
\]
The second factor is a compact group, hence $G(\OO_S)$ (identified with its image under the diagonal embedding into $G_S$) is discrete in $G_S$ whenever $V_\infty\setminus{\mathcal{A}}\subseteq S$.

The above remarks together with Theorem \ref{th:G(k)lattice_char} give

\begin{theorem}\label{th:G(OS)lattice_char}
Let $G$ be a connected reductive $k$-group and let $V_\infty\setminus{\mathcal{A}(G)}\subseteq S$.
Then $G(\OO_S)$ is a lattice in $G_S$ if and only if $X_k(G)=1$.
\end{theorem}

Let us note that for $G$ semisimple, the condition $X_k(G)=1$ is automatically satisfied.
We begin now to investigate the relation between $G(\OO_S)$ being infinite, being Zariski-dense in $G$ and having Strong Approximation.

\begin{proposition}\label{prop:G(OS)_Zariski_dense}
Let $G$ be a connected semisimple $k$-group and suppose $V_\infty\setminus\mathcal{A}(G)\subseteq S\subseteq V_k$. Let $\Gamma$ be an $S$-arithmetic subgroup of the group $G(k)$. Suppose $\sum_{v\in S}\rank_{k_v}G'>0$ for every $k$-simple factor $G'$ of $G$. Then $\Gamma$ is infinite and  Zariski dense in $G$.
\end{proposition}
\begin{proof}
Note that it suffices to consider the case where $G$ is $k$-simple.
Let $H$ be the Zariski closure of $\Gamma$ and $H^0$ be the connected component of the identity in $H$.
By Theorem \ref{lm:sub_G(k)_has_closure_defined_k}, $H$ is defined over $k$ and so is $H^0$ (\ref{pre:connected_components}).
Note that $g\Gamma g^{-1}$ is commensurable with $\Gamma$ for every $g\in G(k)$ (\ref{lm:embeddings_and_arithmetic}).\ref{e_G(k)_commensurates}) and  since $H^0$ is the intersection of all algebraic subgroups of finite index in $H$, it follows that $G(k)$ normalizes $H^0$.
Moreover, since $G(k)$ is Zariski dense in $G$ (\ref{th:G(k)_reductive}.\ref{e_G(k)_Zariski_dense}), $H^0$ must be a normal subgroup of $G$.

Now by assumption $\sum_{v\in S}\rank_{k_v}G>0$ and this implies (recall Theorem \ref{th:kv_anisotropic_compact}) that $\prod_{v\in S}G(k_v)$ is not compact.
Since $G(\OO_S)$ is a lattice in the locally compact group $G_S$ (Theorem \ref{th:G(OS)lattice_char}), it follows that $G(\OO_S)$, $\Gamma$  and $H$ are all infinite.
But then $H^0$ is an infinite normal $k$-subgroup of $G$, which is a $k$-simple group. Thus, we must have $H^0=G$, as desired.
\end{proof}

An algebraic $k$-group $G$ is said to have \textbf{Strong Approximation} with respect to a finite set of valuations $S\subseteq V_f$ if the group $G(k)G_S$ is dense in $G(\mathbb{A}_k)$. Note that this in particular implies that $G(\OO_S)$ is dense in $\prod_{v\in V_f\setminus{S}}G(\OO_v)=G(\widehat{\OO}_S)$.
We will make use of the Strong Approximation Theorem.
This was shown by Kneser (\cite{Kne}) and Platonov (\cite{Pla}) for $\ch k=0$ and by Margulis (\cite{Mar1} and \cite{Mar}) and Prasad (\cite{Pra}) for $\ch k>0$.

\begin{theorem}[Strong Approximation]\label{th:Strong_Approximation}
Let $G$ be a connected simply connected semisimple $k$-group.
Let $S\subset V$ be a finite set of valuations of $k$.
Suppose that $G^i_S=\prod_{v\in S}G^i(k_v)$ is non-compact for every almost $k$-simple factor $G^i$. Then $G(k)\cdot G_S$ is dense in $G(\mathbb{A}_k)$.
\end{theorem}

As a direct consequence of the previous results we obtain:

\begin{theorem}\label{th:G(OS)_infinite_iff_G_S_noncompact}
Let $G$ be a connected semisimple $k$-group and suppose that $V_\infty\setminus\mathcal{A}\subseteq S\subseteq V$. Then $G(\OO_S)$ is infinite if and only if $G_S$ is noncompact. Moreover, if $G$ is simply connected and $k$-simple, then $G(\OO_S)$ is infinite if and only if $G$ has Strong Approximation.
\end{theorem}
\begin{proof}
By Theorem \ref{th:G(OS)lattice_char}, $G(\OO_S)$ is a lattice in $G_S$. Since a compact group does not admit an infinite lattice, the direct implication follows. For the converse note that Theorem \ref{th:kv_anisotropic_compact} forces one of the almost $k$-simple factors of $G$ to be noncompact.
But then by Proposition \ref{prop:G(OS)_Zariski_dense} $G(\OO_S)$ must be infinite.
For $G$ simply connected and $k$-simple Theorem \ref{th:Strong_Approximation} gives that $G$ has Strong Approximation if and only if $G_S$ is noncompact.
\end{proof}

Finally the following lemma describes the relations between arithmetic groups under restriction of scalars, see section \ref{sec:restriction_scalars}.

\begin{lemma}[{\cite[3.1.4]{Mar}}]\label{lm:restriction_scalars_commensurable}
Let $k'$ be a finite separable field extension of $k$, let $H$ be a $k'$ group and let $S\subset V_k$.
For each $v\in V_k$, we denote by $v'$ the set of valuations of the field $k'$ extending the valuation $v$.
Put $S'=\bigcup_{v\in S}v'$.
\begin{enumerate}[label=\roman*)]
\item \label{e-kv-iso-restriction}For each $v\in V_k$ there exists a natural $k_v$-isomorphism 
\[
f_v:\ R_{k'/k}H\to\prod_{w\in v'}R_{k'_w/k_v}H.
\]
The isomorphisms $f_v^{-1}\circ R^0_{k'_w/k_v}$, for $v\in V_k$, $w\in v'$ induce a topological group isomorphism of the adele groups associated with $H$ and $R_{k'/k}H$, whose restriction to $H(k')$ agrees with $R^0_{k'/k}$.
\item \label{e-commensurable-restriction}
The subgroups $R^0_{k'/k}(H(\OO_{S'}))$ and $(R_{k'/k}(H))(\OO_S)$ are commensurable.
\end{enumerate} 
\end{lemma}

\section{The general formulation}\label{sec:CSP_formulation}
Let $k$ be a global field, fix $S$ a non-empty set of valuations containing $V_\infty$ and let $\OO_S$ be the ring of $S$-integers.
Let $G$ be an algebraic $k$-group which is assumed to be embedded in $\GL_N$ or to have an affine group $\OO_S$-scheme structure as in section \ref{sec:arithmetic_groups}.
Let us consider the arithmetic group $G(\OO_S)$. 

Let $\mathfrak{N}_a$ be the family of all normal subgroups of finite index of $G(\OO_S)$.
One can consider $\mathfrak{N}_a$ as a fundamental system of neighbourhoods of the identity (see Proposition \ref{pre:base_group_topology_char}), this induces a topology $\tau_a$ on $G(\OO_S)$, which is compatible with the group structure, known as the profinite topology.

A rich source of normal subgroups of finite index of $G(\OO_S)$ is provided by the arithmetic of the ring $\OO_S$.
Namely, for every ideal $\mathfrak{q}$ of $\OO_S$ we can consider the principal $S$-congruence subgroup $G(\mathfrak{q})$.
If we take the family $\mathfrak{N}_c$ of all principal $S$-congruence subgroups of $\Gamma$ as a fundamental system of neighbourhoods for the identity, this defines a topology $\tau_c$ on $\Gamma$, which is compatible with the group structure, called the $S$-congruence topology.
Note that every $S$-congruence subgroup has finite index in $G(\OO_S)$ (\ref{lm:embeddings_and_arithmetic}.\ref{e_finite_index}) , so we have $\mathfrak{N}_c\subseteq\mathfrak{N}_a$.

The {\bf Congruence Subgroup Problem} arises from the following  question:
\begin{center}
Is every finite index subgroup in $\Gamma$ an $S$-congruence subgroup?
\end{center}
Clearly, an affirmative answer is equivalent to $\tau_c=\tau_a$.
This equality can be reinterpreted as follows.
Observe that $G(\OO_S)$ admits completions $\widehat{G(\OO_S)}$ and $\overline{G(\OO_S)}$ with respect to the  (left) uniform structures induced by $\tau_a$ and $\tau_c$, see section \ref{sec:completions_and_group_completions}.
By Proposition \ref{pre:pro:profinite_uniform_completion_coincide} we can describe these completions via projective limits as follows:
\[
\widehat{G(\OO_S)}=\varprojlim_{N\in\mathfrak{N}_a}G(\OO_S)/N\quad \mbox{and}\quad \overline{G(\OO_S)}=\varprojlim_{\mathfrak{0}\neq\mathfrak{q}} G(\OO_S)/G(\mathfrak{q}).
\]
Since $\tau_a$ is finer than $\tau_c$ we have a surjective (uniformly) continuous homomorphism  $\pi: \widehat{G(\OO_S)}\twoheadrightarrow{\overline{G(\OO_S)}}$ (Corollary \ref{pre:cor:extension_to_completion}).
It follows that $\tau_a=\tau_c$ if and only if the $S$-congruence kernel
\[C(G,S):=\ker\pi\]
is the trivial group.

Note that the definitions of $G(\OO_S)$ and $G(\mathfrak{q})$ depend on the embedding $\iota$.
Nevertheless $C(G,S)$ is independent of the chosen embedding.
To see this let us reconstruct $C(G,S)$ in terms of the group $G(k)$.
We will analogously use two families of subgroups of $G(k)$.

We say that a subgroup $\Gamma\subset G(k)$ is an $S$-congruence subgroup if $G(\mathfrak{q})\subseteq\Gamma$ for some $\mathfrak{q}\neq 0$.
The family $\widetilde{\mathfrak{N}}_c$ of all $S$-congruence subgroups of $G(k)$ serves as a fundamental system of neighbourhoods of the identity and defines a topology $\widetilde{\tau}_c$ on $G(k)$.
Indeed, $\widetilde{\mathfrak{N}}_c$ clearly satisfies \ref{GB1} and \ref{GB2} in \ref{pre:base_group_topology_char} and \ref{GB3} follows from Lemma \ref{lm:embeddings_and_arithmetic}.\ref{e_congruence_well_defined} and the fact that conjugation by an element $g\in G(k)$ is a $k$-automorphism of $G$.
Clearly $\mathfrak{N}_c\subset\widetilde{\mathfrak{N}}_c$, hence $G(\OO_S)$ with the topology $\widetilde{\tau}_c$ is a on open subgroup of $G(k)$.

A subgroup $\Gamma\subset G(k)$ is an $S$-arithmetic\footnote{In general a subgroup $H\leq G(\overline{k})$ is called $S$-arithmetic if $H$ and $G(\OO_S)$ are commensurable.} subgroup if $\Gamma$ and $G(\OO_S)$ are commensurable.
The family $\widetilde{\mathfrak{N}}_a$ of all $S$-arithmetic subgroups of $G(k)$ serves as a fundamental system of neighbourhoods of the identity and defines a topology $\widetilde{\tau}_a$ on $G(k)$. Indeed, $\widetilde{\mathfrak{N}}_c$ clearly satisfies \ref{GB1} and \ref{GB2}, while \ref{GB3} follows from Lemma \ref{lm:embeddings_and_arithmetic}.\ref{e_G(k)_commensurates}.
Clearly $\mathfrak{N}_a\subset\widetilde{\mathfrak{N}}_a$, hence $G(\OO_S)$ with the topology $\widetilde{\tau}_a$ is an open subgroup of $G(k)$.

It follows from \ref{e_congruence_well_defined} and \ref{e_arithmetic_well_defined} in Lemma \ref{lm:embeddings_and_arithmetic} that the definition of the above topologies on $G(k)$ does not depend on the embedding $\iota$.
As above, $\widetilde{\tau}_a$ and $\widetilde{\tau}_c$ induce two uniform structures in $G(k)$ and these give  respective completions which will be denoted respectively by $\widehat{G(k)}$ and $\overline{G(k)}$. Since $\widetilde{\tau}_a$ is finer than $\widetilde{\tau}_c$ the identity map in $G(k)$ induces a continuous surjective homomorphism 
\[\widetilde{\pi}: \widehat{G(k)}\to\overline{G(k)}.\]
Moreover, note that $\widehat{G(\OO_S)}$ ,  respectively $\overline{G(\OO_S)}$, is naturally isomorphic to the closure of $G(\OO_S)$ with respect to $\widetilde{\tau}_a$, respectively $\widetilde{\tau}_c$, as a subgroup of $\widehat{G(k)}$ , respectively $\overline{G(k)}$--  since $\bigcap_{N\in \mathfrak{N}_c}=1$ the group $G(k)$, and in particular $G(\OO_S)$, embedds both in $\widehat{G(k)}$ and $\overline{G(k)}$.

\begin{remark}
Note that by construction $\overline {G(k)}$ is nothing but the closure of $G(k)$ in the topological group $G(\mathbb{A}_S)$, see section \ref{sec:analytic_varieties_adeles_varieties}. 
\end{remark}

\begin{lemma} \label{lm:CSP_reduces_to_G(Os)}
Let $\widetilde{\pi}$ and $\pi$ be as above. Then $\ker\widetilde{\pi}\cong\ker{\pi}$. In particular, $C(G,S)$ is independent of the embedding $\iota: G\hookrightarrow \GL_N$.
\end{lemma}

\begin{proof}
Suppose $x\in \widehat{G(k)}$ is a minimal Cauchy filter on $G(k)$ converging to $1$ with respect to $\widetilde{\tau}_c$.
This implies that $\widetilde{\mathfrak{N}}_c$ and, in particular, all principal congruence subgroups and $G(\OO_S)$ belong to $x$.
Hence, if we  look at $x\cap G(\OO_S):=\{F\cap G(\OO_S)\;:\; F\in x\}$ this gives a filter on $G(\OO_S)$ which is Cauchy with respect to (the uniform structure given by) $\tau_c$ and converges to $1$ in $\tau_c$.
It follows that $x\cap G(\OO_S)$ gives rise to a point in $\widehat{G(\OO_S)}$, which is isomorphic to the closure of $G(\OO_S)$ in $\widehat{G(k)}$.
This implies that $x\in\widehat{G(\OO_S)}$ and so $x\in C(G,S)$.

\end{proof}

Let us note that the proof of the previous lemma actually shows the following.

\begin{lemma}\label{lm:CSP_reduces_to_open_subgroup}
Let $\Gamma\leq G(k)$ be a subgroup, and denote by $\widehat{\Gamma}$ and $\overline{\Gamma}$ the closure of $\Gamma$ in $\widehat{G(k)}$ and $\overline{G(k)}$. We have a natural map $\widehat{\Gamma}\to\overline{\Gamma}$.
Suppose that $\Gamma$ is open both in $\widetilde{\tau}_a$ and $\widetilde{\tau}_c$.
Then $\ker\widetilde{\pi}\cong C(G,S)\cong \ker (\widehat{\Gamma}\to\overline{\Gamma})$.
\end{lemma}

The Congruence Subgroup Problem therefore reduces to the study of the group
\[
C(G,S)=\ker\left( \widehat{G(\OO_S)}\to\overline{G(\OO_S)}\right).
\]

Whenever $C(G,S)$ is trivial we say that $G(\OO_S)$ has the \textbf{Congruence Subgroup Property (CSP)}.
If $C(G,S)$ is finite, we say that $G(\OO_S)$ has the \textbf{weak Congruence Subgroup Property (wCSP)}.
Even the finiteness of $C(G,S)$ has some interesting consequences, see Theorem \ref{th:CSP_implies_PRG} and  \ref{th:log_subgroup_growth}.

\section{Reductions and necessary conditions}

We begin by describing the behaviour of $C(G,S)$ for different $G$.
\begin{proposition} \label{lm:congruence_kernel_functor}
The assignment $G\mapsto C(G,S)$ gives a  functor $C(\_,S)$ from the category of affine $k$-groups to the category of profinite groups. 
\end{proposition}
\begin{proof}
It follows from Lemma \ref{lm:CSP_reduces_to_G(Os)} that $C(G,S)=\ker\pi$ is a closed subgroup of $\widehat{G(\OO_S)}$, hence a profinite group.
Let $f:G\to H$ be a $k$-morphism of $k$-groups. 
Note that $f_{|G(k)}:G(k)\to H(k)$ is continuous with respect to both the congruence and arithmetic topologies ( \ref{lm:embeddings_and_arithmetic}.\ref{e_congruence_continuous} and \ref{lm:arithmetic_subgroups}.\ref{e_arithmetic_continuous}) and so it extends to a continuous function between the respective completions.
Given $x\in \widehat{G(k)}$ a minimal Cauchy filter, since $f_{|G(k)}$ is continuous, $f(x):=\{f(F)\;:\;F\in x\}$ gives a Cauchy filter basis on $H(k)$.
Thus we define $C(\_,S):\widehat{G(k)}\to \widehat{H(k)}$ by  assigning to every element $x\in \widehat{G(k)}$ the minimal Cauchy filter containing $f(x)$. 
Let us check that under this map $C(G,S)$ maps to $C(H,S)$.
If $H(\mathfrak{p})$ is a principal $S$-congruence subgroup in $H(k)$. Then by Lemma \ref{lm:embeddings_and_arithmetic}.\ref{e_arithmetic_well_defined} there exists an ideal $\mathfrak{q}\subset\OO_S$ such that $f(G(\mathfrak{q}))\subseteq H(\OO_S)$. 
Now if $x\in C(G,S)$ then $x$ converges to $1$ in $\widetilde{\tau}_c$ and this implies $G(\mathfrak{q})\in x$.
Hence $f(G(\mathfrak{q}))\in f(x)$ and thus any filter contaning $f(x)$ contains $H(\mathfrak{p})$. It follows that $C(\_,S)(f)(C(G,S))\subseteq C(H,S)$ as desired.

\end{proof}

\begin{lemma}\label{lm:CSP_connected}
If $G^0$ is the connected component of the identity in $G$, the inclusion map $i:G^0\to G$ induces an isomorphism $C(G^0,S)\to C(G,S)$.
\end{lemma}
\begin{proof}
First let us note that $G^0$ is a normal $k$-subgroup of $G$ of finite index (Proposition \ref{pre:connected_components}).
Hence we may keep the chosen embedding for $G^0$ as well and we have $G^0(\OO_S)=G(\OO_S)\cap G(k)$ and $G^0(\mathfrak{q})= G(\mathfrak{q})\cap G^0(k)$ for every ideal $\mathfrak{q}\subseteq\OO_S$.
Moreover the quotient $G(k)/G^0(k)$ is a finite group, hence $G^0(\OO_S)$ has finite index in $G(\OO_S)$.
This implies that $G^0(\OO_S)$ and $G^0(k)$ are open subgroups of $G(k)$ with respect to the arithmetic topology of $G(k)$.
By Lemma \ref{lm:CSP_reduces_to_open_subgroup} it suffices to show that $G^0(k)$ is also open with respect to the congruence topology in $G(k)$, i.e., there exists an ideal $\mathfrak{q}\subseteq\OO_S$ such that $G(\mathfrak{q})\subseteq G^0(k)$.

To prove the claim let us denote by $G^0$, $G^1$, $\ldots$, $G^n$ the connected components of the algebraic $k$-group $G$.
The identity component $G^0$ is defined over $k$ (Proposition \ref{pre:connected_components}) and this implies that $G':=\cup_{i=1}^{n}G^i$ is defined over $k$ as well\footnote{In an algebraic group the connected and irreducible components coincide and are disjoint. The irreducible components are defined over the separable closure $k_s$ and a component is defined over $k$ if and only it is stable under the action of $\Gal(k_s/k)$, which permutes the irreducible components. Now since $G^0$ is defined over $k$ this implies that $G\setminus{G^0}$ is stable under $\Gal(k_s/k)$ and hence defined over $k$, see \cite[Proposition I.1.2]{Bor} and references therein for details.}.
It follows that we have a decomposition $k[G]=k[G^0]\oplus k[G']$ as a direct sum of algebras.
Now this implies that $1=e_0+e$ for some $e_0\in k[G^0]$ and $e\in k[G']$.

It follows that $e(g)=1$, for $g\in G'(k)$, and $e(g)=0$ for $g\in G^0(k)$.
Now we know that $k[G]=k[x_{ij},\,\det(x_{ij})^{-1}\,|\,1\leq i,j\leq n]$, where $x_i(1)=0$ for every $i\in I$. Then $e=f(x_{i})$ for some polynomial $f\in k[X_i]$.
Pick an ideal $\mathfrak{q}\subseteq\OO_S$ such that $c\notin\mathfrak{q}$ for every coefficient $c\in k$ of $f$.
We claim that $G(\mathfrak{q})\cap G'(k)=\emptyset$ and hence $G(\mathfrak{q})\subseteq G^0(k)$.
Indeed, since $e(1)=0$, we must have $e(g)=0$ $\mod\mathfrak{q}$ for $g\in G(\mathfrak{q})$, but $e(g)=1$ for every $g\in G'(k)$.
Hence we must have $G(\mathfrak{q})\cap G'(k)=\emptyset$.
\end{proof}

The following lemma describes the behaviour of $C(G,S)$ when $G$ is a semidirect product of $k$-groups.

\begin{lemma}[{\cite{Rag}}] \label{lm:CSP_semidirect}
Let $G$ be an algebraic $k$-group, such that $G$ is the semidirect product of a normal $k$-subgroup $N$ and a $k$-subgroup $H$.
Then $C(G,S)$ is a semidirect product of $C(H,S)$ and a quotient of $C(N,S)$. In particular, if $C(N,S)$ is trivial then $C(G,S)\cong C(H,S)$.
\end{lemma}

\section{The characteristic 0 case}

Throughout this section we assume that $\ch k=0$, that is, $k$ is a number field.
The aim of this section is to reduce the Congruence Subgroup Problem to the case of semisimple algebraic groups.
By Lemma \ref{lm:CSP_connected} we may restrict to connected groups. We begin by investigating the additive and multiplicative  groups $\mathbb{G}_a$ and $\mathbb{G}_m$.

\begin{lemma}\label{lm:CSP_additive_trivial}

Let $G=\mathbb{G}_a$, then for any $S\subseteq V_f$ we have  $C(G,S)=1$. 
\end{lemma}
\begin{proof}
Suppose $H\leq \OO_S$ has index $m$. Then $m\OO_S\leq H$.
\end{proof}

To show that the congruence kernel for the multiplicative group is also trivial we will apply the following theorem by Chevalley (\cite[Theorem 1]{Che}).

\begin{theorem} \label{th:Chevalley}
Suppose $H\leq k^*$ is a finitely generated subgroup. Then, given $m\in\NN$ and any non-zero ideal $\mathfrak{p}\subset\OO$, there exists a non-zero ideal $\mathfrak{q}\subset\OO$ coprime to $\mathfrak{p}$ such that for every $x\in H$ such that $x\equiv 1\mod\mathfrak{q}$, we have $x\in H^m$.
The statement $x\equiv 1\mod \mathfrak{q}$ is to be interpreted in $k$, that is, $x-1\in\frac{a}{b}\mathfrak{q}$ where $a,b\in\OO$ and $(b)$ is coprime to $\mathfrak{q}$.
\end{theorem}

\begin{lemma}\label{lm:CSP_multiplicative_trivial}
Let $G=\mathbb{G}_m$. Then for any finite $S\subseteq V_f$ we have $C(G,S)=1$.
\end{lemma}
\begin{proof}
Let $N\leq \OO_S^*$ be a subgroup of index $m$. Then we have $(\OO_S^*)^m\leq N$. It is well known that $\OO_S^*$ is a finitely generated group.
Hence we can apply Theorem \ref{th:Chevalley} to $H=\OO_S^*$ and $\mathfrak{p}=\prod_{v\in S}\mathfrak{p}_v$ to obtain an ideal $\mathfrak{q}$ coprime to $\mathfrak{p}$ such that for every $x\in \OO_S^*$ such that $x\equiv 1\mod\mathfrak{q}$ we have $x\in(\OO_S^*)^m$. Consider $\mathfrak{q}^S:=\mathfrak{q}\OO_S$. Now observe that $\mathfrak{q}^S$ is an ideal in $\OO_S$ ($\mathfrak{q}$ is coprime to $\mathfrak{p})$ and that $G(\mathfrak{q}^S)\subseteq\{x\in\OO^*\ :\ x\equiv 1\mod \mathfrak{q}\}\leq (\OO_S^*)^m\leq N$, i.e., $N$ is an $S$-congruence subgroup.

\end{proof}

\begin{lemma}\label{lm:CSP_unipotent_trivial}
Let $G$ be a connected unipotent $k$-group. Then for any $S\subseteq V_f$ we have $C(S,G)=1$.
\end{lemma}
\begin{proof}

We may suppose that $G$ is an algebraic $k$-subgroup of $\UU_n$ \cite[Theorem 8.3]{Wat}.
One can construct a central normal series of algebraic $k$-groups
\[
\UU_n=U_0\supset U_1\supset\ldots \supset U_{\frac{n(n-1)}{2}}=1
\]
such that $U_i/U_{i+1}\cong \mathbb{G}_a$ for every $i<\frac{n(n-1)}{2}$.
This gives a central normal series
\[
G=G_0\supset G_1\supset \ldots\supset G_d=1
\]
such that each quotient $G_i/G_{i+1}$ is a isomorphic to $\mathbb{G}_a$ (note that $\mathbb{G}_a$ has no algebraic subgroups) and $d=\dim G$.
We apply induction on $d$. For $d=1$ this is Lemma \ref{lm:CSP_additive_trivial}.
Suppose the result holds for $d=r-1$, in particular $C(G_1,S)=1$. Then $G$ is of the form $\mathbb{G}_a\ltimes G'$ where $G'$ is a unipotent $k$-group of dimension $d-1$. 
By induction and Lemma \ref{lm:CSP_additive_trivial} we can apply Lemma \ref{lm:CSP_semidirect} to obtain $C(G,S)=0$.
\end{proof}

Let us now consider the case of a connected $k$-group $G$.
By Theorem \ref{th:mostow}, $G$ is the semidirect product of a semisimple $k$-subgroup and its unipotent radical.
Hence, the previous lemma together with Lemma \ref{lm:CSP_semidirect} shows that for number fields the study of $C(G,S)$ reduces to the case $G$ reductive.
The following lemma shows that we can futher reduce it to the case of a semisimple group.

\begin{lemma}[{\cite[2.4]{Pra_Rap_developments}}]
Let $G$ be a connected reductive $k$-group.
Then
\[
C(G,S)\cong C([G,G],S).
\]
\end{lemma}

\section{The semisimple case}

Throughout this section we assume that $G$ is a connected semisimple $k$-group.
Note that by Lemma \ref{lm:CSP_connected}, the assumption $G$ connected is harmless regarding CSP. 
By Theorem \ref{pre:th:ss_decomposes_simple_factors}, $G$ is the almost direct product of its almost $k$-simple factors $G_1,\ldots,G_m$. We will further assume that $G_i(\OO_S)$ is infinite, equivalently $(G_i)_S$ is non-compact (Theorem \ref{th:G(OS)_infinite_iff_G_S_noncompact}), for every $1\leq i\leq m$.
Let us note that otherwise $C(G_i,S)=1$. 

The aim of this section is to reduce, when $\ch k=0$, the Congruence Subgroup Problem for semisimple $k$-groups to the following setting:
\begin{enumerate}[label=(A\arabic*)]
\item $G$ is simply connected.\label{A_simply}
\item $G$ is almost $k$-simple. \label{A_k-simple}
\item $G$ has the Strong Approximation Property.\label{A_strong}
\item $G$ is absolutely almost simple.\label{A_abs_simp}
\item For every $v\in S\cap V_f$ the group $G(k_v)$ is noncompact, equivalently (Theorem \ref{th:kv_anisotropic_compact}), for every $v\in S\cap V_f$ the group $G$ is $k_v$-isotropic.\label{A_anisotropics}
\end{enumerate}

\begin{lemma}\label{lm:CSP_implies_simplyconnected}
Let $G$ be as above and suppose $\ch k=0$.
Suppose further that $C(S,G)$ is finite.
Then $G$ is simply connected.
\end{lemma}
\begin{proof}
By Theorem \ref{pre:th:ss_decomposes_simple_factors}.\ref{e-k-semisimple-in-k-simples}, $G$ is the almost direct product of its almost $k$-simple factors $G_1,\ldots,G_r$.
By Theorem \ref{pre:th:simply_covering}, there exists a connected simply connected semisimple $k$-group $\widetilde{G}$ which decomposes as a direct product of its almost $k$-simple factors $\widetilde{G}_1,\ldots,\tilde{G}_k$ and a central $k$-isogeny $p:\widetilde{G}\to G$ with finite kernel $\mu$ such that $p(\widetilde{G}_i)=G_i$.
Note that $\rank_{k_v} \widetilde{G}_i=\rank_{k_v} G_i$ and so $(\widetilde{G}_i)_S$ is not compact.
Hence, by Theorem \ref{th:Strong_Approximation} $\widetilde{G}$ has Strong Approximation, that is, $\overline{\widetilde{G}(k)}=\widetilde{G}(\mathbb{A}_S)$. 

Applying the functor $C(\_,S)$ we obtain the following commutative diagram.

\qnote{
\[
\begin{tikzcd}
		&	& 1\ar[d] & 1\ar[d]\\
1\ar[r]	&	C(\mu,S)\ar[r]\ar[d
]				&	\widehat{\mu(k)}\ar[r
]\ar[d
]	&	\mu(\mathbb{A}_S)\ar[d
]\\
1 \ar[r]	&	C(\widetilde{G},S)\ar[r]\ar[d
]	&	\widehat{\widetilde{G}(k)}\ar[r
] \ar[d
]	&	\widetilde{G}(\mathbb{A}_S)\ar[d
]\\
1\ar[r] &	C(G,S)\ar[r]						&	\widehat{G(k)}\ar[r
]				&	G(\mathbb{A}_S)\\			
\end{tikzcd}
\]
}
Note that by Theorem \ref{th:CSP:arithmetic_open_map} the map $\widehat{p(k)}$ is open and has finite kernel $\widehat{\mu(k)}=\mu(k)$.
Hence the horizontal and vertical sequences starting at $1's$ are exact and since $\widetilde{G}$ has Strong Approximation the map $\pi_{\widetilde{G}}$ is surjective.
It follows that $C(G,S)$ contains a subgroup isomorphic to $\pi_{\widetilde{G}}^{-1}(\mu(\mathbb{A}_S))/\widehat{\mu(k)}$.
Now observe that $\mu(\mathbb{A}_S)\subset \widetilde{G}(\mathbb{A}_S)$ is an infinite group.
Indeed by Lemma \ref{pre:lm:totally_split_v} if $x\in\mu$ then $x\in\mu(k_v)$ for infinitely many $v$'s.
Since $\mu$ is nontrivial, this implies that $\mu(\mathbb{A}_S)$ is infinite.
Hence $C(G,S)$ contains an infinite subgroup, a contradiction.
Therefore $G$ must be simply connected.
\end{proof}

We can therefore assume from now on that $G$ is a simply connected semisimple group. But then $G$ is the direct product of its almost $k$-simple factors (see Theorem \ref{pre:th:ss_decomposes_simple_factors}).
It is then clear that $C(G,S)$ is finite if and only if $C(G_i,S)$ is finite for every $k$-simple factor $G_i$.
Hence we may assume \ref{A_k-simple}.

Now $G$ is a connected simply connected almost $k$-simple group and since we are assuming that $G(\OO_S)$ is infinite, Theorem \ref{th:G(OS)_infinite_iff_G_S_noncompact} gives $G_S$ is non compact, so by the Strong Approximation Theorem $G$ has Strong Approximation, i.e., we may assume \ref{A_strong}.
The following lemma shows that we may assume \ref{A_abs_simp}.

\begin{lemma}\label{lm:CSP_absolutely_simple}
Let $G$ be a connected simply connected almost $k$-simple $k$-group.
Then there exists a finite separable extension $k'$ of $k$ and a connected simply connected absolutely almost simple $k'$-group $H$ such that $G=R_{k'/k}(H)$ and $C(G,S)\cong C(H,S')$ where $S'=\cup_{v\in S}\{v'\in V_{k'}\; :\; v'_{|k}=v\}$.
\end{lemma}
\begin{proof}
By Theorem \ref{pre:th:k_simple_is_restriction_abs_simple} there exist a finite separable extension $k'$ and a connected simply connected absolutely almost simple $k'$-group $H$ such that $G=R_{k'/k}H$.
Now Lemma \ref{lm:restriction_scalars_commensurable}.\ref{e-commensurable-restriction} shows that $G(\OO_S)$ and $R_{k'/k}^0(H(\OO_{S'}))$ are commensurable and this implies that the restriction of the respective $S$-arithemtic topologies to $G(\OO_S)\cap R_{k'/k}^0(H(\OO_{S'}))$ induce the same topology.
Moreover by \ref{lm:restriction_scalars_commensurable}.\ref{e-kv-iso-restriction} the respective congruence topologies define the same topology when restricted to the intersection as well.
Hence by Lemma \ref{lm:CSP_reduces_to_open_subgroup} we have $C(G,S)\cong C(H,S')$.
\end{proof} 

Finally we show that \ref{A_anisotropics} is necessary for $G$ to have CSP.
\begin{lemma}\label{lm:kv_anisotropic}
Let $G$ be a semisimple $k$-group and assume that $G(\OO_S)$ is infinite. Put $\mathcal{A}:=\{v\in V_f\ :\ G\ \mbox{is}\ k_v\mbox{-anisotropic}\}$.
If $\mathcal{A}\cap S$ is non-empty, then $C(G,S)$ is infinite.
\end{lemma}

\begin{proof}
By Theorem \ref{th:kv_anisotropic_compact} $G(k_v)$ is compact for every $v\in\mathcal{A}$.
In particular $G(\mm_v^i)\subseteq G(k_v)$ gives a descending sequence of finite index open subgroups with trivial intersection.
Note that for $v\in S$ we have $G(\OO_S)\subseteq G(k_v)$. Hence, for $v\in\mathcal{A}\cap S$, $G(\OO_S)\cap G(\mm_v^i)$ gives a descending sequence of finite index subgroups of $G(\OO_S)$ with trivial intersection. 
Since $v\in S$, none of them are $S$-congruence subgroups so we can build a surjection
\[
C(G,S)\twoheadrightarrow \varprojlim_{i}G(\OO_S)/(G(\OO_S)\cap G(\mm_v^i)).
\]
Since the latter inverse limit contains $G(\OO_S)$ as a subgroup it follows that $C(G,S)$ is infinite.
\end{proof}

Let us finish this chapter by stating famous Serre's Conjecture.
\begin{conjecture*}[Serre]\label{conj:Serre}
Let $G$ be a simply connnected absolutely almost simple $k$-group. Then $C(G,S)$ is finite if and only if $\rank_S G=\sum_{v\in S}\rank_{k_v}G\geq 2$ and $G(k_v)$ is noncompact for every $v\in S\cap V_f$.
\end{conjecture*}

\clearpage{\pagestyle{empty}\cleardoublepage}

\chapter{Representation Growth of Arithmetic Groups}


\section{Representation Growth}\label{sec:RG}

In this section we introduce the necessary notation for the study of representation growth and present some well known results on representation theory.
We then focus on the representation growth of $p$-adic analytic and $\F_p[[t]]$-analytic groups.
Theorem \ref{th:RG_p-adic_perfect_type_c} shows that every compact $p$-adic analytic group with perfect Lie algebra is of representation type $c$ for some $c$, this was shown by Lubotzky and Martin in \cite{LuMa}.
Jaikin showed in \cite{JaiPre} that the same result holds for $\F_p[[t]]$-standard groups under the weaker condition (in chase $\ch k>0$) of having a finitely generated Lie algebra (see Theorem \ref{th:RG_Fp[[t]]-standard_type_c}).
We present both results for completion.
Jaikin's avoidance of a perfectness hypothesis in Theorem \ref{th:RG_Fp[[t]]-standard_type_c} paves the way to extend Theorem 1.2 in \cite{LuMa} to any semisimple group over a global field.     
We refer to \ref{sec:representation_theory} for notations and definitions concerning representation theory.

Let $\Gamma$ be a group. We want to study the representations of the group $\Gamma$, in particular, in this work we are interested in representations with finite image. From now on $\Rep(\Gamma)$ (respectively $\Irr(\Gamma)$) denotes the set of isomorphism classes of (irreducible) representations of $\Gamma$ with finite image. Evidently this is no restriction if $\Gamma$ is a finite group. The following lemma shows that for profinite groups this is no restriction either.
\begin{lemma}\label{lm:RG_profinite_finite_image}
Let $\Gamma$ be a profinite group. Then every representation of $\Gamma$ factors through a finite index open subgroup. Hence, every representation of $\Gamma$ has finite image. In particular, if $\widehat{\Gamma}$ is the profinite completion of $\Gamma$ we have $r_n(\Gamma)=r_n(\widehat{\Gamma})$.
\end{lemma}

\begin{proof}
Let $\rho\in \Rep_n(\Gamma)$.
Recall that $\rho$ is a continuous homomorphism $\rho: \Gamma\to \GL_n(\mathbb{C})$, in particular, $\ker\rho$ is a closed subgroup of $\Gamma$ and since $\Gamma$ is compact, so is $\rho(\Gamma)$.
This implies that $\rho(\Gamma)$ is a compact Lie group isomorphic to the profinite group $\Gamma/\ker\rho$, which is a totally disconnected group (Theorem \ref{pre:th:profinite_char}).
It follows that the identity component of $\rho(\Gamma)$ must be trivial. Hence $\rho(\Gamma)$ must be finite, since it is compact. This implies that $\{1\}$ is open in $\rho(\Gamma)$ and hence $\ker\rho\leq\Gamma$ is a finite index open subgroup.
\end{proof}

The previous lemma shows that most results of the representation theory of finite groups apply to the representation theory of profinite groups as well.
If $\Gamma$ is finite, then $\Irr(\Gamma)$ is finite as well, however for infinite $\Gamma$ we may have $\Irr_n(\Gamma)\neq\emptyset$ for infinitely many $n$. Let us define

\begin{align*}
&r_n(\Gamma)=|\Irr_n(\Gamma)|,\\
&R_n(\Gamma)=\sum_{i=1}^{n}r_n(\Gamma)=|\Irr_{\leq n}(\Gamma)|.
\end{align*}
We call $r_n(\Gamma)$ the \textbf{representation growth function} of $\Gamma$.
Note that, in general, $r_n(\Gamma)$ may be infinite, e.g., $\Gamma=\mathbb{Z}$.

We say that $\Gamma$ has \textbf{polynomial representation growth} (PRG for short) if the sequence $r_n(\Gamma)$ is polynomially bounded, that is, if there exist constants $c_1,\,c_2\in\mathbb{R}$ such that $r_n(\Gamma)\leq c_1n^{c_2}$ for every $n$. In particular, this requires $r_n(\Gamma)<\infty$ for every $n$.
Note that $r_n(\Gamma)$ is polynomially bounded if and only if $R_n(\Gamma)$ is.

We say that $\Gamma$ is of {\bf representation type $\mathbf{c}$} if there exist  $c\in\mathbb{R}$ such that $|\Gamma/K_n(\Gamma)|\leq cn^c$ for every $n$, where
\[
K_n(\Gamma):=\bigcap_{\rho\in\Rep_{\leq n}(\Gamma)}\ker\rho=\bigcap_{\rho\in\Irr_{\leq n}(\Gamma)}\ker\rho.
\]

\begin{lemma}\label{RG:type_c_implies_PRG}
Suppose $\Gamma$ is of representation type $c$, then $\Gamma$ has PRG.
\end{lemma}
\begin{proof}
Note that every $\rho\in\Irr_{\leq n}(\Gamma)$ factors through $K_n(\Gamma)$, hence it can be seen as a representation of $\Gamma/K_n(\Gamma)$.
If $\Gamma$ is of representation type c, we can see $\rho\in\Irr_{\leq n}(\Gamma)$ as a representation of the finite group $\Gamma/K_n(\Gamma)$ which is of order at most $cn^c$. It follows that $|\Irr_{\leq n}(\Gamma)|\leq cn^c$, hence $\Gamma$ has PRG.
\end{proof}

\begin{lemma}\label{lm:RG_finite_index_subgroup}
Let $\Gamma$ be a profinite group and $\Gamma'\leq\Gamma$ a finite index subgroup.
Put $m=|\Gamma:\Gamma'|$.
Then for every $n\in\mathbb{N}$ we have
\begin{enumerate}[label=\roman*)]
\item $R_n(\Gamma)\leq mR_n(\Gamma')$.\label{1}
\item $R_n(\Gamma')\leq mR_{nm}(\Gamma)$.\label{2}

\end{enumerate}
As a consequence, if $\Gamma$ and $\Gamma$' are commensurable then $\Gamma$ has PRG if and only if $\Gamma'$ has PRG.
\end{lemma}

\begin{proof} 
Let $\rho\in\Irr_{\leq n}(\Gamma)$ and consider $\rho'\in\Irr_{\leq n}(\Gamma')$ an irreducible constituent of $\rho_{|\Gamma'}$.
Note that $N:=\ker\rho\cap\Gamma'$ has finite index in $\Gamma$, so we may assume that we are working with representations of the finite groups $\Gamma/N$ and $\Gamma'/N$.
Hence by Frobenius reciprocity we know that $\rho$ is an irreducible constituent of $\Ind_{\Gamma'}^\Gamma\rho'$. Moreover note that the number of irreducible constituents of $\Ind_{\Gamma'}^\Gamma\rho'$ is at most $m$ ( each of them has $\rho'$ as a constituent when restricted to $\Gamma'$ by Frobenius reciprocity and $\dim\Ind_{\Gamma'}^\Gamma\rho'=m\dim\rho'$). Hence we have a map $f:\Irr_{\leq n}(\Gamma)\to \Irr_{\leq n}(\Gamma')$ whose fibers have size at most $m$. This proves \ref{1}.

Now consider $\rho'\in\Irr_{\leq n}(\Gamma')$ and let $\rho\in\Irr(\Gamma)$ be an irreducible constituent of $\Ind_{\Gamma'}^\Gamma\rho'$.
Clearly $\rho$ has dimension at most $nm$.
Moreover for every such $\rho$, $\rho_{|\Gamma'}$ can have at most $m$ irreducible constituents (if $k$ is the number of irreducible constituents of $\rho_{|\Gamma'}$ and $\lambda$ is one of minimal dimension, say $\delta$, then $k\delta\leq\dim\rho\leq m\dim\lambda$, where the last inequality follows from the fact that $\rho$ is a constituent of $\Ind_{\Gamma'}^\Gamma\lambda$ by Frobenius reciprocity).
Hence we have a map $f:\Irr_{\leq n}(\Gamma')\to\Irr_{\leq nm}(\Gamma)$ whose fibers have size at most $m$, which proves \ref{2}.  
\end{proof}

\begin{lemma}\label{lm:RG_type_c_finite_index}
Let $\Gamma$ be a profinite group and $\Gamma'\leq\Gamma$ a finite index open subgroup such that $|\Gamma:\Gamma'|= m$. Then $\Gamma$ is of representation type $c$ for some $c\in\mathbb{R}_{\geq 0}$ if and only if $\Gamma'$ is of representation type $c'$ for some $c'\in\mathbb{R}_{\geq 0}$.
As a consequence, if $\Gamma$ and $\Gamma'$ are commensurable and any of them is of representation type $c$, there exists $d$ such that $\Gamma$, $\Gamma'$ and $\Gamma\cap\Gamma'$ are of representation type $d$.
\end{lemma}
\begin{proof}
First let us suppose that $\Gamma'$ is representation type $c'$.
Given $\rho\in\Irr_{\leq n}(\Gamma)$, we have $\rho_{|\Gamma'}=\lambda_1\oplus\ldots\oplus\lambda_r$ for some $\lambda_i\in\Irr_{\leq n}(\Gamma')$. By definition, $K_n(\Gamma')\leq\ker\rho_{|\Gamma'}$ and since 
$K_n(\Gamma')$ is a characteristic subgroup of $\Gamma'$ we obtain
\[
K_n(\Gamma')\subseteq \bigcap_{g\in\Gamma}(\ker\rho_{|\Gamma'})^g=\ker\Ind_{\Gamma'}^\Gamma\left(\rho_{|\Gamma'}\right)\subseteq\ker\rho.
\]
It follows that $K_n(\Gamma')\subseteq K_n(\Gamma)$ and so 
\[
|\Gamma/K_n(\Gamma)|\leq |\Gamma:K_n(\Gamma')|=|\Gamma:\Gamma'||\Gamma':K_n(\Gamma')|\leq m c'n^{c'}\leq cn^c
\]
 for some $c\in\mathbb{R}_{\geq 0}$ as required.

Suppose now that $\Gamma$ is of representation type $c$. 
If $\rho'\in\Irr_{\leq n}(\Gamma')$ then $\Ind_{\Gamma'}^\Gamma\rho'\in\Rep_{nm}(\Gamma)$. Hence $K_{nm}(\Gamma)\subseteq\ker\Ind_{\Gamma'}^\Gamma\rho'$.
By Frobenius reciprocity $\rho'$ is a constituent of $\left(\Ind_{\Gamma'}^\Gamma\rho'\right)_{|\Gamma'}$, hence $K_{nm}(\Gamma)\cap \Gamma'\subseteq\ker\rho'$.
We thus have $K_{nm}(G)\cap\Gamma'\subseteq K_n(\Gamma')$. This gives
\[
|\Gamma'/K_n(\Gamma')|\leq |\Gamma'/(K_{nm}(\Gamma)\cap\Gamma')|\leq |\Gamma/K_{nm}(\Gamma)|\leq c(nm)^c\leq c'n^{c'}
\]
for some $c'\in\mathbb{R}_{\geq 0}$ as desired.

\end{proof}
\begin{lemma}\label{lm:RG_pronilpotent_monomial}

Let $\Gamma$ be a pro-nilpotent group. Then for every $\rho\in\Irr(\Gamma)$ there exists $\Gamma'\leq\Gamma$ with $|\Gamma:\Gamma'|=\dim\rho$ and $\rho'\in\Irr(\Gamma')$ such that $\rho=\Ind_{\Gamma'}^\Gamma\rho'$.
\end{lemma}
\begin{proof}
If $\rho\in\Irr(\Gamma)$ then by Lemma \ref{lm:RG_profinite_finite_image}, $\rho$ factors through a finite index open subgroup $N$. Then we can regard $\rho$ as a representation of the finite nilpotent group $\Gamma/N$.
For a finite nilpotent group every irreducible representation is induced from a linear representation of a subgroup (see \cite[Corollary 6.14]{Isa}). Suppose $\rho\cong\Ind_{\Gamma'N/N}^{\Gamma/N}{\lambda}$ for some subgroup $\Gamma'N/N$ and representation $\lambda\in\Irr_1(\Gamma'N/N)$. Then the result follows if we regard $\lambda$ as a representation of $\Gamma'N$.
\end{proof}

\begin{lemma}\label{lm:RG_induction_stabilizer}

Let $\Gamma$ be a profinite group and $N\trianglelefteq_o\Gamma$ an open normal subgroup.
Then $\Gamma$ acts on $\Irr(N)$ by conjugation and if $H\leq \Gamma$ is the stabilizer of $\lambda\in\Irr(N)$ the map
\[
\begin{array}{cccc}
\Ind: & \{\psi\in\Irr(H)\; :\; \langle\psi_{|N},\lambda\rangle\neq 0\}&\to& \{\rho\in\Irr(\Gamma)\; :\; \langle\rho_{|N},\lambda\rangle\neq 0\}\\
  & \psi & \mapsto & \Ind_N^\Gamma\psi\\
\end{array}
\]
is a bijection.
\end{lemma}
\begin{proof}
This is nothing but to extend Theorem \ref{pre:th:Clifford} to profinite groups.
If $\rho\in\Irr(\Gamma)$, by Lemma \ref{lm:RG_profinite_finite_image} $\rho$ factors through some finite index open subgroup $M$.
Moreover we may assume that $M\leq N$, otherwise take $M\cap N$.
Suppose $\lambda$ is an irreducible constituent of $\rho_{|N}$, i.e., $\langle\lambda,\rho_{|N}\rangle\neq 0$.
If we regard $\rho$ as a representation of $\Gamma/M$ and $\lambda\in\Irr(N)$ as a representation of $N/M$ we can apply Theorem \ref{pre:th:Clifford} to obtain a bijection
\[
\begin{array}{ccc}
\{\psi\in\Irr(H/M)\;:\;\langle\psi_{|N/M},\lambda\rangle\neq 0\}&\to&\{\rho\in\Irr(\Gamma/M)\; :\; \langle\rho_{|N/M},\lambda\rangle\neq 0\}\\
\psi&\mapsto &\Ind_{H/M}^{\Gamma/M}\psi\\
\end{array}
\]
Hence, regarding $\psi\in\Irr(H/M)$ as a representation of $H$, we obtain the desired bijection if $\lambda\in\Irr(N)$ factors through $M$.
Since every $\lambda\in\Irr(N)$ factors through some such $M$ the result follows.
\end{proof}

\begin{lemma}\label{lm:RP_direct_prod}
Suppose $\Gamma_1$ and $\Gamma_2$ are profinite groups. Then every irreducible representation of $\Gamma_1\times\Gamma_2$ is of the form $\rho_1\otimes\rho_2$, where $\rho_i$ is an irreducible representation of $\Gamma_i$. In particular, if $\Gamma_1$ and $\Gamma_2$ have PRG then $\Gamma_1\times\Gamma_2$ has PRG.
\end{lemma}
\begin{proof}
 The result is true if we assume the groups $\Gamma_1$ and $\Gamma_2$ are finite. 
Using again Lemma \ref{lm:RG_profinite_finite_image} we see that every $\rho\in\Irr(\Gamma_1\times\Gamma_2)$ factors through a finite index open subgroup $N$.
It follows that $(\Gamma_1\times\Gamma_2)/N$ is isomorphic to $\Gamma_1/(N\cap\Gamma_1)\times\Gamma_2/(N\cap\Gamma_2)$, which is a finite group.
If we regard $\rho$ as a representation of this finite group, we have $\rho\cong\lambda_1\otimes\lambda_2$ for some $\lambda_i\in\Irr(\Gamma_i/N\cap\Gamma_i)$. If we regard $\lambda_1$ and $\lambda_2$ as representations of $\Gamma_1$ and $\Gamma_2$ we obtain the desired decomposition.

Now if $\rho_1\in\Irr_m(\Gamma_1)$ and $\rho_2\in\Irr_k(\Gamma_2)$ we have by construction $\rho_1\otimes\rho_2\in\Irr_{mk}(\Gamma_1\times\Gamma_2)$.
Hence, if $R_n(\Gamma_1)\leq c_1n^{c_1}$ and $R_n(\Gamma_2)\leq c_2n^{c_2}$, we have 
\[
R_n(\Gamma_1\times\Gamma_2)\leq \sum_{mk=n}R_m(\Gamma_1)R_k(\Gamma_2)\leq c_1c_2n^{c_1+c_2+1}
\]
and $\Gamma_1\times\Gamma_2$ has PRG.
\end{proof}
\begin{theorem}[{\cite[2.7]{LuMa}}] \label{th:RG_p-adic_perfect_type_c}
Let $G$ be a compact $p$-adic analytic group and suppose that the Lie algebra of $G$ is perfect. Then $G$ is of representation type $c$ for some $c\in\mathbb{R}_{\geq 0}$.
\end{theorem}
\begin{proof}

Every $p$-adic analytic group has an open standard subgroup $U$ (Theorem \ref{pre:th:analytic_open_standard}), and since $G$ is compact,  every open subgroup has finite index. Thus, by Lemma \ref{lm:RG_type_c_finite_index} it suffices to show that $U$ is of representation type $c$ for some $c\in\mathbb{R}_{\geq 0}$.
Moreover, note that, since $U$ is an open subgroup of $G$, both $G$ and $U$ have the same associated Lie $\mathbb{Q}_p$-algebra.
Therefore we may assume that $G$ is a standard group.

This means we have a global chart for $G$, i.e., a homeomorphism
\[
\begin{array}{cccc}
f:& G& \to &(p\mathbb{Z}_p)^d\\
 & g&\mapsto& {\bf g},\\
\end{array}
\]
where $d=\dim G$ (as a $\mathbb{Q}_p$-analytic group).
Moreover, multiplication in $G$ is given by an analytic function $m:\mathbb{Z}_p^d\to\mathbb{Z}_p^d$  with coefficients in $\mathbb{Z}_p$ that can be written as
\[
m({\bf x},{\bf y})={\bf x}+{\bf y} + B({\bf x},{\bf y})+ O(3),
\]
where $B$ is a $\mathbb{Z}_p$-bilinear form and $O(3)$ is a power series with monomials of total order at least $3$, see section \ref{sec:standard_groups}.
We will make an abuse of notation and identify an element $g\in G$ with its coordinates ${\bf g}\in (p\mathbb{Z}_p)^d$.
The Lie algebra $\LL=\mathcal{L}_G$ associated to the standard group $G$  is $(\mathbb{Z}_p)^d$ with Lie bracket given by $[{\bf x},{\bf y}]_L=B({\bf x},{\bf y})- B({\bf x},{\bf y})$ for $\mathbf{x},\mathbf{y}\in(\mathbb{Q}_p)^d$.
By \ref{pre:formal_group_laws}.\ref{e_formal_commutator} the coordinates of the commutator of two elements $x,y\in G$ are given by
\[
[{\bf x},{\bf y}]:=f([x,y])= [{\bf x},{\bf y}]_L+ O(3).
\]

Let us write $G_i=f^{-1}((p^i\mathbb{Z}_p)^d)$, which gives a natural filtration of $G$ by finite index open normal 
subgroups. Moreover, by the proof of Theorem 8.31 in \cite{AnalyticProP} we have $G_{i+1}=G^{p^i}$, where the latter subgroup denotes the closed subgroup generated by $p^i$-powers of elements of $G$.

Now since $\mathcal{L}(\mathbb{Q}_p)=\LL\otimes_{\ZZ_p}\QQ_p$ is perfect, this implies that $p^r\mathcal{L}\subseteq [\mathcal{L},\mathcal{L}]$ for some $r\in\NN$ and so $p^{2n+r}\mathcal{L}\subseteq [p^n\mathcal{L},p^n\mathcal{L}]$ for every $n\in\NN$.
This implies that $G_{2n+r}\subseteq [G_n,G_n]G_{3n}$.
Indeed, any element ${\bf g}\in p^{2n+r}\mathcal{L}$ can be written as a sum of $d$ elements of the form $[{\bf x_1},{\bf y_1}]_{L}+\ldots [{\bf x_d},{\bf y_d}]_{L}$ for some ${\bf x_i},\,{\bf y_i}\in p^n\mathcal{L}$.
Thus, applying the above formula for the coordinates of commutators and multiplication, we have that
\[
[{\bf x_1},{\bf y_1}]\cdots[{\bf x_d},{\bf y_d}]={\bf g}+ O(3).
\]
Therefore $g\in[G_n,G_n]G_{3n}$ and $G_{2n+r}\subseteq [G_n,G_n]
G_{3n}$.

If $n\geq r+2$ then $2(n+1)+r\leq 3n$ and we can apply the above inclusion to obtain $G_{3n}\leq G_{2(n+1)+r}\leq [G_n,G_n]G_{3(n+1)}$.
Since this holds for every $n\geq r+2$, applying induction we obtain that $G_{3n}\subseteq[G_n,G_n]G_{3m}$ for every $m\geq n$, which implies $G_{3n}\leq [G_n,G_n]$.

Consider now $\rho\in\Irr_{p^n}(G)$ where $n\geq r+2$.
Since $G$ is a pro-$p$ group and hence pro-nilpotent, $\rho$ is induced from $\rho'\in\Irr_1(H)$ for some subgroup $H\leq U$ of index $p^n$ (Lemma \ref{lm:RG_pronilpotent_monomial}).
Recall that $G_{n+1}=G^{p^n}\leq H$, so we have $G_{3(n+1)}\leq[G_{n+1},G_{n+1}]\leq [H,H]$.
Hence, $G_{3(n+1)}\leq\ker\rho'$ and since $G_{3(n+1)}$ is normal this gives $G_{3(n+1)}\leq\ker\rho$.

It follows that $|G/K_n(G)|\leq G/G_{3(n+1)}=p^{3(n+1)d}$ for $n\geq r+2$. Therefore, we can find $c\in\mathbb{R}_{\geq 0}$ such that $|G/K_n(G)|\leq cn^c$ for every $n\in\NN$.
\end{proof}

The following theorem is an unpublished result of A. Jaikin. We present it here for completion.
\begin{theorem}[{\cite[Theorem 3.1]{JaiPre}}]\label{th:RG_Fp[[t]]-standard_type_c}
Let $G$ be an $\F_p[[t]]$-standard group and suppose that $\mathcal{L}_G$ is finitely generated as a $G$-module.
Then $G$ is of representation type $c$ for some $c\in\mathbb{R}_{\geq 0}$.
\end{theorem}
\begin{proof}

We will use the notation of Theorem \ref{th:RG_p-adic_perfect_type_c}. If $\LL=\LL_G$ is the $\Fp[[t]]$-Lie algebra of the standard group $G$, we have now a homeomorphism from $G$ to $t\mathcal{L}=(tF_p[[t]])^{d}$ given by
\[
\begin{array}{cccc}
f:& G &\to &(t\F_p[[t]])^d\\
 & g&\mapsto& {\bf g},
\end{array}
\]
where $d=\dim G$.
Let us denote the adjoint action of $G$ by $\Ad=\Ad_{\F_p[[t]]}:G\to \Aut_{\Fp[[t]]}(\mathcal{L})$, see section \ref{sec:standard_groups}.

If we let $G$ act by conjugation on the abelian group $G_k/G_{2k}$, from \ref{pre:formal_group_laws}.\ref{e_formal_adjoint} and the description of the adjoint action (see \eqref{eq:adjoint_locally} in section \ref{sec:standard_groups}) we obtain an isomorphism of $G$-modules
\begin{equation}\label{iso:standard}
\begin{array}{ccc}
G_k/G_{2k}&\to& \mathcal{L}/t^k\mathcal{L}\\
 g&\mapsto& t^{-k}{\bf g}
\end{array}
\end{equation}
Let us write $A$ for the $\F_p[[t]]$-subalgebra of $\End_{\F_p[[t]]}(\mathcal{L})$ generated by $\{\Ad(g)\, :\, g\in G\}$  and for every subset $H\subseteq G$ let $\I(H)$ denote the left ideal of $A$ generated by $\{\Ad(h)-\Id\;:\; h\in H\}$. We will write $\I_k=\I(G_k)$. Note that since $G_k$ is a normal subgroup, $\I_k$ is an ideal.
Under this notation $[\mathcal{L},G_k]=\mathcal{L}I_k$ if we see $\mathcal{L}$ as an $\End_{\F_p[[t]]}(\mathcal{L})$ module. 
Let us state and prove a series of claims that are required for the proof.
\begin{claim}
There exists $s$ such that for every $k\geq s$ we have
\[G_{k+s}\leq[G_k,G].
\]
\end{claim}
By assumption $\mathcal{L}$ is finitely generated as a $G$-module, equivalently (see \ref{lm:finite_generation_pro_p_modules}) $|\mathcal{L}/[\mathcal{L},G]|$ is finite.
Hence there exists $s\in\NN$ such that $t^s\mathcal{L}\subseteq[\mathcal{L},G]$.
From this and isomorphism \eqref{iso:standard} it follows that for every $k\geq s$ we have
\[
G_{k+s}\leq[G_k,G]G_{2k}.
\]
If we take $n>k$ this gives
\[
G_{2n}\leq G_{(n+1)+s}\leq [G_{n+1},G]G_{2(n+1)}\leq[G_k,G]G_{2(n+1).}
\]
Hence, whenever $k\geq s$ we can apply an inductive argument to obtain
\[
G_{k+s}\leq [G_k,G]G_{2n}
\]
for every $n>k$ and since $G_{k+s}$ is closed this gives $G_{k+s}\leq[G_k,G]$.

\begin{claim}
There exists $l\in\NN$ such that $\I_1^l\leq tA$.
\end{claim}
Note that $\F_q[[t]]$ is a Noetherian local ring and $A$ is an $\F_q[[t]]$-submodule of the finitely generated $\F_q[[t]]$-module $\End_{\F_q[[t]]}(\mathcal{L})$, hence $A$ is finitely generated as well.
We can apply Nakayama Lemma to get that $A/tA$ is finite.
Now note that from the description of the adjoint action (see \eqref{eq:adjoint_locally} in section \ref{sec:standard_groups}) it follows that $\Ad(G)-\Id\subseteq t\End_{\F_p[[t]]}(\mathcal{L})$, which implies $\I_1\subseteq t\End_{\F_p[[t]]}(\mathcal{L})$, i.e., $\I_1^i\subseteq t^i\End_{\F_p[[t]]}(\mathcal{L})$.
Thus $\I_1^i$ gives a strictly decreasing sequence of submodules, so there exists $l\in\NN$ such that $\I_1^l\leq tA$.
\begin{claim}
For every $u>s$ we have $\I_{u+s}\leq\I_1\I_u+\I_u\I_1$.
\end{claim}
By Claim 1 we know that $G_{u+s}\leq [G_u,G]$.
Suppose $g=[g_u,g_1]$, where $g_u\in G_u,\ g_1\in G_1$.
Then
\begin{align*}
\Ad(g)-\Id&=\Ad([g_u,g_1])-\Id\\
&= \Ad(g_u^{-1}g_1^{-1})[\Ad(g_u)-\Id,\Ad(g_1)-\Id]_L\in \I_{u}\I_{1}+\I_{1}\I_{u}.
\end{align*}
If $g=fh$ is a product of commutators in $[G_u,G]$, then
\begin{align}\label{eq:RG_prod_of_ads}
\Ad(fh)-\Id=(\Ad(f)-\Id)(\Ad(h)-\Id)+\Ad(f)-\Id-(\Ad(h)-\Id)
\end{align}
and by the previous case $g\in \I_1\I_u+\I_u\I_1$.
\begin{claim}
For every $u>s$ we have $\I_{u+(2l-1)s}\leq t\I_u$.
\end{claim}
Applying the previous claim $(2l-1)$ times we obtain
\[
\I_{u+s(2l-1)}\subseteq \sum_{i+j=2l-1}\I_{1}^i\I_{u}\I_{1}^j.
\]
It follows from Claim 2 that $\I_{u+s(2l-1)}\leq t\I_{u}$.

\begin{claim}
Let $H$ be an open subgroup of $G$, then $\I_1/\I(H)$ is finite.
\end{claim}
Let $f_1,\ldots,f_{d^2}$ be a set of free generators of $\End_{\Fq [[t]]}(\mathcal{L})$ as an $\Fq [[t]]$-module.
Then every element $E\in\End_{\Fq [[t]]}(\LL)$ can be expressed as
linear combination of the $f_i$ with coefficients in $\Fp [[t]]$.
If we assign to every such $E$ its coefficients with respect to the $f_i$, this gives $\End_{\Fq [[t]]}(\mathcal{L})$ the structure of an $\Fq[[t]]$-analytic variety.
Note that the map $\Ad-\Id:G\to\End_{\Fp[[t]]}(\LL)$ is then an analytic map.
We can consider any power series in the $f_i$ with coefficients in $\Fp[[t]]$ as an analytic function on $\End_{\Fq[[t]]}(\LL)$.
Since $\I_1$ is finitely generated, to prove the claim it suffices to show that $t^j\I_1\subseteq\I(H)$ for some $j\in\NN$.
Suppose that there exists no such $j$, this means that $\I(H)$ is a free $\F_p[[t]]$-module of rank strictly less than that of $\I_1$.
So there exist $\alpha_1,\ldots,\alpha_{d^2}$ such that $\psi({\bf f})=\sum \alpha_i f_i$ satisfies $\psi({\Ad(\bf h)-\Id})=0$ for every $h\in H$ but there exists $g\in G$ such that $\psi(\Ad(g)-\Id)\neq 0$. 
that is, $H$ is contained in the zeroes of some analytic equation but $G$ is not. But this is impossible by \cite[Theorem 3.3]{JaiKlopsch}

\begin{claim}
Let $s(G):=\{g\in G\;:\;\Ad(g)\;\textrm{is semisimple}\}$. Then the closed subgroup generated by $s(G)$ is open in $G$.
\end{claim}
Note that $g^{p^d}\in s(G)$ for every $g\in G$.
We can apply Zelmanov's positive solution to the Restricted Burnside Problem to obtain that $G^{p^d}$ is a finite index subgroup of $G$ and hence open.
It follows that the closed subgroup generated by $s(G)$ is open as well.
\begin{claim}
Suppose $h\in s(G)$. Then $A(\Ad(h)-\Id)^2$, the left ideal of $A$ generated by $(\Ad(h)-\Id)^2$, has finite index in $\I(h)=A(\Ad(h)- \Id)$.
\end{claim}
From the point of view of $\F_p[[t]]$-modules the claim holds if and only if $t^kA(\Ad(h)-\Id)\subseteq A(\Ad(h)-\Id)^2$ for some $k$ and this happens if and only if $\Fp((t))\otimes A(\Ad(h)-\Id)=\Fp((t))\otimes A(\Ad(h)-\Id)^2$.
This is an equality of vector spaces and hence reduces to an equality of their respective dimensions, which are stable under extension of scalars, so this holds if and only if $\overline{\Fp((t))}\otimes A(\Ad(h)-\Id)=\overline{\Fp((t))}\otimes A(\Ad(h)-\Id)^2$.
Now an appropriate change of coordinates allows us to assume that $\Ad(h)-\Id$ is a diagonal matrix, and so for every element $M\in\End_{\overline{\Fp((t))}}(\overline{\Fp((t))}^d)$ we have $M(\Ad(h)-\Id)=0$ if and only if $M(\Ad(h)-\Id)^2=0$.
Hence $\overline{\Fp((t))}\otimes A(\Ad(h)-\Id)=\overline{\Fp((t))}\otimes A(\Ad(h)-\Id)^2$ and we are done.
\begin{claim}
There exist $c_1$ and $c_2$ such that $t^{c_1+kc_2}\I_1\subseteq\I_k$ for every $k$.
\end{claim}

By Claim $5$ and $6$ we know  that $\I(s(G))$ is of finite index in $\I_1$.
Pick $g_1,\ldots,g_t\in s(G)$ such that  $\Ad(g_i)-\Id$ generate $\I(s(G))$.
By Claim $7$ we know that $A(\Ad(g_i)-\Id)/A(\Ad(g_i)-\Id)^2$ is finite.
Thus there exists $c$ such that for every $i$ we have $t^{c}A(\Ad(g_i)-\Id)\subseteq A(\Ad(g_i)-\Id)^2$.
Hence for every $j\geq 0$ we have
\[
t^{c(p^j-1)}\I(s(G))\leq\sum_i A(\Ad(g_i)-\Id)^{p^j}=\sum_i A(\Ad(g_i)^{p^j}-\Id)\leq\I_{p^j}.
\]
The last inequality follows from the fact that $G^{p^j}\leq G_{p^j}$, see \ref{pre:formal_group_laws}.\ref{e_power_map}.
Now it is clear that there exist constants $c_1$ and $c_2$ satisfying the claim.

Consider now $\rho\in\Irr(G)$ and suppose $G_{2(k-1)}\nsubseteq \ker\rho $ but $G_{2k}\subseteq\ker\rho$.

Let $\lambda\in\Irr(G_k)$ be an irreducible constituent of $\rho_{|G_k}$ and let $H\leq G$ be the stabilizer of $\lambda$ with index $|G/H|=p^m$.

\begin{claim}
$G_{s(2l-1)(m+1)}\leq H$.
\end{claim}
Consider the chain of subgroups
\[ H\leq HG_{s(2l-1)(m+1)}\leq HG_{s(2l-1)m}\leq\ldots\leq HG_{s(2l-1)}.\]
Note that $|G:H|=p^m$ implies that there must be an equality in the above chain.
If $H=HG_{s(2l-1)(m+1)}$ we are done.
If not, suppose $HG_{s(2l-1)i}=HG_{s(2l-1)(i+1)}$ for some $1\leq i\leq m$.
From \eqref{eq:RG_prod_of_ads} it follows that $\I(JN)=\I(J)+I(N)$ whenever $N\trianglelefteq G$ and $J\leq G$.
We can now apply the previous equality of subgroups and Claim 4 to obtain
\[\I(H)+\I_{s(2l-1)i}=\I(H)+\I_{s(2l-1)(i+1)}=\I(H)+t\I_{s(2l-1)i}.\]
If we apply the Nakayama Lemma to the above equality of $\F_q[[t]]$-modules we get $\I_{s(2l-1)i}\leq \I(H)$.
In particular, $\LL\I_{s(2l-1)i}\leq \LL\I(H)$.
Therefore, isomorphism \eqref{iso:standard} gives
\[ [G_k,G_{s(2l-1)i}]G_{2k}\leq [G_,H]G_{2k}.\]
In particular, $G_{s(2l-1)i}$ stabilizes $\lambda$ and $G_{s(2l-1)(m+1)}\leq G_{s(2l-1)i}\leq H$.

We are now ready to finish the proof.
Suppose that $k\geq c_1+c_2s(2l-1)(m+1)+2$. 
It follows from Claim 8 that
\[
 t^{c_1+c_2s(2l-1)(m+1)}\LL\I_1\leq \LL\I_{s(2l-1)(m+1)}.
 \]
Now first applying Claim 1 and then the previous inequality together with isomorphism \eqref{iso:standard} we obtain
\begin{align*}
G_{c_1+c_2s(2l-1)(m+1)+k+s}&\leq[G_{c_1+c_2s(2l-1)(m+1)+k},G_1]G_{2k}\\
&\leq[G_k,G_{s(2l-1)(m+1)}]G_{2k}.
\end{align*}
However, since by Claim 9 we have $G_{s(2l-1)(m+1)}\leq H$, we must have
\[
G_{k+c_1+c_2s(2l-1)(m+1)}\leq\ker\lambda.
\]
But recall that $G_{2(k-1)}\nleq\ker\lambda$ and hence
\[
2k-2\leq k+c_1+c_2s(2l-1)(m+1),
\]
equivalently, $k\leq c_1+c_2s(2l-1)(m+1)+2$, which contradicts our assumption. 
Therefore, we must have
\[k\leq c_1+c_2s(2l-1)(m+1)+2.
\]
It follows that $G_{2(c_1+c_2s(2l-1)(m+1)+2)}\leq\ker\rho$.
By Lemma \ref{lm:RG_induction_stabilizer} $\rho$ is induced from an irreducible representation of $H$ and so $\dim\rho\geq p^m$ so we have $G_{2(c_1+c_2s(2l-1)(\log_p n+1)+2)}\leq\ker\rho$ 
for some constants $d_1$, $d_2$ which only depend on $c_1$ and $c_2$.
This implies
\[
G_{d_1+d_2\log_p n}\leq K_n(G).
\]
Thus, $|G/K_n(G)|\leq|G/G_{d_1+d_2\log_p n}|=p^{\dim G(d_1+d_2\log_p n)}\leq cn^{c}$ for some $c\in\mathbb{R}_{\geq 0}$ as desired.
\end{proof}
\begin{remark}
Note that the above proof works equally well for $\mathbb{Z}_p$-standard groups if we assume that $p^k\LL(G)\leq[\LL,G]$ for some $k\in\mathbb{N}$, nevertheless for $p$-adic analytic groups this condition is equivalent to $\LL(G)\otimes\mathbb{Q}_p$ being perfect.
\end{remark}

\section{CSP and Representation Growth of Arithmetic Groups}\label{sec:CSP_vs_RG}

Let $k$ be a global field and $G$ a simply connected semisimple algebraic $k$-group.
Let $S\subset V_k$ be a non-empty finite set of valuations such that $V_\infty\subseteq S$ and $\Gamma$ and $S$-arithmetic subgroup of $G$.
In this section we investigate the relation between the weak Congruence Subgroup Property (namely the finiteness of $C(G,S)$) and the representation growth of $\Gamma$.
The main theorem is the following.
\begin{theorem}\label{th:CSP_implies_PRG}
Let $k,\,S,\,\Gamma$ and $G$ be as above.
If $G$ has the weak Congruence Subgroup Property with respect to $S$, then $\Gamma$ has Polynomial Representation Growth.
\end{theorem}

As in the previous sections we fix a $k$-embedding $\iota:G\to\GL_n$. We also give $G$ the structure of an affine group $\OO_S$-scheme through this embedding, see section \ref{sec:arithmetic_groups}.

We will make use of the previous sections to reduce Theorem $\ref{th:CSP_implies_PRG}$ to a particular setting.
First let us note that, by definition, $\Gamma$ is commensurable with $G(\OO_S)$, so applying Lemma \ref{lm:RG_finite_index_subgroup} it follows that $\Gamma$ has PRG if and only if $G(\OO_S)$ does. Hence we may assume that $\Gamma=G(\OO_S)$.
If $G(\OO_S)$ is finite, then so is $\Gamma$ and Theorem \ref{th:CSP_implies_PRG} becomes trivial, so from now on we will assume that $G(\OO_S)$ is infinite.

Let $G^0$ be the connected component of $G$. Since $G^0$ is of finite index in $G$ (Theorem \ref{pre:connected_components}), it follows that $G^0(\OO_S)$ has finite index in $G(\OO_S)$. Hence applying Lemma \ref{lm:RG_finite_index_subgroup} again it suffices to study the representation growth of $G^0(\OO_S)$. Moreover by Lemma \ref{lm:CSP_connected} $G$ has wCSP if and only if $G^0$ does, so we may assume that $G$ is connected.

Since $G$ is simply connected, Theorem \ref{pre:th:ss_decomposes_simple_factors} gives that $G$ is a direct product of its almost $k$-simple factors, say  $G_1,\ldots, G_r$.
Now Lemma \ref{lm:RP_direct_prod} shows that $G(\OO_S)$ has PRG if and only if each $G_i(\OO_S)$ has.
Hence, we may assume that $G$ is almost $k$-simple.
Moreover Theorem \ref{th:G(OS)_infinite_iff_G_S_noncompact} now implies that $G$ has Strong Approximation.
Recall that this means that $G(k)G_S$ is dense in $G(\mathbb{A}_k)$, which in particular implies that $G(\OO_S)$ is dense in $G(\widehat{\OO}_S):=\prod_{v\in V_f\setminus{S}}G(\OO_v)$. Hence we have $\overline{G(\OO_S)}\cong G(\widehat{\OO}_S)$.

Moreover, Lemmata \ref{lm:CSP_absolutely_simple} and \ref{lm:restriction_scalars_commensurable} (together with one furhter application of Lemma \ref{lm:RG_finite_index_subgroup}) allow us to further assume that $G$ is absolutely almost simple.

Now if $G(\OO_S)$ has wCSP, then $C(G,S)=\ker\left(\widehat{G(\OO_S)}\to G(\widehat{\OO}_S)\right)$ is finite. 
Hence we can find $\Gamma_{0}$ of finite index in $G(\OO_S)$ such that its profinite closure $\widehat{\Gamma}_0$ has trivial intersection with $C(G,S)$.
Thus, we can see $\widehat{\Gamma}_0$ as a finite index subgroup in $G(\widehat{\OO}_S)$.
It follows from Lemma \ref{lm:RG_finite_index_subgroup} and Lemma \ref{lm:RG_profinite_finite_image}  that $G(\OO_S)$ has PRG if and only if $\Gamma_0$ does, if and only if $\widehat{\Gamma}_0$ does , if and only if $G(\widehat{\OO}_S)$ does.
Therefore to prove Theorem \ref{th:CSP_implies_PRG} it suffices to show the following theorem.

\begin{theorem}\label{th:PRG}
Let $k$ be a global field and $S\subset V_k$ a finite set of valuations.
Let $G$ be a connected simply connected absolutely almost simple $k$-group.
Then $G(\widehat{\OO}_S)$ has Polynomial Representation Growth.
\end{theorem}

\begin{remark}
Note that in Theorem \ref{th:PRG} we do not need to assume that $G$ has wCSP with respect to $S$.
\end{remark}

Let us write $V_f^S:=\{v\in V_f\ :\ v\notin S\}$.
Recall that
\[
G(\widehat{\OO}_S)\cong\prod_{v\in V_f^S}G(\OO_v),
\]
so we want to better understand the groups $G(\OO_v)$ for $v\in V_f^S$.
The idea is to find a uniform bound for the representation growth of the groups $G(\OO_v)$, that is, to find $c\in\mathbb{R}_{\geq 0}$ such that for every "good valuation" $v$ the group $G(\OO_v)$ has representation type $c$, and then apply Lemma \ref{lm:RP_direct_prod} to obtain a polynomial bound for the representation growth of $G(\widehat{\OO}_S)$.
By Proposition \ref{th:deltaP1} we can control the effect of $G(\OO_v)/G(\mm_v)$.
Then we show that for almost every $v$ the group $G(\mm_v)$ is $\OO_v$-standard (Lemma \ref{lm:Nv1_standar}) and we can "uniformly" apply  Theorem \ref{th:RG_Fp[[t]]-standard_type_c} to obtain Theorem \ref{th:allvaluations}.

First, let us briefly recall how $G(k_v)$ has the structure of a $k_v$-analytic group for every $v\in V_f^S$.
If we see $G\subseteq\GL_N\subseteq\mathbb{A}_{N^2+1}$, then $G$ is determined by an ideal given by some polynomials $f_1,\ldots,f_m\in k[x_1,\ldots,x_{N^2+1}]$.
Since $G$ is an algebraic variety the set of regular points is Zariski dense in $G$, moreover since $G$ acts transitively on itself by automorphism, it follows that all points in $G$ are regular.
Hence, by the Jacobian Criterion (see for instance \cite[Theorem 2.19]{Liu}, the Jacobian of $f_1,\ldots,f_m$ has rank $N^2+1-\dim G$.
Thus, applying the Inverse Function Theorem for local fields (\cite[Part II, Chapter III]{Ser}) it follows that at every point $x\in G(k_v)$ the projection onto the, say first, $\dim G$ coordinates defines a chart on an open neighbourhood of $x$. This gives the structure of a $k_v$-analytic variety to $G(k_v)$ and multiplication and inversion, which are given by maps of algebraic $k$-groups, are analytic with respect to this structure.

Now on the one hand we have $\mathcal{L}(G(k_v))$, the $k_v$-Lie algebra associated to the analytic group $G(k_v)$ and on the other hand we have $L_G(k_v)=k_v\otimes L(k)$ the $k_v$-Lie algebra associated to the algebraic $k$-group $G$.
Since $\Der_{k_v}(k_v[G(k_v)])=k_v\otimes \Der_k(k[G])$ we have $L(k_v)\cong\LL(G(k_v))$.

Equivalently, if we regard $G$ as an affine group $\OO_S$-scheme then
\[
L_G(R)=\Der_{\OO_S}(\OO_S[G],R)\cong \Hom_{\OO_S}(I_{\OO_S[G]}/I_{\OO_S[G]}^2),R)
\]
for every $\OO_S$-algebra $R$, see section \ref{sec:affine_group_schemes}.
In particular,
\[
L_G(k_v)=L_G(k)\otimes k_v=L_G(\OO_S)\otimes k_v\cong\LL(G(k_v)).
\]

Let us introduce some useful notation and state a series of results which will allow us to prove Theorem \ref{th:PRG}.
Given $v\in V_f^S$, recall that $\mm_v$ stands for the maximal ideal of $\OO_v$ and let us write $\mathfrak{p}_v:=\{x\in\OO_S\ :\ x\in\OO_v\} $ for the corresponding maximal ideal of $\OO_S$.
We will fix a uniformizer $\pi_v\in\OO_v$ such that $\mm_v=(\pi_v)$.
Recall that $\F_v=\OO_v/\mm_v$ stands for the corresponding residue field of size $q_v:=|\F_v|$.
For every $k\in\NN$ let us write
\[
N_v^k=G(\mm_v^k):=G(\OO_v)\cap \ker\left(\GL_N(\OO_v)\to \GL_N(\OO_v/\mm_v^k)\right),
\]
for the principal congruence subgroup of level $k$ of $G(\OO_v)$.
If we see $G$ embedded in $\GL_N$ this is nothing but the points of $G$ whose coordinates lie in $\mm_v^k$.
If one thinks of $G$ as an affine group $\OO_S$-scheme then
\[
G(\mm_v^k)=\{\varphi\in\Hom_{\OO_S}(\OO_S[G],\OO_v)\,:\, \varphi(I_{\OO_S[G]})\subseteq\mm_v^k\},
\] where $I_{\OO_S[G]}=\ker\epsilon$ and $\epsilon:\OO_S[G]\to \OO_S$ represents the identity of $G_{\OO_S}$, see section \ref{sec:affine_group_schemes}.

We want to apply Theorem \ref{th:RG_Fp[[t]]-standard_type_c} to the groups $G(\OO_v)$, so we begin by investigating the adjoint action in semisimple algebraic $k$-groups. 
\begin{lemma}\label{lm:fgLieJaiPre}
Let $Q$ be a connected simply connected semisimple algebraic group defined over a field $F$ and $H\leq Q(F)$ a Zariski dense subgroup.
If $L$ is the Lie algebra of $Q$, then 
\[L(F)=[L(F),H].\]
\end{lemma}

\begin{proof}
Since we are dealing  with vector spaces, the equality amounts to equality of dimensions.
Extension of scalars does not affect dimension and hence if $\bar{F}$ is an algebraic closure of $F$ it suffices to show $\bar{F}\otimes_{F}L(F)=\bar{F}\otimes_F[L(F),H]$, that is, $L(\bar{F})=[L(\bar{F}),H]$.
Moreover, since the adjoint action is Zariski continuous and $[L(\bar{F}),H]$, being a vector subspace, is Zariski closed we have $[L(\bar{F}),H]=[L(\bar{F}),Q(\bar{F})]$.
We may therefore assume without loss of generality that $F$ is algebraically closed and that $H=Q(F)$.

In this setting $Q$ can be obtained as a Chevalley group scheme over $F$ (see the comments before Corollary 5.1 in \cite{Steinberg}).
Let $\alpha$ be a root of an associated  (absolute )root system $\Phi$ of $Q$ and let $X_\alpha$ denote the corresponding unipotent subgroup, then $G$ is generated by $\{X_{\alpha}\;:\;\alpha\in\Phi\}$, see \cite[Chapter 3]{Steinberg}.
Then, since $Q$ is simply connected, the subgroup $\langle X_\alpha,X_{-\alpha}\rangle$ acts on its Lie algebra $\mathfrak{x}_{-\alpha}\oplus \mathfrak{h}_\alpha\oplus\mathfrak{x}_\alpha$ as $\SL_2(F)$ acts on its Lie algebra $\mathfrak{sl}_2$ (\cite[Corollary 3.6]{Steinberg}).
Moreover $L(F)=\langle \sum_{\alpha\in\Phi}\mathfrak{x}_{-\alpha}+\mathfrak{h}_\alpha+\mathfrak{x}_\alpha\rangle$, so it suffices to check the equality in the case $Q=\SL_2$ and this reduces to the following quick calculation.
\begin{align*}
&\begin{pmatrix} k^{-1} & 0\\ 0 & k\end{pmatrix}
\begin{pmatrix} 0 & 1\\ 0 & 0\end{pmatrix}
\begin{pmatrix} k & 0\\ 0 & k^{-1} \end{pmatrix}
-
\begin{pmatrix}  0 & 1\\ 0 &0\end{pmatrix}
&=&
\begin{pmatrix} 0 & k^{-2}-1\\ 0 & 0\end{pmatrix}
\\
&\begin{pmatrix} 1 & 0\\ -k & 1\end{pmatrix}
\begin{pmatrix} 0 & 1\\ 0 & 0\end{pmatrix}
\begin{pmatrix} 1 & 0\\ k & 1 \end{pmatrix}
-
\begin{pmatrix}  0 & 1\\ 0 &0\end{pmatrix}
&=&k\begin{pmatrix} 1 & 0\\ 0 & -1\end{pmatrix}
+
\begin{pmatrix} 0 & 0\\ -k^2& 0 \end{pmatrix}
\end{align*}
From the first calculation we obtain that all elements of the form $\begin{pmatrix}
0& x\\ 0 &0
\end{pmatrix}$
belong to $[\mathfrak{sl}_2,\SL_2]$ and by a symmetric operation the same holds for elements of the form
$\begin{pmatrix}
0 & 0\\ x& 0
\end{pmatrix}$.
Now from the second calculation we obtain that all elements of the form
$\begin{pmatrix}
x & 0\\
0 &-x
\end{pmatrix}$
belong to $[\mathfrak{sl}_2,\SL_2]$ and this finishes the proof.
\end{proof}

As a direct application, we obtain the following lemma.

\begin{lemma}\label{lm:L(Ov)_fg_module}
Let $k$, $S$ and $G$ be as in Theorem \ref{th:PRG}. Then for every $v\in V_f^S$, $L(\OO_v)$ is finitely generated as an $N_v^1$-module. Moreover, for almost every $v\in V_f^S$ we have
\[L(\OO_v)=[L(\OO_v),G(\OO_v)].\]
\end{lemma}
\begin{proof}
$G(\OO_S)$ is Zariski-dense by Proposition \ref{prop:G(OS)_Zariski_dense}. 
Thus, we can apply Lemma \ref{lm:fgLieJaiPre} to obtain $L(k)=[L(k),G(\OO_S)]$.
Since $L(k)=k\otimes_{\OO_S}L(\OO_S)$, this implies  
\[\pp_{v_1}^{i_1}\cdots\pp_{v_{t}}^{i_t}L(\OO_S)\subseteq[L(\OO_S),G(\OO_S)]\]
for certain prime ideals $\pp_{v_i}\subseteq\OO$ corresponding to distinct valuations $v_1,\,\ldots,\,v_t\in V_f^S$.
It follows that $L(\OO_v)=[L(\OO_v),G(\OO_v)]$ for every $v\in V_f^S\setminus\{v_1,\,\ldots,\,v_t\}$.

In general, note that, for $v\in V_f^S$, $G(\OO_v)$ is Zariski-dense since $G(\OO_S)$ is.
It follows that $N_v^1=G(\mm_v)$, being a finite index subgroup in $G(\OO_v)$, is Zariski-dense.
Applying Lemma \ref{lm:fgLieJaiPre} we obtain $L(k_v)=[L(k_v),N_v^1]$.
Hence, $[L(\OO_v),N_v^1]$ has finite index in $L(\OO_v)$, which implies that $L(\OO_v)$ is a finitely generated $N_v^1$-module (note that $N_v^1$ is a pro-$p$ group for $p=\ch \mathbb{F}_v$ and apply Lemma \ref{lm:finite_generation_pro_p_modules} to $L(\OO_v)$ as an $N_v^1$-module).
\end{proof}

\begin{lemma} \label{lm:G(Fv)_simple}
Let $G$ be a connected simply connected absolutely almost simple algebraic $k$-group.
Then for almost every $v\in V_f^S$ the group $
G(\OO_v)/N_v^1$ is perfect and is a central extension of a finite simple group of Lie type $H(\F_{v})/Z(H(\F_v))$, where $H$ is of the Lie type of $G$.

\end{lemma}

\begin{proof}[Sketch of proof]
The proof of this result rests on the relative theory of reductive group schemes over rings which lies beyond the general scope of this work. We give here a sketch of the proof and point out to references for the details.

Recall that we gave $G$ the structure of an affine group $\OO_S$-scheme.
Given $v\in V_f^S$, let us write $G_{\mathfrak{p}_v}$ for the reduction of $G$ modulo $\mathfrak{p}_v$, that is, $G_{\mathfrak{p}_v}$ is an affine group $\F_v$-scheme represented by $\OO_S[G]\otimes_{\OO_S}\OO_S/\mathfrak{p_v}=\OO_S[G]\otimes_{\OO_S}\F_v$, where $\OO_S[G]$ represents $G$.
Since $G$ is a connected and irreducible algebraic variety (as an affine algebraic group), it is reduced (as an affine group $\OO_S$-scheme, that is $\OO_S[G]$ has no nilpotent elements) and this implies that the same holds for $G_{\mathfrak{p}_v}$ for almost every $v\in V_f^S$, see \cite[Proposition 9.7.7 and 9.7.8]{EGAIV}.
If $G_{\mathfrak{p}_v}$ is reduced, then it is an algebraic variety and hence has some regular point, but it is also an algebraic $\F_v$ group, so indeed all points are regular and $G_v$ is smooth.
Hence we can apply \cite[18.5.17]{EGAIV} and obtain that the reduction map $G(\OO_v)\to G_{\mathfrak{p}_v}(\F_v)$ is surjective for almost every $v\in V_f^S$.

Now if we see $k$ as the direct limit of the rings $\OO_S[1/r]$ for $r\in\OO_S$ and since $G$ is a semisimple $k$-group, we can apply \cite[Corollary 3.1.11]{Con} to obtain that there exists $s\in\OO_S$ (consisting of those "bad" primes) such that $G_{\OO_S[1/s]}$ (which is the affine group $\OO_S[1/s]$-scheme represented by $\OO_S[G]\otimes_{\OO_S}\OO_S[1/s]$) is a semisimple affine group $\OO_S[1/s]$-scheme, that is, for every $\mathfrak{p}\in\Spec(\OO[1/s])$ the group represented by $\OO_S[1/s][G]\otimes_{\OO_S[1/s]}(\OO_S[1/s]/\mathfrak{p})$ is a semisimple $(\OO_S[1/s]/\mathfrak{p})$-group.
But note that $\OO_S[1/s][G]\otimes_{\OO_S[1/s]}(\OO_S[1/s]/\mathfrak{p})=\OO_S[G]\otimes\F_v$  where $v$ is the valuation corresponding to the prime ideal $\mathfrak{p}\cap\OO_S$.
This means that for almost every $v\in V_f^S$ the group $G_{\mathfrak{p}_v}$ is semisimple.

Now maybe after considering a finite extension $k'$ of $k$ ( or the corresponding finite integral extension $\OO_{S'}$ of $\OO_S$) one can find a split maximal $k'$-torus $T$ in $G_{k'}$ ($G_{\OO_{S'}}$) and applying a similar direct limit argument  one shows that $T_{\mathfrak{p}}$ is a split maximal torus in $G_{\mathfrak{p}}$ for almost every $\mathfrak{p}\in\Spec(\OO_{S'})$ (see section 3 in \cite{Con} and in particular Lemmata 3.2.5 and 3.2.7).

This shows that for almost every $v\in V_f^S$ the root systems associated to $G$ and $G_{\mathfrak{p_v}}$ are the same. Since the root system detects being simply connected and absolutely almost simple, this implies that $G_{\mathfrak{p}_v}$ is simply connected and absolutely almost simple for almost every $v\in V_f^S$.

Therefore, for almost every $v\in V_f^S$ the $\F_v$-group $G_{\mathfrak{p}_v}$ is semisimple, simply connected and absolutely almost simple and we have $G(\OO_v)/N_v^1\cong G_{\mathfrak{p}_v}(\F_v)$.
Given a connected simply connected absolutely almost simple algebraic group $H$ defined over a finite field $\Fq$, $H(\Fq)$ is perfect and $H(\Fq)/Z(H(\Fq))$ is a finite simple group of Lie type (except for a finite number of $q$'s), see Theorem \ref{pre:th:FGLT_simple}.
Hence, noting that for a given prime power $q$ there is a finite number of valuations $v$ such that $q=q_v$ (Lemma \ref{pre:lm:numberofval}), it
follows that the statement is true for almost every $v\in V_f^S$.
\end{proof}

\begin{lemma} \label{lm:Nv1_standar}
Let $G$ be an algebraic group defined over $k$.
Then for almost every $v\in V_f^S$, $G(\mm_v)$ is an $\OO_v$-standard subgroup of $G(\OO_v)$.
\end{lemma}
\begin{proof}
Recall that we  have given $G$ a structure of affine group $\OO_S$-scheme.
Multiplication and inverse operations in $G$ are given by $\OO_S$-algebra maps $m^*: \OO_S[G]\to \OO_S[G]\otimes_{\OO_S}\OO_S[G]$ and $i^*: \OO_S[G]\to \OO_S[G]$.
Recall that the group $G(k_v)$ inherits the structure of a $k_v$-analytic group, where $G(\OO_v)$ and $G(\mm_v)$ are open subgroups. We can interpret $G(\mm_v)$ as 
\[
G(\mm_v)=\{\varphi\in\Hom_{\OO_S}(\OO_S[G],\OO_v)\,:\, \varphi(I_{\OO_S[G]})\subseteq\mm_v\},
\]
where $I_{\OO_S[G]}=\ker\epsilon$ and $\epsilon:\OO_S[G]\to\OO_S$ represents the identity element.
If we regard $G\subseteq\GL_N$, this is nothing but the points of $G$ all of whose coordinates lie in $\mm_v$.
Let us write $I=I_{\OO_S[G]}$ and define
\[
I_v:=\langle \mathfrak{p}_v,I\cap\OO_S[G]\rangle\subset\OO_S[G].
\]
Consider $\OO_S[G]_{I_v}$, the localization of $\OO_S[G]$ at $I_v$.
Let $\mathcal{R}_v$ denote the completion of $\OO_S[G]_{I_v}$ with respect to its unique maximal $I_v\OO_S[G]_{I_v}$ and write $\mathcal{M}_v$ for the maximal ideal of $\mathcal{R}_v$.
We then have 
\[
G(\mm_v)=\{\varphi\in\Hom_{\OO_S}(\OO_S[G],\OO_v)\,:\, \varphi(I_{\OO_S[G]})\subseteq\mm_v\}\cong\Hom(R_v,\OO_v),\]
where the latter are assummed to be continuous homomorphisms with respect to the $\mathcal{M}_v$ and $\mm_v$-adic topologies.
The isomorphism follows from the fact that any $\varphi\in\Hom_{\OO_S}(\OO_S[G],\OO_v)$ such that $\varphi(I_{\OO_S[G]})\subseteq\mm_v$ uniquily extends to an elelement in $\Hom(\mathcal{R}_v,\OO_v)$. 
Note that the isomorphisms above are indeed group isomorphisms, where multiplication and inversion naturally extend from $G(\mm_v)$ -- $m^*$ and $i^*$ extend to $\mathcal{R}_v$ and this represents affine group schemes over $\OO_v$.

If we show that $\mathcal{R}_v\cong \OO_v[[y_1,\ldots,y_d]]$, the ring of power series over $d$ indeterminates with coefficients in $\OO_v$, this would induce an isomorphism of groups $G(\mm_v)\cong\Hom(\OO_v[[y_1,\ldots,y_d]],\OO_v)$.
Note that there is a natural identification $\Hom(\OO_v[[y_1,\ldots,y_d]],\OO_v)=\mm_v^d$ and under this identification the induced multiplication and inversion in the latter are given by analytic functions on $d$ variables with coefficients in $\OO_v$, that is, $G(\mm_v)$ is a standard subgroup of $G(\OO_v)$.

We claim that such isomorphism exists for almost every $v\in V_f^S$. 
First let us note that if $\ch k=0$ then $\ch \F_v\notin\mathcal{M}_v^2$, otherwise $I_{\OO_S[G]}$ would not be a proper ideal.
Hence we may assume that $\ch\F_v\notin\mathcal{M}_v^2$.
Now it suffices to show that $\mathcal{R}_v$ is regular, since in this case $\mathcal{R}_v$ is a complete regular local ring such that $\ch\F_v\notin\mathcal{M}_v^2$ and by Cohen's Structure Theorem of complete regular local rings \cite[Theorem 29.7]{Mat} $\mathcal{R}_v\cong\OO_v[[y_1,\ldots,y_d]]$.
Now $\mathcal{R}_v$ is the completion of $\OO[G]_{\I_v}$, so $\mathcal{R}$ is regular if and only if $\OO[G]_{I_v}$ is (\cite[Proposition 11.24]{AtiMac}).
If $d=\dim G$ (as an affine $k$-group) then $\dim \OO[G]=\dim k[G]+1$ ($\dim A$ stands for the Krull dimension of a ring $A$) and $\dim \OO_S[G]_{I_v}=d+1$ for every $v\in V_f^S$
 since $\OO[G]$ is equicodimensional, see \cite[10.6.1]{EGAIV}.
Now note that $\dim_{\F_v}{I_v/I_v^2}=1+\dim_{\F_v}(I\otimes\F_v/I^2\otimes\F_v)$ and the latter equals  $1+\dim G$ for almost every $v\in V_f^S$ ( $I/I^2$ is isomorphic as a vector space to the tangent space at $1$ of $G$ or equivalently its Lie algebra, which is given by linear equations with coefficients in $\OO_S$, these equations will clearly give a $\F_v$-vector space of dimension $\dim G$ for almost every $v\in V_f^S$).
Hence $\OO[G]_{I_v}$ is regular for almost every $v\in V_f^S$ and the proof is complete.

\end{proof}
Putting together the previous results we obtain the following lemma.

\begin{lemma}\label{lm:goodvaluations}
Let $\mathcal{P}$ be the subset of all $v\in V_f^S$ such that:
\begin{enumerate}[label=\roman*)]
\item\label{lm:good_valuations_G(Fv)_simple} $G(\OO_v)/N_v^1$ is perfect and is a central extension of a finite simple group of Lie type $H(\F_{v})$, where $H$ is of the Lie type of $G$.
\item\label{lm:good_valuations_Nv1_standard}
$N_v^1$ is a standard subgroup of $G(\OO_v)$.
\item\label{lm:good_valuations_[G(O_v),Nv1]=Nv1} for every $n\in\NN$, $[N_v^n,G(\OO_v)]=N_v^n$.
\item $G(\OO_v)$ is a perfect group, that is, $G(\OO_v)=[G(\OO_v),G(\OO_v)]$. \label{lm:good_valuations_G(O_v)_perfect}

\end{enumerate}
Then $V_f\setminus\mathcal{P}$ is finite, i.e. , statements \ref{lm:good_valuations_G(Fv)_simple}, \ref{lm:good_valuations_Nv1_standard}, \ref{lm:good_valuations_[G(O_v),Nv1]=Nv1} and \ref{lm:good_valuations_G(O_v)_perfect} are satisfied for almost every $v\in V_f^S$.
\end{lemma}

\begin{proof}
Note that \ref{lm:good_valuations_G(Fv)_simple} and \ref{lm:good_valuations_Nv1_standard} are Lemma \ref{lm:G(Fv)_simple} and \ref{lm:Nv1_standar}  respectively.

We prove \ref{lm:good_valuations_[G(O_v),Nv1]=Nv1}.
Given $v\in V_f^S$, suppose that  $L(\OO_v)=[L(\OO_v),G(\OO_v)]$ and that $N_v^1$ is a standard subgroup of $G(\OO_v)$.
The group $N_v^1=G(\mm_v)$ is a normal subgroup of $G(\OO_v)$, so $G(\OO_v)$ acts on it by conjugation.
Recall that for an $\OO_v$-standard group we have a natural identification $f:N_v^1\to \mm_vL(\OO_v)$ and under this identification we have for every $k\in\NN$ an isomorphism of $G(\OO_v)$-modules
\[
N_v^k/N_v^{2k}\cong L(\OO_v)/\mm_v^kL(\OO_v).
\]
In particular, $L(\OO_v)=[L(\OO_v),G(\OO_v)]$ implies that for every $k\in\NN$, we have
\[
[N_v^k,G(\OO_v)]N_v^{2k}=N_v^k
\]
Now fix $n\in\NN$. Note that we have
\[
[N_v^n,G(\OO_v)]N_v^{2n}=[N_v^n,G(\OO_v)][N_v^{2n},G(\OO_v)]N_v^{4n}.\]
Hence an easy induction argument gives that  for every $k>n$ we have $[N_v^n,G(\OO_v)]N_v^k=N_v^n$ and this implies $[N_v^n,G(\OO_v)]=N_v^n$.
From \ref{lm:good_valuations_Nv1_standard} and Lemma \ref{lm:L(Ov)_fg_module} it follows that our initial assumptions hold for almost every $v\in V_f^S$ and the proof is complete.

For \ref{lm:good_valuations_G(O_v)_perfect} note that \ref{lm:good_valuations_G(Fv)_simple} implies that $[G(\OO_v),G(\OO_v)]N_v^1=G(\OO_v)$. But \ref{lm:good_valuations_[G(O_v),Nv1]=Nv1} implies that $[G(O_v),N_v^1]=N_v^1$, hence $N_v^1\subset[G(\OO_v),G(\OO_v)]$ and $[G(\OO_v),G(\OO_v)]=G(\OO_v)$ for almost every $v\in V_f^S$.

\end{proof}

We are now ready to prove the following proposition, which will be used for the proof of Theorem \ref{th:PRG}.
\begin{proposition}\label{th:deltaP1}
Let $\mathcal{P}$ be as in Lemma \ref{lm:goodvaluations}. There exists $\delta>0$ such that for every $v\in\mathcal{P}$, every nontrivial irreducible representation of $G(\OO_v)$ is of dimension at least $q_v^{\delta}$.
\end{proposition}

\begin{proof}
Pick $\delta$ as in Lemma \ref{lm:projrep}, that is, every projective representation of a perfect central extension of a finite simple group of Lie type, say $H(\Fq)/Z(H(\Fq))$, has dimension at least $q^\delta$.
Suppose $v\in\mathcal{P}$ and take $\rho\in\Irr(G(\OO_v))$.
Let $\rho_{|N_v^1}$ be the restriction of $\rho$ to $N_v^1$.

If $\rho_{|N_v^1}$ is trivial, that is, if $N_v^1\subseteq\ker\rho$, we can see $\rho$ as a representation of $G(\OO_v)/N_v^1$.
By Lemma \ref{lm:goodvaluations}.i) we can apply Lemma \ref{lm:projrep} to get $\dim\rho\geq q_v^{\delta}$.

Suppose now that there is a nontrivial $\mu \in\Irr(N_v^1)$ such that $\mu$ is a constituent of $\rho_{|N_v^1}$.
Let $H\leq G(\OO_v)$ be the stabilizer of $\mu$. We claim that $H$ is a proper subgroup of $G(\OO_v)$.
For suppose not, and let $k\in\NN$ be maximal such that $\mu_{|N_v^k}$ is nontrivial.
By assumption $[N_v^k,N_v^k]\subseteq N_v^{2k}\subseteq\ker\mu$. Hence any nontrivial constituent $\lambda\in\Irr(N_v^k)$ of $\mu_{|N_v^k}$ is linear.
Indeed, there is only one such constituent, since $[N_v^k,N_v^1]\leq N_v^{k+1}\leq \ker\mu$ implies $N_v^k\subseteq \Z(\mu)$, the center of $\mu$, see \cite[2.26 and 2.27]{Isa}.
It follows that $G(\OO_v)$ stabilizes $\lambda$.
But since $\lambda$ is linear this implies $[N_v^k,G(\OO_v)]\subseteq\ker\lambda$.
By lemma \ref{lm:goodvaluations}.\ref{lm:good_valuations_[G(O_v),Nv1]=Nv1}, we have $N_v^k=[N_v^k,G(\OO_v)]$, i.e., $N_v^k\subseteq\ker\lambda$.
But $\lambda$ is an arbitrary constituent of $\mu_{|N_v^k}$, thus $\mu_{|N_v^k}$ must be trivial, which contradicts our choice of $k$.
Therefore $N_v^1\leq H<G(\OO_v)$.
Now by Clifford's Theory (Lemma \ref{lm:RG_induction_stabilizer}) , we know that $\rho$ is induced from an irreducible representation $\tilde{\rho}\in\Irr(H)$, in particular $\dim\rho=\dim\tilde{\rho}|G(\OO_v):H|$.
We can see $H/N_v^1$ as a subgroup of $G(\OO_v)/N_v^1$ and by Lemma \ref{lm:goodvaluations}.\ref{lm:good_valuations_G(Fv)_simple}  we can apply Lemma \ref{lm:projrep} to obtain $|G(\OO_v):H)|\geq q_v^{\delta}$.
It follows that $\dim\rho\geq q_v^{\delta}$ and the proof is complete.
\end{proof}


\begin{lemma}\label{lm:typec}
For every $v\in V_f^S$, there exists $c>0$ such that $G(\OO_v)$ is of representation type $c$, i.e., $|G(\OO_v)/K_n(G(\OO_v))|\leq cn^c$ for every $n\in\NN$.
In particular $G(\OO_v)$ has PRG.
\end{lemma}

\begin{proof}
Recall that the group $G(\OO_v)$ has a natural structure of $k_v$-analytic group.
Hence there exists $H\leq_o G(\OO_v)$ such that $H$ is a $\OO_v$-standard group.
Since $H$, being an open subgroup of a profinite group, is a finite index subgroup, it suffices to show the claim for such a group $H$ (Lemma \ref{lm:RG_finite_index_subgroup}).
For $\ch k>0$, $H$ can be given the structure of an $\F_p[[t]]$-analytic group (Lemma \ref{pre:lm:Zp_and_Fpt_analytic}) whose Lie algebra is finitely generated as an $H$-module by Lemma \ref{lm:L(Ov)_fg_module}, so the result follows from Theorem \ref{th:RG_Fp[[t]]-standard_type_c}.
Analogously if $\ch k=0$, then $H$ can be given the structure of a compact $p$-adic analytic group for $p=\ch \F_v$ (Lemma \ref{pre:lm:Zp_and_Fpt_analytic}) with perfect Lie algebra $L_G$ ($G$ is a simple algebraic $k$-group and for $\ch k=0$ this implies its Lie algebra is perfect).
Thus the claim follows from Theorem \ref{th:RG_p-adic_perfect_type_c}.
\end{proof}

\begin{theorem}\label{th:allvaluations}
There exists $c>0$ such that for every $v\in V_f^S$, $G(\OO_v)$ is of representation type $c$.
\end{theorem}
\setcounter{claim}{0}
\begin{proof}
Note that thanks to Lemma \ref{lm:typec} it suffices to show that there exists $c>0$ such that $G(\OO_v)$ is of type $c$ for almost every $v\in V_f^S$.
Hence we want to give a polynomial bound for $|G(\OO_v)/K_n(G(\OO_v))|$ that holds for almost every $v\in V_f^S$.
In particular, we may assume that $v\in\mathcal{P}$ (see Lemma \ref{lm:goodvaluations} for the definition of $\mathcal{P}$).

Consider $\rho\in\Irr_{\leq n}(G(\OO_v))$ and put $r'= \lceil{\log_{q_v}n} \rceil$. We claim that there exist $c_1,c_2\in\NN$ such that for almost every $v\in V_f^S$ we have $N_v^{c_1r'+c_2}\subseteq \ker \rho$, in particular,
\[ N_v^{c_1r'+c_2}\subseteq K_n(G(\OO_v)).\]
This will finish the proof, for we would have
\begin{align*}
|G(\OO_v)/K_n(G(\OO_v))|& \leq|G(\OO_v)/N_v^{c_1r'+c_2}|=|G(\OO_v)/N_v^1||N_v^1/N_v^{c_1r'+c_2}|.
\end{align*}
From our initial embedding $\iota:G\to\GL_N$ it follows that $G(\OO_v)/N_v^1$ embedds in $\GL_N(\F_v)$ and this implies $|G(\OO_v)/N_v^1|\leq q_v^{N^2}$.
On the other hand the group $N_v^1$ is $\OO_v$-standard, hence $|N_v^k/N_v^{k+1}|=q_v^{\dim G}$ for every $k$ and we have $|N_v^1/N_v^{c_1r'+c_2}|\leq q_v^{(c_1r'+c_2)\dim{G}}$.
Finally note that for $v\in\mathcal{P}$, Proposition \ref{th:deltaP1} gives $q_v\leq n^{1/\delta}$.
All together this implies
\[
|G(\OO_v)/K_n(G(\OO_v))| \leq q_v^{N^2}q_v^{(c_1r'+c_2)\dim G}\leq n^{c_1\dim G+(c_2\dim G+N^2)/\delta}\]
for every $v\in\mathcal{P}$, as we required.

To prove the claim let us first observe that since $N_v^1$ is a normal subgroup of $G(\OO_v)$, it suffices to show
\[
N_v^{c_1r'+c_2}\subseteq K_n(N_v^1).
\]
Indeed, suppose $\rho\in\Irr(G(\OO_v))$ and let $\rho'\in\Irr(N_v^1)$ be an irreducible constituent of $\rho_{|N_v^1}$.
Then we would have $N_v^{c_1r'+c_2}\subseteq\ker\rho'$, and being normal, $N_v^{c_1r'+c_2}\subseteq\ker\Ind_{N_v^1}^{G(\OO_v)}\rho'$, which by Frobenius reciprocity has $\rho$ as an irreducible constituent.
Hence $N_v^{c_1r'+c_2}\subset\ker\rho$ as we required.
To prove the claim we will follow the proof of Theorem \ref{th:RG_Fp[[t]]-standard_type_c}.

Recall that for every $v\in\mathcal{P}$, $N_v^1$ is a standard subgroup of the $\OO_v$-analytic group $G(\OO_v)$ with associated Lie algebra $L(\OO_v)$.
Hence, we have a homeomorphism from $N_v^1$ to $\pi_v L(\OO_v)$ given by
\[
\begin{array}{ccc}
N_v^1	&\to	&\pi_vL(\OO_v)\\
g		&\mapsto&{\bf g},
\end{array}
\]
where $\pi_v$ stands for a uniformizer of the local ring $\OO_v$.
Let us denote the adjoint action of $N_v^1$ by $\Ad_v=\Ad_{\OO_v}:N_v^1\to \Aut_{\OO_v}(L(\OO_v))$.
Recall that for every $k\in\NN$, $N_v^k/N_v^{2k}$ is an abelian normal subgroup of $N_v^1/N_v^{2k}$, where $N_v^1$ acts by conjugation and we have an isomorphism of $N_v^1$-modules $N_v^k/N_v^{2k}\cong L(\OO_v)/\mm_v^k L(\OO_v)$  given by
\begin{align}
\begin{array}{ccc} \label{iso}
N_v^k/N_v^{2k}&\to & L(\OO_v)/\pi_v^{k} L(\OO_v)\\
g&\mapsto & \pi_v^{-k}{\bf g},
\end{array}
\end{align}
see the proof of Theorem \ref{th:RG_p-adic_perfect_type_c}.
For every $v\in V_f^S$, let us write $\A_v$ for the $\OO_v$-subalgebra of $\End_{\OO_v} L(\OO_v)$ generated by $\{\Ad_v(g):g\in N_v^1\}$ and for every subset $H\subseteq G(\OO_v)$ let $\I_v(H)$ denote the left ideal of $A_v$ generated by $\{\Ad_v(h)-\Id\;:\; h\in H\}$.
We will write $\I_{v,k}$ for the ideal $\I_v(N_v^k)$.

We will state and prove a series of claims that are required for the proof.

\begin{claim}
For every $v\in V_f^S$, $[L(\OO_v),N_v^1]$ has finite index in $L(\OO_v)$.
\end{claim}
\noindent This is Lemma \ref{lm:L(Ov)_fg_module}.

\begin{claim}
There exists $s\in\NN$ such that for almost every $v\in V_f^S$ and every $k>0$, we have
\[
\mm_v^{sk} L(\OO_v)\subseteq [\mm_v^k L(\OO_v),N_v^k]=\mm_v^k L(\OO_v)\I_{v,k}.
\]
\end{claim}

\noindent Since $G$ is an algebraic group, the map $\Ad-\id$ is a morphism of algebraic varieties given by
\[
\begin{array}{rclcc}
L&\times&G&\to&L\\
(l&,&g)&\mapsto&[l,g]\\
(l_1,\ldots,l_d)&,&(g_1,\ldots,g_{N^2})&\mapsto&  (f_1(l_1,\ldots,l_d,g_1,\ldots,g_{N^2}),\\
& & & & \ldots,\\
& & & & f_d((l_1,\ldots,l_d,g_1,\ldots,g_{N^2})),\\
\end{array}
\]
where $d=\dim G$, we have chosen a basis for $L$ such that $(l_1,\ldots,l_d)$ are the coordinates of $l$ and $(g_1,\ldots,g_{N^2})$ stand for the usual coordinates induced from the embedding $G\hookrightarrow \GL_N$.
Moreover,the $f_i$ are polynomials with coefficients in $\OO_S$ (note that $G$ is an affine group $\OO_S$-scheme), which give the coordinates of $[l,g]$ in $L$ with respect to the chosen basis.
Note that $\Ad_v$ is given by the same polynomials for every $v\in V_f^S$.

Put ${\bf x}=(x^1,\ldots,x^d)$, where $x^j=(x_1^j,\ldots,x_d^j)$, ${\bf y}=(y^1,\ldots,y^{d})$, where $y^j=(y_1^j,\ldots,y_{N^2}^j)$ and similarly for $\ {\bf l}\in (L(\OO_v))^d, \ {\bf g}\in (N_v^1)^d.$
By abuse of notation we will identify elements of $L(\OO_v)$ and $G(\OO_v)$ with their corresponding coordinates.
Consider the matrix $\M({\bf x},{\bf y})=(m_{ij}({\bf x},{\bf y}))_{1\leq i,j\leq d}$, where $m_{ij}({\bf x},{\bf y})=f_i(x^j, y^j)$.
Then $\M({\bf l},{\bf g})$ represents the coordinates of $d$ elements of the form $[l,g]$, where $l\in L(\OO_v)$ and $g\in N_v^1$.
Write $V({\bf l},{\bf g})$ for the $\OO_v$-module generated by these $d$ elements.

Note that for ${\bf l}\in(\mm_v^k L(\OO_v))^d$ and ${\bf g}\in (N_v^r)^d$ we have $V({\bf l},{\bf g})\subseteq [\mm_v^k L(\OO_v),N_v^r]$.
Moreover, we claim that there exists a pair $({\bf l},{\bf g})\in (\mm_v^k L(\OO_v))^d \times (N_v^r)^d$ such that equality holds.
Indeed, $[\mm_v^k L(\OO_v),N_v^k]\subseteq \mm_v^k L(\OO_v)$ and $\mm_v^kL(\OO_v)$ is a free $\OO_v$-module of rank $d$ generated by elements of the form $[l,g]$ for $l\in \mm_v^kL(\OO_v)$ and $g\in N_v^r$.
Since $\OO_v$ is a local ring, this implies that $[\mm_v^k L(\OO_v),N_v^r]$ is generated by at most $d$ elements of the form $[l,g]$ where $l\in \mm_v^kL(\OO_v)$, $g\in N_v^r$.

Now from Claim $1$ it follows that $[L(\OO_v),N_v^1]$ has finite index in $L(\OO_v)$.
Therefore, $[L(\OO_v),N_v^1]$ is  a free $\OO_v$-module of rank $d$ as well.
But observe that $V({\bf l},{\bf g})$ is free of rank $d$ if and only if $\det \M({\bf l},{\bf g})\neq 0$.
Hence, there must exist a pair $({\bf l},{\bf g})\in(\mm_vL(\OO_v))^d\times(N_v^1)^d$ such that $\det \M({\bf l},{\bf g})\neq 0$.
In particular $\det \M({\bf x},{\bf y})$ is a nontrivial polynomial with coefficients in $\OO_S$.
We can see this determinant as a regular function of the affine variety $L\times G$.
Moreover, let us note that if we put $r=\ord_v(\det \M({\bf l},{\bf g}))$, then $\mm_v^r L(\OO_v)\subseteq V({\bf l},{\bf g})$.
Therefore it suffices to show that there exists $s\in\NN$ such that for almost every $v\in V_f^S$ and every $k\geq 1$, there exists a pair $({\bf l_v^k},{\bf g_v^k})\in (\mm_v^k L(\OO_v))^d\times(N_v^k)^d$ that satisfies $\ord_v(\det \M({\bf l_v^k},{\bf g_v^k}))=ks$.
This follows from Lemma \ref{lm:polval} below.
\begin{claim}

For every $u>s$ and for almost every $v\in V_f^S$, we have
\[N_v^{u+s}\leq[N_v^u,N_v^1].\]
\end{claim}

\noindent
Recall that we may assume $v\in\mathcal{P}$, in particular, we may assume that $N_v^1$ is a standard subgroup of $G(\OO_v)$.
From the previous claim it follows that $\pi_v^{u+s}L(\OO_v)\leq[\pi_v^{u+1}L(\OO_v),N_v^1]$.
Hence, for $u>s$ the isomorphism (\ref{iso}) gives $N_v^{u+s}\leq[N_v^u,N_v^1]N_v^{2u}$.
Applying induction on $l$ as in Claim 1 in Theorem \ref{th:RG_Fp[[t]]-standard_type_c} we obtain that if $u>s$, then for every $l\geq 2u$, $N_v^{u+s}\leq[N_v^u,N_v^1]N_v^l$.
Therefore $N_v^{u+s}\leq[N_v^u,N_v^1]$.

\begin{claim}
There exists $l\in\NN$ such that for every $v\in V_f^S$, we have 
\[\I_{v,1}^l\leq\pi_v\A_v.\]
\end{claim}

\noindent 
Note that $\A_v$ is an $\OO_v$-submodule of the finitely generated $\OO_v$-module $\End_{\OO_v}(L(\OO_v))$.
Since $\OO_v$ is Noetherian, this implies that $\A_v$ is a finitely generated $\OO_v$-module as well.
In particular, by Nakayama Lemma $\A_v/\pi_v\A_v$ must be finite.
Now note that from the description of the adjoint action (see \eqref{eq:adjoint_locally} in \ref{sec:standard_groups}) it follows that $\Ad(G)-\Id\subseteq \pi_v\End_{\OO_v}(L(\OO_v))$, which implies $\I_{v,1}\subseteq \pi_v\End_{\OO_v}(L(\OO_v))$, i.e., $\I_{v,1}^i\subseteq \pi_v^i\End_{\OO_v}(L(\OO_v))$.
Thus $\I_{v,1}^i$ gives a strictly decreasing sequence of submodules, so there exists $l\in\NN$ such that $\I_{v,1}^l\leq \pi_vA_v$.
But note that since $\End_{\OO_v}(L(\OO_v))$ is generated by $d^2$ ($d=\dim G$) elements it suffices to take $l=d^2$, which is independent of $v$.

\begin{claim}
For every $u>s$ and for almost every $v\in V_f^S$, we have
\[\I_{v,u+s}\subseteq\I_{v,1}\I_{v,u}+\I_{v,u}\I_{v,1}.\]
\end{claim}

If $g\in N_v^{u+s}$, then $g\in [N_v^u,N_v^1]$ for almost every $v\in V_f^S$ by Claim 3.
Now apply the same argument as in Claim 3 of Theorem \ref{th:RG_Fp[[t]]-standard_type_c}.
\begin{claim}
For every $u>s$ and  for almost every $v\in V_f^S$ we have 
\[
\I_{v,u+s(2l-1)}\subseteq \pi_v\I_{v,u}.
\]
\end{claim}
\noindent Applying the previous claim $(2l-1)$ times we obtain \[\I_{v,u+s(2l-1)}\subseteq \sum_{i+j=2l-1}\I_{v,1}^i\I_{v,u}\I_{v,1}^j.\]
It follows from Claim 4 that $\I_{v,u+s(2l-1)}\leq\pi_v\I_{v,u}$.

So far we have shown that the set of all $v\in V_f^S$ that satisfy all of the above claims is cofinite in $V_f^S$. 
Hence we may assume that all of the above claims hold for $v$.

Consider now $\rho\in\Irr(N_v^1)$ and suppose that $N_v^{2(k-1)} \nsubseteq \ker\rho $ but $N_v^{2k}\subseteq\ker\rho$.
We may assume that $k\geq s$, otherwise $N_v^{2s}\subseteq\ker\rho$ and we may add $2s$ to the constant $c_2$.
Let $\lambda\in\Irr(N_v^k)$ be an irreducible constituent of $\rho_{|N_v^k}$ and let $H\leq N_v^1$ be the stabiliser of $\lambda$.
Pick $\delta$ as in Lemma \ref{lm:projrep} and put $r:=\lceil\log_{q_v}|N_v^1:H|\rceil$. Note that if $n=\dim \rho$ then $|N_v^1:H|\leq n$ by Lemma \ref{lm:RG_induction_stabilizer} and so $r\leq r'$.

\begin{claim}
Let $M\leq H$ and suppose $M\trianglelefteq G(\OO_v)$. Put $m:=\lceil\log_{q_v}|N_v^1:M|\rceil$.
Then for almost every $v\in V_f^S$ we have $N^{s(2l-1)(\lceil\frac{m}{\delta}\rceil+1)}\leq H$.
\end{claim}
\noindent
For convenience we will put $c:=s(2l-1)$.
Consider the chain of subgroups
\[ M\leq MN_v^{c(\lceil\frac{m}{\delta}\rceil+1)}\leq MN_v^{c\lceil\frac{m}{\delta}\rceil}\leq MN_v^{c(\lceil\frac{m}{\delta}\rceil-1)}\leq\ldots\leq  MN_v^{c}.\]
Suppose that there is some equality in the above chain, say $MN_v^{c(i+1)}=MN_v^{ci}$ for some $1\leq i\leq \lceil\frac{m}{\delta}\rceil$.
From the previous equality and Claim $6$ (see Claim $9$ in Theorem \ref{th:RG_p-adic_perfect_type_c} as well) it follows that
\[
\I(M)+\I_v^{ci}=\I(M)+\I_v^{c(i+1)}=\I(M)+\pi_v\I_v^{ci}.
\]
We can apply Nakayama Lemma to the above equality of $\OO_v$-modules to obtain $\I_v^{ci}\leq \I(M)$.
In particular, $L(\OO_v)\I_v^{ci}\leq L(\OO_v)\I(M)$.
Therefore, the isomorphism \eqref{iso} gives
\[ [N_v^k,N_v^{ci}]N_v^{2k}\leq [N_v^k,M]N_v^{2k}.\]
In particular, $N_v^{ci}$ stabilizes $\lambda$ and $N_v^{c(\lceil\frac{m}{
\delta}\rceil+1)}\leq N_v^{ci}\leq H$.
We will show that there must be an equality in the chain.
Suppose that all inequalities in the chain are strict.
We claim that if $MN_v^{c(i+1)}<MN_v^{ci}$ then $|MN_v^{c(i+1)}:MN_v^{ci}|\geq q_v^{\delta}$.
If this holds, then $|N_v^1:M|\geq \left( q_v^{\delta}\right)^{\lceil\frac{m}{\delta}\rceil+1}>q_v^m\geq |N_v^1:M|$, which contradicts the choice of $m$.
Suppose then that $MN_v^{c(i+1)}<MN_v^{ci}$ and consider the following refinement.
\[
MN_v^{c(i+1)}\leq MN_v^{c(i+1)-1}\leq MN_v^{c(i+1)-2}\leq\ldots \leq N_v^{ci}.
\]
Then for some $c(i+1)\leq j\leq ci+1$ we have $MN_v^{j+1}<MN_v^{j}$.
Note that since $M$ is normal in $G(\OO_v)$, $G(\OO_v)$ acts by conjugation on the abelian group $MN_v^j/MN_v^{j+1}$.
The action is non trivial, for $[G(\OO_v),N_v^j]=N_v^j$ by \ref{lm:goodvaluations}.\ref{lm:good_valuations_[G(O_v),Nv1]=Nv1}, but $N_v^1$ does act trivially.
Hence, $MN_v^j/MN_v^{j+1}$ is a nontrivial $G(\OO_v)/N_v^1$ module.
By \ref{lm:goodvaluations}.\ref{lm:good_valuations_G(Fv)_simple} we can apply Lemma \ref{lm:projrep} to obtain $|MN_v^j/MN_v^{j+1}|\geq q_v^\delta$, which finishes the proof.

\begin{claim}
For almost every $v\in V_f^S$ we have $N_v^{s(2l-1)\lceil\frac{dr}{\delta}\rceil+1}\leq H$.
\end{claim}
\noindent Put $\displaystyle H':=\bigcap_{g\in G(\OO_v)}H^g$, where $H^g:=g^{-1}Hg$.
Suppose $|G(\OO_v):H'|\leq|G(\OO_v):H|^d$, where $d=\dim G$.
Then, since $H'\trianglelefteq G(\OO_v)$ we can apply the previous claim to obtain $N_v^{s(2l-1)(\lceil\frac{dm}{\delta}\rceil+1)}\leq H$.
We claim that $\displaystyle H'=\bigcap_{g_1,\ldots,g_d}H^g$ for some $g_1,\ldots,g_d\in G(\OO_v)$, which implies $|G(\OO_v):H'|\leq|G(\OO_v):H|^d$.
Let us write $\Irr(N_v^k|N_v^{2k}):=\{\psi\in\Irr(N_v^k)\;|\; N_v^{2k}\subseteq\ker\psi\}$.
Note that $\Irr(N_v^k|N_v^{2k})\cong\Irr(N_v^k/N_v^{2k})$.
Recall that $\lambda\in\Irr(N_v^k|N_v^{2k})$ and since $N_v^k\trianglelefteq G(\OO_v)$ we have an action of $G(\OO_v)$ in $\Irr(N_v^k)$. Indeed, $G(\OO_v)$ acts on $\Irr(N_v^k|N_v^{2k})$ as well.
It is clear that $H^g$ is the stabiliser of $\lambda^g$ under this action.
Now by the isomorphism (of $G(\OO_v)$-modules) \eqref{iso} we obtain an isomorphism of $G(\OO_v)$-modules
\[
\Irr(N_v^k|N_v^{2k})\cong\Irr(N_v^k/N_v^{2k})\cong\Irr(L(\OO_v)/\pi_v^kL(\OO_v)).
\]
We can give $\Irr(L(\OO_v)/\pi_v^kL(\OO_v))$ an $\OO_v$-module structure given by $r\lambda(l):=\lambda(rl)$ for every $r\in\OO_v,\lambda\in\Irr(L(\OO_v)/\pi_v^kL(\OO_v))$ and $l\in L(\OO_v)$.
Since the action of $G(\OO_v)$ on $L(\OO_v)$ is the adjoint action, this action is $\OO_v$-linear and it follows that the $\OO_v$-structure on $\Irr(L(\OO_v)/\pi_v^kL(\OO_v))$ is compatible with the $G(\OO_v)$-action.
It is then clear that for $g\in G(\OO_v)$ and $\lambda\in\Irr(L(\OO_v)/\pi_v^kL(\OO_v))$, $g$ stabilizes $\lambda$ if and only if $g$ stabilizes $\langle\lambda\rangle_{\OO_v}$.
Now $L(\OO_v)/\pi_v^kL(\OO_v)$ is an $\OO_v$-module of rank $d$ and so is $\Irr(L(\OO_v)/\pi_v^kL(\OO_v))$.
It follows that for $\lambda\in\Irr(L(\OO_v)/\pi_v^kL(\OO_v))$, $\langle \lambda^g\;:\; g\in G(\OO_v)\rangle_{\OO_v}$ is an $\OO_v$-submodule generated by $\lambda^{g_1},\ldots,\lambda^{g_d}$ for some $g_1,\ldots,g_d\in G(\OO_v)$, which finishes the proof.

We are finally ready to prove our very first claim.
Recall that $r:=\lceil\log_{q_v}|N_v^1:H|\rceil$ and let us set $c':=\lceil \frac{d}{\delta}\rceil$.
Suppose that $k\geq s^2(2l-1)(c'r+1)+3$.
Then it follows from Claim 2 that
\[ \pi_v^{s^2(2l-1)(c'r+1)}L(\OO_v)\I_{v,1}\leq L(\OO_v)\I_{v,s(2l-1)(c'r+1)}.\]
Now from Claim 3 and isomorphism \eqref{iso}, we obtain
\[ N_v^{k+s^2(2l-1)(c'r+1)+s}\leq [N_v^{k+s^2(2l-1)(c'r+1)},N_v^1]N_v^{2k}\leq [N_v^k, N_v^{s(2l-1)(c'r+1)}]N_v^{2k}.\]
However, since $N_v^{s(2l-1)(c'r+1)}\leq H$ (Claim $8$), we must have \[N_v^{k+s^2(2l-1)(c'r+1)}\leq\ker\lambda.\]
But recall that $N_v^{2(k-1)}\nleq\ker\lambda$, hence $2k-2\leq k+s^2(2l-1)(c'r+1)$, or equivalently, $k\leq s^2(2l-1)(c'r+1)+2$, which contradicts our assumption. 
Therefore we must have
$k\leq s^2(2l-1)(c'r+1)$, so there exist $c_1,c_2$ such that
\[
N_v^{c_1r'+c_2}\leq N_v^{c_1r+c_2}\leq K_n(N_v^1),\]
which finishes the proof.
\end{proof}

We are finally ready to Prove Theorem \ref{th:PRG}.
\begin{proof}[Proof of Theorem \ref{th:PRG}.]
We want to show that $G(\widehat{\OO}_S)\cong\prod_{v\in V_f^S}G(\OO_v)$ has polynomial representation growth.
Let us consider the set of valuations $\mathcal{P}$, see Lemma \ref{lm:goodvaluations}.
We can write
\[ G(\widehat{\OO}_S)\cong \prod_{v\in V_f^S\setminus \mathcal{P}}G(\OO_v)\times \prod_{v\in\mathcal{P}}G(\OO_v).\]
Recall that by Lemma \ref{lm:typec} $G(\OO_v)$ has PRG for every $v\in V_f^S$.
As $\mathcal{P}$ is confinite in $V_f^S$, the first term is a finite product of groups each of which having PRG and hence it has itself PRG (Lemma \ref{lm:RP_direct_prod}).
Therefore it suffices to show that the second term in the above isomorphism has PRG.

Let $\rho$ be an irreducible $n$-dimensional representation of $\prod_{v\in\mathcal{P}}G(\OO_v)$.
Then $\rho\cong\rho_1\otimes\ldots\otimes\rho_t$, where $\rho_i$ is an $n_i$-dimensional irreducible representation of $G(\OO_{v_i})$ for some $v_i\in\mathcal{P}$ and $n_1\ldots n_t=n$.
By \ref{lm:goodvaluations}.\ref{lm:good_valuations_G(O_v)_perfect} $n_i>1$  and by Lemma \ref{lm:ndecom} below there are at most $n^d$ possible configurations of the form $(n_1,\ldots,n_t)$ such that $n_1\ldots n_t=n$.
Given such a configuration $(n_1,\ldots,n_t)$, the number of valuations $v_i$ such that $G(\OO_{v_i})$ has a nontrivial representation of dimension $n_i$ is bounded by $n_i^{b/\delta}$ (Proposition \ref{th:deltaP1} and Lemma \ref{pre:lm:numberofval}).
Hence the number of choices for $(v_1,\ldots,v_t)$ is bounded by $\prod n_i^{b/\delta}=n^{b/\delta}$.
Now given $(n_1,\ldots,n_t)$ and $(v_1,\ldots,v_t)$, we know that $G(\OO_{v_i})$ has at most $cn_i^c$ irreducible representations of dimension $n_i$ (Theorem \ref{th:allvaluations}).
Therefore $\prod_{v\in\mathcal{P}}G(\OO_v)$ has at most $n^dn^{b/\delta}c^{\log n}n^c$ irreducible $n$-dimensional representations.
In particular, it has polynomial representation growth, which finishes the proof.

\end{proof}

\subsubsection{Auxiliary results}

We present here some auxiliary results that we use in the proofs of the results of the previous section.
Lemma \ref{lm:polsol} is needed for Lemma \ref{lm:polval}, which is the final step of Claim 2 in Theorem \ref{th:allvaluations}.

\begin{lemma} \label{lm:polsol}
Let $f=f(x_1,\ldots,x_n)$ be a nontrivial polynomial with coefficients in a finite field $\F$.
Put $d=\max\{\deg_{x_i}f:1\leq i\leq n\}$.
If $|\F|>d$, then there exists ${\bf y_0}\in\F^n$ such that $f({\bf y_0})\neq 0$.
\end{lemma}
\begin{proof}
We will apply induction on the number of variables $n$.

For $n=1$ it is clear that the polynomial $f$ can have at most $d$ solutions, which implies the claim.

Now suppose the lemma holds for $n-1$ variables.
We may see $f(x_1,\ldots,x_n)=f(x_n)(x_1,\ldots,x_n)$ as a polynomial in $1$ variable with coefficients in $\F(x_1,\ldots,x_{n-1})$ and of degree at most $d$.
This polynomial can have at most $d$ solutions in $\F$. Hence, there exists $y_n\in\F$ such that $f(y_n)(x_1,\ldots,x_{n-1})$ is a nontrivial polynomial with coefficients in $\F$ and of degree at most $d$.
By induction, there exist $y_1,\ldots,y_{n-1}\in\F$ such that $f(y_n)(y_1,\ldots,y_{n-1})$ is nonzero.
Then, setting ${\bf y_0}=(y_1,\ldots,y_n)$, we have $f({\bf y_0})\neq 0$ and we are done.
\end{proof}
\begin{lemma}\label{lm:polval}




Given $\OO_S[\mathbb{A}^n]:=\OO_S[x_1,\ldots,x_n]$, put ${\tilde{I}:=\langle x_1,\ldots,x_n\rangle}\triangleleft \OO_S[\mathbb{A}^n]$ and consider a radical prime ideal $J\triangleleft \tilde{I}$ such that $\OO_S[X]:=\OO_S[\mathbb{A}^n]/J$ is a finitely generated reduced $\OO_S$-algebra. Suppose $\OO_S[X]\otimes k$ gives an affine $k$-variety and we further assume that this variety is smooth.
For every $\OO_S$-algebra $R$ let us write $X(R)=\Hom_{\OO_S}(\OO_S[X],R)$ for the "$R$-points" of $\OO_S[X]$.

Suppose that $f\in I=\tilde{I}/J\triangleleft\OO_S[X]$ is non trivial. Then there exists $s\in\NN$ such that for almost every $v\in V_f^S$ and  every $k\geq 1$ there exists $x_v^k\in X(\mm_v^k)$ with $\ord_v(f(x_v^k))=ks$.

\end{lemma}

\begin{proof}

The affine $k$-algebra $k[X]:=\OO_S(X)\otimes_{\OO_S}k$ determines an affine algebraic variety defined over $k$, which, by construction, is embedded in the affine space $\mathbb{A}^n$. By abuse of notation we will denote this algebraic variety by $X$.
Suppose that $J=\langle g_1,\ldots,g_r\rangle$, let $l_i$ be the linear component (in terms of the $x_i$'s) of $g_i$ and write $L=\langle l_1,\ldots,l_r\rangle\triangleleft \OO_S[\mathbb{A}^n]$
for the ideal generated by the $l_i$ in $\OO_S[\mathbb{A}^n]$.
Then $k[\mathcal{T}_0(X)]:=\left(\OO_S[\mathbb{A}^n]/L\right)\otimes_{\OO_S}k$ is an affine $k$-algebra which determines the tangent space of $X$ at $0$.

The ring $\OO_S[X]$ is a Noetherian domain, so we have $\bigcap I^n=0$ (\cite[Corollary 10.18]{AtiMac}).
Hence, there exists $s$ such that $f\in I^s\setminus I^{s+1}$.
Pick $\tilde{f}\in \tilde{I}^s\setminus\tilde{I}^{s+1}$ such that $\pi(\tilde{f})=f$, where $\pi$ is the natural projection $\pi:\OO_S[\mathbb{A}^n]\to \OO_S[X]$, and let $\tilde{h}$ be the homogeneous component of degree $s$ of $\tilde{f}$.
We claim that $\tilde{h}\notin L$.
Indeed, suppose that $\tilde{h}=\sum l_im_i$.
Then $\tilde{h}':=\sum(g_i-l_i)m_i\in\tilde{I}^{s+1}$ and $\tilde{f}':=\tilde{f}-\tilde{h}+\tilde{h}'\in\tilde{I}^{s+1}$.
But note that $\pi(\tilde{h})=\tilde{h}'$, hence $\pi(\tilde{f})=\pi(\tilde{f}')=f$, which implies $f\in I^{s+1}$, a contradiction. 

Now for every prime ideal $p_v\subseteq\OO_S$ we have a reduction map
 $\tilde{r}_v:\OO_S[\mathbb{A}^n]\to\F_v[\mathbb{A}^n]:=\OO_S[\mathbb{A}^n]\otimes_{\OO_S}\OO_S/p_v$ and analogously for $\OO_S[X]\to\F_v[X_v]:=\OO_S[X]\otimes_{\OO_S}\OO_S/p_v$.
We say that $\F_v[X_v]$ is the reduction modulo $p_v$ of $\OO_S[X]$.
Since $\OO_S[X]$ is a reduced domain, it follows that for almost every $v\in V_f^S$, $\F_v[X_v]$ is a reduced $\F_v$-algebra with no zero divisors.
For those $v$, $\F_v[X]$ determines an affine algebraic variety defined over $\F_v$, which will be denoted by $X_v$, with tangent space at $0$ given by the affine algebra 
\[
\F_v[\mathcal{T}_0(X_v)]=:\left(\OO_S[\mathbb{A}^n]/L\right)\otimes_{\OO_S}\OO_S/p_v\cong \F_v[\mathbb{A}^n]/\left(L\otimes_{\OO_S}\OO_S/p_v\right).
\]
If we set $\tilde{h}_v=\tilde{r}_v(\tilde{h})$ and $L_v=\tilde{r}_v(L)$, since $\tilde{h}\notin L$ it follows that $\tilde{h}_v\notin L_v$ for almost every $v\in V_f^S$.
It follows that $l(\tilde{h}_v)$, the image of $\tilde{h}_v$ in $\F_v[\mathcal{T}_0(X_v)]$, is non trivial.

Suppose that $\dim k[X]=d$, then since $X$ is a smooth variety we have $\dim\mathcal{T}_0(X)=d$, hence $\dim\F_v[\mathcal{T}_0(X_v)]=d$ for almost every $v\in V_f^S$.
Thus, for almost every $v\in V_f^S$ there is an $\F_v$-isomorphism $\rho_v:\F_v[\mathcal{T}_0(X_v)]\to\F_v[x_1,\ldots,x_d] $.
By construction $\rho_v(l(\tilde{h}_v))$ is a polynomial of degree $s$.
It follows now from Lemma \ref{lm:polsol}, that for every $v$ such that $|\F_v|>s$, there exists ${\bf y_v'}\in\F_v^s$ such that $\rho_v(l(\tilde{h}_v))(y_v')\neq 0$, equivalently there exists ${\bf y_v}\in \mathcal{T}_0(X_v)(\F_v)$ such that ${\bf y_v}(l(\tilde{h}_v)\neq 0$, i.e., there exists ${\bf x_v^0}\in \F_v^n$ such that $\tilde{h}_v({\bf x_v^0})\neq 0$.

If we now reinterpret $\mathcal{T}_0(X_v)(\F_v)$ via dual numbers (for a ring $R$, $R[\epsilon]\cong R[x]/(x^2)$) and recall that $\F_v[X_v]=\OO_S[X]\otimes \OO_S/p_v$ we have natural bijections
\begin{align*}
\mathcal{T}_0(X_v)(\F_v)&=\Hom(\F_v[\mathcal{T}_0(X_v)],\F_v)\\
& \leftrightarrow \{\varphi\in\Hom_{\OO_S}(\F_v[X_v],\F_v[\epsilon])\; :\; \varphi(\tilde{r}_v(\tilde{I}))\subseteq \epsilon\F_v[\epsilon]\}\\
& \leftrightarrow \Hom_{\OO_S}(\F_v[X_v],\mm_v^k/\mm_v^{k+1})\\ & \leftrightarrow \Hom_{\OO_S}(\OO_S[X],\mm_v^k/\mm_v^{k+1}),
\end{align*}
where the bijections ($\leftrightarrow$) are induced by the evaluation maps at the $x_i$'s.
Recall now (the proof of) Lemma \ref{lm:Nv1_standar} and set  $I_v=\langle p_v,I\rangle$ so that we have
\[
\Hom_{\OO_S}(\OO_S[X],\mm_v^k/\mm_v^{k+1})\leftrightarrow\Hom_{\OO_S}(\widehat{\OO_S[X]}_{I_v},\mm_v^{k}/\mm_v^{k+1}),
\]
where $\widehat{\OO_S[X]}_{I_v}$ denotes the completion of the ring $\OO_S[X]_{I_v}$ with respect to its unique maximal ideal and the latter are assummed to be continous homomorphisms with respect to the $I_v$-adic topology.
It follows from the proof of Lemma \ref{lm:Nv1_standar} that for almost every $v\in V_f^S$ there is an isomorphism
$\widehat{\OO_S[X]}_{I_v}\cong \OO_v[[y_1,\ldots,y_d]]$, which implies that $\Hom_{\OO_S}(\widehat{\OO_S[X]}_{I_v},\mm_v^k)\leftrightarrow(\mm_v^k)^d$ and so for those $v$'s we have a projection
\[
\Hom_{\OO_S}(\widehat{\OO_S[X]}_{I_v},\mm_v^k)\twoheadrightarrow\Hom_{\OO_S}(\widehat{\OO_S[X]}_{I_v},\mm_v^k/\mm_v^{k+1})
\] again induced by the evaluation maps at the $x_i$'s under the corresponding homomorphisms.

Now, following the constructions, this gives for almost every $v\in V_f^S$ a point $x_v^k\in X(\mm_v^k)$ such that $f(x_v^k)\neq 0\mod \mm_v^{k+1}$.
But recall that $\tilde{h}$ has degree $s$, so $\ord_v(f(x_v^k))=ks$ as desired.
\end{proof}

The following lemma is used in the proof of Theorem \ref{th:PRG}.
\begin{lemma}[{\cite[Lemma 4.7]{LuMa}}] \label{lm:ndecom}
Given $n\in\NN$, define $f(n)$ to be the number of tuples $(n_1,\ldots,n_t)$ of integers such that $n_i>0$ and $n=n_1\ldots n_t$.
Then there exists $d\in\NN$ such that $f(n)\leq n^d$ for every $n\in\NN$.
\end{lemma}

\clearpage{\pagestyle{empty}\cleardoublepage}

\chapter{Subgroup Growth of Arithmetic Groups}


\section{Subgroup Growth}
Given a group $\Gamma$ let us write $a_n(\Gamma)$ for the number of subgroups of $\Gamma$ of index $n$ and $s_n(\Gamma)$ for the number of subgroups of $\Gamma$ of index at most $n$, i.e., $s_n(G)=\sum_{i=1}^n a_i(G)$.
If $\Gamma$ is a topological group, we will only consider closed subgroups.
Note that if $\Gamma$ is a profinite group  and $H\leq \Gamma$ is a subgroup of finite index, then $H$ is open if and only if it is closed.
Hence, $s_n(\Gamma)$ equals the number of closed (equivalently open) subgroups of index at most $n$.

We recall some standard properties concerning subgroup growth in the following lemma.

\begin{lemma}\label{lm:SG_evident}
\begin{enumerate}[label=\roman*)]
\item If $N\trianglelefteq \Gamma$ ,then $a_n(\Gamma/N)\leq a_n(\Gamma)$ and $s_n(\Gamma/N)\leq s_n(\Gamma).$\label{e-sg-finite-index}
\item If $H\leq\Gamma$ and  $|\Gamma:H|=m$,  then $a_n(H)\leq a_{mn}(\Gamma)$ and $s_n(H)\leq s_{mn}(\Gamma).$
\item Suppose $a_n(\Gamma)$ is finite, then  $a_n(\Gamma)=a_n(\Gamma/N)$ for some finite index normal subgroup $N\trianglelefteq\Gamma$.
In particular, $a_n(\Gamma)=a_n(\widehat{\Gamma})$, where $\widehat{\Gamma}$ denotes the profinite completion of $\Gamma$.\label{e-SG-equiv-profinite}
\end{enumerate}
\end{lemma}

The following lemma gives an upper bound for the subgroup growth of a finite group.
\begin{lemma}\label{lm:SG_finite_group}
Let $\Gamma$ be a finite group. Then $s_n(\Gamma)\leq|\Gamma|^{\log |\Gamma|}$.
\end{lemma}
\begin{proof}
If $d(H)$ denotes the minimum number of generators of a group $H$, let us define the rank of $G$ as
\[
\rk(\Gamma):=\max\{d(H)\; :\; H\leq \Gamma\}.
\]
Thus, every subgroup of $\Gamma$ is generated by at most $\rk(\Gamma)$ element, so we have
\[
s_n(\Gamma)\leq |\Gamma|^{\rk(\Gamma)}.
\]
Moreover $\rk(\Gamma)\leq\log|\Gamma|$. Indeed, suppose $d(H)=r$ for some $H\leq\Gamma$, say $H=\langle h_1,\ldots,h_r\rangle$ and put $H_i:=\langle h_1,\ldots,h_i\rangle$. Then
\[
|\Gamma|\geq|H|=|H_r:H_{r-1}|\ldots|H_2:H_1||H_1|\geq 2^r.
\]
Therefore $s_n(\Gamma)\leq |\Gamma|^{\rk(\Gamma)}\leq |\Gamma|^{\log|\Gamma|}$ as desired.
\end{proof}

\begin{lemma}\label{lm:SG_super_of_subgroups}
Let $U$ be a finite index subgroup of $\Gamma$ with $|\Gamma:U|=m$.
Then for each $k$ the number of subgroups $H$ of $\Gamma$ such that $U\leq H$ and $|H:U|\leq k$ is  bounded by $m^{\log k}$.
\end{lemma}
\begin{proof}
The maximal length of a chain of subgroups between $U$ and $H$ is
at most $\log k$, so $H$ can be generated by $U$ and at most $\log k$ further elements.
Moreover, note that for $u_1,\ldots,u_{\log k}\in U$ we have
$\langle U,h_1,\ldots,h_{\log k}\rangle=\langle U,u_1h_1,\ldots,u_{\log k}h_{\log k}\rangle$ so the number of distinct subgroups of this form is at
most $|\Gamma:U|^{\log k} = m^{\log k}$.
\end{proof}

\begin{lemma}\label{lm:SG_finite_index_subgroup}

Let $H$ be a finite index subgroup of $\Gamma$ such that $|\Gamma:H|=m$.
Then for each $n\in\NN$ we have
\[
s_n(\Gamma)\leq (mn)^{\log m}s_{n}(H).
\]

\end{lemma}
\begin{proof}
To each subgroup $S\leq\Gamma$ we associate the subgroup $S\cap H$.
If $|\Gamma:S|\leq n$, then $|H:(H\cap S)|\leq n$ as well. 
Suppose that $U\cap H=S\cap H$, then $|U:(H\cap S)|\leq m$.
By the previous lemma the number of such possible $U$ is bounded by $(mn)^{\log m}$, so we have
\[
s_n(\Gamma)\leq (mn)^{\log m}s_{n}(H).
\]
\end{proof}

\begin{lemma}\label{lm:SG_commensurable}
Suppose $\Gamma_1$ and $\Gamma_2$ are commensurable subgroups of a group $\Gamma$.
Then $s_n(\Gamma_1)\leq n^{D_1\log n}$ for some $D_1\in\mathbb{R}_{\geq 0}$ if and only if $s_n(\Gamma_2)\leq n^{D_2\log n}$ for some $D_2\in\mathbb{R}_{\geq 0}$.
\end{lemma}
\begin{proof}
This follows by repeated applications of the inequalities in \ref{lm:SG_finite_group} and \ref{lm:SG_finite_index_subgroup} to the groups $\Gamma_1$,$\Gamma_2$ and $\Gamma_1\cap\Gamma_2$.
\end{proof}

The following lemmata explain how one can obtain information on the subgroup growth of a group from its representation growth.

\begin{lemma}\label{lm:type_c_log_subg_growth}
Let $\Gamma$ be a group.
If $\Gamma$ is of representation type $c$, i.e., $|\Gamma:K_n(\Gamma)|\leq n^c$, then $s_n(\Gamma)\leq n^{c^2\log n}$.

\end{lemma}

\begin{proof}
i) Let $H\leq \Gamma$ be a subgroup of index at most $n$.
Then $\Gamma$ permutes the right cosets of $H$ by right multiplication.
The associated permutation representation $\rho:\Gamma\to \GL_{|\Gamma:H|}(\mathbb{C})$ is of degree $|\Gamma:H|\leq n$.
It follows that $K_n(\Gamma)\subseteq\ker\rho\subseteq H$.
This is true for any such $H$, so we obtain
\[
|\Gamma:\bigcap_{|\Gamma:H|\leq n}H|\leq |\Gamma:K_n(\Gamma)|\leq n^c.
\]

If we set $S_n:=\bigcap_{|\Gamma:H|\leq n}H$, it follows that $S_n\subseteq H$ for any subgroup $H\leq \Gamma$ of index at most $n$ and we have $s_n(\Gamma)=s_n(\Gamma/S_n)$.
Now $\Gamma/S_n$ is a finite group of order at most $n^c$ and we can apply Lemma \ref{lm:SG_finite_group} to obtain

\[
s_n(\Gamma)=s_n(\Gamma/S_n)\leq (n^c)^{\log n^c}=n^{c^2\log n}.
\]
\end{proof}

\begin{lemma}\label{lm:Fratt_p_subgroup_of_PRG_group}
Let $\Gamma$ be a profinite group with Polynomial Representation Growth, say, $r_n(\Gamma)\leq n^c$.
Then for every subgroup $P\leq \Gamma$ of index at most $n$ we have $|P:\overline{[P,P]}|\leq n^{c+1}$. 
\end{lemma}
\begin{proof}
Note that $|P:\overline{[P,P]}|=r_1(P)$.
On the other hand, since $|\Gamma:P|\leq n$,
by Lemma \ref{lm:RG_finite_index_subgroup} we have $r_1(P)\leq n r_n(\Gamma)\leq n^{c+1}$. 
\end{proof}
\section{CSP and Subgroup Growth of Arithmetic Groups}

In this section we show how Theorem \ref{th:PRG} can be used to obtain results on subgroup growth of arithmetic groups.
Let $k$ be a global field of positive characteristic and fix a nonempty finite set of valuations $S\subset V_k$.
We aim at proving the following theorem.

\begin{theorem}\label{th:log_subgroup_growth}
Let $G$ be a simply connected almost simple $k$-group.
Let $\Gamma\leq G$ be an $S$-arithmetic subgroup and suppose that $G$ has the weak Congruence Subgroup Property with respect to $S$.
Then there exists a constant $D>0$ such that $s_n(\Gamma)\leq n^{D\log n}$.
\end{theorem}

If $\Gamma$ is an $S$-arithmetic group of $G$, then it is commensurable with $G(\OO_S)$.
By Lemma \ref{lm:SG_commensurable} we may assume $\Gamma=G(\OO_S)$.
We can apply the same scheme of reductions as in section \ref{sec:CSP_vs_RG} by using Lemma \ref{lm:SG_commensurable} instead of Lemma \ref{lm:RG_finite_index_subgroup} to show that we may assume that $G$ is a connected simply connected absolutely almost simple $k$-group with Strong Approximation.

Recall that since $G$ has wCSP with respect to $S$, the map
\[
\pi:\widehat{G(\OO_S)}\twoheadrightarrow G(\widehat{\OO}_S).
\]
has finite kernel.
Thus, there exists a finite index subgroup $\Gamma'\leq \widehat{G(\OO_S)}$ that embedds in $G(\widehat{\OO}_S)$ as a finite index subgroup.
Therefore repeated applications of Lemma \ref{lm:SG_commensurable} allow us to obtain Theorem \ref{th:log_subgroup_growth} as a consequence of the following theorem.

\begin{theorem}\label{th:PSG}

Let $G$ be a connected simply connected absolutely almost simple $k$-group.
Then  there exists a constant $D$ such that
\[
s_n(G(\widehat{\OO}_S))\leq n^{D\log n}.
\]
\end{theorem}

\noindent Hence, we are reduced to the study of the subgroup growth of the group $G(\widehat{\OO}_S)=\prod_{v\in V_f^S}G(\OO_v)$.

We obtain the following Corollary as direct application of Theorem \ref{th:allvaluations} and Lemma \ref{lm:type_c_log_subg_growth}.
\begin{corollary}\label{cor:all_valuations_subgroup_type_dlog}
There exists $d>0$ such that for every $v\in V_f^S$ we have
\[
s_n(G(\OO_v))\leq n^{d\log n}.
\]
\end{corollary}

To prove Theorem \ref{th:PSG} we basically follow \cite{NiAlSze} and apply Lemma \ref{lm:Fratt_p_subgroup_of_PRG_group} in the final step.
The following proposition is \cite[Proposition 4.3]{Lu}. Let us note that in \cite{Lu} $G$ is assumed to be split, however the same argument works without this assumption.

\begin{theorem}\label{th:Lu_subnormal}

Let $G$ be as in Theorem \ref{th:PSG}. If $H$ is a subgroup of index $n$ in $G(\widehat{\OO}_S)$, then $H$ contains a sub-normal subgroup of $G(\widehat{\OO}_S)$ of index at most $n^{d_1}$ for some constant $d_1$.
\end{theorem}

The following proposition describes the subnormal subgroups of the local factors $G(\OO_v)$ for $v\in\mathcal{P}$ (see Lemma \ref{lm:goodvaluations} for the definition of $\mathcal{P}$).

\begin{proposition} \label{prop:subnormal_subgr_local}
Let $\mathcal{P}$ be as in Lemma \ref{lm:goodvaluations}.
Then for every $v\in\mathcal{P}$ the following hold:
\begin{enumerate}[label=\roman*)] 
\item If $H$ is a subnormal subgroup of $G(\OO_v)$ such that $HN_v^1=G(\OO_v)$ then $H=G(\OO_v)$.\label{prop:subnormal_subgr_local_Nv1_sub_Frattini}
\item Let $Z_v$ denote the inverse image in $G(\OO_v)$ of the center of $G(\OO_v/\mm_v)$.\label{prop:subnormal_subgr_local_proper_sub_central}
If $H$ is a proper subnormal subgroup of $G(\OO_v)$ then $H\leq Z_v$.
\end{enumerate}
\end{proposition}
\begin{proof}
i) A proper sub-normal subgroup is contained in a proper normal subgroup, so we may assume that $H$ is normal in $G(\OO_v)$.
By Lemma \ref{lm:goodvaluations}.\ref{lm:good_valuations_[G(O_v),Nv1]=Nv1}, we know that $[N_v^k,G(\OO_v)]=N_v^k$ for every $k$.
We claim that $N_v^1\subseteq HN_v^k$ for every $k\geq 1$. The claim implies that $N_v^1\subseteq H$ (recall that we assume $H$ to be closed) and hence $H=HN_v^1=G(\OO_v)$ as required.

To prove the claim it suffices to show that $N_v^k\subseteq HN_v^{k+1}$ for every $k\geq 1$.
We know that $[N_v^k,G(\OO_v)]=[N_v^k,HN_v^{1}]=N_v^k$.
But note that $[N_v^{k},HN_v^{1}]\subseteq [N_v^{k},H][N_v^{k},N_v^1]\subseteq HN_v^{k+1}$, which proves the claim.

ii) We may assume as above that $H$ is a normal subgroup.
By Lemma \ref{lm:goodvaluations}.\ref{lm:good_valuations_G(Fv)_simple} the group $G(\F_v)/Z(G(\F_v))\cong G(\OO_v)/Z_v$ is a finite simple group.
Hence, a proper normal subgroup $H$ either satisfies $HZ_v=G(\OO_v)$, whence $H=G(\OO_v)$ by i), or $H\leq Z_v$.
\end{proof}

We can now derive some properties of subnormal subgroups in the global case.

\begin{proposition}\label{prop:subnormal_subgr_global}
Let $H$ be an open subnormal subgroup of $G(\widehat{\OO}_S)$ of index $n$. Let $\mathcal{P}$ as in Lemma \ref{lm:goodvaluations} and define $V(H):=\{v\in\mathcal{P}\ |\ H\cap G(\OO_v)\neq G(\OO_v)\}$.
Then for every $v\in V(H)$ we have
\[
H\cap G(\OO_v)\leq Z_v
\]
and there exists a constant $d_2$ such that
\[
\sum_{v\in V(H)}\log q_v\leq \frac{1}{d_2}\log n.
\]
\end{proposition}
\begin{proof}
Let $v\in V_f^S$.
By definition and Proposition \ref{prop:subnormal_subgr_local}.\ref{prop:subnormal_subgr_local_proper_sub_central}, if $v\in\mathcal{P}\setminus V(H)$, then $G(\OO_v)\subseteq H$.
On the other hand if $v\in V(H)$, then by Proposition \ref{prop:subnormal_subgr_local}.\ref{prop:subnormal_subgr_local_proper_sub_central} we have $H\cap G(\OO_v)\leq Z_v$. Moreover we claim that $\pi_v(H)\leq Z_v$, where $\pi_v$ denotes the projection from $G(\widehat{\OO}_S)$ to $G(\OO_v)$. Note that to prove the claim we may assume that $H$ is a normal subgroup. But for $H$ normal, if $\pi_v(H)=G(\OO_v)$ then $H\cap G(\OO_v)=G(\OO_v)$, since $G(\OO_v)$ is a perfect group (Lemma \ref{lm:goodvaluations}.\ref{lm:good_valuations_G(O_v)_perfect}).
We have shown that 
\[
\prod_{v\in\mathcal{P}\setminus V(H)} G(\OO_v)\subseteq H\subseteq\prod_{v\in V_f^S\setminus V(H)}G(\OO_v)\times \prod_{v\in V(H)}Z_v.\]
Therefore, if we write $Z(H):=\prod_{v\in V(H)}Z_v$ and $E:=\prod_{v\in V_f^S\setminus\mathcal{P}}G(\OO_v)$, we may see $H$ as a subgroup of $Z(H)\times E$.

Let us put $\widetilde{G}:=\prod_{v\in V(H)}G(\OO_v)\times E$.
On the one hand $M:=|E:\prod_{v\in V_f^S\setminus\mathcal{P}}G(\mm_v)|$ is independent of $H$.
On the other hand $|Z(H):\prod_{v\in V(H)}G(\mm_v)|$ is negligible compared to $|\prod_{v\in V(H)}G(\OO_v):\prod_{v\in V(H)}G(\mm_v)|$.
Let us write 
\[
\widetilde{G}(1):=\prod_{v\in V(H)\cup (V_f^S\setminus\mathcal{P})}G(\mm_v).
\]
Since $|\widetilde{G}:Z(H)\times E|\leq |\widetilde{G}:H|$ we can find a constant $d'_1$ (independent of $H$) such that $|\widetilde{G}:\widetilde{G}(1)|\leq n^{d'_1}$.

Now note that $|\widetilde{G}:\widetilde{G}(1)|$ is approximately $M\prod_{v\in V(H)}q_v^{d'_2}$, where $d'_2=\dim G$.
Moreover we also have 
\[
\prod_{v\in V(H)}q_v^{d_2}\leq|\widetilde{G}:H|\leq n^{d_1},
\]
where $d_2$ is a constant slightly less than $d_2'$.
In particular, we have
\begin{equation}
 \sum_{v\in V(H)}\log q_v\leq \frac{1}{d_2}\log n
\end{equation}

\end{proof}
We can finally give a proof of Theorem \ref{th:PSG}.
\begin{proof}[Proof of Theorem \ref{th:PSG}]
Let $H$ be an open subgroup of $G(\widehat{\OO}_S)$ of index at most $n$.
By Theorem \ref{th:Lu_subnormal} and (the proof of) Proposition \ref{prop:subnormal_subgr_global} there exist constants $d_1$, $d_2$ (independent of $H$), a subnormal subgroup $H'$ and a finite set of valuations $V(H')$ such that
\[
|G(\widehat{\OO}_S):H'|\leq n^{d_1},
\]
\[
|\widetilde{G}:\widetilde{G}(1)|\leq n^{d_1},
\]
\[
\prod_{v\in\mathcal{P}\setminus V(H')} G(\OO_v)\subseteq H'\subseteq\prod_{v\in V_f^S\setminus V(H')}G(\OO_v)\times \prod_{v\in V(H')}Z_v,\]
\[
\sum_{v\in V(H')}\log q_v\leq \frac{1}{d_2}\log n^{d_1},
\]
where $\displaystyle{\widetilde{G}=\prod_{v\in V(H')}G(\OO_v)\times E}$ and $\displaystyle{\widetilde{G}(1)=\prod_{v\in V(H')\cup (V_f^S\setminus\mathcal{P})} G(\mm_v)}$ are as in Proposition \ref{prop:subnormal_subgr_global}.
Then, for fixed $V(H')$, by Lemma \ref{lm:SG_super_of_subgroups} we can reduce the problem to counting subgroups $H$ of index at most $n$ in the group $\widetilde{G}$.
By \cite[Proposition 1.3.2]{Subgroup_Growth} we know that 
\begin{equation}\label{eq:subgr_extension}
s_n(\widetilde{G})\leq s_n(\widetilde{G}/\widetilde{G}(1))s_n(\widetilde{G}(1))n^{\rk(\widetilde{G}/\widetilde{G}(1))}
\end{equation}
where $\rk(\widetilde{G}/\widetilde{G}(1))$ is the maximal cardinality of a minimal generating set of a subgroup of $\widetilde{G}/\widetilde{G}(1)$.
Clearly 
\[
\rk(\widetilde{G}/\widetilde{G}(1))\leq \log |\widetilde{G}/\widetilde{G}(1)|\leq \log n^{d_1}=d_1\log n.\]
Applying Lemma \ref{lm:SG_finite_group} to the finite group $\widetilde{G}/\widetilde{G}(1)$ we obtain
\[
s_n(\widetilde{G}/\widetilde{G}(1))\leq n^{d_1 \log n^{d_1}}=n^{d_1^2\log n}.
\]
To prove the appropriate bound for $s_n(\widetilde{G}(1))$ we apply Lemma \ref{lm:Fratt_p_subgroup_of_PRG_group}.
By Theorem \ref{th:PRG} $G(\widehat{\OO}_S)$ has PRG, say $r_n(G(\widehat{\OO}_S)\leq n^{C}$.
By construction, it follows that $r_n(\widetilde{G})\leq n^C$ as well.
Now given a subgroup $K\leq \widetilde{G}(1)$ such that $|\widetilde{G}(1):K|\leq n$, we have $|\widetilde{G}:K|\leq n^{d_1+1}$. By Lemma \ref{lm:Fratt_p_subgroup_of_PRG_group} applied to $K$ and $\widetilde{G}$ we get
\begin{equation}\label{eq:fratt}
|K:\Phi(K)|\leq (n^{d
_1+1})^{C+1}\leq n^E
\end{equation}
for some constant $E$, where $\Phi(K)$ denotes the Frattini subgroup of $K$.
Now if $P$ is a subgroup of the pro-$p$ group $\widetilde{G}(1)$ of index at most $n$, find a sequence
\[
\widetilde{G}(1)=K_0\geq K_1\geq \ldots\geq K_{\lfloor \log_p n\rfloor}=P
\]
such that $|K_s:K_{s+1}|=1$ or $p$.
For every $s$, $K_s/\Phi(K_s)$ is an elementary abelian $p$-group which can be interpreted as an $\F_p$-vector space of dimension $d(K_s)$.
By construction $\Phi(K_s)\leq K_{s+1}$, hence $K_{s+1}$ can be interpreted as a subspace of codimension at most $1$ in $K_s/\Phi(K_s)$.
Hence given $K_s$ the number of choices for $K_{s+1}$ is
\[
1+\frac{p^{d(K_s)}-1}{p-1}\leq p^{d(K_s)}=|K_s:\Phi(K_s)|\leq n^E,\]
where the last inequality follows from \eqref{eq:fratt}.
Hence the number of possibilities for $P$ is at most $n^{E\log_p n}\leq n^{E_2\log n}$ for some constant $E_2$.

If we substitute the above bounds in \eqref{eq:subgr_extension} we obtain
\begin{equation} \label{eq:subgr_standar}
s_n(\widetilde{G})\leq n^{d_1^2\log n} n^{E_2\log n}n^{d_1\log n}\leq n^{d_3\log n}
\end{equation}
for some constant $d_3$.

We have found the appropriate bound for a fixed $V(H')$, so it remains to estimate the number of possible $V(H')$.
Recall that $\sum_{v\in V(H')}\log q_v\leq \frac{1}{d_2}\log n^{d_1}$, in particular, $|V(H')|\leq \frac{d_1}{d_2}\log n$.
On the other hand by Lemma \ref{pre:lm:numberofval} the number of valuations $v$ with $q_v\leq n^{d_1/d_2}$ is bounded by $n^{bd_1/d_2}$. Therefore the number of possible subsets $V(H')$ is bounded by $n^{B\log n}$ for some constant $B$.
This together with \eqref{eq:subgr_standar} gives
\[
s_n(G(\widehat{\OO}_S))\leq n^{d_3\log n}n^B\leq n^{D\log n}
\]
for some constant $D$, which finishes the proof.  

\end{proof}
\clearpage{\pagestyle{empty}\cleardoublepage}

\chapter{Algebra groups and the Fake Degree Conjecture}
\section{Introduction}
The theory and results presented in this chapter are motivated by the so-called Fake Degree Conjecture.
The Fake Degree Conjecture appears in the context of algebra groups, a class of groups introduced by Isaacs which can be seen as a generalization of unitriangular groups.
It is natural to try to adapt Kirillov's Orbit Method  for this class of groups and this was first made, under certain hypothesis, by Kazhdan.
Kirillov's method suggests a correspondence between 
the so-called coadjoint orbits and the irreducible characters of the group under study.
The Fake Degree Conjecture claims that in the class of algebra groups one can use this correspondence to read the irreducible character degrees from the square root of the sizes of the coadjoint orbits (the fake degrees).

To introduce the class of algebra groups let us first consider the group $U_n(\Fq)$, this is the group of unitriangular  matrices of size $n\times n$ with coefficients in the finite field $\Fq$.
The unitriangular groups and their representation theory have been object of study for several reasons.
Let us note that for $p=\ch\Fq$ the group $U_n(\Fq)$ is a Sylow $p$-subgroup of $\GL_n(\Fq)$.
This implies that any $p$-group can be embedded in some $U_n(\Fq)$, for every finite group can be represented as a permutation matrix group in some $\GL_n(\ZZ)$.
The understanding of this class of groups may therefore give information about the family of $p$-groups.
For instance, a lower bound for the number of conjugacy classes of $U_n(\Fp)$ produces an estimate for the number of isomorphism classes of $p$-groups of order $p^n$, see \cite{Hig} and \cite{VaA}.
On the other hand, unitriangular groups have a rich representation theory.
Indeed, classifying their irreducible representations is known to be a wild problem.
Recent work on the representation theory of unitriangular groups can be found in \cite{DiaIsaacs}.

The first trivial observation about the representation theory of $U_n(\Fq)$ is that, being a $p$-group, every irreducible representation has $p$-power degree.
However it was soon observed that all known irreducible representations of $U_n(\Fq)$ have  indeed $q$-power degree.
The corresponding conjecture was publicize, among others, by Thompson and read as follows.
\begin{conjecture}\label{conj:th}
All irreducible representations of $U_n(\Fq)$ have $q$-power degree.
\end{conjecture}

In 1977 D. Kazhdan  gave in \cite{kazhdan} a description of all irreducible characters of the groups $U_n(\Fq)$ in the case where $n\leq p$. Kazhdan adapted Kirillov's Orbit Method thanks to the fact that for $n\leq p$ the logarithm and exponentinal functions are well defined bijections between the unitriangular group $U_n(\Fq)$ and the strictly uppertriangular matrices $\mathfrak{u}_n(\Fq)$.
In this setting we have a natural action of $U_n(\Fq)$ on $\mathfrak{u}_n(\Fq)$ and on its dual $\widehat{\mathfrak{u}_n(\Fq)}:=\Irr((\mathfrak{u}_n(\Fq),+))$ (see section \ref{sec:KOM} for more details).
For $\lambda\in\Irr(\mathfrak{u}_n(\Fq))$ let $\Omega_\lambda$ denote its corresponding $U_n(\Fq)$-orbit.

\begin{theorem}[{\cite[Proposition 1]{kazhdan}}]\label{th:KazhdanSpringer}
Suppose $n\leq p=\ch\Fq$.
For every $\chi\in\Irr(U_n(\Fq))$ there exists an $\Fq$-subalgebra $\mathfrak{h}\subset\mathfrak{u}_n(\Fq)$ and a linear character $\lambda\in\Irr(\exp(\mathfrak{h}))$ such that
\[
\chi=\lambda^{U_n(\Fq)}=\frac{1}{|\Omega_\lambda|^{2}}\sum_{\mu\in\Omega_\lambda}\mu\circ\log.
\]
\end{theorem}
\noindent In particular, it follows that, for $n\leq p$, every irreducible representation of $U_n(\Fq)$ has degree $|U_n(\Fq)|/|\mathfrak{h}|$ (for some $\Fq$-subalgebra $\mathfrak{h}\subseteq\mathfrak{u}_n(\Fq)$), which is a $q$-power.

However, for arbitrary $n$ the problem was open until 1995, when I.M. Isaacs presented a proof of a more general statement. In \cite{Isaacs}
Isaacs generalized the class of groups under study by introducing the concept of "algebra groups", which contains, in particular, the class of unitriangular groups.

Let $R$ be a commutative ring with unit and let us consider a finite associative $R$-algebra $J$ which is assumed to be nilpotent, i.e., $J^n=0$ for some $n\in\NN$.
Let us consider the set of formal objects
\[
1+J:=\{\ 1+x\ :\ x\in J\ \}.
\]
Then $1+J$ is easily seen to be a group with respect to the natural
multiplication $(1 + x)(1 + y) = 1+x + y + xy$.
In fact, $1+J$ is a subgroup of the
group of units of the algebra $A= R\cdot 1+J$, in which $J$ is the Jacobson radical.
The
group $1+J$ constructed in this way is the $R$-algebra group based on $J$.
It is clear that $U_n(\Fq)\cong 1+\mathfrak{u}_n(\Fq)$, so unitriangular groups are algebra groups.
This more general definition allowed Isaacs to apply inductive arguments to a broader class of groups.
This way he obtained an affirmative answer to Conjecture \ref{conj:th}.

\begin{theorem}[{\cite[Theorem A]{Isaacs}}]\label{th:Isaaqdeg}
For every $\Fq$-algebra group $P$ (and in particular for any untriangular group $U_n(\Fq)$) every irreducible representation of $P$ has $q$-power degree.
\end{theorem}

Isaacs's result and his new approach motivated further investigation of algebra groups.
Theorem \ref{th:Isaaqdeg} can be interpreted as a partial generalization of Kazhdan's aforementioned result for certain unitriangular groups to the class of algebra groups.
It was natural to study to what extend were Kazhdan's results in \cite{kazhdan} extendible to the class of $\Fq$-algebra groups.
In the algebra group setting there is a natural bijection given by $1+x\mapsto x$ between the group $1+J$ and the algebra $J$, which may be used to apply Kirillov's Orbit Method (see section \ref{sec:KOM} for more details).
In this sense Halasi showed the following partial generalization of Theorem \ref{th:KazhdanSpringer} to algebra groups.

\begin{theorem}[{\cite[Theorem 1.2]{Halasi}}]\label{th:Halasi}
Let $1+J$ be an $\Fq$-algebra group and $\chi\in\Irr(1+J)$.
Then there exist an $\Fq$-algebra $H\leq J$ and a linear character $\lambda\in\Irr(1+H)$ such that $\chi=\lambda^{1+J}$.

\end{theorem}

Nevertheless, obtaining a explicit description of all irreducible characters of a given algebra group in terms of coadjoint orbits as in Theorem \ref{th:KazhdanSpringer} is not possible (see section \ref{sec:KOM}).
However it was suggested that obtaining the irreducible character degrees might be possible in the following way.
Note that in Theorem \ref{th:KazhdanSpringer}, for every $\chi\in\Irr(U_n(\Fq))$, $\chi(1)=|\Omega_{\lambda}|^{1/2}$ for some $\lambda\in\Irr(\mathfrak{u}_n(\Fq))$ and corresponding $U_n(\Fq)$-orbit $\Omega_\lambda$.
In the case of an algebra group $1+J$ we similarly have a $(1+J)$-action on both $J$ and $\widehat{J}:=\Irr((J,+))$.
Consider the list of integers obtained by taking the square roots of the sizes of
the conjugation orbits of $1+J$ on $\Irr(J)$.
The sum of these numbers is $|J|=|1+J|$ and the length of the list is the number of conjugacy classes of $1+J$.
In other words, this list of square roots of orbit sizes resembles the list of degrees of the irreducible characters of $1+J$.
In an unpublished set of notes circulated
by Isaacs in 1997, these numbers were called the “fake character degrees” of the algebra group $1+J$.
It was suggested that perhaps they are always the actual irreducible character degrees.
This has been later known as the Fake Degree Conjecture.

\begin{conjecture}[Fake Degree Conjecture]\label{con:fd}
For every prime power $q$ and every $\Fq$-algebra group $1+J$ the irreducible character degrees of $1+J$ coincide, counting multiplicities, with the square root of the cardinalities of the $(1+J)$-orbits in $\Irr(J)$.
\end{conjecture}

Note that an immediate corollary of this conjecture (see Lemma
\ref{lm:abelianization}) is that the orders of $[J,J]_L$ and $[1+J,1+J ]$ have
to be equal (we write $[a ,b]_L=ab-ba$ for Lie brackets when necessary to avoid confusion with group commutators).
Thus in order to understand Conjecture \ref{con:fd} one should first answer the following question.

\begin{question}
\label{q:abel}
Is it true that the size of $(1+J)_{\ab}$, the abelianization of $1+J$, coincides with the index of $[J,J]_L$ in $J$?
\end{question}

In \cite{Ja1} Jaikin gave an explicit counterexample to the Fake Degree Conjecture by constructing an $\FF_2$-algebra group that gives a negative answer to the above question.
However, this approach was not sufficient to disprove the conjecture for odd characteristics.
The particular behaviour of the prime $p=2$ regarding character correspondences and computations for odd $p$ suggested that the Fake Degree Conjecture might hold for algebra groups defined over fields of odd characteristic.

This was our motivation for looking at the following family of
examples. Let $\pi$ be a finite $p$-group. Given a ring $R$ we will set $\I_{R}$
to be the augmentation ideal of the group ring $R[\pi]$. If we take $R=\Fq$,
then $\I_{\Fq}$ is a nilpotent algebra and  $1 + \I_{\FF_q}$ is the group of
normalized units of the modular group ring $\Fq[\pi]$. It is not difficult to
see that the index of $[\I_{\FF_q}, \I_{\FF_q}]_L$ in $\I_{\FF_q}$ is equal to
$q^{\kk(\pi)-1}$, where $\kk(\pi)$ is the number of conjugacy classes of $\pi$
(see Lemma \ref{lm:FD_ab_lie_conjugacy_classes}).
In section \ref{sec:disproof} we give a disproof for the Fake Degree Conjecture by showing that the aforementioned question cannot have a positive answer for every $\Fq$-algebra group.
The disproof points out to a possible relation between $\kk(\pi)$ and $(1+\I_{\Fq})_{ab}$.
Our next result describes this relation by determining the size of $(1+\I_{\FF_q})_{\ab}$, the abelianization of $1 + \I_{\FF_q}$.

\begin{theorem}\label{th:sizeequality}
Let $\pi$ be a finite $p$-group. Then
\[
|(1 + \I_{\FF_q})_{\ab}| = q^{\kk(\pi) - 1} |\!\B_0(\pi)|.
\]
\end{theorem}

The group $\B_0(\pi)$ that appears in the theorem is the {\em Bogomolov
multiplier} of $\pi$. It is defined as the subgroup of the Schur multiplier
$\HH^2(\pi,\QQ/\ZZ)$ of $\pi$ consisting of the cohomology classes vanishing
after restriction to all abelian subgroups of $\pi$.  The Bogomolov multiplier
plays an important role in birational geometry of quotient spaces $V/\pi$ as it
was shown by Bogomolov in \cite{Bo}. In a dual manner, one may view the group
$\B_0(\pi)$ as an appropriate quotient of the homological Schur multiplier
$\HH_2(\pi,\ZZ)$, see \cite{Mor12}. We were surprised to discover that, in this
form, the Bogomolov multiplier had appeared in the literature much earlier in a
paper of W. D. Neumann \cite{Ne}, as well as in the paper of R. Oliver
\cite{Oli80} that plays an essential role in our proofs. The latter paper
contains various results about Bogomolov multipliers that were only subsequently
proved in the cohomological framework.

There are plenty of finite $p$-groups with nontrivial Bogomolov multipliers
(see for example \cite{Kan14}). Thus we obtain explicit counterexamples to the Fake
Degree Conjecture for all primes.

Our next result provides a conceptual explanation for the equality in Theorem \ref{th:sizeequality}.
Note that multiplication and inversion in $1+\I_{\F_q}$ extend to $1+\I_{\F_q}\otimes R$ for every $\F_q$-algebra $R$.
Hence we can associate to $(1+\I_{\F_q})$ an algebraic group $G$ defined over $\F_q$.
It is clear that $G$ is a unipotent group.
A direct calculation shows that  the $\F_q$-Lie algebra $L_G$ of $G$ is isomorphic to $\I_{\F_q}$.
The derived subgroup $G'$ of $G$ is also a unipotent algebraic group defined over $\F_q$ (\cite[Corollary I.2.3]{Bor}), and so by
\cite[Remark A.3]{KMT74}, $|G^\prime(\F_q)|=q^{\dim G^\prime}$.
Note that in general we have only an inclusion $$(1+\I_{\FF_q})^\prime=
G(\Fq)^\prime \subseteq G^\prime(\Fq),$$ but not the equality.

\begin{theorem} 
\label{t:algebraicgroup} Let $\pi$ be a finite $p$-group and $G$ the associated algebraic $\F_q$-group as above. 
\begin{enumerate}
\item
We have
\[
\dim G^\prime=\dim_{\F_q}[L_G,L_G]_L=|\pi|-\kk(\pi).
\]
In particular,
\[
|G(\FF_q): G^\prime(\FF_q)|=q^{\kk(\pi)-1}.
\]

\item  For every $q=p^{n}$, we have
\[
G'(\Fq)/G(\Fq)'\cong \B_{0}(\pi).
\]
\end{enumerate}

\end{theorem}

Our hope is that the second statement of the theorem would help better
understand the structure of the Bogomolov multiplier. As an example of this
reasoning, recall that a classical problem about the Schur multiplier asks what
is the relation between the exponent of a finite group and of its Schur
multiplier (\cite{Sch04}). Standard arguments reduce this question to the case
of $p$-groups.  It is known that the exponent of the Schur multiplier is bounded
by some function that depends only on the exponent of the group (\cite{Mor07}),
but this bound  is obtained from  the bounds that appear in the solution of the
Restricted Burnside Problem and so it is probably very far from being optimal.
Applying the homological description of the Bogomolov multiplier,  it is not
difficult to see that the exponent of the Schur multiplier is at most the
product of the exponent of the group by the exponent of the Bogomolov
multiplier. Thus, we hope that the following theorem would help obtain a
better bound on the exponent of the Schur multiplier.

 \begin{theorem}\label{exponent}
Let $\pi$ be a finite $p$-group and $G$ the associated algebraic $\F_p$-group associated to $1+\I_{\F_p}$ as above.
For every $q=p^{n}$, we have
\[
\exp \B_0(\pi) = \min \{ m \mid G^\prime (\FF_q) \subseteq G(\FF_{ q^{m} })^\prime \}.
\]
\end{theorem}

\section{Algebra groups and Kirillov's Orbit Method}\label{sec:KOM}

Throughout this section we mainly follow \cite{Sa04}. Fix a prime $p$ and let $R$ be a commutative ring with unit and $\ch R=p$.
Let us consider a finite associative $R$-algebra $J$ and assume further that $J$ is nilpotent, i.e. $J^n=0$ for some $n\in\NN$.
We define the $R$-algebra group based on $J$ as \[1+J=\{1+j\ :\ j\in J\},\]
where group multiplication and inverse are given by
\begin{align*}
(1+j)(1+h)&=1+j+h+jh, \\ (1+j)^{-1}&=\sum_{k=0}^{n-1} (-1)^kj^k.
\end{align*}
Associativity of the algebra multiplication guarantees that the group operation is associative and thanks to the nilpotency of $J$ the inverse operation is well defined.
Note that $|J|=|1+J|$ and hence $1+J$ is a finite $p$-group.
Groups constructed in this manner are called $R$-algebra groups.

Let us note that given any finite associative $R$-algebra $A$, $1+\rad(A)$ is an $R$-algebra group.
From now on we will take $R=\Fq$ ($q$ being some $p$-power) and in this case $J$ becomes an $\Fq$-vector space.
Nevertheless, some of the results presented below hold for more general $R$. 
Let us present some examples of algebra groups that will be useful later on.

\begin{example}
Take $J=\mathfrak{u}_n(\Fq)$ be the strictly upper triangular matrices of size $n\times n$ with entries in $\Fq$.
Then $1+J=1+U_n(\Fq)$ is the unitriangular group of dimension $n$ over $\Fq$.
\end{example}

\begin{example}
Let $\pi$ be a finite $p$-group and consider the group algebra $\Fq[\pi]$.
Let us write $\I_{\Fq[\pi]}:=\ker \left(\Fq[\pi]\to \Fq\right)$ for the augmention ideal.
Then $\rad \Fq[\pi]=\I_{\Fq[\pi]}$ and $1+\I_{\Fq[\pi]}$ is an algebra group.
Note that we can describe the units of $\Fq[\pi]$ as $\left(\Fq[\pi]\right)^*\cong \Fq^*\times \left(1+\I_{\Fq[\pi]}\right)$.
For this reason $1+\I_{\Fq[\pi]}$ is called the group of normalized units of $\Fq[\pi]$.
\end{example}

As explained in the introduction Isaacs introduced this class of groups motivated by Conjecture \ref{conj:th} on unitriangular groups. Isaac's Theorem \ref{th:Isaaqdeg} can be seen as a generalization of Kazhdan's results on unitriangular groups (Theorem \ref{th:KazhdanSpringer}).
Kazhdan's work followed the spirit of Kirillov's Orbit Method.
Kirillov presented in \cite{Kirillov} a method to describe the irreducible unitary representations of a connected simply connected nilpotent Lie group $N$.
If we denote by $\mathfrak{n}$ the Lie algebra of the Lie group $N$, the adjoint action on $\mathfrak{n}$ induces an $N$-action on $\hat{\mathfrak{n}}$, the unitary dual of $(\mathfrak{n},+)$.
Kirillov established a correspondence between $N$-orbits in $\hat{\mathfrak{n}}$ and irreducible unitary representations of $N$ (up to isomorphism).
Kirillov's Orbit Method has been applied to several class of groups, among others $p$-groups and pro-$p$ groups (\cite{Jon}, \cite{Howe} and \cite{Ja2}), finitely generated nilpotent groups (\cite{Voll}) and arithmetic groups (\cite{AKOV3}), see \cite{MeritsKOM} for a general discussion.
We explain how to adapt Kirillov's Orbit Method to the class of $\F_q$-algebra groups and show that it works best for groups of nilpotency class strictly less than $p$.

Let $J$ be as above and $1+J$ the associated $\F_q$-algebra group.
There is a natural action of $1+J$ on $J$ given by \[a^{(1+b)}=(1+b)^{-1}a(1+b)\quad \mbox{for every}\ a,b\in J.\]
Note that this action is equivalent to the conjugation action of $1+J$ on itself, the equivalence being given by $a\mapsto 1+a$.
The $(1+J)$-action on $J$, induces a $(1+J)$-action on $\widehat{J}:=\Irr(J,+)$, called the coadjoint action, given by
\[\lambda^{1+b}(a)=\lambda(a^{(1+b)^{-1}})\ \mbox{ for every}\ \lambda\in \Irr(J)\quad \mbox{and every}\ a,b\in J.\]

Let us write $\Fun (1+J)^{1+J}$ for the set of complex functions on $1+J$ that are constant on conjugacy classes.
From the above setting one can immediately construct an assignment 
\[
\begin{array}{ccc}
(1+J)\textrm{-orbits in}\ \Irr(J)&\to &\Fun (1+J)^{1+J}\\
\Omega &\mapsto & \chi_\Omega(1+j):= |\Omega|^{-1/2}\sum_{\lambda\in\Omega} \lambda(j), 
\end{array}
\] 
which associates to each coadjoint orbit $\Omega$ a class function of $1+J$.
The following lemma states some remarkable properties of the functions $\chi_\Omega$.
 
\begin{lemma} \label{lm:coadorbits}
Let $J$ be as above and $1+J$ the associated algebra group.
\begin{enumerate}[label=\roman*)]
\item	\label{itm1:lm:coadorbits} The number of $(1+J)$-orbits in $J$ and in $\Irr(J)$ equals $\kk(1+J)$, the number of conjugacy classes of $1+J$.
\item	\label{itm2:lm:coadorbits} Let $\Omega_1$ and $\Omega_2$ be two $(1+J)$-orbits in $\Irr(J)$.
Then 
\[
\langle\chi_{\Omega_1},\chi_{\Omega_2}\rangle=
	\begin{cases}
									1\quad &\mbox{if}\ \Omega_1=\Omega_2,\\
									0\ &\mbox{otherwise}.
	\end{cases}
\]
				\end{enumerate}

\end{lemma}
\begin{proof}
\ref{itm1:lm:coadorbits}
It was pointed out above that the $(1+J)$-action on $J$ is equivalent to the conjugation action of $1+J$ on itself.
Hence the number of $(1+J)$-orbits in $J$ equals $\kk(1+J)$.
Now by the well known orbit-counting formula often
known as Burnside's Lemma, the number of orbits of a $(1+J)$-action equals the average number of fixed points of the elements of $1+J$ in this action.
Hence it suffices to show that an element $g\in 1+J$ has the same number of fixed points in $J$ and in $\Irr(J)$.

Given $g\in 1+J$, note that the action of $g$ on $J$ is $\F_q$-linear, hence $f_g: J\to J$, $j\mapsto j^g-j$, is an $\F_q$-linear homomorphism.
Now $g$ fixes $\lambda\in\hat{J}$ if and only if $\lambda(j^g-j)=0$ for every $j\in J$, i.e., if and only if $\im f_g\subseteq\ker\lambda$.
Since 
$\im f_g$ is clearly a subgroup of $(J,+)$, the number of characters $\lambda\in\Irr(J)$ vanishing in $\im f_g$ equals $|J/\im f_g|$.
On the other hand, note that  $j\in J$ is fixed by $g$ if and only if $j\in\ker f_g$. 
Clearly $|\ker f_g|=|J/\im f_g|$ and so the number of fixed points of $g$ in $J$ and $\Irr(J)$ coincide.

\ref{itm2:lm:coadorbits}
We claim that for every $\lambda,\phi\in\Irr(J)$, $\langle\lambda,\phi\rangle=1$ if $\lambda=\phi$ and $\langle\lambda,\phi\rangle=0$ otherwise.
Indeed, note that every $\chi_{\Omega}$ is the sum of the distinct $\lambda\in\Omega$ and the summands are disjoint for distinct orbits $\Omega$, so orthogonality of irreducible characters proves the claim. The result is now clear.
\end{proof}

It follows from Lemma \ref{lm:coadorbits} that $\{\chi_\Omega:\Omega\ \textrm{is a}\ (1+J)\textrm{-orbit}\}$ is a set of $\kk(1+J)$ orthonormal functions in $\Fun(1+J)^{1+J}$.
It is tempting to think that this could be the set of irreducible characters of $1+J$. 
However, note that $\chi_\Omega(1)=|\Omega|^{1/2}$.
Thus for $\chi_\Omega$ to be a character, the number of elements in every orbit must be a square.
Computing the length of an orbit amounts to knowing the order of the stabilizer of any of its elements.
The interplay between the algebra and group structures will allow us to obtain a nice description of the stabilizer of a given character $\lambda\in\Irr(J)$.
At this point let us remark that together with the associative multiplication, the algebra $J$ inherits a Lie algebra structure given by the bracket $[a,b]=ab-ba$.

Let us now introduce for every $\lambda\in\Irr(J)$ the form $B_{\lambda}:J\times J\to \CC^*$ given by $(a,b)\mapsto\lambda([a,b])$.
Put $\Rad {B_{\lambda}}:=\{j\in J\ :\ B_\lambda(j,u)=1\ \forall u\in J\}$.
We will say that $B_\lambda$ is non-degenerate if $\Rad {B_{\lambda}}=0$.
For every $\F_q$-subspace $H\subseteq J$ we will write $H^\perp=\{j\in J\ : \ B_\lambda(j,h)=1\ \forall h\in H\}$ and we will call an $\F_q$-subspace $H\subseteq J$ isotropic if $H\subseteq H^\perp$.

The following lemma states some important properties of the forms $B_\lambda$ and explains how to use them to obtain the stabilizer of a character $\lambda\in\Irr(J)$.

\begin{lemma}\label{lm:bilinear_form_stabilizer}
Given $\lambda\in \Irr(J)$, let $B_{\lambda}:J\times J\to \CC^*$ be as above.
Then:
\begin{enumerate}[label=\roman*)]
\item $B_\lambda(u+v,w)=B_\lambda(u,w)B_\lambda(v,w)$ , $B_\lambda(u,v+w)=B_\lambda(u,v)B_\lambda(u,w)$, $B_\lambda(ru,u)=1$, $B_\lambda(ru,v)=B_\lambda(u,rv)$ and $B_\lambda(uv,w)=B_\lambda(u,vw)B_\lambda(v,wu)$ for every $u,v,w\in J$, $r\in \F_q$. \label{e-bilinear_form_properties}

\item For every $\F_q$-linear subspace $H\subseteq J$, $H^\perp$ is an $\F_q$-linear subspace.
In particular, $\Rad {B_\lambda}=J^\perp$ is an $\F_q$-linear subspace.\label{e-linearsubspace}

\item $\St_{1+J}(\lambda)=1+\Rad {B_{\lambda}}.$ \label{e-stabilizer}

\item \label{e-radical_index} If H is a maximal isotropic $\F_q$-linear subspace, then $|J/H|=|H/\Rad{B_\lambda}|$.
In particular $|J/\Rad{B_\lambda}|=|\Omega_\lambda|$ is a square.

\end{enumerate}

\end{lemma}

\begin{proof}
\ref{e-bilinear_form_properties} Since both $\lambda$ and the Lie bracket are bilinear with respect to addition the first two equalities hold. 
For the third, note that the Lie bracket is $\F_q$-linear and hence $[ru,u]=r[u,u]=0$ for every $u\in J,\ r\in\F_q$.
From the previous equalities it follows that $B_\lambda(r(u+v),u+v)=B_\lambda(ru,v)B_\lambda(rv,u)=1$ (in particular $B_\lambda(u,v)^{-1}=B_\lambda(v,u)$) for every $u,v\in J$, $r\in\F_q$.
It follows that $B_\lambda(ru,v)=B_\lambda(rv,u)^{-1}=B_\lambda(u,rv)$.
The last equality follows from the Jacobi identity of Lie algebras.

\ref{e-linearsubspace} For every $u,v\in H^\perp$ and $w\in H$, applying \ref{e-bilinear_form_properties} we have
\[
B_\lambda(ru+v,w)=B_\lambda(ru,w)B_\lambda(v,w)=B_\lambda(u,rw)B_\lambda(v,w)=1.
\]
Hence $H^\perp$ is an $\F_q$-linear subspace.

\ref{e-stabilizer} Let $g=1+u$ be an element of $1+J$.
Then $g$ fixes $\lambda$ if and only if for every $v\in J$, $\lambda(gvg^{-1})=\lambda(v)$ or equivalently $\lambda(gvg^{-1}-v)=0$.
Since multiplication by $g$ acts bijectively on $J$ we may change $v$ for $vg$ in the above equality so this amounts to $\lambda(gv-vg)=\lambda(uv-vu)=\lambda([u,v])=0$ for every $v\in J$, i.e., $u\in\Rad_{B_{\lambda}}$.

\ref{e-radical_index} Let us first assume that $\Rad B_\lambda=0$
.
Then $B_\lambda$ induces an isomorphism:
\[
\begin{array}{rcccccc}
\phi:& (J,+)	&	\to 	& (\Irr(J),\cdot)	\\
     & j		&	\mapsto & B_\lambda^j: 			&J	&\to &\mathbb{C}^*\\
& & & & u&\mapsto & B_\lambda(j,u)\\
\end{array}
\]
Note that for an $\F_q$-subspace $H\subseteq J$, we have an isomorphism $\phi(H^{\perp})\cong \Irr(J/H)$, in particular, $|H^{\perp}|=|J/H|$.
If we assume $H$ isotropic, then $H\subseteq H^{\perp}$.
Moreover, equality holds precisely when $H$ is a maximal isotropic subspace. Indeed, if $h\in H^\perp$ then $H+\F_q h$ is isotropic, hence $h\in H$. It follows that for $H$ maximal isotropic we have $|H|=|J/H^\perp|=|J/H|$.

If $\Rad B_\lambda\neq 0$ then $\tilde{\lambda}:J/\Rad B_\lambda\times J/\Rad B_\lambda\to\CC^*$ given by $(u+\Rad B_\lambda,v+\Rad B_\lambda)\mapsto \lambda([u,v])$ is a nondegenerate form and the claim easily follows from the previous claim.
\end{proof}

The previous lemma gives further evidence that the $\chi_\Omega$ could be the irreducible characters of $1+J$.
Now since $1+J$ is a finite $p$-group, by \cite[Corollary 6.14]{Isaacs} every irreducible character of $1+J$ is induced from a linear character of a subgroup $N\leq 1+J$.
Hence it will be reasonable to try to obtain $\chi_\Omega$ as the induced character of a linear character $\psi$ of a subgroup $1+H_\Omega$, where $H_\Omega\subseteq J$ is a subalgebra of index $|\Omega|^{1/2}$.
Note that for $\lambda\in\Omega$, we have 
\[
\chi_\Omega(1+u)=\sum_{g\in (1+J)/St_{1+J}(\lambda)}\lambda^g(u),
\]
thus it will be natural to pick $\psi_\lambda\in \Irr(H_\Omega)$ such that $\psi_\lambda(1+j):=\lambda(j)$.

However this approach has serious obstructions. One can easily observe that this will not work even for an abelian $R$-algebra $J$.
Indeed, in this case, for any coadjoint orbit $\Omega$, $|\Omega|=1$, which would imply $H_\Omega=J$ for every $\Omega$.
But note that for $\psi_\lambda$ to be a character we must have $\lambda(u+j)=\lambda(u)\lambda(j)=\psi_\lambda(1+u)\psi_\lambda(1+j)=\psi_\lambda(1+u+j+uj)=\lambda(u+j+uj)=\lambda(u+j)\lambda(uj)$ for every $u,j\in J$, that is $J^2\subseteq\ker\lambda$.
If this were to hold for every $\lambda\in\hat{J}$, then $J^2=0$ and this is certainly a very strong restriction which does not hold for every commutative $\F_q$-algebra $J$.

Nevertheless another approach is possible.
Let us observe that the abovementioned strategy will make sense equally well if we set a correspondence of the form
\[
\begin{array}{ccc}
(1+J)-\mbox{orbits in}\ \Irr(J)&\to &\Fun (1+J)^{1+J}\\
\Omega &\mapsto & \chi_\Omega(e(j)):= \frac{1}{|\Omega|^{\frac{1}{2}}}\sum_{\lambda\in\Omega} \lambda(j), 
\end{array}
\] 
where $e:\ J\to 1+J$ is any bijection preservig the $(1+J)$-actions.
Observe that the assigment $j\mapsto 1+j$ clearly satisfies this property.
It is not clear how to find an alternative map different from the obvious one, however if we assume that $J^p=0$ (i.e. if $n\leq p$) we can consider the exponential map given by 
\[
\begin{array}{ccc}
\exp:\ &J\to &1+J\\
\ &j\mapsto &\exp(j)=\sum_{k=0}^{n-1} \frac{j^k}{k!} 
\end{array}
\]
Note that $(J,+)$ is a $p$-group and hence the map $j\mapsto \frac{j}{k}$ is to be considered as the inverse of the map $j\mapsto kj$, which is well defined.
As usual, its inverse is given by the logaritm 
\[
\begin{array}{cccc}
\log:\ &1+J&\to &J\\
&1+j&\mapsto &\log(1+j)=\sum_{k=1}^{n-1}(-1)^{k+1} \frac{j^k}{k} 
\end{array}
\]
It is clear that both maps respect the $(1+J)$-action on both $J$ and $1+J$.
For the rest of the section we will assume that $J^p=0$. We will show that, in this case, the outlined strategy works and we have a bijection
\[
\begin{array}{cccl}
(1+J)\textrm{-orbits in} \Irr(J)	&\to 		& &\Irr(1+J)\\
\Omega 									&\mapsto	&\exp(j)&\mapsto |\Omega|^{-1/2}\sum_{\lambda\in\Omega}\lambda(j)\\
& & 1+j & \mapsto |\Omega|^{-1/2} \sum_{\lambda\in\Omega}\lambda(\log (1+j))
\end{array}
\]

First let us recall that thanks to the universal Baker-Campbell-Hausdorff Formula (see \cite[Lemma 9.15]{Khu}), for every $u,j\in J$ we have $\exp(u)\exp(j)=\exp(u+j+\gamma(u,j))$, where $\gamma(u,j)$ is a sum of Lie words on the elements $u$ and $j$.
It follows that given $H_\lambda\subseteq J$ a subalgebra of index $|\Omega|^{1/2}$, $\exp(H_\lambda)=1+H_\lambda$ is a subgroup of $1+J$ with the same index.
Let us now define $\psi_\lambda(\exp(j)):=\lambda(j)$.
Then for $\psi_\lambda$ to be a linear character of $\exp(H_\lambda)$ we must have
\begin{equation}\label{eq:lambdapsi}
\begin{split}
\lambda(u+j)=\lambda(u)\lambda(j)& =
\psi_\lambda(\exp(j))\psi_\lambda(\exp(u))\\
& =\psi_\lambda(\exp(u+j+\gamma(u,j))=\lambda(u+j+\gamma(u,j)).
\end{split}
\end{equation}
Thus, if $[H_\lambda,H_\lambda]\subseteq\ker\lambda$, then $\psi_\lambda$ is a linear character of $\exp(H_\lambda)$. 
Therefore, to construct $\psi_\lambda$ with the desired properties it suffices to find a maximal isotropic subalgebra with respect to $B_\lambda$.
This is the content of the following proposition.

\begin{proposition} \label{prop:maxisosubalgebra}
Let $\lambda\in\Irr(J)$ and consider the form $B_\lambda$ as above. Then there exists an $\F_q$-subalgebra $H_\lambda$ such that $H_{\lambda}$ is a maximal isotropic subspace of $J$.
\end{proposition}
\begin{proof}
We will apply induction on $\dim_{\F_q} J$. The result is clear if $B_\lambda$ is trivial.
By refining the chain of ideals $\{0\}=J^p\subseteq J^{p-1}\subseteq\ldots J^2\subseteq J$ we can obtain a chain of ideals
\[
\{0\}=J_{k+1}\subseteq J_k\subseteq\ldots\subseteq J_2\subseteq J_1=J\]
such that $J_i/J_{i+1}$ has dimension $1$ for every $i\leq k$.
Since $B_\lambda$ is nontrivial, $J$ is not isotropic, so there exists $i\leq k+1$ such that $J_i$ is isotropic but $J_{i-1}$ is not.
Note that $J_i\nsubseteq \Rad B_\lambda$, for $J_i$ has codimension $1$ in $J_{i-1}$ and so the latter would be isotropic as well.
Hence $J_i^{\perp}\subsetneq J$.
Note that $J_i^\perp$ is a subalgebra (recall that $B_\lambda(uv,w)=B_\lambda(u,vw)B_\lambda(v,wu)$ by \ref{lm:bilinear_form_stabilizer}.\ref{e-bilinear_form_properties}).
If we consider the restriction of $B_\lambda$ to $J_i^{\perp}$ as a form on $J_i^{\perp}$, by induction there exists a subalgebra $H_\lambda\subseteq J_i^{\perp}$ that is a maximal isotropic subspace of $J_i^{\perp}$.
We claim that $H_\lambda$ is a maximal isotropic subspace of $J$ as well.
Indeed, suppose $H_\lambda\subseteq U$ for some maximal isotropic subspace $U$.
First note that $J_i\subseteq H_\lambda$ for $H_\lambda$ is maximal within $J_i^{\perp}$ and $J_i+H_\lambda\subseteq J_i^{\perp}$ is isotropic.
This implies $J_i\subseteq U$ and since $U$ is isotropic, we have $U\subseteq J_i^{\perp}$.
Since $H_\lambda$ was maximal, this forces $H_\lambda=U$ and the claim follows.
\end{proof} 

It only remains to show that $\psi_\lambda^{1+J}(1+j)=\chi_\Omega(\exp(j))$, which will show that the latter is a character of $1+J$.
We will first need the following proposition.

\begin{proposition} \label{prop:transitiveaction}
Let $\lambda,\phi\in\Irr(J)$ an consider a maximal isotropic subspace $H_\lambda\subseteq J$ for $B_\lambda$ as in Proposition \ref{prop:maxisosubalgebra}.
Then $\phi_{|H_\lambda}=\lambda_{|H_\lambda}$ if and only if $\lambda$ and $\phi$ are $1+H_\lambda$-conjugate.
\end{proposition}
\begin{proof}
Note first that by construction $(B_\lambda)_{|H_\lambda}$ is trivial and so by Lemma \ref{lm:bilinear_form_stabilizer}.\ref{e-stabilizer}  if $\lambda=\phi^g$ for some $g\in 1+H_\lambda$, then for any $h\in H_\lambda$ we have
\[
\lambda(h)=\phi^{g}(h)=\phi(h^{g^{-1}})=\phi(h).
\]
Now let $\Sigma$ be the set of elements $\phi\in\Irr(J)$ such that $\lambda_{|H_\lambda}=\phi_{|H_\lambda}$, then $|\Sigma|=|J/H_{\lambda}|$.
On the other hand, $1+H_\lambda$ acts on $\Sigma$ and the $(1+H_\lambda)$-orbit of $\lambda$ has size $|H_{\lambda}|/|\St_{1+J}(\lambda)|=|H_\lambda|/|\Rad B_\lambda|=|J/H_\lambda|=|\Sigma|$
by Lemma \ref{lm:bilinear_form_stabilizer}.\ref{e-radical_index}.
Therefore $\Sigma$ is the $(1+H_\lambda)$-orbit of $\lambda$ and we are done.
\end{proof}

\begin{theorem}\label{th:mainKOM}
Let $1+J$ be an algebra group with $J^p=0$. Consider $\lambda\in\Irr(J)$ and let $\Omega$ be its corresponding $(1+J)$-orbit.
Then
\[
\chi_{\Omega,\exp}(j):=\chi_{\Omega}(\exp(j))=|\Omega|^{-1/2}\sum_{\lambda\in\Omega}\lambda(j) 
\]
is an irreducible character of $1+J$. Moreover, all irreducible characters of $1+J$ are of this form.
\end{theorem}
\begin{proof}
For a $(1+J)$-orbit $\Omega$ take $\lambda\in\Omega$ and let $H_\lambda$ be as in Proposition \ref{prop:maxisosubalgebra}.
By \eqref{eq:lambdapsi} the function $\psi_\lambda$ given by $\psi_\lambda(\exp(j))=\lambda(j)$ is an irreducible character of $1+H_\lambda$.
We claim that $\psi_\lambda^{1+J}=\chi_\Omega(\exp(j))$.
This would show that $\chi_{\Omega,\exp}$ is a character of $1+J$.
Let us define $\widetilde{\psi}_\lambda:1+J\to\mathbb{C}$ that maps $1+h\to \psi_\lambda(1+h)$ for $1+h\in 1+H_\lambda$ and vanishes elsewhere.
 By definition we have 
\begin{align*}
\psi_\lambda^{1+J}(\exp(j))&=\frac{1}{|H_\lambda|}\sum_{g\in 1+J}\widetilde{\psi}_\lambda(\exp(j)^{g^{-1}})=\frac{1}{|H_\lambda|}\sum_{g\in 1+J}\widetilde{\psi}_\lambda(\exp(j^{g^{-1}}))\\
&=\frac{1}{|H_\lambda|}\sum_{\substack{g\in 1+J\\ j^{g^{-1}}\in H_\lambda}}\lambda(j^{g^{-1}})
\end{align*}
Let us consider $\Irr(J/H_\lambda)$ and regard its elements as members of $\Irr(J)$. By orthogonality relations of irreducible characters we obtain
$\sum_{\nu\in\Irr(J/H_\lambda)}\nu(j)=|J/H_\lambda|$ if $j\in H_\lambda$ and $0$ otherwise.
Hence, we can rewrite
\begin{align*}
\psi_\lambda^{1+J}(\exp(j))&=\frac{1}{|J|}\sum_{\substack{g\in 1+J\\ \nu\in\Irr(J/H_\lambda)}}\lambda(j^{g^{-1}})\nu(j^{g^{-1}})=\frac{1}{|J|}\sum_{\substack{g\in 1+J\\ \nu\in\Irr(J/H_\lambda)}}(\lambda\nu)^g(j)\\
&=\frac{1}{|J|}\sum_{\substack{g\in 1+J\\ \nu\in\Sigma}}\nu^{g}(j)=\frac{|\Sigma|}{|J|}\sum_{g\in 1+J}\lambda^g(j),
\end{align*}
where $\Sigma$ is the set of $\nu\in\Irr(J)$ such that $\nu_{|H_\lambda}=\lambda_{|H_\lambda}$ and we have $|\Sigma|=|H_\lambda|/|\Rad B_\lambda|$, see (the proof of) Proposition \ref{prop:transitiveaction}.
Now recall that $\St_{1+J}(\lambda)=1+\Rad B_\lambda$ by Lemma \ref{lm:bilinear_form_stabilizer}.\ref{e-stabilizer}, so we obtain
\[
\psi_\lambda^{1+J}(\exp(j))=\frac{|\Sigma|}{|J|}\sum_{g\in 1+J}\lambda^g(j)=\frac{|\St_{1+J}(\lambda)|}{|H_\lambda|}\sum_{\mu\in\Omega_\lambda}\mu(j)=|\Omega_\lambda|^{-1/2}\sum_{\mu\in\Omega_\lambda}\mu(j).
\]
Therefore $\chi_{\Omega,\exp}$ is indeed a character. So we obtain $\kk(1+J)$ characters of the form $\chi_{\Omega,\exp}$ which, applying the same idea as in Lemma \ref{lm:coadorbits}.\ref{itm2:lm:coadorbits}, are irreducible, whence the theorem.
\end{proof}
By using the Lazard Correspondence one can extend Theorem \ref{th:mainKOM} to the class of $p$-groups of nilpotency class less than $p$, see \cite{Sa04}.
Nevertheless the approach presented above is more convenient for our results. In general it is not possible to extend Theorem \ref{th:mainKOM} to algebra groups where $J^p$ is not necessarily equal to $0$.
Indeed, suppose that $J$ is an $\Fp$-algebra, then by construction of $\chi_{\Omega}$ all character values of $\chi_{\Omega}$ lie in $\mathbb{Q}(\xi)$ where $\xi$ is a nontrivial $p^{\textrm{th}}$ root of unity. This would force all character values to be real if $p=2$. A well known result from Character Theory says that if $\chi$ takes real values for every $\chi\in\Irr(1+J)$ then every element $g\in 1+J$ is conjugate to its inverse. However, one can find $\F_2$-algebra groups where the last condition does not hold, see \cite{IsaacsKa}.

Note that in Theorem \ref{th:mainKOM}, the character degrees are given by $|\Omega_{\lambda}|^{1/2}$. So even though Kirillov's Orbit Method cannot provide all of the irreducible characters, it might be possible  that one can obtain the irreducible character degrees by looking at the size of the coadjoint orbits $|\Omega_{\lambda}|$. 
This is the content of the Fake Degree Conjecture, which we present in the following section.

\section{The Fake Degree Conjecture: a disproof} \label{sec:disproof}

Kirillov's corespondence between $(1+J)$-orbits in $\Irr(J)$ and $\Irr(1+J)$ suggests a correspondence
\[
\begin{array}{ccc}
(1+J)-\textrm{orbits in}\, \Irr(J)&\to &\{\chi(1)\;:\chi\in\Irr(1+J)\}\\
\Omega &\mapsto &|\Omega|^{\frac{1}{2}}\\
\end{array}
\]
The Fake Degree Conjecture suggests that this correspondence is well defined and gives a bijection counting multiplicities for every algebra group $1+J$.
To show that the Fake Degree Conjecture is false, we will focus on linear  characters.
The correspondence then restricts to a bijection between linear characters of $1+J$ and fixed points of $\Irr(J)$ under the coadjoint action of $1+J$. The idea of the proof is to show that this leads to a contradiction.

\begin{lemma}\label{lm:abelianization} Let $J$ be a finite dimensional nilpotent algebra over a finite field $\FF$. Then the number of fixed points in $\Irr(J)$ under the coadjoint action of $1+J$ equals the index of $[J,J]_L$ in $J$. In particular if the Fake Degree Conjecture holds,
\begin{equation}\label{eq:ab_equality}
|J /[J,J]_{L}|=|(1+J)_{\ab}|.
\end{equation}
\end{lemma}  

\begin{proof}
By Lemma \ref{lm:bilinear_form_stabilizer}.\ref{e-stabilizer}, $\lambda\in\Irr(J)$ is fixed under the coadjoint action of $1+J$ if and only if $\rad B_{\lambda}=J$, which amounts to $\lambda([J,J]_{L})=0$.
It is now clear that the number of fixed points in $\Irr(J)$ equals the number of linear forms vanishing on $[J,J]_{L}$ and hence if the Fake Degree Conjecture holds
\[|J /[J,J]_{L}|=|\{\text{fixed points of\ } \Irr(J)\}|=|\Lin(1+J)|=|(1+J)_{\ab}|.\]
\end{proof}

The following family of examples will prove to be key for the study of the equality \eqref{eq:ab_equality}.
Let $\pi$ be a finite $p$-group and  $\mathbb{F}_{q}[\pi]$ the group algebra over the finite field  $\mathbb{F}_{q}$.
The augmentation ideal 
\[
\I_{\F_q}:=\{\sum_{h\in \pi}c_{h}h\ : \sum_{h\in \pi}c_{h}=0\}
\]
coincides with the Jacobson radical $\rad \mathbb{F}_{q}[\pi]$ of $\F_q[\pi]$ and hence it is a nilpotent ideal.
We can then consider the algebra group $1+\I_{\F_q}$, these are the normalized units of $\mathbb{F}_{q}[\pi]$.
If the Fake Degree Conjecture holds, we know $$|[1+\I_{\F_q},1+\I_{\F_q}]|=|[\I_{\F_q},\I_{\F_q}]_{L}|.$$
The following lemma gives a relation between the size of $[\I_{\F_q},\I_{\F_q}]_{L}$ and the number of conjugacy classes of the group $\pi$.
\begin{lemma}
\label{lm:FD_ab_lie_conjugacy_classes}
Let $\pi$ be a finite group and $\f$ a field. Then
\[
\dim_{\f} \I_{\f}/[\I_{\f},\I_{\f}]_L= \kk(\pi)-1.
\]
\end{lemma}
\begin{proof}
It is clear that the set $\pi$ is an $\f$-basis for $\f[\pi]$. We first claim
that
\[
\dim_{\f} \f[\pi]/ [\f[\pi],\f[\pi]]_L= \kk(\pi).
\]

Let $x_{1},\ldots,x_{\kk(\pi)}$ be representatives of conjugacy classes of
$\pi$. Observe that for any $x,y,g\in\pi$ with $y=g^{-1}xg$, we have
$x-y=[g,g^{-1}x]_{L}$. The elements
$\bar x_{1},\ldots,\bar x_{\kk(\pi)}$ therefore span $\f[\pi]/
[\f[\pi],\f[\pi]]_L$.

Set $\lambda_{i}$ to be the linear functional on $\f[\pi]$ that takes the value
$1$ on the elements corresponding to the conjugacy class of $x_{i}$ and vanishes
elsewhere. Observe that for any $g,h\in\pi$, we have $[g,h]_{L}=g(hg)g^{-1}-hg$
and hence each $\lambda_{i}$ induces a linear functional on $\f[\pi]/
[\f[\pi],\f[\pi]]_L$. Now if $\sum_j\alpha_{j}\bar x_{j}=0$ for some $\alpha_j
\in \f$, then $\alpha_i = \lambda_{i}(\sum_j \alpha_j \bar x_j)=0$ for each $i$.
It follows that $\bar x_{1},\ldots,\bar x_{\kk(\pi)}$ are also linearly
independent and hence a basis. This proves the claim.


Now, it is clear that $\{g-1:g\in\pi\setminus\{1\}\}$ is an $\f$-basis for
$\I_{\f}$. Since for any $g,h\in\pi$, we have $[g,h]_{L}=[g-1,h-1]_{L}$, it
follows that $[\f[\pi],\f[\pi]]_{L}=[\I_{\f},\I_{\f}]_{L}$, whence the lemma.
\end{proof}

As an immediate consequence of the above results we obtain

\begin{proposition}\label{prop:consequence_FD}
Let $\pi$ be a finite $p$-group and $\I_{\F_q}$ the augmentation ideal of $\F_q[\pi]$.
If the Fake Degree Conjecture holds, then for any finite $p$-group $\pi$, $$|1+\I_{\F_q}/[1+\I_{\F_q},1+\I_{\F_q}]|=|\I_{\F_q}/[\I_{\F_q},\I_{\F_q}]_{L}|=q^{\kk(\pi)-1}.$$
\end{proposition}

\begin{theorem}
The Fake Degree Conjecture is false for any prime p.
\end{theorem}
\begin{proof}
Let us assume that the Fake Degree Conjecture holds for $\Fp$-algebra groups for some prime $p$.
Take the $p$-group
\[
\pi:=\langle x_{1},x_{2},x_{3},x_{4}:x_{i}^p=[x_{j},x_{k},x_{l}]=1\ \text{for all}\ i,j,k,l\in\{1,2,3,4\}\rangle.
\]
Consider the element $z:=[x_{1},x_{2}][x_{3},x_{4}]\in Z(\pi)$ and write $\widetilde{\pi}:=\pi/\langle z\rangle$.
We denote by $\I_{\F_p}$ the augmentation ideal of $\F_p[\pi]$ and by $\widetilde{\I}_{\F_p}$ the augmentation ideal of $\F_p[\widetilde{\pi}]$.
We will try to estimate $|1+\I_{\F_p}/[1+\I_{\F_p},1+\I_{\F_p}]|$. For this consider the series of subgroups
$$A_{i}:=1+(z-1)^{i}\mathbb{F}_{p}[\pi] \unlhd 1+\I_{\F_p}\ , \ i\in\{1,\ldots,p-1\}.$$
We then  have:
\begin{multline*}
|1+\I_{\F_p}/[1+\I_{\F_p},1+\I_{\F_p}]|=|1+\I_{\F_p}/A_{1}[1+\I_{\F_p},1+\I_{\F_p}]|\cdot \\ \prod_{i=1}^{p-1} |A_{i}[1+\I_{\F_p},1+\I_{\F_p}]/A_{i+1}[1+\I_{\F_p},1+\I_{\F_p}]|.
\end{multline*}
The first term is just $|1+\widetilde{\I}_{\F_p}/[1+\widetilde{\I}_{\F_p},1+\widetilde{\I}_{\F_p}]|=p^{\kk(\widetilde{\pi})-1}$, by the preceding proposition.
For the rest of the terms observe that we have isomorphisms:
\[
A_{i}[1+\I_{\F_p},1+\I_{\F_p}]/A_{i+1}[1+\I_{\F_p},1+\I_{\F_p}]\cong A_{i}/(A_{i+1}[1+\I_{\F_p},1+\I_{\F_p}]\cap A_{i}).
\]
Moreover we have an isomorphism of $1+\I_{\F_p}$-modules between $A_{i}/A_{i+1}$ and $\Fp[\widetilde{\pi}]$,
where $1+\I_{\F_p}$ acts naturally by conjugation and the isomorphism is given by $1+(z-1)x\mapsto \bar{x}$ for every $x\in \pi$.
We observe that $1+\I_{\F_p}$ acts trivially on the quotient  $A_{i}/(A_{i+1}[1+\I_{\F_p},1+\I_{\F_p}]\cap A_{i})$
and hence $|A_{i}/(A_{i+1}[1+\I_{\F_p},1+I_{\F_p}]\cap A_{i})|\leq|\Fp[\widetilde{\pi}]/[\Fp[\widetilde{\pi}],1+\I_{\F_p}]|$.
However
\[
[\Fp[\widetilde{\pi}],1+\I_{\F_p}]=[\Fp[\widetilde{\pi}],1+\widetilde{\I}_{\F_p}]=[\Fp[\widetilde{\pi}],\widetilde{\I}_{\F_p}]_{L}=[\widetilde{\I}_{\F_p},\widetilde{\I}_{\F_p}],
\]
since for every $h\in\widetilde{\pi}$ and $g\in 1+\widetilde{\I}_{\F_p}$ we have $h-h^{g}=[g^{-1},gh]_{L}$.
This in turn implies that for every $i\in\{1,\ldots,p-1\}$ we have
\begin{align*}
|A_{i}/(A_{i+1}[1+\I_{\F_p},1+\I_{\F_p}]\cap A_{i})|& \leq|\Fp[\widetilde{\pi}]/[\Fp[\widetilde{\pi}],1+\I_{\F_p}]|\\
& =|\Fp[\widetilde{\pi}]/[\Fp[\widetilde{\pi}],\Fp[\widetilde{\pi}]]_{L}|=p^{\kk(\widetilde{\pi})}.
\end{align*}
We can in fact sharpen this inequality for $i=1$.
Observe that ${z=1+(z-1)1}$ is a commutator and hence belongs to ${A_{2}[1+\I_{\F_p},1+\I_{\F_p}]\cap A_{1}}$,
however $1\notin[\Fp[\widetilde{\pi}],\Fp[\widetilde{\pi}]]_{L}$.
Hence 
\[
|A_{1}/(A_{2}[1+\I_{\F_p},1+\I_{\F_p}]\cap A_{1})|\leq p^{\kk(\widetilde{\pi})-1}.
\]
Therefore we have shown the following inequality
\[
|1+\I_{\F_p}/[1+\I_{\F_p},1+\I_{\F_p}]|\leq p^{p\kk(\widetilde{\pi})-2}.
\]
The following lemma establishes the relation between $\kk(\pi)$ and $\kk(\widetilde{\pi}):$
\begin{lemma}
Let $\pi$ and $\widetilde{\pi}$ be as above, then $\kk(\pi)=p\kk(\widetilde{\pi})$.
\end{lemma}
\begin{proof}
Let $\{\overline{x^{h_{0}=1}},\overline{x^{h_{1}}},\ldots,\overline{x^{h_{n}}}\}$ be a conjugacy class in $\widetilde{\pi}$, where $\overline{x}$ stands for $x\langle z\rangle\in\widetilde{\pi}$.
Let us show that $\{x,x^{h_{1}},\ldots,x^{h_{n}}\}$ is a conjugacy class in $\pi$.
If there exists $h\in \pi$ such that $x^{h}\notin\{x,x^{h_{1}},\ldots,x^{h_{n}}\}$, then $x^{h}=x^{h_{i}}z^{k}$ for some $i\in\{0,\ldots,n\},\ k\in\{1,\ldots,p-1\}$.
Hence $x=x^{h_{i}h^{-1}}z^{k}$, which implies $z=[h_{i}h^{-1},x^{1/k}]$, but $z$ is not a simple commutator. 
Since $x$ was an arbitrary representative of $\overline{x}$, it follows that  $\overline{x}^{\widetilde{\pi}}$ splits into $p$ distinct conjugacy classes of $\pi$.
\end{proof}
Substituting this in the previous inequality we obtain
$$|1+\I_{\F_p}/[1+\I_{F_p},1+\I_{\F_p}]|\leq p^{\kk({\pi})-2},$$
which contradicts the previous proposition.
\end{proof}

\section{Abelianization of groups of units and the Bogomolov multiplier} \label{sec:Main&Bogomolov}
\sectionmark{Groups of units and the Bogomolov multiplier}
Motivated by the results in the previous section, we want to better understand the groups $(1+\I_{\F_q})_{\ab}$.
To this end we will apply Oliver's results in $\cite{Oli80}$.
This will allow us to prove Theorem $\ref{th:exactseq}$ which gives enough information about the group $(1+\I_{\F_q})$ to prove Theorem \ref{th:sizeequality}.

\subsection{The group $(1+\I_{\mathbb{F}_q})_{\ab}$}
Given a ring $\R$, recall that $\I_{\R}$ denotes the augmention ideal of the group ring $\R[\pi]$, that is,
\[
\I_{\R}:=\ker\left(\R[\pi]\to \R\right)=\{\sum_{g\in\pi}r_g g\; :\; \sum_{g\in\pi}r_g=0\}.
\]
We will approach the group $(1+\I_{\F_q})_{\ab}$ via $\K$-theory.
Given a ring $\R$, the first $\K$-theoretical group of $\R$ is defined as 
\[
\K_1(\R):=\GL(\R)_{\ab}.
\]
When $\R$ is a local ring, we have $\K_1(\R)\cong (\R^*)_{\ab}$ (see \cite[Corollary 2.2.6]{Ros94}), the abelianization of the group of units of the ring $\R$.
Hence we have $\K_1(\F_q[\pi])\cong \F_q^*\times(1+\I_{\F_q})_{\ab}$.
We will apply the results in \cite{Oli80}, to obtain information about the group $\K_1(\F_q[\pi])$.

Let us first introduce the appropriate setting.
Let $\OO$ be the ring of integers of a $p$-adic field $k$.
We define
\begin{align*}
\SK_1(\OO[\pi])&:=\ker\left(\K_1(\OO[\pi])\to\K_1(k[\pi])\right)\\
\Wh(\OO[\pi])&:=\K_1(\OO[\pi])/\left(\OO^*\times\pi_{\ab}\right)\\
\Wh'(\OO[\pi])&:=\Wh(\OO[\pi])/\SK_1(\OO[\pi])\\
\end{align*}
The main theorem in \cite{Oli80} is the following.
\begin{theorem}\label{th:oliver}
Let $\pi$ be a $p$-group and $\OO$ the ring of integers of a $p$-adic local field. Then
\[
\SK_1(\OO[\pi])\cong B_0(\pi).
\]
\end{theorem}

Recall that the Bogomolov multiplier of the group $\pi$ is defined as 
\[
B_0(\pi)=H_2(\pi,\mathbb{Z})/H_2^{\ab}(\pi,\mathbb{Z}),
\]
where $H_2^{\ab}(\pi,\mathbb{Z})$ denotes the subgroup of $H_2(\pi,\mathbb{Z})$ generated by $H_2(\tau,\mathbb{Z})$ for every abelian subgroup $\tau\leq\pi$.

Let us introduce the appropriate setting to apply Theorem \ref{th:oliver}.
Let $\zeta_{q-1}$ be a primitive $\small q-1$-root of unity.
If $p$ is a prime and $q $ is a power of $p$, let $\R_q=\ZZ_p[\zeta_{q-1}]$ be a finite extension of the $p$-adic integers $\ZZ_p$.
Note that $\R_q/p\R_q \cong \FF_q$.
Fix a $\ZZ_p$-basis $ B_{q}=\{\lambda_j \mid 1 \leq j \leq n \}$ of $\R_q$ and let $\varphi$ be a generator of $\Aut_{\ZZ_p}(\R_q) \cong \Gal(\FF_q/\FF_p)$ such that $\varphi(\lambda) \equiv\lambda^p \pmod{p}$.
Let us define
\[
\bar\I_{\R_q} = \I_{\R_q}/\langle x - x^g \mid x \in \I_{\R_q}, \, g \in \pi
\rangle.
\] 
Set $\mathcal C$ to be a set of nontrivial conjugacy class representatives of
$\pi$.
Then $\bar\I_{\R_q}$ can be regarded as a free $\ZZ_p$-module with basis $\{ \lambda \overline{(1 - r)} \mid \lambda\in B_{q}, \, r \in \mathcal C \}$.
Finally define the abelian group $\M_q$ to be
\[
\M_q=\bar\I_{\R_q}/\langle p \lambda \overline{(1 - r)} - \varphi(\lambda)\overline{(1 - r^p)} \mid \lambda\in B_{q},\ r \in \mathcal C\rangle.
\]
The proof of Theorem \ref{th:sizeequality} rests on the following structural description of the group $(1 + \I_{\FF_q})_{\ab}$.

\begin{theorem}
\label{th:exactseq}
Let $\pi$ be a finite $p$-group. There is an exact sequence
\begin{equation*}
\label{eq:extheorem}
\xymatrix{ 
1 \ar[r] & \B_0(\pi) \times \pi_{\ab} \ar[r] & (1 + \I_{\FF_q})_{\ab} \ar[r] & \displaystyle \M_q \ar[r] & \pi_{\ab} \ar[r] & 1.
}
\end{equation*}
\end{theorem}
\begin{proof}
Our proof relies on inspecting the connection between $\K_1(\FF_q[\pi])$ and $\K_1(\R_q[\pi])$ by utilizing the results of \cite{Oli80}.
By Theorem \ref{th:oliver}, we know that $\SK_1(\R_q[\pi])\cong B_0(\pi)$.

The crux of understanding the structure of the group $\K_1(\R_q[\pi])$ is in the
short exact sequence (see \cite[Theorem 2]{Oli80})
\begin{equation*}
\xymatrix{ 
1 \ar[r] & \Wh'(\R_q[\pi]) \ar[r]^-{\Gamma} & \bar\I_{\R_q} \ar[r] & \pi_{\ab} \ar[r] & 1,
}
\end{equation*}
where the map $\Gamma$ is defined by composing the $p$-adic logarithm with a
linear automorphism of $\bar\I_{\R_q} \otimes \QQ_p$. More precisely, there is a
map $\Log \colon 1 + \I_{\R_q} \to \I_{\R_q} \otimes \QQ_p$, which induces an
injection $\log \colon \Wh'(\R_q[\pi]) \to \bar\I_{\R_q} \otimes \QQ_p$. Setting
$\Phi \colon \I_{\R_q} \to \I_{\R_q}$ to be the map  $\sum_{g \in \pi} \alpha_g
g \mapsto \sum_{g \in \pi} \varphi(\alpha_g) g^p$, we define $\Gamma \colon
\Wh'(\R_q[\pi]) \to \bar\I_{\R_q} \otimes \QQ_p$ as the  composite of $\log$
followed by the linear map $1 - \frac{1}{p} \Phi$. It is shown in
\cite[Proposition 10]{Oli80} that  $\im \Gamma \subseteq \bar\I_{\R_q}$, i.e.,
$\Gamma$ is integer-valued. We thus have a diagram
\begin{equation}
\label{eq:diagWh}
\xymatrix{ 
 & 1 + \I_{R_q} \ar[d] \ar@/^/[dr]^-{(1 - \frac{1}{p}\Phi)\circ \log} \\
1 \ar[r] & \Wh'(\R_q[\pi]) \ar[r]^-{\Gamma} & \bar\I_{R_q}.
}
\end{equation}
The group $\Wh'(\R_q[\pi])$ is torsion-free (cf. \cite {Wa}), so we have an
explicit description
\begin{equation}
\label{eq:k1explicit}
\K_1(\R_q[\pi]) \cong  \R_q^*\times\SK_1(\R_q[\pi]) \times \pi_{\ab} \times \Wh'(\R_q[\pi]).
\end{equation}

To relate the results above to $\K_1(\FF_q[\pi])$, we use the functoriallity of $\K_1$.
Given a ring $\R$ and and ideal $I\subset\R$, the projection map $\R\to\R/I$ induces a map $\mu:\K_1(\R)\to \K_1(\R/I)$ whose kernel is denoted by $\K_1(\R,I)$.
Let $\mathfrak{p}$ be the ideal in $\R_q[\pi]$ generated by $p$.
Since $\F_q[\pi]=\R_q[\pi]/\mathfrak{p}$ we obtain an exact sequence
\begin{equation}
\label{eq:exseqK}
\xymatrix{ 
\K_1(\R_q[\pi], \mathfrak{p}) \ar[r]^-{\partial} & \K_1(\R_q[\pi]) \ar[r]^-{\mu} & \K_1(\FF_q[\pi]) \ar[r] & 1.
}
\end{equation}

Since $\R_q[\pi]$ is a local ring (\cite[Corollary XI.1.4]{Bas68}), we have $\K_1(\R_q[\pi])\cong \R_q^*\times(1+\I_{\R_q})$ and so the map $\mu$ indeed surjects onto $\K_1(\F_q[\pi])=\F_q^*\times(1+\I_{\F_q})$.
Moreover $\K_1(\R_q[\pi], \mathfrak{p}) = (1 + p\R_q) \times \K_1(\R_q[\pi],p\I_{\R_q})$
and $\R_q^*/(1 + p\R_q) \cong \FF_q^* $. 
Hence \eqref{eq:k1explicit} and \eqref{eq:exseqK} give a reduced exact sequence
\begin{equation}
\label{eq:exseqKreduced}
\xymatrix{ 
\K_1(\R_q[\pi], p\I_{\R_q}) \ar[r]^-{\partial} & \Wh'(\R_q[\pi]) \ar[r]^-{\mu} & \displaystyle \frac{(1 + \I_{\FF_q})_{\ab}}{\mu(\SK_1(\R_q[\pi]) \times \pi_{\ab})} \ar[r] & 1.
}
\end{equation}
To determine the structure of the relative group $\K_1(\R_q[\pi], p\I_{\R_q})$
and its connection to the map $\partial$, we make use of \cite[Proposition
2]{Oli80}. The restriction of the logarithm map $\Log$ to $1 + p\I_{\R_q}$
induces an isomorphism $\log \colon \K_1(\R_q[\pi], p\I_{\R_q}) \to
p\bar\I_{\R_q}$ such that the following diagram commutes:
\begin{equation}
\label{eq:diagK1rel}
\xymatrix{ 
 & 1 + p\I_{\R_q} \ar[d] \ar@/^/[dr]^-{\log} \\
1 \ar[r] & \K_1(\R_q[\pi],p\I_{\R_q}) \ar[r]^-{\log} & p\bar\I_{\R_q}.
}
\end{equation}
In particular, the group $\K_1(\R_q[\pi],p\I_{\R_q})$ is torsion-free, and so
$\mu(\SK_1(\R_q[\pi]) \times \pi_{\ab}) \cong \SK_1(\R_q[\pi]) \times
\pi_{\ab}$. Note that by \cite[Theorem V.9.1]{Bas68}, the vertical map $1 +
p\I_{\R_q} \to \K_1(\R_q[\pi],p\I_{\R_q})$ of the above diagram is surjective.

We now collect the stated results to prove the theorem. First combine the
diagrams \eqref{eq:diagWh} and \eqref{eq:diagK1rel} into the following diagram:
\begin{equation}
\label{eq:diagCompare}
\xymatrix{ 
1 + p\I_{\R_q} \ar[d] \ar[rr] \ar@/^/[ddr]^{\log} && 1 + \I_{\R_q} \ar[d] \ar@/^/[ddr]^{(1 - \frac{1}{p}\Phi)\circ \log} \\
\K_1(\R_q[\pi],p\I_{\R_q}) \ar[rr]^-{\partial} \ar[dr]^-{\log} && \Wh'(\R_q[\pi]) \ar[dr]^-{\Gamma} \\
& p\bar\I_{\R_q} \ar[rr]^-{1 - \frac{1}{p}\Phi} && \bar\I_{\R_q}.
}
\end{equation}
Since the back and top rectangles commute and the left-most vertical map is
surjective, it follows that the bottom rectangle also commutes.  Whence $\coker
\Gamma\circ\partial \cong \coker (1 - \frac{1}{p}\Phi)$. Observing that the latter group is
isomorphic to $\M_q$, the exact sequence \eqref{eq:exseqKreduced} gives an exact
sequence
\begin{equation}
\label{eq:exfinal}
\xymatrix{ 
1 \ar[r] & \B_0(\pi) \times \pi_{\ab} \ar[r] & (1 + \I_{\FF_q})_{\ab} \ar[r] & \M_q \ar[r] & \pi_{\ab} \ar[r] & 1.
}
\end{equation}
The proof is complete.
\end{proof}

We now derive Theorem \ref{th:sizeequality} from Theorem \ref{th:exactseq}.

\begin{proof}[Proof of Theorem \ref{th:sizeequality}]
The exact sequence of Theorem \ref{th:exactseq} implies that $|(1 +
\I_{\FF_q})_{\ab}| = |\!\B_0(\pi)| \cdot |\!\M_q\!|$. Hence it suffices to
compute $|\!\M_q\!|$. To this end, we filter $\M_q$ by the series of its
subgroups
\[
\M_q\supseteq p\M_q\supseteq p^2\M_q\supseteq\cdots.
\]
Note that the relations $p \lambda \overline{(1 - r)} -
\varphi(\lambda)\overline{(1 - r^p)}=0$ imply $(\exp\pi) \M_q=0$, so the above
series reaches $0$.

For each $i \geq 0$, put
\[
\pi_i=\{x^{p^{i}}\mid x\in\pi\} \textrm{\ and \ }  \mathcal{C}_{i}=\mathcal{C}\cap(\pi_i\setminus \pi_{i+1}).
\]
Note that $\mathcal{C}_i$ is empty for $p^i \geq \exp \pi$.
Now,
\[
 p^{i}\M_q/p^{i+1}\M_q=\langle\lambda \overline{(1-r)}\mid\lambda\in B_{q},\ r\in\mathcal{C}_i\rangle \cong \bigoplus_{\mathcal{C}_i} C_p^{n}. 
\]
It follows that $|\!\M_q\!|=q^{|\mathcal{C}|}=q^{\kk(\pi)-1}$ and the proof is complete.
\end{proof}

\qnote{
\begin{example}
Let $\pi$ be the group given by the polycyclic generators $\{ g_i \mid 1 \leq i
\leq 7 \}$ subject to the power-commutator relations
\[
\begin{aligned}
g_{1}^{2} = g_{4}, \,
 g_{2}^{2} = g_{5}, \,
 g_{3}^{2} = g_{4}^{2} = g_{5}^{2} = g_{6}^{2} = g_{7}^{2} = 1, \\
[g_{2}, g_{1}]  = g_{3}, \,
 [g_{3}, g_{1}]  = g_{6} ,\,
 [g_{3}, g_{2}]  = g_{7} , \,
 [g_{4}, g_{2}]  = g_{6}, \,
 [g_{5}, g_{1}]  = g_{7},
\end{aligned}
\]
where the trivial commutator relations have been omitted. The group $\pi$ is of
order $128$ with $\pi_{\ab} \cong C_4 \times C_4$. Its Bogomolov multiplier is
generated by the commutator relation $[g_{3}, g_{2}] = [g_{5}, g_{1}]$ of order
$2$, see \cite[Family 39]{JM14}. We have $\kk(\pi) = 26$ and by inspecting the
power structure of conjugacy classes, we see that $\M_q \cong C_2^{13} \times
C_4^{6}$. On the other hand, using the available computational tool
\cite{LAGUNA}, it is readily verified that we have $(1 + \I_{\FF_q})_{\ab} \cong
C_2^{13} \times C_4^{5} \times C_8$. Following the proof of Theorem
\ref{th:exactseq}, the embedding of $\B_0(\pi) \times \pi_{\ab}$ into $(1 +
\I_{\FF_q})_{\ab}$ maps the generating relation $[g_{3}, g_{2}] = [g_{5},
g_{1}]$ of $\B_0(\pi)$ into the element $\exp((1 - g_7)(g_3 - g_5))$, which
belongs to $(1 + \I_{\FF_q})_{\ab}^4$. In particular, the embedding of
$\B_0(\pi) \times \pi_{\ab}$ into $(1 + \I_{\FF_q})_{\ab}$ may not be split.
\end{example}
}

\subsection{$B_0$ and rationality questions in unipotent groups}
As explained in the introduction, for a linear algebraic $\F_q$-group we have $G(\F_q)'=[G(\F_q),G(\F_q)]\subseteq G'(\F_q)$, where $G'$ is the derived group of $G$. 
Suppose now $G$ is the group defined by the algebra group $1+\I_{\F_p}$ where $\I_{\F_p}$ is the agumentation ideal of $\F_p[\pi]$ for some $p$-group $\pi$. Theorem \ref{t:algebraicgroup} and \ref{exponent} give a precise description of the behaviour of the difference between $G'(\Fq)$ and $G(\Fq)'$ for this class of groups.
\begin{proof}[Proofs of Theorem \ref{t:algebraicgroup} and \ref{exponent}]
We will consider an extension $\FF_{l}$ of $\FF_q$ of degree $m$. The inclusion
$G(\Fq)\subseteq G(\Fl)$ induces a map $f \colon G(\Fq)_{\ab}\to G(\Fl)_{\ab}$
with
\[
\ker f= (G(\Fq)\cap G(\Fl)')/ G(\Fq)'.
\]
Note that there exists a large enough $m$  such that $G'(\Fq)= G(\Fq)\cap G(\Fl)'$, and hence
$\ker f = G'(\FF_q)/ G(\FF_q)'$. For this reason we want to understand $\ker f$
for a given $m$. 


The inclusion $\Fq\subseteq \Fl$ induces a map $$\incl\colon\K_{1}(\Fq[\pi])\to
\K_{1}(\Fl[\pi]).$$ Note that $f$ is just the restriction of $\incl$ to
$(1+\I_{\Fq})_{\ab}$. Recalling sequence \eqref{eq:exseqKreduced} from the proof
of Theorem \ref{th:exactseq}, we set
\[
\SK_{1}(\Fl[\pi])=\mu(\SK_{1}(\R_{l}[\pi])\subseteq(1 + \I_{\Fl})_{\ab}=G(\Fl)/G(\Fl)'.
\]
Commutativity of the diagram
\begin{equation}
\label{eq:diagcomrstr}
\xymatrix{ 
 \K_{1}(\R_{q}[\pi])\ar[r]^-{\mu} \ar[d]^-{\incl} & \K_{1}(\Fq[\pi])\ar[d]^-{\incl}\\
 \K_{1}(\R_{l}[\pi])\ar[r]^-{\mu} & \K_{1}(\Fl[\pi])
 }
\end{equation}
shows that $\incl$ restricts to a map
$\incl\colon\SK_{1}(\Fq[\pi])\to\SK_{1}(\Fl[\pi])$. Recall that
$\SK_{1}(\R_l[\pi])\cong\B_{0}(\pi)$,so we obtain from sequence \eqref{eq:exfinal}
the commutative diagram
\begin{equation*} 
\xymatrix{
1 \ar[r] & \SK_{1}(\Fq[\pi])\times \pi_{\ab} \ar[r]\ar[d]^-{\incl\times \id} & G(\Fq)_{\ab} \ar[r]\ar[d]^-{f} & \M_q \ar[r]\ar[d]^-{\iota} & \pi_{\ab} \ar[r] & 1\\
1 \ar[r] & \SK_{1}(\Fl[\pi]) \times \pi_{\ab} \ar[r] &  G(\Fl)_{\ab} \ar[r] & \M_l \ar[r] & \pi_{\ab} \ar[r] & 1,
}
\end{equation*}
where $\iota$ is the map induced by the inclusion $\I_{\R_{q}}\subseteq \I_{\R_{l}}$.

We will now show that $\ker\iota=0$. This will imply $\ker
f\subseteq\SK_{1}(\Fq[\pi])$. Without loss of generality, we may assume that
there is an inclusion of bases $B_{q}\subseteq B_{l}$. As in the proof of Theorem \ref{th:sizeequality}, let us consider the series
\[
\M_l\supseteq p\M_l\supseteq p^2\M_l\supseteq\cdots.
\]
Observe again that for each $i \geq 0$ we have
\[
\begin{aligned}
p^{i}\M_q/p^{i+1}\M_q&=\langle\lambda\overline{(1-r)}\mid\lambda\in B_{q},\ r\in\mathcal{C}_i\rangle,\\
p^{i}\M_l/p^{i+1}\M_l&=\langle\lambda \overline{(1-r)}\mid\lambda\in B_{l},\ r\in\mathcal{C}_i\rangle.
\end{aligned}
\]
If we consider the graded groups associated to the series above, we get an induced map 
\[
\gr(\iota)\colon \bigoplus_{i \geq 0} p^{i}\M_q/p^{i+1}\M_q\to \bigoplus_{i \geq 0} p^{i}\M_l/p^{i+1}\M_l.
\]
By construction $\iota$ is induced by the assignments
$\lambda \overline{(1-r)}\mapsto\lambda\overline{(1-r)}$, for every $\lambda\in B_{q},\
r\in\mathcal{C}$. Hence $\gr(\iota)$ is injective in every component and
therefore injective. This implies $\ker\iota=0$, as desired. In particular, we obtain that 
\begin{equation}\label{manyelements}
|G(\Fq)/ G'(\Fq)|\ge|\M_q|= q^{\kk(\pi)-1}.
\end{equation}

We are now ready to show the first statement of Theorem \ref{t:algebraicgroup}.
Observe that $G$ is a unipotent connected algebraic group defined over $\Fp$ and so is $G'$ (\cite[Corollary I.2.3]{Bor}).
 Hence $ G'\cong_{\Fp}\mathbb{A}^{\dim G'}$ (cf. \cite[Remark A.3]{KMT74}) and so
$|G'(\Fp)|=p^{\dim G'}$. By (\ref{manyelements}), 
we have $|G(\Fp)/ G'(\Fp)|\ge p^{\kk(\pi)-1}$, whence $\dim G'\le  |\pi|-\kk(\pi)$. 
On the other
hand we have  $ [L_G,L_G]_{L}= [\I_{\F},\I_{\F}]_{L},$  which, by Lemma \ref{lm:FD_ab_lie_conjugacy_classes}, has dimension $|\pi|-\kk(\pi) $.
It is known that for an algebraic group, $\dim G'\geq\dim [L_G,L_G]_{L}$ (see \cite[Corollary 10.5]{Hum75}).
Thus $ \dim G'=|\pi|-\kk(\pi)$.

Let us set $e=\exp\B_{0}(\pi)$, the exponent of the group $\B_0(\pi)$.
We now claim that $\ker f=\SK_{1}(\Fq[\pi])$ if and only if $e$ divides $m=|\Fl:\Fq| $.
This will imply the second statement of Theorem \ref{t:algebraicgroup} and also  Theorem \ref{exponent}.

Let us consider
$\Fl[\pi]\cong\bigoplus_{i=1}^{m}\Fq[\pi]$ as a free $\Fq[\pi]$-module.  This
gives a natural inclusion $\GL_{1}(\Fl[\pi])\to \GL_{m}(\Fq[\pi])$, which
induces the transfer map
\[
\trf \colon \K_{1}(\Fl[\pi])\to \K_{1}(\Fq[\pi]).
\]
Note that if $x\in\K_{1}(\Fq[\pi])$, then $(\trf\circ\incl)(x)=x^{m}$.
By commutativity of \eqref{eq:diagcomrstr} the transfer map restricts to a map
\begin{equation*}
\trf \colon \SK_{1}(\Fl[\pi])\to \SK_{1}(\Fq[\pi]).
\end{equation*}
Moreover, by \cite[Proposition 21]{Oli80} the transfer map is an isomorphism.
It thus follows that $\incl(\SK_{1}(\Fq[\pi]))=1$ if and only if $e$ divides $m$.
Hence $\ker f=\SK_{1}(\Fq[\pi])$ if and only if $e$ divides $m $ and we are done.
 \end{proof}
\qnote{
\begin{example}
Let $\pi$ be the group given by the polycyclic generators $\{ g_i \mid 1 \leq i
\leq 7 \}$ subject to the power-commutator relations
\[
\begin{aligned}
g_{1}^{2} = g_{4}, \,
 g_{2}^{2} = g_{5}, \,
 g_{3}^{2} = g_{4}^{2} = g_{5}^{2} = g_{6}^{2} = g_{7}^{2} = 1, \\
[g_{2}, g_{1}]  = g_{3}, \,
 [g_{3}, g_{1}]  = g_{6} ,\,
 [g_{3}, g_{2}]  = g_{7} , \,
 [g_{4}, g_{2}]  = g_{6}, \,
 [g_{5}, g_{1}]  = g_{7},
\end{aligned}
\]
where the trivial commutator relations have been omitted. The group $\pi$ is of
order $128$ with $\pi_{\ab} \cong C_4 \times C_4$. Its Bogomolov multiplier is
generated by the commutator relation $[g_{3}, g_{2}] = [g_{5}, g_{1}]$ of order
$2$, see \cite[Family 39]{JM14}. We have $\kk(\pi) = 26$ and by inspecting the
power structure of conjugacy classes, we see that $\M_q \cong C_2^{13} \times
C_4^{6}$. On the other hand, using the available computational tool
\cite{LAGUNA}, it is readily verified that we have $(1 + \I_{\FF_q})_{\ab} \cong
C_2^{13} \times C_4^{5} \times C_8$. Following the proof of Theorem
\ref{th:exactseq}, the embedding of $\B_0(\pi) \times \pi_{\ab}$ into $(1 +
\I_{\FF_q})_{\ab}$ maps the generating relation $[g_{3}, g_{2}] = [g_{5},
g_{1}]$ of $\B_0(\pi)$ into the element $\exp((1 - g_7)(g_3 - g_5))$, which
belongs to $(1 + \I_{\FF_q})_{\ab}^4$. In particular, the embedding of
$\B_0(\pi) \times \pi_{\ab}$ into $(1 + \I_{\FF_q})_{\ab}$ may not be split.
\end{example}
}
\clearpage{\pagestyle{empty}\cleardoublepage}

\chapter{Finite groups of Lie type}
\section{Introduction}

In this chapter we investigate the representation growth of a certain class of groups arising from finite simple groups. The results presented in this chapter were obtained in collaboration with B. Klopsch.
We refer to Section \ref{sec:RG} for notation and concepts relating to representation growth.
As above, we only consider continuous representations.
Let $H$ be a group and suppose that $r_n(H)<\infty$ for every $n$.
Then one can formally define the representation zeta function of $H$ as the Dirichlet series
\[\zeta_H(s)=\sum_{n\in\NN}r_n(G)n^{-s}.\]
We write $\alpha(H)$ for the abscissa of convergence of $\zeta_H$.
Recall that given a group $\Gamma$ we define $R_n(\Gamma)=\sum_{i=1}^n r_n(\Gamma)$. The invariant $\alpha(H)$ is related to the asymptotic behaviour of the sequence $R_n(H)$ by the equation
\[\alpha(H)=\limsup \frac{\log R_n(H)}{\log n}.\]

Let us consider the class of groups 
\[
\mathcal{C}:=\{\mathbf{H}=\prod_{i\in I}S_i\ :\ S_i \ \mbox{is a nonabelian finite simple group}\}
.\]
We want to study the representation growth of groups in this class.
In particular, we want to understand under which conditions a group of this class has polynomial representation growth. 
Given a cartesian product of finite groups $\mathbf{H}=\prod_{i\in I}S_i$, let us define $L_{\mathbf{H}}(n):=\{i\in I: R_n(S_i)>1\}$ and put $l_{\mathbf{H}}(n):=|L_{\mathbf{H}}(n)|$.
Hence, $l_{\mathbf{H}}(n)$ counts the number of factors in $\mathbf{H}$ with a nontrivial irreducible representation of dimension at most $n$.
We will usually write $l(n)=l_{\mathbf{H}}(n)$ when there is no danger of confusion.
We immediately obtain the following necessary condition.
\begin{lemma}\label{prop:nec}
Let $\mathbf{H}=\prod_{i\in I}S_i$.
If $l_{\mathbf{H}}(n)$ is not polynomially bounded then $\mathbf{H}$ does not have PRG.
\end{lemma}
We aim at proving that for $\mathbf{H}\in\mathcal{C}$ the converse is also true.

\begin{theorem}\label{th:CharPRG}
Let $\mathbf{H}=\prod_{i\in I} S_i$ be a cartesian product of nonabelian finite simple groups and let $l_{\mathbf{H}}(n):=|\{i\in I : R_n(S_i)>1\}|$.
Then $\mathbf{H}$ has PRG if and only if $l_{\mathbf{H}}(n)$ is polynomially bounded.
\end{theorem}

Theorem \ref{th:CharPRG} provides us with a way to construct numerous examples of groups within $\mathcal{C}$ having PRG.
We will denote this  subclass of groups by $\mathcal{C}_{_{PRG}}$.
It is natural to ask what types of representation growth can occur among groups in $\mathcal{C}_{_{PRG}}$. 
In \cite{KaNi}, Kassabov and Nikolov studied questions related to the class $\mathcal{C}$, and more in particular about the subclass $\mathcal{A}\subset\mathcal{C}$ where
\[
\mathcal{A}:=\{\mathbf{H}=\prod_{i\in I}S_i\ :\ S_i=\Alt(m) \ \mbox{for some} \  m\geq 5\}
,\]
where $\Alt(m)$ denotes the alternating group on $m$ letters.
In particular, they showed the following theorem.

\begin{theorem}[{cf. \cite[Theorem 1.8]{KaNi}}]\label{th:alt_abscissa}
For any $b>0$, there exists a group $G$ such that $\alpha(G)=b$.
\end{theorem}

The groups appearing in the proof of Theorem \ref{th:alt_abscissa} are frames (see \cite{KaNi})for groups of the form
\[
\mathbf{H}=\prod_{i\geq 5}\Alt(i)^{f(i)},
\]
for some $f:\NN\to \NN$.
In particular the simple factors of $\mathbf{H}$ have unbounded rank \footnote{For $S$ a nonabelian finite simple groups we define $\rk S=m$ if $S=\Alt(m)$ and $\rk S=\rk L$ if $S$ is a simple group of Lie type $L$.} as alternating groups.
We want to obtain a similar result for the subclass $\mathcal{L}\subset\mathcal{C}$, where
\[
\mathcal{L}:=\{\mathbf{H}=\prod_{i\in I}S_i\ :\ S_i \ \mbox{is a finite simple group of Lie type}\}
.\]
We do this in the following theorem, where it is shown that an analogous result holds for this class. In our case, all of the simple factors may have the same Lie rank.

\begin{theorem}\label{th:Lie_abscissa}

For any $c>0$, there exists a group $\mathbf{H}=\prod_{i\in\NN}S_i\in\mathcal{L}_{_{PRG}}$ such that $\alpha(\mathbf{H})=c$. Moreover $\mathbf{H}$ can be chosen so that only one Lie type occur among its factors.

\end{theorem}

It is natural to ask whether the cumbersome groups appearing in the proof of Theorem \ref{th:Lie_abscissa} are finitely generated as profinite groups. We give a positive answer in the following theorem.

\begin{theorem}\label{th:PRG_iff_fg}
Let $\mathbf{H}\in\mathcal{C}$.
If $\mathbf{H}$ has PRG then $\mathbf{H}$ is finitely generated as a profinite group.
\end{theorem}

The next natural question is whether such a group is the profinite completion of finitely generated residually finite group. If this is true, we say that the group in question is a profinite completion.
This question was answered by Kassabov and Nikolov for cartesian products with factors of unbounded rank (see \cite[Theorem 1.4]{KaNi}).

\begin{theorem}
Let $\mathbf{H}=\prod_{i\in\NN}
S_{i}^{c(i)}$
where the $\{S_i\}$ are an infinite family of finite simple groups such that $\rk S_i\to\infty$.
If $\mathbf{H}$ is topologically finitely generated then it is a profinite completion.
\end{theorem}

Moreover they also explain why the condition $\rk S_i\to\infty$ is necessary.
In particular, the groups appearing in the proof of Theorem \ref{th:Lie_abscissa} are not profinite completions.
It is not clear however whether Theorem $\ref{th:Lie_abscissa}$ remains true if we require $\mathbf{H}$ to be a profinite completion.

\section{Representation Growth of Finite Groups of Lie Type}

We refer to section \ref{sec:FGLT} for definitons and concepts regarding finite groups of Lie type.
Let $G$ be a connected simply connected absolutely almost simple affine algebraic group with Steinberg endomorphism $F:G\to G$, then (except for a finite number of exceptions) the group $G^F$ is quasisimple and in particular $G^F/Z(G^F)$ is a simple group (see Theorem \ref{pre:th:FGLT_simple}).
Finite simple groups of this form are called finite simple groups of Lie type.
According to the classification of finite simple groups, all but a finite number of the nonabelian finite simple groups are either an alternating group or a finite simple group of Lie type.
Therefore the asymptotic behaviour of the representation growth of groups of the form $G^F$ may be used to derive asymptotic properties for the representation growth of finite simple groups of Lie type.
The representation growth of the quasisimple groups $G^F$ has been studied in \cite{LieSha} by Liebeck and Shalev using Deligne-Lusztig theory.
For a prime power $q$ and Lie type $L$, let $L(q)$ denote (if there exists one) a finite quasisimple group of type $L$ defined over $\Fq$.
\begin{theorem}[{\cite[Theorem 1.1]{LieSha}}] \label{th:repgrowq}
Fix a Lie type $L$ and let $h$ be the corresponding Coxeter number.
Then for any fixed real number $t>2/h$, we have 
\[\zeta_{L(q)}(t)\to 1\ \hbox{as} \ q\to\infty.\]
\end{theorem}

\begin{theorem}[{\cite[Theorem 1.2]{LieSha}}]\label{th:repgrowL}
Fix a real number $t>0$. Then there is an integer $r(t)$ such that for any $L(q)$ with $\rk L\geq r(t)$, we have
\[\zeta_{L(q)}(t)\to 1\ \mbox{as}\ |L(q)|\to\infty.\]
\end{theorem}
The corresponding result for alternating groups reads as follows.
\begin{theorem}[{\cite[Corollary 2.7]{LieSha2}}]\label{th:repgrwoAn}
Fix a real number $t>0$.
Let $A_n$ denote the alternating group on $n$ elements, then
\[\zeta_{A_n}(t)\rightarrow 1\ \mbox{as}\ n\to\infty.\] 
\end{theorem}



Note that for a finite simple group of Lie type $S=L(q)/Z(L(q))$ for some Lie type $L$ and primer power $q$, we have $\zeta_S(t)\leq\zeta_{L(q)}(t)$ for every $t>0$.
Therefore the above theorems provide information about the asymptotic behaviour of the representation zeta functions for all the infinite series of nonabelian finite simple groups.
As a corollary, one obtains a general result for the representation growth of any nonabelian finite simple group.

\begin{corollary}\label{cor:growthsimples}
Let $S$ denote a nonabelian finite simple group.
\begin{enumerate}[label=\roman*)]
\item For $t>1$, we have
\[\zeta_S(t)\to 1\ \mbox{as}\ |S|\to\infty.\]
\item There exists an absolute constant $c$ such that $r_n(S)<cn$ for every nonabelian finite simple group $S$.
\end{enumerate}
\end{corollary}

For the proof of Theorem \ref{th:Lie_abscissa} we will need a more precise study of the representation zeta functions
$\zeta_{L(q)}(s)$.
Let us first introduce some concepts which will be useful for comparing different zeta functions, cf. \cite[Definition 2.4]{AKOV_base_change}.
Let $f(s)=\sum_{n=1}^\infty a_nn^{-s}$ and $g(s)=\sum_{n=1}^\infty b_nn^{-s}$ be Dirichtlet generating series, i.e., Dirichlet series with integer coefficients $a_n, b_n\geq 0$.
Let $C\in\RR$. We write $f\lesssim_C g$ if $f$ and $g$ have the same abscissa of convergence $\alpha$ and $f(s)\leq C^{1+s}g(s)$ for every $s>\alpha$.
We write $f\sim_C g$ if $f\lesssim_C g$ and $g\lesssim_C f$.

In \cite{AKOV_base_change} Avni, Klopsch, Onn and Voll obtained a precise description of the asymptotic behaviour of the representation zeta functions of finite groups of Lie type.

\begin{theorem}[{\cite[Theorem 3.1 (3.1)]{AKOV_base_change}}] \label{th:AKOV_finite_Lie}
For every connected semi-simple algebraic group $G$ defined
over a finite field $\Fq$ with absolute root system $\Phi$
\[
\zeta_{G(\Fq)}(s) \sim_{C}  1 + q^{\rk\Phi-|\Phi^+|s}=1+q^{k(1-\frac{h}{2}s)}.
\]
\end{theorem}

The following lemma provides a way to study infinite products of representation zeta functions.

\begin{lemma}\label{lm:prod_zeta_fun}
Let $f,g$ be Dirichlet generating series and denote by $\alpha_f$  and $\alpha_g$ their corresponding abscissae of convergence.
Suppose that ${f=
\prod_{m=1}^\infty(1 +f_m)}$ and ${g=\prod_{
m=1}^\infty(1+g_m)}$, where $f_m,g_m$ are Dirichlet
generating series with vanishing constant terms.
Suppose further that there exists $C\in\RR_{>0}$
such that, for all $m$, $f_m\lesssim_{C}g_m$.
Then $\alpha_f\leq\alpha_g$.
\end{lemma}

Finally the following theorem will be used several times to obtain the necessary bounds to apply Theorem \ref{th:CharPRG}, see \cite{Lan}.

\begin{theorem}\label{th:min_deg_proj_rep}
There are absolute constants $d$, $e>1$ such that for any finite simple group of Lie type $S$ and any irreducible projective representation $\chi$ of $S$ we have
\[
\chi(1)>dq^{er}
\]
where $r=\rk L$, and $S$ is of type $L$ and defined over $\Fq$.
\end{theorem}
\section{Polynomial Representation Growth in $\mathcal{C}$}

We begin by giving a proof of Lemma \ref{prop:nec}.
\begin{proof}[{\bf Proof of Lemma \ref{prop:nec}}]
Suppose that $l(n)$ is not polynomially bounded.
Observe that for some (distinct) $i_1,\ldots,i_t\in I$ and  $\rho_{i_j},$ an $n_{i_j}$-dimensional irreducible representation of the group  $S_{i_j}$, the tensor product $\rho_{i_1}\otimes\ldots\otimes\rho_{i_t}$ naturally defines an irreducible $n$-dimensional representation of the group $\prod_{j=1}^{t}S_{i_j}$ (and in particular of $\mathbf{H}$), where $n=n_{i_1}\ldots n_{i_t}$.
Indeed, every irreducible representation of $\mathbf{H}$ has this form. 
In particular, it follows that
\[
R_{n^2}(\mathbf{H})\geq\displaystyle \frac {l(n)(l(n)-1)}{2}.
\]
Hence $R_n(\mathbf{H})$ cannot be polynomially bounded, as we assumed that $l(n)$ is not.
This implies that $r_n(\mathbf{H})$ is not polynomially bounded.
\end{proof}

\begin{proof}[{\bf Proof of Theorem \ref{th:CharPRG}}]
The direct implication follows from Lemma \ref{prop:nec}.

For the converse take $a$ as in Lemma \ref{lm:ndecom}.
Since $i(n)$ is polynomially bounded and $l_H(1)=1$ ( the $S_i$ being nonabelian and simple), we can take $b\in\NN$ such that $l_H(n)\leq n^b$.
Similarly, by Corollary \ref{cor:growthsimples} we can take $c\in\NN$ such that $r_n(S)\leq n^c$ for every nonabelian finite simple group $S$. 

Now let $\rho$ be an irreducible $n$-dimensional representation of $\mathbf{H}=\prod_{i\in I}S_i$.
Recall that $\rho\cong\rho_{i_1}\otimes\ldots\otimes\rho_{i_t}$ for some (distinct) $i_1,\ldots,i_t\in I$, where $\rho_{i_j}$ is a nontrivial $n_{i_j}$-dimensional irreducible representation of $S_{i_j}$ and $n_{i_1}\cdots n_{i_t}=n$.
Moreover, $n_{i_j}>1$ since the groups $S_{i_j}$ are simple.
By the previous lemma there are at most $n^a$ possible configurations of the form $(n_1,\ldots,n_t)$ such that $n_1\cdots n_t=n$.
Given such a configuration $(n_{i_1},\ldots,n_{i_t})$, the number of factors $S$ of $\bf{H}$ such that $S$ has a representation of dimension $n_{i_j}$ is at most $l_H(n_{i_j})\leq n_{i_j}^b$.
Hence the number of choices of factors $(i_1,\ldots,i_t)\in I^t$ for such a configuration is bounded by $\prod n_{i_j}^b=n^b$.
Now each group $S_{i_j}$ has at most $n_{i_j}^c$ irreducible representations of dimension $n_{i_j}$.
Therefore $r_n(\mathbf{H})\leq n^an^bn^c=n^{a+b+c}$ and $\mathbf{H}$ has polynomial representation growth.
\end{proof}

Given a finite group $G$ let us denote by $d(G)$ the minimal number of generators of $G$.
It is known that $d(S)=2$ for every nonabelian finite simple group.
The following result in \cite{WieIV} will be used for the proof Theorem \ref{th:PRG_iff_fg}.

\begin{theorem}\label{th:Wiegold_d(S)}
Let $S$ be a nonabelian finite simple group.
For every $e\in\NN$, $d\left(S^{|S|^{e}}\right)=e+2$.
\end{theorem}
\begin{proof}[{\bf Proof of Theorem \ref{th:PRG_iff_fg}}]
Suppose $\mathbf{H}$ has PRG.
For $S$ a nonabelian finite simple group, let $c(S)$ denote the number of copies of $S$ appearing as a composition factor of $\mathbf{H}$.
By Theorem \ref{th:CharPRG} we know that $l(n)\leq n^e$ for some $e\in\NN$.
Since for every finite group $S$, $S$ has a nontrivial irreducible representation of dimension less than $|S|$, it follows that $c(S)\leq l(|S|)\leq |S|^e$ for every $S$.
Note that we can write $\mathbf{H}=\prod_{S}S^{c(S)}$, where $S$ runs along the nonabelian finite simple groups.

Applying Theorem \ref{th:Wiegold_d(S)}, we obtain that, for every $S$, $S^{c(S)}$ is generated by $e$ elements.
Let us pick for every $S$ a tuple $(g_1^S,\ldots,g^S_e)\in S^{c(S)}$ such that $<g^S_1,\ldots,g^S_e>=S^{c(S)}$.
We claim that the $e$-tuple $(\mathbf{g_1}=\prod_S g_1^S,\ldots,\mathbf{g_e}=\prod_S g_e^S)$ generates $\mathbf{H}$ as a profinite group.
Put $\mathbf{G}=<\mathbf{g_1},\ldots,\mathbf{g_e}>$.
It follows by construction that $\mathbf{G}\times \prod_{Q\neq S}Q^{c(Q)} =\mathbf{H}$.
Thefore
\[
\mathbf{H}/\prod_{Q\ncong S}Q^{c(Q)}\cong S^{c(S)}
\]
 is in the top of $\mathbf{G}$ for every $S$.
But this clearly implies $\mathbf{G}=\mathbf{H}$.

\end{proof}

\section{Types of representation growth}
\begin{proof}[{\bf Example}]
Let us first look at an example to ilustrate the idea of the construction.
Let us consider the groups $\SL_2(\Fq)$.
The representation theory of these groups is well understood (see \cite[Chapter 15]{DigMi}) and we know that for $q=p^n$ ($p$ odd prime) their representation zeta functions are given by
\begin{align*}
\zeta_{\SL_2(\Fq)}=&1+q^{
-s}+\frac{q-3}{2}(q+1)^{-s+2}+2\left(\frac{q + 1}{2}\right)^{-s}
+\\
&\frac{q-1}{2}
(q-1)^{-s}+2\left(\frac{q-1}{2}\right)^{-s}.
\end{align*}
For convinience, we will encode the nontrivial part as
\[
f(q,s)=\zeta_{\SL_2(\Fq)}(s)-1
\]
Fix an odd prime $p$ and consider now the group $\mathbf{H}=\prod_{i\in\NN}\SL_2(\F_{p^i})$.
Note that all factors of the cartesian product have rank $1$.
Hence, applying Theorem \ref{th:min_deg_proj_rep} we readily obtain that if $\SL_2(\Fq)$ has a representation of dimension $n$, then $n>dq^e$.
It follows that $l_H(n)<\log_p (n/d)$, in particular, $l_H(n)$ is polynomially bounded and $\mathbf{H}$ has PRG.
Recalling once again that irreducible representations of $\mathbf{H}$ are given by finite tensor products of irreducible representations of the factors $\SL_2(\Fq)$, we obtain that the representation zeta function of $\mathbf{H}$ is given by
\[
\zeta_{\mathbf{H}}(s)=\prod_{i\in\NN}\zeta_{\SL_2(\F_{p^i})}(s)=\prod_{i\in\NN}(1+f(p^i,s)).
\]
It follows that $\zeta_{\mathbf{H}}(s)$ converges if and only if $\sum_{i\in\NN}f(p^i,s)$ does.
An analysis of the summands ocurrying in each $f(p^i,s)$ gives that the latter series converges if and only if the series $\sum_{q=p^i}{q}^{-s+1}=\sum_{i\in\NN}(p^{-s+1})^i$ does.
The latter expression, being a geometric series, converges if and only if $p^{-s+1}<1$, that is, if and only if $s>1$.
Therefore $\alpha(\mathbf{H})=1$.
\end{proof}

\begin{theorem}\label{th:arb_abscissa_quasisimple}

Given $c\in\RR_{>0}$, there exists a Lie type $L$, a prime number $p$ and  $\mathbf{H}=\prod_{i\in\NN} S_i$, where for every $i\in\NN$, $S_i=L(q)$ for some $p$-power $q$, such that $\mathbf{H}$ has PRG and $\alpha(\mathbf{H})=c$.
\end{theorem}
\begin{proof}

Given $c\in\RR_{>0}$, pick a Lie type $L$ with Coxeter number $h$ and $k=\rk\Phi$ such that $khc>2$.
Let us assume that $h$ is an even integer (this holds for all root systems except for those of type $A_{2k}$, moreover and easy adaptation of the proof also works for $h$ odd).
Construct a sequence of natural numbers $\{a_n\}\subseteq\NN$ such that $c-\frac{a_n}{n}<\frac{1}{n}$ for every $n\in\NN$.
Clearly $a_n\to c$ and there exists $N\in\NN$ such that $hka_n>2$ for every $n>N$.
Pick a prime number $p$.
Consider the function $f:\NN\to\NN$ given by
$i\mapsto {p}^{\frac{k(ha_i-2i)}{2}}$ if $i>N$, $i\mapsto 0$ otherwise.
Recall that $L(q)$ denotes a quasisimple group of type $L$ defined over $\Fq$.
Finally consider the group
\[
\mathbf{H}=\prod_{i\in\NN}L(p^i)^{f(i)}.
\]
We claim that $\alpha(\mathbf{H})=c$.
First observe that, by Theorem \ref{th:min_deg_proj_rep}, if $L(p^i)$ has an irreducible representation of dimension $n$, then $n>dp^{ier}$, that is, $i<\frac{\log (n/d)}{re\log p}\leq C\log n$ for some constant $C$.
On the other hand note that $f(i)=p^{\frac{ik(ha_i-2i)}{2i}}\leq p^{\frac{ikhc-2i}{2}}=p^{Di}$ for some constant $D$.
Thus
\[
l_H(n)<C\log n f(C\log n)\leq C\log n p^{Dlog n}\leq C\log n n^{D\log p},
\]
that is, $l_H(n)$ is polynomially bounded.
By Theorem \ref{th:CharPRG} $\mathbf{H}$ has PRG and we have
\[
\zeta_{\mathbf{H}}=\prod_{i\in\NN}\zeta_{L(p^i)}^{f(i)}.
\]
Now by Theorem \ref{th:AKOV_finite_Lie} we know that
\[
\zeta_{L(p^i)}\lesssim_C 1+p^{i(\rk\Phi-|\Phi^+|s)}=1+p^{ik\left(1-\frac{h}{2}s\right)}.
\]
Therefore an application of Lemma \ref{lm:prod_zeta_fun} gives
\[
\alpha(\mathbf{H})=\alpha\left(\prod_{i\in\NN}\left(1+p^{ik\left(1-\frac{h}{2}s\right)}\right)^{f(i)}\right).
\]
The latter product convergences if and only if the series
\[\sum_{i\in\NN}f(i)p^{ik\left(1-\frac{h}{2}s\right)}=\sum_{i>N}{p}^{\frac{k(ha_i-2i)}{2}}p^{ik\left(1-\frac{h}{2}s\right)}=
\sum_{i>N}p^{\frac{ikh}{2}\left(\frac{a_i}{i}-s\right)}
\]
converges.
We claim that the latter series converges if and only if $s>c$.
Suppose first that $s>c$, then $\frac{a_i}{i}-s<-\epsilon$ for some $\epsilon>0$. Thefore
\[
\sum_{i>N}p^{\frac{ikh}{2}\left(\frac{a_i}{i}-s\right)}\leq \sum_{i>N}p^{\frac{ikh}{2}\left(c-s\right)}. 
\]
Since the latter is a subseries of a geometric series of ratio $p^{\frac{kh}{2}\left(c-s\right)}<1$, it converges.
Suppose now that $s<c$. Then there exists $N<M\in\NN$ and $\epsilon >0$ such that $\frac{a_i}{i}-s>\epsilon$ for every $i>M$.
Hence
\[
\sum_{i>N}p^{\frac{ikh}{2}\left(\frac{a_i}{i}-s\right)}\geq\sum_{i>M}p^{\frac{ikh}{2}\left(\frac{a_i}{i}-s\right)}\geq
\sum_{i>M}p^{\frac{ikh}{2}\left(\epsilon\right)}.
\]
Since the latter series is a tale of a geometric series of ratio $p^{\frac{kh}{2}\left(\epsilon\right)}>1$, it does not converge.
\end{proof}

Theorem \ref{th:Lie_abscissa} is now a corollary of Theorem \ref{th:arb_abscissa_quasisimple}.
\begin{proof}[{\bf Proof of Theorem \ref{th:Lie_abscissa}}]
Let us pick $L=(A_p,\id)$ such that $\rk\Phi=p$ is a prime number and $pc>2$.
Now by (the proof of) Theorem \ref{th:arb_abscissa_quasisimple} there exists $H=\prod_{i\in\NN}S_i$ such that for every $i\in\NN$, $S_i=L(p^e)$ for some $e\in\NN$ and $\alpha(\mathbf{H})=c$. 
Note that in this case we have $L(p^e)=\SL_p(\mathbb{F}_{p^e})$ and $Z(\SL_p(\mathbb{F}_{p^e}))=1$, since there are no nontrivial $p^{\mbox{{\tiny th}}}$-roots of unity in a field of characteristic $p$.
Hence the groups $L(p^e)$ are indeed simple and the theorem follows.
Note that one can similarly find an appropriate Lie type within the other infinite families ($B_i$, $C_i$, $D_i$, $i\in\NN)$ and mimic the above construction for that type. 
\end{proof}

\clearpage{\pagestyle{empty}\cleardoublepage}







\backmatter

\newpage
\pagestyle{empty}
\textcolor{white}{Se acabó}
\vspace*{14cm}
\cleardoublepage






\end{document}